\documentclass[a4paper]{article}

\usepackage[english]{babel}
\usepackage[utf8]{inputenc}
\usepackage[T1]{fontenc}
\usepackage{mathtools}
\usepackage{amsthm}
\usepackage{amsfonts}
\usepackage{amssymb}
\usepackage[margin=3cm]{geometry}
\usepackage[pdftex,pdfborder={0 0 0},linktoc=all]{hyperref}
\usepackage{subfig}
\usepackage[figurename=Figure]{caption}
\usepackage{tikz}

\usetikzlibrary{graphs}
\usetikzlibrary{arrows}

\theoremstyle{plain}
	\newtheorem{thm}{Theorem}[section]
	\newtheorem{cor}[thm]{Corollary}
	\newtheorem{lem}[thm]{Lemma}
	\newtheorem{prop}[thm]{Proposition}

\theoremstyle{definition}
	\newtheorem{dfn}[thm]{Definition}
	\newtheorem{ntn}[thm]{Notation}
	\newtheorem{dfns}[thm]{Definitions}
	\newtheorem{ntns}[thm]{Notations}

\theoremstyle{remark}
	\newtheorem{rem}[thm]{Remark}
	\newtheorem{rems}[thm]{Remarks}
	\newtheorem{ex}[thm]{Example}
	\newtheorem{exs}[thm]{Examples}

\numberwithin{equation}{section}

\newcommand{\C}{\mathbb{C}}
\newcommand{\I}{\mathcal{I}}
\newcommand{\J}{\mathcal{J}}
\newcommand{\K}{\mathcal{K}}
\newcommand{\N}{\mathbb{N}}
\newcommand{\Q}{\mathbb{Q}}
\newcommand{\R}{\mathbb{R}}

\newcommand{\dx}{\dmesure\!}
\newcommand{\deron}[2]{\frac{\partial #1}{\partial #2}}
\newcommand{\esp}[2][]{\mathbb{E}_{#1}\!\left[ #2 \right]}
\newcommand{\espcond}[3][]{\mathbb{E}_{#1}\!\left[ #2\hspace{-1mm} \mvert \! #3 \right]}
\newcommand{\mvert}{\mathrel{}\middle|\mathrel{}}
\newcommand{\norm}[1]{\left\lvert #1 \right\rvert}
\newcommand{\Norm}[1]{\left\lVert #1 \right\rVert}
\newcommand{\pa}{\mathcal{P}\!}
\newcommand{\pp}{\mathcal{P\!P}\!}
\newcommand{\prsc}[2]{\left\langle #1\,, #2 \right\rangle}
\newcommand{\trans}[1]{\prescript{\text{t}}{}{#1}}
\newcommand{\var}[1]{\Var\!\left( #1 \right)}

\renewcommand{\P}{\mathbb{P}}

\renewcommand{\bar}{\overline}
\renewcommand{\epsilon}{\varepsilon}
\renewcommand{\geq}{\geqslant}
\renewcommand{\leq}{\leqslant}
\renewcommand{\tilde}{\widetilde}

\DeclareMathOperator{\card}{Card}
\DeclareMathOperator{\dmesure}{d}
\DeclareMathOperator{\ev}{ev}
\DeclareMathOperator{\Id}{Id}
\DeclareMathOperator{\sinc}{sinc}
\DeclareMathOperator{\sym}{Sym}
\DeclareMathOperator{\Var}{Var}

\author{Michele Ancona\,\thanks{Michele Ancona, Tel Aviv University, School of Mathematical Sciences; e-mail: \url{michi.ancona@gmail.com}. Michele Ancona is supported by the Israeli Science Foundation through the ISF Grants 382/15 and 501/18.} \and Thomas Letendre\,\thanks{Thomas Letendre, Université Paris-Saclay, CNRS, Laboratoire de Mathématiques d’Orsay, 91405 Orsay, France; e-mail: \url{letendre@math.cnrs.fr}. Thomas Letendre is supported by the French National Research Agency through the ANR grants UniRaNDom (ANR-17-CE40-0008) and SpInQS (ANR-17-CE40-0011).}}
\date{\today}
\title{Zeros of smooth stationary Gaussian processes}

\begin{document}

\maketitle

\begin{abstract}
Let $f:\R \to \R$ be a stationary centered Gaussian process. For any $R>0$, let $\nu_R$ denote the counting measure of $\{x \in \R \mid f(Rx)=0\}$. In this paper, we study the large $R$ asymptotic distribution of $\nu_R$. Under suitable assumptions on the regularity of $f$ and the decay of its correlation function at infinity, we derive the asymptotics as $R \to +\infty$ of the central moments of the linear statistics of $\nu_R$. In particular, we derive an asymptotics of order $R^\frac{p}{2}$ for the $p$-th central moment of the number of zeros of $f$ in $[0,R]$. As an application, we prove a functional Law of Large Numbers and a functional Central Limit Theorem for the random measures~$\nu_R$. More precisely, after a proper rescaling, $\nu_R$ converges almost surely towards the Lebesgue measure in weak-$*$ sense. Moreover, the fluctuation of $\nu_R$ around its mean converges in distribution towards the standard Gaussian White Noise. The proof of our moments estimates relies on a careful study of the $k$-point function of the zero point process of~$f$, for any $k \geq 2$. Our analysis yields two results of independent interest. First, we derive an equivalent of this $k$-point function near any point of the large diagonal in~$\R^k$, thus quantifying the short-range repulsion between zeros of $f$. Second, we prove a clustering property which quantifies the long-range decorrelation between zeros of $f$.
\end{abstract}




\tableofcontents


\section{Introduction}
\label{sec introduction}

Let $Z$ denote the zero set of a smooth centered stationary Gaussian process $f$ on $\R$. A classical problem in probability is to understand the number of zeros of $f$ in a growing interval, that is the asymptotics of $\card(Z \cap [0,R])$ as $R \to +\infty$. This problem has a long history, starting with the articles of Kac~\cite{Kac1943} and Rice~\cite{Ric1944} who computed the mean number of zeros of $f$ in an interval. We refer to Section~\ref{subsec related works} below for further discussion of related works.

In this paper, we compute the large $R$ asymptotics of the central moments of any order of $\card(Z \cap [0,R])$, under suitable conditions on $f$. The starting point of our analysis is the Kac--Rice formula, which allows to write the $k$-th factorial moment of $\card(Z \cap [0,R])$ as the integral over $[0,R]^k$ of the $k$-point function of the random point process $Z$. Most of the paper is devoted to the study of this $k$-point function $\rho_k$, that we believe to be of independent interest. A priori, $\rho_k$ is only well-defined on $\R^k \setminus \Delta_k$, where $\Delta_k$ denotes the large diagonal in $\R^k$. We prove that $\rho_k$ vanishes along $\Delta_k$, which is the sign of a repulsion between the zeros of $f$. In fact, we characterize this repulsion by deriving an equivalent of $\rho_k(x)$ as $x \to y$, for any $y \in \Delta_k$. We also prove that $\rho_k$ satisfies a clustering property if the correlation function of the process $f$ decays fast enough. This clustering property can be interpreted as a clue that zeros of $f$ in two disjoint intervals that are far from one another are quasi-independent. Our main tool in the study of $\rho_k$ and its singularities are the divided differences associated with $f$. We believe that the methods we develop below regarding these divided differences can have applications beyond the scope of this paper.


\subsection{Linear statistics associated with the zeros of a Gaussian process}
\label{subsec linear statistics associated with the zeros of a Gaussian process}

Let us introduce quickly the object of our study. More details are given in Section~\ref{sec framework}. Let $f: \R \to \R$ be a stationary centered Gaussian process of class $\mathcal{C}^1$. Let $\kappa:x \mapsto \esp{f(0)f(x)}$ denote the \emph{correlation function} of $f$. We assume that $f$ is \emph{normalized} so that $\kappa(0)=1=-\kappa''(0)$ (see Section~\ref{subsec stationary Gaussian processes and correlation functions}). The zero set $Z= f^{-1}(0)$ is then almost surely a closed discrete subset of $\R$ (see Lemma~\ref{lem Z as closed discrete}). 

We denote by $\nu = \sum_{x \in Z} \delta_x$ the \emph{counting measure} of $Z$, where $\delta_x$ is the unit Dirac mass at $x$. Let $\phi:\R \to \R$, we denote by $\prsc{\nu}{\phi} = \sum_{x \in Z} \phi(x)$ whenever this makes sense. Besides, for any $R>0$, we denote by $\phi_R:x \mapsto \phi(\frac{x}{R})$. Finally, for any $A \subset \R$, we denote by $\mathbf{1}_A$ the indicator function of $A$. Then, for any $R>0$, we have:
\begin{equation*}
\card(Z \cap [0,R]) = \prsc{\nu}{\mathbf{1}_{[0,R]}} = \prsc{\nu}{(\mathbf{1}_{[0,1]})_R}.
\end{equation*}
More generally, we can consider the asymptotics of $\prsc{\nu}{\phi_R}$ as $R \to +\infty$, where $\phi:\R \to \R$ is a nice enough test-function. It turns out that the dual point of view is more relevant, and this is the one we adopt in this paper: instead of integrating $\phi_R$ over $Z$, we consider the integral of a fixed test-function $\phi$ over homothetical copies of $Z$. Let $R >0$, we denote by $Z_R = \{ x \in \R \mid f(Rx)=0\}$ and by $\nu_R = \sum_{x \in Z_R} \delta_x$ its counting measure. Then, for all $\phi:\R \to \R$, we have $\prsc{\nu}{\phi_R} = \prsc{\nu_R}{\phi}$. In particular, $\card(Z \cap [0,R]) = \prsc{\nu_R}{\mathbf{1}_{[0,1]}}$. Quantities of the form $\prsc{\nu_R}{\phi}$ are called the \emph{linear statistics} of $\nu_R$. In the following, we study the large $R$ asymptotic distribution of the random measure $\nu_R$, mostly through the central moments of its linear statistics.


\subsection{Moments asymptotics}
\label{subsec moments asymptotics}

Our first theorem describes the large $R$ asymptotics of the central moments of the linear statistics $\prsc{\nu_R}{\phi}$ of the random measure $\nu_R$. To the best of our knowledge, this is the first result of this kind for Gaussian processes on $\R$, even in the simplest case of $\prsc{\nu_R}{\mathbf{1}_{[0,1]}}=\card\left(Z \cap [0,R]\right)$. We will consider the following quantities, that are slightly more general.

\begin{dfn}[Central moments]
\label{def mp nu R}
Let $p \geq 2$ be an integer and let $R >0$. For any test-functions $\phi_1,\dots,\phi_p$, we denote by
\begin{equation*}
m_p(\nu_R)(\phi_1,\dots,\phi_p) = \esp{\prod_{i=1}^p \left(\prsc{\nu_R}{\phi_i}-\rule{0em}{2.5ex}\esp{\prsc{\nu_R}{\phi_i}}\right)},
\end{equation*}
whenever the right-hand side makes sense. In particular, $m_p(\nu_R)(\phi,\dots,\phi)$ equals $m_p(\prsc{\nu_R}{\phi})$, the $p$-th central moment of $\prsc{\nu_R}{\phi}$, when this quantity is well-defined. When $p=2$, we use the standard notation $\var{\prsc{\nu_R}{\phi}} = m_2(\prsc{\nu_R}{\phi})$ for the variance of $\prsc{\nu_R}{\phi}$.
\end{dfn}

In this paper, we use the following class of test-functions. Note that this class contains, among others, the indicator functions of bounded intervals and the continuous functions decreasing fast enough at infinity.

\begin{dfn}[Test-function]
\label{def test-function}
We say that a measurable function $\phi:\R \to \R$ is a \emph{test-function} if $\phi$ is integrable, essentially bounded and continuous almost everywhere with respect to the Lebesgue measure.
\end{dfn}

In order to say something about central moments, we need to make some assumptions on the random process $f$. These assumptions are further discussed in Section~\ref{subsec stationary Gaussian processes and correlation functions}, and in Appendix~\ref{sec examples of smooth non-degenerate processes} where we build examples of processes satisfying these conditions. For now, let us just give one concrete example. The hypotheses of all the theorems in the present paper are satisfied by the so-called \emph{Bargmann--Fock process}. This process is the centered stationary Gaussian process on $\R$ whose correlation function is $x \mapsto e^{-\frac{1}{2}x^2}$. See Appendix~\ref{sec examples of smooth non-degenerate processes} for more details, especially Examples~\ref{ex non-degenerate and fast-decreasing}.

\begin{ntn}[$\mathcal{C}^k$-norms]
\label{ntn norm kappa}
Let $k \in \N$ and let $g:\R \to \R$ be a $\mathcal{C}^k$-function such that $g$ and all its derivatives of order at most $k$ are bounded on $\R$. For any $\eta \geq 0$, we denote by:
\begin{equation*}
\Norm{g}_{k,\eta} = \sup \left\{\norm{g^{(l)}(x)} \mvert 0 \leq l \leq k, \norm{x} \geq \eta \right\},
\end{equation*}
where $g^{(l)}$ stands for the $l$-th derivative of $g$. If $\eta=0$, we simply denote $\Norm{g}_{k,0}$ by $\Norm{g}_k$.
\end{ntn}

If $f$ is a $\mathcal{C}^p$-process, then its correlation function satisfies $\kappa \in \mathcal{C}^{2p}(\R)$. Moreover, by Cauchy-Schwarz's Inequality, $\kappa^{(k)}$ is bounded for all $k \in \{0,\dots,2p\}$. Hence $\Norm{\kappa}_{k,\eta}$ is well-defined for any $k \in \{0,\dots,2p\}$ and $\eta \geq 0$. We refer to Section~\ref{subsec stationary Gaussian processes and correlation functions} for more details.

\begin{dfn}[Partitions]
\label{def partitions}
Let $A$ be a non-empty finite set, a \emph{partition} of the set $A$ is a family $\I = \{I_1,\dots,I_m\}$ of non-empty disjoint subsets of $A$ such that $\bigsqcup_{i=1}^m I_i =A$. We denote by $\pa_A$ (resp.~$\pa_k$) the set of partitions of~$A$ (resp.~$\{1,\dots,k\}$). A \emph{partition into pairs} of $A$ is a partition $\I \in \pa_A$ such that $\card(I)=2$ for all $I \in \I$. We denote by $\pp_A$, (resp.~$\pp_k$) the set of partitions into pairs of $A$ (resp.~$\{1,\dots,k\}$). We also use the convention that $\pa_\emptyset = \{ \emptyset \} = \pp_\emptyset$.
\end{dfn}

\begin{ntn}[Gaussian moments]
\label{ntn mu p}
For all $p \in \N$, we denote by $\mu_p$ the $p$-th moment of a centered Gaussian variable of variance $1$. Recall that we have $\mu_p = \card(\pp_p)$, that is $\mu_p = 2^{-\frac{p}{2}}p! \left(\frac{p}{2}!\right)^{-1}$ if $p$ is even, and $\mu_p=0$ if $p$ is odd.
\end{ntn}

\begin{thm}[Central moments asymptotics]
\label{thm moments}
Let $p \geq 2$ be an integer. Let $f$ be a normalized stationary centered Gaussian $\mathcal{C}^p$-process and let $\kappa$ denote its correlation function. We assume that, $\Norm{\kappa}_{p,\eta} = o(\eta^{-4p})$ as $\eta \to +\infty$. For all $R>0$, we denote by $\nu_R$ the counting measure of $\{x \in \R \mid f(Rx)=0\}$.

Let $\phi_1,\dots, \phi_p$ be test-functions in the sense of Definition~\ref{def test-function}. Then, as $R \to +\infty$, we have:
\begin{equation*}
m_p(\nu_R)(\phi_1,\dots,\phi_p) = \sum_{\left\{ \{a_i,b_i\} \mvert 1 \leq i \leq \frac{p}{2}\right\} \in \pp_p} \prod_{i=1}^\frac{p}{2} m_2(\nu_R)(\phi_{a_i},\phi_{b_i}) + o(R^\frac{p}{2}).
\end{equation*}
In particular, if $\phi$ is a test-function then, as $R \to +\infty$, we have:
\begin{equation*}
m_p\left(\prsc{\nu_R}{\phi}\right) = \mu_p \var{\prsc{\nu_R}{\phi}}^\frac{p}{2} + o(R^\frac{p}{2}).
\end{equation*}
\end{thm}

\begin{rem}
\label{rem moments}
If $p$ is odd, then $\pp_p=\emptyset$ and $\mu_p=0$. In this case, only the term $o(R^\frac{p}{2})$ remains on the right-hand side of the asymptotics in Theorem~\ref{thm moments}. On the other hand, if $p$ is even, then $\pp_p\neq \emptyset$ and $\mu_p>0$. In this latter case, the leading terms in the asymptotics of Theorem~\ref{thm moments} are of order $R^\frac{p}{2}$ for generic test-functions, see Proposition~\ref{prop variance} below. Note however that we obtain $o(R^\frac{p}{2})$ if, for example, the $(\phi_i)_{1\leq i \leq p}$ are pairwise orthogonal for the $L^2$ inner product defined by the Lebesgue measure.
\end{rem}

In order to interpret this result, let us describe the expectation and the covariance structure of the linear statistics of $\nu_R$. First, we describe the expectation of $\nu_R$ for any fixed $R>0$. Note that Proposition~\ref{prop expectation} below is a natural extension of the results of Kac~\cite{Kac1943} and Rice~\cite{Ric1944}, who computed the expectation of $\card\left(Z \cap [0,R]\right)$. Recall that a \emph{Radon measure} is a continuous linear form on $\left(\mathcal{C}^0_c(\R),\Norm{\cdot}_\infty\right)$, the space of compactly supported continuous functions equipped with the sup-norm.

\begin{prop}[Expectation of the linear statistics]
\label{prop expectation}
Let $f$ be a normalized stationary centered Gaussian $\mathcal{C}^1$-process. Let $R>0$ and let $\nu_R$ denote the counting measure of $\{x \in \R \mid f(Rx)=0\}$. Let $\phi:\R \to \R$ be a Borel-measurable function. If $\phi$ is non-negative or Lebesgue-integrable then,
\begin{equation*}
\esp{\prsc{\nu_R}{\phi}} = \frac{R}{\pi} \int_{-\infty}^{+\infty} \phi(x) \dx x.
\end{equation*}
In particular, as Radon measures $\esp{\nu_R} = \frac{R}{\pi} \dx x$, where $\dx x$ stands for the Lebesgue measure of $\R$.
\end{prop}

\begin{rem}
\label{rem as well defined}
If $\phi: \R \to \R$ is Lebesgue-integrable we can apply Proposition~\ref{prop expectation} to $\norm{\phi}$. This proves that, for all $R>0$, almost surely $\prsc{\nu_R}{\norm{\phi}} <+\infty$. Thus, the random variable $\prsc{\nu_R}{\phi}$ is almost surely well-defined. Moreover $\norm{\prsc{\nu_R}{\phi}} \leq \prsc{\nu_R}{\norm{\phi}}$, so that $\prsc{\nu_R}{\phi}$ is an integrable random variable, and it makes sense to consider its expectation.
\end{rem}

\begin{rem}
\label{rem integrable not L1}
In this paper, we consider quantities of the form $\prsc{\nu}{\phi} = \sum_{x \in Z} \phi(x)$, where $Z$ is discrete. If $\phi$ is only defined up to modifications on a negligible set these quantities are ill-defined. However, let $\nu_R$ be is as in Proposition~\ref{prop expectation} and let $\phi_1$ and $\phi_2$ be test-functions such that $\phi_1= \phi_2$ almost everywhere. Then, we have $\prsc{\nu_R}{\phi_1}=\prsc{\nu_R}{\phi_2}$ almost surely. Indeed, by Proposition~\ref{prop expectation} we have $\esp{\prsc{\nu_R}{\norm{\phi_1-\phi_2}}}=0$, so that $\prsc{\nu_R}{\norm{\phi_1-\phi_2}}=0$ almost surely. Then, the result follows from $\norm{\prsc{\nu_R}{\phi_1}-\prsc{\nu_R}{\phi_2}} \leq \prsc{\nu_R}{\norm{\phi_1-\phi_2}}$.
\end{rem}

The following result gives the large $R$ asymptotics of the covariance of $\prsc{\nu_R}{\phi_1}$ and $\prsc{\nu_R}{\phi_2}$, where $\phi_1$ and $\phi_2$ are test-functions (see Definition~\ref{def test-function}). To the best of our knowledge, this result was only known for $\phi_1 = \phi_2 = \mathbf{1}_{[0,1]}$ until now, see~\cite[Theorem~1]{Cuz1976}. Note that the positivity of the leading constant $\sigma$ in not proved in~\cite{Cuz1976}.

\begin{prop}[Covariances asymptotics]
\label{prop variance}
Let $f$ be a normalized stationary centered Gaussian $\mathcal{C}^2$-process and let $\kappa$ denote its correlation function. We assume that $\kappa$ and $\kappa''$ are square-integrable and that $\Norm{\kappa}_{2,\eta}$ tends to $0$ as $\eta \to +\infty$.

Then there exists $\sigma >0$ such that, for any test-functions $\phi_1$ and $\phi_2$ we have:
\begin{equation}
\label{eq asymp m2}
m_2(\nu_R)(\phi_1,\phi_2) = R \sigma^2 \int_{-\infty}^{+\infty} \phi_1(x)\phi_2(x) \dx x + o(R)
\end{equation}
as $R \to +\infty$. Moreover, we have:
\begin{equation}
\label{eq def sigma}
\sigma^2 =\frac{1}{\pi} + \frac{2}{\pi^2} \int_0^{+\infty}\left(\frac{1-\kappa(t)^2-\kappa'(t)^2}{\left(1-\kappa(t)^2\right)^\frac{3}{2}}\left(\sqrt{1-a(t)^2} + a(t) \arcsin(a(t))\right) - 1\right)\dx t,
\end{equation}
where $a:(0,+\infty) \to [-1,1]$ is the map defined by: $\forall t >0$, $a(t) = \dfrac{\kappa(t)\kappa'(t)^2 - \kappa(t)^2\kappa''(t) + \kappa''(t)}{1 - \kappa(t)^2 -\kappa'(t)^2}$.
\end{prop}

\begin{rem}
\label{rem explicit lower bound}
The fact that $\sigma >0$ is non-trivial. It is proved in Section~\ref{subsec positivity of the leading constant}, using the Wiener--Itô expansion of $\card\left(Z\cap [0,R]\right)$ derived in~\cite{KL1997}. In Corollary~\ref{cor sigma positive}, we obtain the following explicit lower bound:
\begin{equation*}
\sigma^2 \geq \frac{1}{\pi^2} \int_0^{+\infty}(\kappa(z) + \kappa''(z))^2 \dx z >0.
\end{equation*}
If we consider the Bargmann--Fock process $f_{BF}$, that is if $\kappa:x \mapsto e^{-\frac{1}{2}x^2}$, the previous lower bound gives $\sigma_{BF}^2 \geq (2\pi^3)^{-\frac{1}{2}} \simeq 0.12\dots$. In~\cite[Proposition~3.1 and Remark~1]{Dal2015}, Dalmao computed $\sigma_{BF}^2$ up to a factor $\frac{1}{\pi}$. Using his result, we get $\sigma_{BF}^2 \simeq 0,18\dots$. Note that these values are smaller than~$\frac{1}{\pi}$, hence the integral on the right-hand side of Equation~\eqref{eq def sigma} is negative.
\end{rem}


\subsection{Clustering for the \texorpdfstring{$k$}{}-point functions}
\label{subsec clustering for the k-point functions}

Let $p \geq 2$ be an integer and let $f$ be as above a normalized centered stationary Gaussian process. The first step in the proof of our moments asymptotics (Theorem~\ref{thm moments}) is to derive a tractable integral expression of the central moments $m_p(\nu_R)(\phi_1,\dots,\phi_p)$ that we want to estimate. Using the Kac--Rice formula (see Proposition~\ref{prop Kac-Rice formula}), we write $m_p(\nu_R)(\phi_1,\dots,\phi_p)$ as a linear combination of terms of the form
\begin{equation}
\label{eq intro clustering}
\int_{\R^k} \Phi_R(x) \rho_k(x) \dx x,
\end{equation}
where $1 \leq k \leq p$ and $\Phi:\R^k \to \R$ is an integrable function built from the $(\phi_i)_{1 \leq i \leq p}$. In this equation, the function $\rho_k$ is the Kac--Rice density of order $k$ (cf.~Definition~\ref{def Kac-Rice densities}). In order to give some meaning to this density, notice that it coincides with the $k$-point function of the random point process $Z= f^{-1}(0)$, see Lemma~\ref{lem k point function}. By this we mean that, for any $x=(x_i)_{1 \leq i \leq k} \in \R^k$ such that $\rho_k(x)$ is well-defined, we have:
\begin{equation*}
\frac{1}{(2\epsilon)^k} \esp{\prod_{i=1}^k \card\left(Z \cap [x_i-\epsilon,x_i+\epsilon]\right)} \xrightarrow[\epsilon \to 0]{} \rho_k(x).
\end{equation*}

The core of the proof of Theorem~\ref{thm moments} is to understand the large $R$ asymptotics of integrals of the form~\eqref{eq intro clustering}. This leads to a detailed study of the Kac--Rice densities $(\rho_k)_{k \in \N^*}$. Given $k \in \N^*$, Definition~\ref{def Kac-Rice densities} allows to define $\rho_k(x)$ for any $x=(x_i)_{1 \leq i \leq k} \in \R^k$ such that the Gaussian vector $(f(x_i))_{1 \leq i \leq k}$ is non-degenerate. In particular, if the correlation function $\kappa$ of $f$ tends to $0$ at infinity, as in Theorem~\ref{thm moments} and Proposition~\ref{prop variance}, the ergodicity of $f$ implies that $\rho_k$ is well-defined on $\R^k \setminus \Delta_k$, where $\Delta_k = \left\{(x_1,\dots,x_k) \in \R^k \mvert \exists i,j\in \{1,\dots,k\}, i \neq j\ \text{and}\ x_i = x_j \right\}$ denotes the large diagonal in $\R^k$ (cf.~Lemma~\ref{lem non-degeneracy} for more details). In general, $\rho_k$ is a continuous symmetric function defined on some symmetric open subset of $\R^k \setminus \Delta_k$, see Lemmas~\ref{lem Kac-Rice densities symmetric} and~\ref{lem Dk and Nk continuous}.

Interpreting $\rho_k$ as the $k$-point function of $Z$, some of the intermediate results in the proof of Theorem~\ref{thm moments} appear to be of independent interest. Theorems~\ref{thm vanishing order} and~\ref{thm clustering} below are analogous to the main results of~\cite{NS2012}, where Nazarov and Sodin studied the $k$-point function of a Gaussian Entire Function. Note however that our methods are completely different. In particular, we do not require any form of analyticity.

\begin{thm}[Vanishing order of the $k$-point function]
\label{thm vanishing order}
Let $k \in \N^*$, let $f$ be a normalized stationary centered Gaussian $\mathcal{C}^k$-process. Let $y=(y_i)_{1\leq i \leq k} \in \R^k$ and let $\I \in \pa_k$ be the partition defined by:
\begin{equation*}
\forall i,j \in \{1,\dots,k\}, \qquad y_i=y_j \iff \exists I \in \I, \{i,j\} \subset I.
\end{equation*}
For any $I \in \I$, we denote by $\norm{I}$ the cardinality of $I$ and by $y_I \in \R$ the common value of the $(y_i)_{i \in I}$. Let us assume that the Gaussian vector $\left(f^{(i)}(y_I)\right)_{I \in \I, 0 \leq i < \norm{I}}$ is non-degenerate and denote by
\begin{equation}
\label{eq def ly}
\ell(y) = \left(\prod_{I \in \I} \prod_{i=0}^{\norm{I}-1}\frac{i!}{\norm{I}!}\right) \frac{\espcond{\prod_{I \in \I} \norm{f^{(\norm{I})}(y_I)}^{\norm{I}}}{\forall I \in \I, \forall i \in \{0,\dots,\norm{I}-1\}, f^{(i)}(y_I)=0}}{(2\pi)^\frac{k}{2} \det\left(\var{\left(f^{(i)}(y_I)\right)_{I \in \I, 0 \leq i < \norm{I}}}\right)^\frac{1}{2}},
\end{equation}
where $\espcond{\prod_{I \in \I} \norm{f^{(\norm{I})}(y_I)}^{\norm{I}}}{\forall I \in \I, \forall i \in \{0,\dots,\norm{I}-1\}, f^{(i)}(y_I)=0}$ stands for the conditional expectation of $\prod_{I \in \I} \norm{f^{(\norm{I})}(y_I)}^{\norm{I}}$ given that $f^{(i)}(y_I)=0$ for all $I \in \I$ and $i \in \{0,\dots,\norm{I}-1\}$.

Then, there exists a neighborhood $U$ of $y$ in $\R^k$ such that the $k$-point function $\rho_k$ of $f^{-1}(0)$ is well-defined on $U \setminus \Delta_k$ and, as $x \to y$ with $x=(x_i)_{1 \leq i \leq k} \in U \setminus \Delta_k$, we have:
\begin{equation*}
\left(\prod_{I \in \I} \prod_{\{(i,j) \in I^2 \mid i < j\}} \frac{1}{\norm{x_i-x_j}} \right) \rho_k(x) \xrightarrow[x \to y]{} \ell(y).
\end{equation*}
Moreover, if $\left(f^{(i)}(y_I)\right)_{I \in \I, 0 \leq i \leq \norm{I}}$ is non-degenerate, then $\ell(y) >0$.
\end{thm}

If the process $f$ is of class $\mathcal{C}^k$ and such that $\kappa(x) \xrightarrow[x \to +\infty]{}0$, then the non-degeneracy conditions in Theorem~\ref{thm vanishing order} are satisfied for all $y \in \R^k$ (cf.~Lemma~\ref{lem non-degeneracy} below). In this case, $\rho_k$ is well-defined on $\R^k \setminus \Delta_k$, and $\ell(y)$ is positive for any $y \in \R^k$. If $y \in \R^k \setminus \Delta_k$, the partition associated with $y$ is $\I = \{\{i\} \mid 1 \leq i \leq k\} \in \pa_k$. Then, $\ell(y) = \rho_k(y)$ (see Equation~\eqref{eq def ly} and Definition~\ref{def Kac-Rice densities}) and Theorem~\ref{thm vanishing order} only states that $\rho_k$ is continuous at $y$ and that $\rho_k(y) >0$. If $y \in \Delta_k$, Theorem~\ref{thm vanishing order} shows that $\rho_k(x) \xrightarrow[x \to y]{}0$. In particular, under the assumption that $\kappa$ tends to $0$ at infinity, the $k$-point function of $Z$ can be uniquely extended into a continuous function on $\R^k$ that vanishes exactly on $\Delta_k$. In this case, the last part of the theorem gives the vanishing order of $\rho_k$ near any point of the diagonal. The fact that $\rho_k$ vanishes along $\Delta_k$ is interpreted as the sign of a short-range repulsion between zeros of $f$. The estimates of Theorem~\ref{thm vanishing order} quantify this phenomenon.

Let us now consider the long-range correlations between zeros of $f$. We still concern ourselves with the case where $\kappa$ tends to $0$ at infinity. Let $A$ and $B$ be two non-empty disjoint intervals of~$\R$. If $A$ and $B$ are far enough from one another, the values of $f$ on $A$ are essentially uncorrelated with those of $f$ on $B$. It is then reasonable to expect the point processes $Z \cap A$ and $Z \cap B$ to be roughly independent. The independence of $Z \cap A$ and $Z \cap B$ would imply that $\rho_{k+l}(x,y) = \rho_k(x)\rho_l(y)$ for any $k,l \in \N^*$, any $x \in A^k \setminus \Delta_k$ and any $y \in B^l \setminus \Delta_l$. The following result shows that a relation of this form holds, up to an error term.

\begin{thm}[Clustering for $k$-point functions]
\label{thm clustering}
Let $k \in \N^*$, let $f$ be a normalized stationary centered Gaussian $\mathcal{C}^k$-process whose correlation function $\kappa$ satisfies $\Norm{\kappa}_{k,\eta} \xrightarrow[\eta \to +\infty]{}0$. For any $l \in \{1,\dots,k\}$, let $\rho_l$ denote the $l$-point function of $f^{-1}(0)$.

Then, there exists $C >0$ such that for all $x=(x_i)_{1 \leq i \leq k} \in \R^k \setminus \Delta_k$ we have:
\begin{equation*}
0 \leq \rho_k(x) \leq C \prod_{1 \leq i < j \leq k} \min\left(\norm{x_i-x_j},1\right).
\end{equation*}
Moreover, for all $\eta \geq 1$, for all $\I \in \pa_k$, for all $x = (x_i)_{1 \leq i \leq k} \in \R^k \setminus \Delta_k$ satisfying:
\begin{equation*}
\forall I, J \in \I \ \text{such that} \ I \neq J, \ \forall i \in I, \ \forall j \in J, \ \norm{x_i -x _j} > \eta,
\end{equation*}
we have:
\begin{equation*}
\prod_{I \in \I} \rho_{\norm{I}}(\underline{x}_I) = \rho_k(x) \left(1 +O\!\left(\left(\Norm{\kappa}_{k,\eta}\right)^\frac{1}{2}\right)\right),
\end{equation*}
where the constant involved in the error term $O\!\left(\left(\Norm{\kappa}_{k,\eta}\right)^\frac{1}{2}\right)$ does not depend on $\eta$, $\I$ nor $x$. Here, we denoted by $\norm{I}$ the cardinality of $I$ and by $\underline{x}_I = (x_i)_{i \in I}$, for all $I \in \I$.
\end{thm}

An important new idea in the proof of Theorems \ref{thm vanishing order} and \ref{thm clustering} is that we derive a whole family of new expressions for the $k$-point functions $\rho_k$, indexed by the partitions of $\{1,\dots,k\}$. For any point $x \in \R^k$, at least one of these expressions is uniformly non-degenerate in a neighborhood of~$x$. Then, studying $\rho_k$ is mostly a matter of choosing the right expression, depending on the domain we are considering. These new expressions are introduced and studied in Section~\ref{sec Kac-Rice densities revisited and clustering}, using the divided differences introduced in Section \ref{sec divided differences}. A key idea is that divided differences allow to replace the random vector $(f(x_i))_{1 \leq i \leq k}$, appearing the original expression of $\rho_k(x_1,\dots,x_k)$, by another Gaussian vector which is never degenerate even on the diagonal. We will discuss these ideas further in Section \ref{subsec sketch of proof} below.


\subsection{Law of Large Numbers and Central Limit Theorem}
\label{subsec law of large numbers and central limit theorem}

As an application of the moments estimates of Theorem~\ref{thm moments}, we prove a strong Law of Large Numbers and a Central Limit Theorem. These theorems hold in the large $R$ limit, for the linear statistics $\prsc{\nu_R}{\phi}$ with $\phi$ a test-function (cf.~Definition~\ref{def test-function}), but also for the random measures~$\nu_R$.

\begin{rem}
\label{rem weak LLN}
Under the hypotheses of Proposition~\ref{prop variance}, we immediately obtain a weak Law of Large Number for the linear statistics by applying Markov's Inequality and using the variance estimates of Proposition~\ref{prop variance}. That is, for any test-function $\phi$, for all $\epsilon >0$, we have:
\begin{equation*}
\P \left( \norm{\frac{1}{R}\prsc{\nu_R}{\phi} -\frac{1}{\pi}\int_\R \phi(x) \dx x} > \epsilon\right) = O\!\left(R^{-1}\right).
\end{equation*}
\end{rem}

In fact, if the correlation function $\kappa$ of $f$ decays fast enough, we can combine the moments estimates of Theorem~\ref{thm moments} with Markov's Inequality and the Borel--Cantelli Lemma to obtain the following.

\begin{thm}[Strong Law of Large Numbers]
\label{thm LLN}
Let $p \in \N^*$. Let $f$ be a normalized stationary centered Gaussian $\mathcal{C}^{2p}$-process whose correlation function $\kappa$ satisfies $\Norm{\kappa}_{2p,\eta} = o(\eta^{-8p})$ as $\eta \to +\infty$. Finally, let $(R_n)_{n \in \N}$ be a sequence of positive numbers such that $\sum_{n \in \N} R_n^{-p}<+\infty$.

Then, for any test-function $\phi$, the following holds almost surely:
\begin{equation*}
\frac{1}{R_n}\prsc{\nu_{R_n}}{\phi} \xrightarrow[n \to +\infty]{} \frac{1}{\pi} \int_\R \phi(x) \dx x.
\end{equation*}
Moreover, as Radon measures, we have $\displaystyle\frac{1}{R_n}\nu_{R_n} \xrightarrow[n \to +\infty]{} \frac{1}{\pi} \dx x$ almost surely in the weak-$*$ sense.
\end{thm}

Let us now recall some classical definitions before stating our Central Limit Theorem.

\begin{ntn}[Gaussian distributions]
\label{not Gaussian vectors}
Let $n \geq 1$ and let $\Lambda$ be a positive semi-definite square matrix of size $n$. We denote by $\mathcal{N}(0,\Lambda)$ the \emph{centered Gaussian distribution of variance $\Lambda$} in $\R^n$. We denote by $X \sim \mathcal{N}(0,\Lambda)$ the fact that distribution of the random vector $X \in \R^n$ is $\mathcal{N}(0,\Lambda)$.
\end{ntn}

\begin{dfn}[Schwartz space]
\label{def SR}
A function $\phi: \R \to \R$ is said to be \emph{fast-decreasing} if it satisfies $\phi(x) =O(\norm{x}^{-k})$ as $\norm{x}\to +\infty$, for all $k \in \N$. The \emph{Schwartz space} $\mathcal{S}(\R)$ is the space of $\mathcal{C}^\infty$ functions~$\phi$ such that $\phi$ and all its derivatives are fast-decreasing. Finally, we denote by $\mathcal{S}'(\R)$ the space of \emph{tempered generalized functions}.
\end{dfn}

\begin{rems}
\label{rem SR}
\begin{itemize}
\item In this paper, we use the terminology ``generalized function'' instead of ``distribution'' to avoid any possible confusion with the distribution of a random variable.
\item Recall that $\mathcal{S}'(\R)$ is indeed the topological dual of $\mathcal{S}(\R)$, for some topology that we do not recall here.
\item We refer to~\cite{BDW2018} for details about the definition of random elements of $\mathcal{S}'(\R)$ and the notion of convergence in distribution in this space.
\end{itemize}
\end{rems}

\begin{dfn}[White Noise]
\label{def white noise}
The \emph{standard Gaussian White Noise} $W$ is a random element of $\mathcal{S}'(\R)$ whose distribution is characterized by:
\begin{equation*}
\forall \phi \in \mathcal{S}(\R), \qquad \prsc{W}{\phi} \sim \mathcal{N}\left(0,\Norm{\phi}_{L^2}^2\right),
\end{equation*}
where $\prsc{\cdot}{\cdot}$ is the canonical pairing between $\mathcal{S}'(\R)$ and $\mathcal{S}(\R)$, and $\Norm{\phi}_{L^2} = \left(\int_\R \phi(x)^2 \dx x\right)^\frac{1}{2}$ is the $L^2$-norm of $\phi$.
\end{dfn}

\begin{thm}[Central Limit Theorem]
\label{thm CLT}
Let $f$ be a normalized stationary centered Gaussian process of class $\mathcal{C}^\infty$ and let us assume that its correlation function satisfies $\kappa \in \mathcal{S}(\R)$. Let $\sigma>0$ be the constant defined by Equation~\eqref{eq def sigma}.

For any test-function $\phi$ (in the sense of Definition~\ref{def test-function}), we have the following convergence in distribution:
\begin{equation*}
\frac{1}{R^\frac{1}{2}\sigma} \left(\prsc{\nu_R}{\phi}-\frac{R}{\pi} \int_\R \phi(x) \dx x \right) \xrightarrow[R \to +\infty]{} \mathcal{N}\!\left(0,\Norm{\phi}_{L^2}^2\right).
\end{equation*}
Moreover, the following holds in distribution in $\mathcal{S}'(\R)$:
\begin{equation*}
\frac{1}{R^\frac{1}{2}\sigma}\left(\nu_R -\frac{R}{\pi} \dx x\right) \xrightarrow[R \to +\infty]{} W,
\end{equation*}
where $W$ is the standard Gaussian White Noise and $\dx x$ is the Lebesgue measure of $\R$.
\end{thm}

The fact that almost surely $\nu_R \in \mathcal{S}'(\R)$ for all $R >0$ is not obvious. This is proved in Lemma~\ref{lem nu R as tempered distribution} as a consequence of Proposition~\ref{prop expectation}, see Section~\ref{subsec proof of proposition expectation}.


\subsection{Sketch of proof}
\label{subsec sketch of proof}

In this section, we discuss the main ideas of the proofs of our main results (Theorems~\ref{thm moments}, \ref{thm vanishing order} and~\ref{thm clustering}). First, let us outline the proof of Theorem~\ref{thm moments} assuming the results of Theorem~\ref{thm clustering}. The starting point of the proof is the Kac--Rice formula, see Proposition~\ref{prop Kac-Rice formula} below. It allows to write the non-central moments of the linear statistics associated with the random measure $\nu_R$ as follows:
\begin{equation}
\label{eq Kac-Rice intro}
\esp{\prod_{i=1}^k \prsc{\nu_R}{\phi_i}} = \int_{\R^k} \left(\prod_{i=1}^k \phi_i\left(\frac{x_i}{R}\right)\right) \rho_k(x_1,\dots,x_k) \dx x_1 \dots \dx x_k,
\end{equation}
where $(\phi_i)_{1 \leq i \leq k}$ are test-functions satisfying the hypotheses of Theorem~\ref{thm moments} and $\rho_k$ is the function defined by Definition~\ref{def Kac-Rice densities} below. Here we are cheating a bit: Equation~\eqref{eq Kac-Rice intro} is false and the $k$-th non-central moment on the left-hand side should be replaced be the so-called $k$-th factorial moment for this relation to hold. However, the $k$-th non-central moment can be expressed in terms of the factorial moments of order at most $k$ by some combinatorics, so that a more complicated version of Equation~\eqref{eq Kac-Rice intro} holds. Dealing with these combinatorics is one of the difficulties of the proof of Theorem~\ref{thm moments} given in Section~\ref{sec proof theorem moments}. For the sake of clarity, in this sketch of proof we will not give more details about this, and simply pretend that Equation~\eqref{eq Kac-Rice intro} holds. This is enough to understand the main ideas of the proof.

Under the hypotheses of Theorem~\ref{thm moments}, for any $k \in \{1,\dots,p\}$ the density $\rho_k$ is well-defined from $\R^k\setminus \Delta_k$ to $\R$, but it is a priori singular along $\Delta_k$. As discussed in Section~\ref{subsec clustering for the k-point functions}, the Kac--Rice density~$\rho_k$ is equal to the $k$-point function of the zero point process of $f$. By the first point in Theorem~\ref{thm clustering}, it admits a unique continuous extension to $\R^k$ which is bounded. In particular, the right-hand side of Equation~\eqref{eq Kac-Rice intro} is well-defined and finite. Let $A \subset \{1,\dots,p\}$, we denote by $\norm{A}$ its cardinality. Moreover, for any $x=(x_i)_{1 \leq i \leq p} \in \R^p$, we denote by $\underline{x}_A = (x_i)_{i \in A}$. Then, using Equation~\eqref{eq Kac-Rice intro}, we can write $m_p(\nu_R)(\phi_1,\dots,\phi_p)$ as:
\begin{equation}
\label{eq mp integral intro}
m_p(\nu_R)(\phi_1,\dots,\phi_p) = \int_{\R^p} \left(\prod_{i=1}^p \phi_i\left(\frac{x_i}{R}\right)\right) F_p(x) \dx x,
\end{equation}
where,
\begin{equation}
\label{eq def Fp intro}
F_p : x \longmapsto \sum_{A \subset \{1,\dots,p\}} (-1)^{p - \norm{A}} \rho_{\norm{A}}(\underline{x}_A) \prod_{i \notin A} \rho_1(x_i).
\end{equation}
See Lemma~\ref{lem integral expression} for the rigorous statement corresponding to Equation~\eqref{eq mp integral intro}. Note that we only use the notation $F_p$ in the present section. In Section~\ref{sec proof theorem moments}, this function is the one denoted by $F_{\I_{\min}(p)}$.

Apart from proving Theorem~\ref{thm clustering}, the main difficulty in the proof of our moments estimates is to understand the large $R$ asymptotics of the integral appearing in Equation~\eqref{eq mp integral intro}. In order to do so, we cut $\R^p$ into pieces as follows. Let $\eta \geq 0$, for any $x=(x_i)_{1 \leq i \leq p} \in \R^p$, we denote by $G_\eta(x)$ the graph whose set of vertices is $\{1,\dots,p\}$ and such that there is an edge between $i$ and $j$ if and only if $i \neq j$ and $\norm{x_i-x_j} \leq \eta$. We denote by $\I_\eta(x) \in \pa_p$ the partition defined by the connected components of $G_\eta(x)$. This partition encodes how the components of $x$ are clustered in $\R$, at scale $\eta$. Finally, for any $\I \in \pa_p$, we denote by $\R^p_{\I,\eta} = \{ x \in \R^p \mid \I_\eta(x) =\I\}$. We have $\R^p = \bigsqcup_{\I \in \pa_p} \R^p_{\I,\eta}$, so that it is enough to understand the contribution of each $\R^p_{\I,\eta}$ to the integral appearing in Equation~\eqref{eq mp integral intro}.

Since we are interested in the asymptotics as $R \to +\infty$, we choose a scale parameter $\eta(R)$ that depends on $R$. The most convenient choice for $\eta$ is the following. Under the hypotheses of Theorem~\ref{thm moments}, there exists a function $\alpha$ such that, setting $\eta:R \mapsto R^\frac{1}{4}\alpha(R)$, we have the following as $R\to +\infty$: $\eta(R) \to +\infty$, $\alpha(R) \to 0$ and $\Norm{\kappa}_{p,\eta(R)} = o(R^{-p})$. In particular, the error term in Theorem~\ref{thm clustering} becomes $o(R^{-\frac{p}{2}})$. Then, the contribution of $\R^p_{\I,\eta(R)}$ to~\eqref{eq mp integral intro} depends on the combinatorics of $\I$, and one of the following three situations occurs.
\begin{enumerate}
\item The partition $\I$ contains a singleton, say $\{p\} \in \I$. This means that if $x \in \R^p_{\I,\eta(R)}$, then $x_p$ is far from the other components of $x$, at scale $\eta(R)$. In this case, for each $A \subset \{1,\dots,p-1\}$, we regroup the terms indexed by $A$ and $A \sqcup \{p\}$ in Equation~\eqref{eq def Fp intro}. Using the clustering property of Theorem~\ref{thm clustering}, these two terms cancel each other out, up to an error term of order $o(R^{-\frac{p}{2}})$. Summing over $A \subset \{1,\dots,p-1\}$, we obtain $F_p(x)=o(R^{-\frac{p}{2}})$ uniformly on~$\R^p_{\I,\eta(R)}$. This implies that $\R^p_{\I,\eta(R)}$ only contributes $o(R^\frac{p}{2})$ to~\eqref{eq mp integral intro}.
\item If $\I$ does not contain any singletons, we denote by $a$ the number of pairs in $\I$ and by $b$ the number of elements of $\I$ of cardinality at least $3$. In the second situation we consider, we assume that $b \geq 1$. In this case, we prove that the contribution of $\R^p_{\I,\eta(R)}$ to~\eqref{eq mp integral intro} is $O(R^{a+b}\eta(R)^{p-2a-b})$. This bound is obtained by using the clustering property of Theorem~\ref{thm clustering} in a way similar to what we did in the previous case. The dissymmetry between the pairs and the other elements of $\I$ comes from the integrability of the function $z \mapsto F_2(0,z)$ on $\R$. This dissymmetry is crucial in the following. Using the relation $\eta(R) = R^\frac{1}{4}\alpha(R)$ and $2a+3b \leq p$, we have:
\begin{equation*}
R^{a+b}\eta(R)^{p-2a-b} = O(R^\frac{p}{2}\alpha(R)^{p-2a-b}).
\end{equation*}
Since $b \geq 1$, we have $2a +b<p$ and the previous term is $o(R^\frac{p}{2})$. Once again, $\R^p_{\I,\eta(R)}$ only contributes $o(R^\frac{p}{2})$ to~\eqref{eq mp integral intro}.
\item The last situation is when $\I = \left\{\{a_i,b_i\} \mvert 1 \leq i \leq \frac{p}{2}\right\}$ is a partition into pairs, which can only happen if $p$ is even. In this case, the clustering property of Theorem~\ref{thm clustering} implies that $F_p(x) = \prod_{i=1}^\frac{p}{2} F_2(x_{a_i},x_{b_i}) +o(R^{-\frac{p}{2}})$, uniformly on $\R^p_{\I,\eta(R)}$. This implies that the contribution of $\R^p_{\I,\eta(R)}$ to~\eqref{eq mp integral intro} equals:
\begin{equation*}
\prod_{i=1}^\frac{p}{2} \int_{\R^2} \phi_{a_i}\!\left(\frac{x}{R}\right)\phi_{b_i}\!\left(\frac{y}{R}\right)F_2(x,y) \dx x \dx y +o(R^\frac{p}{2}) = \prod_{i=1}^\frac{p}{2} m_2(\nu_R)(\phi_{a_i},\phi_{b_i}) +o(R^\frac{p}{2}).
\end{equation*}
\end{enumerate}
We conclude the proof of Theorem~\ref{thm moments} by summing up over $\I \in \pa_p$ the contributions of each $\R^p_{\I,\eta(R)}$ to the integral in Equation~\eqref{eq mp integral intro}. Note that the leading term comes from the pieces indexed by partitions into pairs.

Let us now consider the proofs of Theorems~\ref{thm vanishing order} and~\ref{thm clustering}. In this sketch of proof, we assume that the correlation function~$\kappa$ of $f$ tends to $0$ at infinity. This ensures that $\rho_k$ is well-defined on $\R^k \setminus \Delta_k$. By Definition~\ref{def Kac-Rice densities}, for any $x=(x_i)_{1 \leq i \leq k} \in \R^k \setminus \Delta_k$ we have $\rho_k(x) = (2\pi)^{-\frac{k}{2}}N_k(x) D_k(x)^{-\frac{1}{2}}$, where $D_k(x)$ is the determinant of the variance matrix of $\left(f(x_1),\dots,f(x_k)\right)$ and $N_k(x)$ is the conditional expectation of $\prod_{i=1}^k \norm{f'(x_i)}$ given that $f(x_1) = \dots = f(x_k)=0$. The density $\rho_k$ is a priori singular along the large diagonal $\Delta_k \subset \R^k$, since $D_k$ vanishes along $\Delta_k$. The main problem here is to understand to behavior of $\rho_k$, that is of $N_k$ and $D_k$, near $\Delta_k$. This is what we focus on in the remainder of this section. Once this is done, the clustering result of Theorem~\ref{thm clustering} is a (non-trivial) consequence of the decay at infinity of $\kappa$ and its derivatives.

Our study of $N_k$ and $D_k$ near $\Delta_k$ relies on the use of the divided differences associated with the process~$f$. Let us explain our strategy on the simplest non-trivial case, that is for $D_2$. A direct computation, using the Taylor expansion of $\kappa$ around $0$, shows that, in the setting of this paper, we have $D_2(x,y) \sim (y-x)^2$ as $y \to x$. This proof is very simple, but its extension to $3$ points or more seems intractable. Here is another proof of the same result that can be generalized to $k \geq 3$. If $y \neq x$, we can write:
\begin{equation}
\label{eq divided differences intro}
\begin{pmatrix}
f(x) \\ f(y)
\end{pmatrix} = \begin{pmatrix}
1 & 0 \\ 1 & y-x
\end{pmatrix} \begin{pmatrix}
f(x) \\ \frac{f(y)-f(x)}{y-x}
\end{pmatrix}.
\end{equation}
As $y \to x$, we have $\left(f(x),\frac{f(y)-f(x)}{y-x}\right) \xrightarrow[]{}(f(x),f'(x))$. By stationarity and normalization of $f$, the matrix $\var{f(x),f'(x)}$ is the identity. Hence, taking the determinant of the variance of~\eqref{eq divided differences intro}, we recover $D_2(x,y) \sim (y-x)^2$ as $y \to x$.

In Equation~\eqref{eq divided differences intro}, by stationarity, normalization and regularity of~$f$, the Gaussian vector $\left(f(x),\frac{f(y)-f(x)}{y-x}\right)$ is uniformly non-degenerate in a neighborhood of $\Delta_2$. Thus, the degeneracy of $(f(x),f(y))$ along $\Delta_2$ is completely accounted for by the degeneracy of the matrix $\left(\begin{smallmatrix}
1 & 0 \\ 1 & y-x
\end{smallmatrix}\right)$, whose coefficients are deterministic polynomial in $(y-x)$. The divided differences generalize this situation to any number of points. By definition of the divided differences $\left([f]_j(x_1,\dots,x_j)\right)_{1 \leq j \leq k}$ associated with $f$ and $x=(x_i)_{1 \leq i \leq k} \in\R^k \setminus \Delta_k$ (see Section~\ref{subsec Hermite interpolation and divided differences}), we have:
\begin{equation*}
\begin{pmatrix}
f(x_1) \\ \vdots \\ f(x_k)
\end{pmatrix} = M(x) \begin{pmatrix}
[f]_1(x_1) \\ \vdots \\ [f]_k(x_1,\dots,x_k)
\end{pmatrix},
\end{equation*}
where $M(x)$ is a matrix whose coefficients are deterministic polynomials in $(x_j-x_i)_{1 \leq i < j \leq k}$. In fact, $\det M(x) = \prod_{1 \leq i < j \leq k} (x_j-x_i)$ and we have:
\begin{equation*}
\left(\rule{0em}{2.5ex}[f]_1(x_1), [f]_2(x_1,x_2), \dots, [f]_k(x_1,\dots,x_k)\right) \xrightarrow[x \to (z,z,\dots,z)]{} \left(f(z),f'(z), \dots, \frac{f^{(k-1)}(z)}{(k-1)!}\right).
\end{equation*}
Our hypotheses ensure that the Gaussian vector on the right-hand side is non-degenerate. Denoting by $D >0$ the determinant of its variance, this proves that $D_k(x) \sim D \prod_{1 \leq i < j \leq k} (x_j-x_i)^2$, as $x \to (z,z,\dots,z)$. Note that $D$ does not depend on $z$, by stationarity. A refinement of this argument shows that $N_k(x) \sim N \prod_{1 \leq i < j \leq k} (x_j-x_i)^2$ for some $N>0$. Hence, as $x \to (z,z,\dots,z)$ we have:
\begin{equation*}
\rho_k(x) \sim \frac{N}{(2\pi)^\frac{k}{2}D^\frac{1}{2}} \prod_{1 \leq i < j \leq k} \norm{x_j-x_i}.
\end{equation*}

The previous discussion explains how the divided differences allow to understand the apparent singularity of $\rho_k$ near $\{(z,z,\dots,z) \mid z \in \R\} \subset \R^k$, which is the stratum of smallest dimension in~$\Delta_k$. Near other strata, the situation is more intricate, yet tractable by similar methods. The key point is that, using the divided differences associated with $f$, we define a family of alternative expressions of $\rho_k$ indexed by the partitions of $\{1,\dots,k\}$, see Definition~\ref{def Kac-Rice densities partition}. Then, for any $y \in \R^k$, we prove the local estimate of Theorem~\ref{thm vanishing order} by choosing the right expression of $\rho_k$, depending on how the components of $y$ are clustered. Precisely, we use the expression indexed by the partition $\I_0(y)$ defined previously, see also Definition~\ref{def I eta}.


\subsection{Related works}
\label{subsec related works}

The study of the zeros of a Gaussian process goes back to Kac~\cite{Kac1943}, who obtained a formula for the mean number of roots of some Gaussian polynomials in an interval. This was generalized to other Gaussian processes by Rice~\cite{Ric1944}. The mean number of zeros in an interval of any continuous stationary Gaussian process was computed by Ylvisaker, see~\cite{Ylv1965}. The proofs of Kac and Rice rely on an integral formula for the mean number of zeros. Extensions of their work lead to what are now called the Kac--Rice formulas. Modern references for these are~\cite{AT2007} and~\cite{AW2009}, but formulas of this kind already appear in~\cite{CL1965}.

Among other things, Kac--Rice formulas were used to derive conditions for the finiteness of the moments of the number of zeros of Gaussian processes. Geman derived a necessary and sufficient condition for the finiteness of the second moment in~\cite{Gem1972}. The case of higher moments was studied by Cuzick in~\cite{Cuz1975,Cuz1978}. Note that~\cite{Cuz1978} already uses divided differences in order to obtain criteria for the finiteness of the moments of the number of zeros of a Gaussian process. The results of~\cite{Cuz1978} do not apply to the Bargmann--Fock field, whose correlation function is $z \mapsto e^{-\frac{1}{2}z^2}$, and which is one of our main example in this paper (cf.~Example~\ref{ex non-degenerate and fast-decreasing}). Much more recently, a necessary condition for the finiteness of the moment of order $p$ was derived in~\cite{AAD+2019}. We emphasize that the methods developed in the present paper allow to prove the finiteness of the higher moments of the number of zeros of a Gaussian process in an interval under three simple conditions: stationarity, sufficient regularity of the process, and fast enough decay at infinity of the correlation function and its first derivatives. While being easy to state and rather general, these conditions are quite strong and probably far from necessary.

In~\cite{Cuz1976}, Cuzick studied the asymptotic variance as $R \to +\infty$ of the number of zeros of a stationary Gaussian process $f$ in $[0,R]$. He obtained the same asymptotics as in Proposition~\ref{prop variance} for $\phi_1=\mathbf{1}_{[0,1]} = \phi_2$, under slightly weaker conditions. However, he did not prove the positivity of the constant $\sigma$ (cf.~Equation~\eqref{eq def sigma}). Assuming that $\sigma >0$, he also derived a Central Limit Theorem for $\card(Z \cap [0,R])$ as $R \to +\infty$. Piterbarg proved similar results and the positivity of $\sigma$ under different assumptions, see~\cite[Theorem~3.5]{Pit1996} for example.

In~\cite{KL2001}, Kratz and Leòn developed a method for proving Central Limit Theorems for level functionals of Gaussian processes. In particular, it should allow to prove Theorem~\ref{thm CLT} under weaker hypotheses than those we gave. The method of~\cite{KL2001} is completely different and relies on the Wiener--Itô expansion of the functional under study. The Wiener--Itô expansion of $\card(Z \cap [0,R])$ was computed in~\cite{KL1997}. The same proof should yield the expansion of $\prsc{\nu_R}{\phi}$ for any Lebesgue-integrable $\phi$. The results of Kratz--Leòn also show that the variance of $\card(Z \cap [0,R])$ is equivalent to $\sigma^2 R$ as $R\to +\infty$, for some $\sigma >0$. In Section~\ref{subsec positivity of the leading constant}, we use the result of~\cite{KL1997} to derive the lower bound on $\sigma^2$ mentioned in Remark~\ref{rem explicit lower bound}. Let us mention that, very recently, Lachièze-Rey~\cite{Lac2020} proved that:
\begin{equation*}
\liminf_{R \to +\infty} \frac{1}{R} \var{\card(Z \cap [0,R])} >0,
\end{equation*}
under essentially no hypothesis on the process $f$. This implies the positivity of $\sigma^2$ in Proposition~\ref{prop variance}. The present paper partially overlaps with~\cite{Lac2020} since we obtained independently a similar lower bound for $\sigma^2$ by the same method, see~\cite[Section~4]{Lac2020} and Corollary~\ref{cor sigma positive} below.

The references cited previously are concerned with the number of zeros of $f$ in an interval. More generally, a lot of them consider the number of crossings, or up-crossings, of a level by~$f$ in an interval. For an in depth survey of the existing literature on the subject we refer to~\cite{Kra2006}.

A special case of~\cite[Theorem~1]{NS2016} gives the strong Law of Large Numbers for the number of zeros of a stationary Gaussian process $f$ in $[0,R]$, under weaker assumptions than Theorem~\ref{thm LLN}. Nazarov and Sodin also studied the $k$-point functions of the point process defined by the complex zeros of a Gaussian Entire Function, see~\cite{NS2012}. Theorems~\ref{thm vanishing order} and~\ref{thm clustering} are analogous to the main results of~\cite{NS2012}, but for the real zeros of a stationary Gaussian process. We stress that our techniques are different from those of~\cite{NS2012}. In particular, in~\cite{NS2012} the authors require the analiticity of the Gaussian process and use techniques from complex analysis, such as the Residue Theorem, whereas we only require our Gaussian fields to be $\mathcal{C}^{k}$ in order to obtain a clustering property of the $k$-point function.

The $k$-point functions $\rho_k$ of the real zeros of the Bargmann--Fock process were studied by Do and Vu. In~\cite[Lemma~9]{DV2020}, they proved that the $\rho_k$ satisfy a clustering property similar to Theorem~\ref{thm clustering}, with an exponentially small error term. They also derived the vanishing order of the $\rho_k$ along the diagonal, see~\cite[Lemma~10]{DV2020}. Their methods build on the work of~\cite{NS2012} and also relies on complex analysis. In particular, it is paramount in their work that the Bargmann--Fock is the restriction to $\R$ of a Gaussian Entire Function.

In both~\cite{NS2012} and~\cite{DV2020}, the authors deduce from their clustering result a Central Limit Theorem for the (compactly supported) linear statistics of the point processes they study. Their proofs rely on the cumulants method. This strategy was generalized in~\cite[Theorem~13]{BYY2019}, where the authors show that a strong clustering property of the kind of Theorem~\ref{thm clustering}, with a fast-decreasing error term, implies a Central Limit Theorem for the compactly supported linear statistics of the underlying point process. Note that one can not deduce the moments estimates of Theorem~\ref{thm moments} from this kind of results, even when the correlation function $\kappa$ lies in $\mathcal{S}(\R)$.

Under the hypotheses of Theorem~\ref{thm moments}, Markov's Inequality implies the concentration in probability of $\frac{1}{R}\card(Z \cap [0,R])$, more generally of the normalized linear statistics, around its mean at polynomial speed in $R$. Under stronger assumptions, in~\cite{BDFZ2020}, the authors proved a large deviation result for $\frac{1}{R}\card(Z \cap [0,R])$, that is concentration around the mean at exponential speed in $R$. Their proof relies on the existence of an analytic extension of~$f$ to horizontal strips in the complex plane. Note that the Bargmann--Fock process satisfies the hypotheses of~\cite[Theorem~1.1]{BDFZ2020}.

In this paper, we study the zeros of a stationary Gaussian process in an interval of size $R$ as $R \to +\infty$. In~\cite{AL2019}, we studied the real zeros of a Gaussian section of the $d$-th tensor power of an ample line bundle over a real algebraic curve, as $d \to +\infty$. The model of Gaussian section considered in~\cite{AL2019} is known as the \emph{complex Fubini--Study ensemble} and was introduced in~\cite{GW2011}. It is the real analogue of the complex model studied by Shiffman--Zelditch in~\cite{SZ1999} and subsequent papers. The idea to study the random measure associated with the zero set of a Gaussian section already appears in~\cite{SZ1999}. In~\cite{BSZ2000}, the authors study the scaling limit of the $k$-point function of the complex zero set of a random section in their model. They also relate this function with the non-central moments of the linear statistics associated with these complex zeros.

In~\cite{AL2019}, we derived the large $d$ asymptotics for the central moments of the linear statistics associated with the real zero set of a Gaussian section of degree $d$ in the complex Fubini--Study ensemble. These results are the counterpart of Theorems~\ref{thm moments}, \ref{thm LLN} and~\ref{thm CLT} in this context. Note that~\cite[Theorem~1.12]{AL2019} generalizes the variance estimate derived by Letendre--Puchol in~\cite{LP2019}, in the case where the ambient dimension is $1$. Its proof relies on results of Ancona, who proved the counterpart of Theorem~\ref{thm clustering} in~\cite[Theorems~4.1 and 5.7 and Proposition~4.2]{Anc2019}. However, note that Theorems~\ref{thm clustering} and \ref{thm vanishing order} are more precise than their counterparts in ~\cite{Anc2019}. For example, \cite[Theorem~5.7]{Anc2019} says that the $k$-point function $\rho_k$ vanishes along the diagonal $\Delta_k$, while in Theorem \ref{thm vanishing order} we also compute the vanishing order of $\rho_k$ along the diagonal $\Delta_k$, also giving conditions on the process $f$ for which this vanishing order is sharp. As explained in the last paragraph of Section \ref{subsec sketch of proof}, one of the fundamental parts of studying the $k$-point function is finding good expression for $\rho_k(x)$, depending on how the components of $x$ are clustered. The expressions used in the present article are different from those used in \cite{Anc2019} (one should compare the expression appearing in Definition \ref{def Kac-Rice densities partition} with the one in \cite[Proposition 5.21]{Anc2019}). The new expressions used in the present paper turn out to be more precise for estimating $\rho_k$ along the diagonal. The results of~\cite{Anc2019,AL2019} apply to the number of real roots of a Kostlan polynomial of degree $d$, see~\cite{Kos1993}. In this case, the variance asymptotics and the Central Limit Theorem were proved by Dalmao~\cite{Dal2015}.

To conclude this section, let us mention that the setting of the present paper is related with that of~\cite{Anc2019,AL2019,GW2011,LP2019}. Indeed, the Bargmann--Fock process introduced previously is the universal local scaling limit, as $d \to +\infty$, of a random section of degree $d$ in the complex Fubini--Study ensemble. See~\cite{AL2019} for more details.


\subsection{Organization of the paper}
\label{subsec organization of the paper}

The content of this paper is organized as follows. In Section~\ref{sec framework}, we introduce our framework and the random measures $\nu_R$ we are interested in. We also introduce some useful notations that will appear throughout the paper. In Section~\ref{sec Kac-Rice formulas and mean number of zeros}, we recall the Kac--Rice formulas. As first applications, we prove that the Kac--Rice density $\rho_k$ is the $k$-point function of the random point process $Z$ and Proposition~\ref{prop expectation}. Section~\ref{sec proof of proposition variance} is dedicated to the proof of the covariance estimates of Proposition~\ref{prop variance}. In Section~\ref{sec divided differences}, we introduce the divided differences associated with a function and study the distribution of the divided differences associated with a stationary Gaussian process. In Section~\ref{sec Kac-Rice densities revisited and clustering}, we use the divided differences to derive alternative expressions of the Kac--Rice densities. In particular, we prove Theorem~\ref{thm vanishing order} in Section~\ref{subsec proof of theorem vanishing order} and Theorem~\ref{thm clustering} in Section~\ref{subsec proof of theorem clustering}. Section~\ref{sec proof theorem moments} is concerned with the proof of Theorem~\ref{thm moments} and Section~\ref{sec limit theorems} is concerned with the proofs of the limit Theorems~\ref{thm LLN} and~\ref{thm CLT}. This paper also contains three appendices. In Appendix~\ref{sec examples of smooth non-degenerate processes}, we build examples of Gaussian processes satisfying the hypotheses of our main theorems. Appendix~\ref{sec properties of the density function F} contains the proofs of some auxiliary results related to the proof of Proposition~\ref{prop variance}. Finally, Appendix~\ref{sec a Gaussian lemma} is dedicated to the proof of a lemma pertaining to the regularity of the Kac--Rice densities.


\section{Framework}
\label{sec framework}

In this section, we introduce the random measures we study in this paper. First, in Section~\ref{subsec partitions, products and diagonal inclusions}, we introduce some notations related to partitions of finite sets and diagonals in Cartesian products. In Section~\ref{subsec stationary Gaussian processes and correlation functions}, we introduce properly the random processes we are interested in and their correlation functions. Finally, in Section~\ref{subsec zeros of stationary Gaussian processes}, we prove that the vanishing locus of the processes introduced in Section~\ref{subsec stationary Gaussian processes and correlation functions} is almost surely nice (see Lemma~\ref{lem Z as closed discrete}), and we introduce several counting measures associated with this random set.


\subsection{Partitions, products and diagonal inclusions}
\label{subsec partitions, products and diagonal inclusions}

Let us first introduce some notations that will be useful throughout the paper. Recall that we already defined the set $\pa_A$ (resp.~$\pa_k$) of partitions of a finite set $A$ (resp.~of $\{1,\dots,k\}$) and the set $\pp_A$ (resp.~$\pp_k$) of its partitions into pairs (see Definition~\ref{def partitions}).

\begin{ntns}
\label{ntn product indexed by A}
Let $A$ be a finite set and let $Z$ be any set.
\begin{itemize}
\item We denote by $\card(A)$ or by $\norm{A}$ the cardinality of $A$.
\item We denote by $Z^A$ the Cartesian product of $\norm{A}$ copies of $Z$, indexed by the elements of $A$.
\item A generic element of $Z^A$ is denoted by $\underline{x}_A =(x_a)_{a \in A}$, or more simply by $x$. If $B \subset A$ we denote by $\underline{x}_B =(x_a)_{a \in B}$.
\item Let $(\phi_a)_{a \in A}$ be functions on $Z$, we denote by $\phi_A = \boxtimes_{a\in A}\phi_a$ the function on $Z^A$ defined by: $\phi_A(\underline{x}_A)=\prod_{a\in A}\phi_a(x_a)$, for all $\underline{x}_A = (x_a)_{a \in A} \in Z^A$. If $A$ is of the form $\{1,\dots,k\}$ with $k \in \N^*$, we use the simpler notation $\phi=\phi_A$.
\end{itemize}
\end{ntns}

\begin{dfn}[Diagonals]
\label{def diagonals}
Let $A$ be a non-empty finite set, we denote by $\Delta_A$ the \emph{large diagonal} of $\R^A$:
\begin{equation*}
\Delta_A = \left\{(x_a)_{a \in A}\in \R^A \mvert \exists a,b \in A \text{ such that } a \neq b \text{ and } x_a = x_b \right\}.
\end{equation*}
Moreover, for all $\I \in \pa_A$, we denote by
\begin{equation*}
\Delta_{A,\I} = \left\{(x_a)_{a \in A}\in \R^A \mvert \forall a,b \in A, \left(x_a=x_b \iff \exists I \in \I \text{ such that } a \in I \text{ and } b \in I \right)\right\}.
\end{equation*}
If $A = \{1,\dots,k\}$, we use the simpler notations $\Delta_k=\Delta_A$ and $\Delta_{k,\I}=\Delta_{A,\I}$.
\end{dfn}

\begin{dfn}[Diagonal inclusions]
\label{def diagonal inclusions}
Let $A$ be a non-empty finite set and let $\I \in \pa_A$. The \emph{diagonal inclusion} $\iota_\I$ is the function from $\R^\I$ to $\R^A$ defined by: for all $\underline{x}_\I= (x_I)_{I \in \I} \in \R^\I$, $\iota_\I(\underline{x}_\I)= \underline{y}_A= (y_a)_{a\in A}$, where for all $I \in \I$, for all $a \in I$, we set $y_a = x_I$.
\end{dfn}

\begin{rem}
\label{rem diagonal inclusions}
With these definitions, we have $\R^A = \bigsqcup_{\I \in \pa_A} \Delta_{A,\I}$ and $\Delta_A = \bigsqcup_{\I \in \pa_A \setminus \{\I_{\min}(A)\}} \Delta_{A,\I}$, where we denoted $\I_{\min}(A)= \left\{ \{a\} \mvert a \in A\right\}$ (this notation comes from the fact that $\I_{\min}(A)$ is the minimum of $\pa_A$ for some partial order, see Definition~\ref{def partial order pa p}). Moreover, for all $\I \in \pa_A$, the map $\iota_\I$ is a smooth diffeomorphism from $\R^\I \setminus \Delta_\I$ onto $\Delta_{A,\I} \subset \R^A$. Note that $\Delta_{A,\I_{\min}(A)}$ is the configuration space $\R^A \setminus \Delta_A$ of $\norm{A}$ distinct points in~$\R$. In the following, we avoid using the notation $\Delta_{A,\I_{\min}(A)}$ and use $\R^A \setminus \Delta_A$ instead.
\end{rem}

\begin{rem}
\label{rem diagonal inclusions and vanishing order}
Let $y \in \R^k$, the partition $\I$ defined in Theorem~\ref{thm vanishing order} is the unique $\I \in \pa_k$ such that $y \in \Delta_{k,\I}$. With the notations of Theorem~\ref{thm vanishing order}, there exists $(y_I)_{I \in \I} \in \R^\I \setminus \Delta_\I$ such that $y= \iota_\I((y_I)_{I \in \I})$.
\end{rem}

Let $Z \subset \R$ be a closed discrete subset. In particular, for any $K \subset \R$ compact, $Z \cap K$ is finite. As in the introduction, we denote by $\nu = \sum_{x \in Z} \delta_x$ the \emph{counting measure} of $Z$. More generally, for any non-empty finite set $A$, we can define the counting measure of $Z^A \subset \R^A$.

\begin{dfn}[Counting measures]
\label{def nu A}
Let $Z \subset \R$ be closed and discrete and let $A$ be a non-empty finite set. We denote~by:
\begin{align*}
\nu^A &= \sum_{x \in Z^A} \delta_x & &\text{and} & \nu^{[A]} &= \sum_{x \in Z^A \setminus \Delta_A} \delta_x,
\end{align*}
where $\delta_x$ is the unit Dirac mass at $x \in \R^A$ and $\Delta_A$ is defined by Definition~\ref{def diagonals}. These counting measures act on a function $\phi:\R^A \to \R$ as follows:
\begin{itemize}
\item if $\phi \geq 0$ or $\sum_{x \in Z^A} \norm{\phi(x)} < +\infty$ then $\prsc{\nu^A}{\phi} = \sum_{x \in Z^A} \phi(x)$,
\item if $\phi \geq 0$ or $\sum_{x \in Z^A \setminus \Delta_A} \norm{\phi(x)} < +\infty$ then $\prsc{\nu^{[A]}}{\phi} = \sum_{x \in Z^A \setminus \Delta_A} \phi(x)$,
\end{itemize}
Quantities of the form $\prsc{\nu^A}{\phi}$ (resp.~$\prsc{\nu^{[A]}}{\phi}$) are called the \emph{linear statistics} of $\nu^A$ (resp.~$\nu^{[A]}$). As usual, if $A = \{1,\dots,k\}$, we denote $\nu^k=\nu^A$ and $\nu^{[k]}=\nu^{[A]}$.
\end{dfn}

Note that $\nu^A$ (resp.~$\nu^{[A]}$) defines a Radon measure on $\R^A$, that is a continuous linear form on $\left(\mathcal{C}^0_c(\R^A),\Norm{\cdot}_\infty\right)$, the space of compactly supported continuous functions on $\R^A$ equipped with the sup-norm. Note also that the measure $\nu^A$ and $\nu^{[A]}$ are completely characterized by the linear statistics $\left\{ \prsc{\nu^A}{\phi} \mvert \phi \in \mathcal{C}^0_c(\R^A) \right\}$ and $\left\{ \prsc{\nu^{[A]}}{\phi} \mvert \phi \in \mathcal{C}^0_c(\R^A) \right\}$ respectively.

\begin{lem}
\label{lem decomposition nu ks}
Let $Z \subset \R$ be closed and discrete and let $A$ be a non-empty finite set. Using the notations introduced above, we have $\nu^A = \sum_{\I \in \pa_A} (\iota_\I)_*\nu^{[\I]}$.
\end{lem}

\begin{proof}
Recall that $\R^A = \bigsqcup_{\I \in \pa_A} \Delta_{A,\I}$. Then, we have:
\begin{equation*}
\nu^A = \sum_{x \in Z^A} \delta_x = \sum_{\I \in \pa_A} \left(\sum_{x \in Z^A \cap \Delta_{A,\I}} \delta_x \right).
\end{equation*}
Let $\I \in \pa_A$, recall that $\iota_\I$ defines a smooth diffeomorphism from $\R^\I \setminus \Delta_\I$ onto $\Delta_{A,\I}$. Moreover, $\iota_\I(Z^\I \setminus \Delta_\I) = Z^A \cap \Delta_{A,\I}$ (see Definition~\ref{def diagonals} and~\ref{def diagonal inclusions}). Hence,
\begin{equation*}
\sum_{x \in Z^A \cap \Delta_{A,\I}} \delta_x = \sum_{y \in Z^\I \setminus \Delta_\I} \delta_{\iota_\I(y)} = \sum_{y \in Z^\I \setminus \Delta_\I} (\iota_\I)_*\delta_y = (\iota_\I)_*\nu^{[\I]}.\qedhere
\end{equation*}
\end{proof}


\subsection{Stationary Gaussian processes and correlation functions}
\label{subsec stationary Gaussian processes and correlation functions}

In this section, we introduce the random processes we study and how they are normalized. Let $f:\R \to \R$ be a stationary centered Gaussian process. By \emph{stationary}, we mean that, for all $t \in \R$, the process $x \mapsto f(x+t)$ is distributed as $f$. Let $K:\R^2 \to \R$ be the \emph{correlation kernel} of $f$, defined by $K:(x,y) \mapsto \esp{f(x)f(y)}$. Since $f$ is centered, its distribution is characterized by $K$. Let $\kappa:x \mapsto K(0,x)$ denote the \emph{correlation function} of $f$. The stationarity of $f$ is equivalent to the fact that $K(x,y) = \kappa(y-x)$ for all $(x,y) \in \R^2$. Note that, since $K$ is symmetric, then $\kappa$ is an even function.

\begin{dfn}[$\mathcal{C}^p$-process]
\label{def Cp process}
Let $p \in \N \cup \{\infty\}$, we say that $f$ is a process \emph{of class $\mathcal{C}^p$} (or a \emph{$\mathcal{C}^p$-process}) if its trajectories are almost surely of class~$\mathcal{C}^p$.
\end{dfn}

Let us assume that $f$ is of class $\mathcal{C}^p$, for some $p \in \N \cup \{\infty\}$. For all $k \in \{0,\dots,p\}$ we denote by $f^{(k)}$ the $k$-th derivative of $f$. We also use the usual notations $f'=f^{(1)}$ and $f''=f^{(2)}$. Then, for all $m \in \N^*$, for all $x_1,\dots,x_m \in \R$, for all $k_1, \dots,k_m \in \{0,\dots,p\}$, the random vector $(f^{(k_j)}(x_j))_{1 \leq j \leq m}$ is a centered Gaussian vector in $\R^m$. Let us denote by $\partial_1$ (resp.~$\partial_2$) the partial derivative with respect to the first (resp.~second) variable for functions from $\R^2$ to $\R$. For all $k$ and $l$ in $\{0,\dots,p\}$, the partial derivative $\partial_1^k\partial_2^l K$ is well-defined and continuous on $\R^2$. Moreover, $\kappa$ is of class $\mathcal{C}^{2p}$ and, for all $k,l \in \{0,\dots,p\}$, for all $x,y \in \R$, we have:
\begin{equation}
\label{eq covariance}
\esp{f^{(k)}(x)f^{(l)}(y)} = \partial_1^k\partial_2^lK(x,y) = (-1)^k\kappa^{(k+l)}(y-x).
\end{equation}
In particular, the variance matrix of $(f^{(k_j)}(x_j))_{1 \leq j \leq m}$ is $\begin{pmatrix}\partial_1^{k_i}\partial_2^{k_j}K(x_i,x_j)\end{pmatrix}_{1 \leq i,j \leq m}$. This material is standard. We refer the interested reader to~\cite[Appendix~A.2 and~A.3]{NS2016} for more details.

Let us now assume that $f$ is a $\mathcal{C}^1$-process. If $\kappa(0)=0$, then for all $x \in \R$, almost surely $f(x)=0$. Then, almost surely, $f$ is continuous and for all $x \in \Q$, $f(x)=0$. Hence $f$ is almost surely the zero function. Similarly, if $\kappa''(0)=0$ then $f'$ is almost surely the zero function. Hence $f$ is almost surely constant, equal to $f(0) \sim \mathcal{N}(0,\kappa(0))$. These degenerate situations are well-understood, and we will not consider them in the following. That is, we assume that $\var{f(0)}=\kappa(0)>0$ and $\var{f'(0)} = -\kappa''(0) >0$. Without loss of generality, up to replacing $f$ by:
\begin{equation*}
x \longmapsto \frac{1}{\sqrt{\kappa(0)}}f\left(\sqrt{-\frac{\kappa(0)}{\kappa''(0)}}x\right),
\end{equation*}
we may assume that $\kappa(0)=1=-\kappa''(0)$.

\begin{dfn}[Normalization]
\label{def normalized process}
We say that a stationary centered Gaussian process $f$ of class $\mathcal{C}^1$ is \emph{normalized} if its correlation function $\kappa$ satisfies $\kappa(0)=1=-\kappa''(0)$.
\end{dfn}

\emph{In the rest of this paper, unless otherwise specified, the random process $f$ is always assumed to be a normalized stationary centered Gaussian process at least of class $\mathcal{C}^1$.}

Recall that, in Theorems~\ref{thm moments} and~\ref{thm clustering}, we consider a normalized Gaussian $\mathcal{C}^k$-process $f$ whose correlation function $\kappa$ satisfies some form of decay at infinity, as well as its first derivatives. In the remainder of this section, we discuss these conditions. Let us first check that they make sense. Let $l \in \{0,\dots,k\}$, for all $x \in \R$ we have:
\begin{equation*}
\norm{\kappa^{(2l)}(x)} = \norm{\esp{f^{(l)}(0)f^{(l)}(x)}} \leq \esp{f^{(l)}(0)^2}^\frac{1}{2} \esp{f^{(l)}(x)^2}^\frac{1}{2} \leq \kappa^{(2l)}(0),
\end{equation*}
and, if $l < k$,
\begin{equation*}
\norm{\kappa^{(2l+1)}(x)} = \norm{\esp{f^{(l+1)}(0)f^{(l)}(x)}} \leq \esp{f^{(l+1)}(0)^2}^\frac{1}{2} \esp{f^{(l)}(x)^2}^\frac{1}{2} \leq \left(\kappa^{(2l+2)}(0)\kappa^{(2l)}(0)\right)^\frac{1}{2}.
\end{equation*}
Hence, $\kappa$ and all its derivatives of order at most $2k$ are bounded on $\R$. Recalling Notation~\ref{ntn norm kappa}, this means that $\Norm{\kappa}_{l,\eta}$ is well-defined for any $l \in \{0,\dots,2k\}$ and $\eta \geq 0$. Moreover, the previous inequalities show that $\Norm{\kappa}_{2k} = \max \left\{\kappa^{(2l)}(0) \mvert 0 \leq l \leq k \right\}$. Note that asking for $\Norm{\kappa}_{k,\eta}$ to decay at some rate as $\eta \to +\infty$, is just a way to require that $\kappa$ and all its derivatives of order at most $k$ decay at said rate at infinity. For example, taking into account the parity of $\kappa$, the hypothesis that $\Norm{\kappa}_{k,\eta} \xrightarrow[\eta \to +\infty]{} 0$ in Theorem~\ref{thm clustering} is equivalent to asking that $\kappa^{(k)}(x) \xrightarrow[x \to +\infty]{}0$ for all $l \in \{0,\dots,k\}$.

The fact that $\kappa$ tends to $0$ infinity ensures the non-degeneracy of the finite-dimensional marginal distributions of the process $f$. Let us make this statement precise.

\begin{lem}[Non-degeneracy of the marginals]
\label{lem non-degeneracy}
Let $p \in \N$ and let $f$ be a stationary centered Gaussian process of class $\mathcal{C}^p$ whose correlation function is denoted by $\kappa$. Let us assume that $\kappa(x) \xrightarrow[x \to +\infty]{}0$. Let $m \in \N^*$, let $x_1,\dots,x_m \in \R$ and let $k_1,\dots,k_m \in \{0,\dots,p\}$ be such that the couples $((x_j,k_j))_{1 \leq j \leq m}$ are pairwise distinct. Then, the random vector $\left(f^{(k_j)}(x_j)\right)_{1 \leq j \leq m}$ is a non-degenerate centered Gaussian vector in $\R^m$.
\end{lem}

\begin{proof}
Let us just sketch the proof here. The details are given in Appendix~\ref{subsec non-degeneracy spectral measure and ergodicity}. The condition that~$\kappa$ tends to $0$ at infinity implies that the process $f$ is ergodic, which is equivalent to the fact that its spectral measure has no atom. In particular, the spectral measure of $f$ has an accumulation point. This condition is enough to ensure the non-degeneracy of $\left(f^{(k_j)}(x_j)\right)_{1 \leq j \leq m}$ as soon as the couples $((x_j,k_j))_{1 \leq j \leq m}$ are pairwise distinct.
\end{proof}

We conclude this section by making a few remarks about the content of this section and its relation to the hypotheses of Theorem~\ref{thm vanishing order}.

\begin{rems}
\label{rems normalized and non-degenerate}
Let $f$ be a normalized stationary centered Gaussian process and let $\kappa$ denote its correlation function.
\begin{itemize}
\item Since $\kappa$ is even, $\kappa'(0)=0$. In particular, for all $x \in \R$, the random vector $(f(x),f'(x))$ is a standard Gaussian vector in $\R^2$. That is, for all $x \in \R$, $f(x)$ and $f'(x)$ are independent $\mathcal{N}(0,1)$ variables.
\item Let $x,y \in \R$ be such that $x \neq y$, the determinant of the variance matrix of $(f(x),f(y))$ equals $1 - \kappa(y-x)^2$. Hence, this Gaussian vector is degenerate if and only if $\norm{\kappa(y-x)}=1$. By Cauchy--Schwarz's inequality, we have $\norm{\kappa(x)} \leq \kappa(0)=1$ for all $x \in \R$. Thus, for $k=2$, the first non-degeneracy condition in Theorem~\ref{thm vanishing order} is equivalent to the fact that $\norm{\kappa(x)} < 1$ for any $x \neq 0$.
\item Let $k \in \N^*$, if $f$ is of class $\mathcal{C}^k$ then, by Lemma~\ref{lem non-degeneracy}, the fact $\kappa(x) \xrightarrow[x \to +\infty]{}0$ is enough to ensure that $f$ satisfies the hypotheses of Theorem~\ref{thm vanishing order} at any point $y \in \R^k$. This condition is sufficient but not necessary, see Lemma~\ref{lem spectral measure and non-degeneracy} below.
\end{itemize}
\end{rems}


\subsection{Zeros of stationary Gaussian processes}
\label{subsec zeros of stationary Gaussian processes}

Let us now introduce more precisely the random sets we study. Let $f$ be a normalized centered stationary Gaussian process and let us denote by $Z= f^{-1}(0)$ its vanishing locus.

\begin{lem}
\label{lem Z as closed discrete}
Let $f: \R \to \R$ be a normalized centered stationary Gaussian process and let $Z= f^{-1}(0)$. Then, almost surely, $Z$ is a closed discrete subset of $\R$.
\end{lem}

\begin{proof}
The process $f$ is almost surely of class $\mathcal{C}^1$. By Bulinskaya's Lemma (see~\cite[Proposition.~1.20]{AW2009}), since $f(x) \sim \mathcal{N}(0,1)$ for all $x \in \R$, we have that $f$ vanishes transversally almost surely. That is, almost surely, for all $x \in \R$ such that $f(x) = 0$ we have $f'(x) \neq 0$. Then, $Z$ is almost surely a closed $0$-dimensional submanifold of $\R$. Equivalently, $Z$ is almost surely a closed discrete subset of $\R$.
\end{proof}

\begin{dfn}
\label{def ZR nuR}
Let $R >0$.
\begin{itemize}
\item We set $Z_R= \frac{1}{R}Z = \{ x \in \R \mid f(Rx) = 0\}$.
\item Let $\nu_R = \sum_{x \in Z_R} \delta_x$ (resp.~$\nu = \sum_{x \in Z} \delta_x$) denote the counting measure of $Z_R$ (resp.~$Z$).
\item As in Definition~\ref{def nu A}, for any non-empty finite set $A$, we denote by $\nu^A$ (resp.~$\nu^{[A]}$) the counting measure of the random set $Z^A$ (resp.~$Z^A \setminus \Delta_A$), .
\end{itemize}
\end{dfn}

In this paper, we study the asymptotic distribution of $\nu_R$ as $R \to +\infty$ through the asymptotics of its linear statistics $\prsc{\nu_R}{\phi}$, where $\phi:\R \to \R$ is a nice enough function.

\begin{ntns}
\label{ntn phiR and indicator}
As in Section~\ref{sec introduction}, we will use the following notations.
\begin{itemize}
\item Let $\Phi:\R^A \to \R$, for any $R>0$ we set $\Phi_R:x \mapsto \phi(\frac{x}{R})$. In particular, if $\phi:\R \to \R$, we have $\prsc{\nu_R}{\phi} = \prsc{\nu}{\phi_R}$.
\item Let $A$ be a subset of some set $B$, we denote by $\mathbf{1}_A:B \to \R$ the indicator function of $A$. For example, if $A \subset \R$, we have $\card(Z \cap A) = \prsc{\nu}{\mathbf{1}_A}$.
\end{itemize}
\end{ntns}

\begin{rem}
\label{rem factorial moments}
Let $k \in \N^*$, then $\nu^{[k]}$ is the counting measure of $Z^k \setminus \Delta_k$. Let $B$ be a Borel subset of $\R$, we denote by $\mathcal{N}_B = \card(Z \cap B) = \prsc{\nu}{\mathbf{1}_B}$. The $k$-th \emph{factorial moment} of~$\mathcal{N}_B$ is defined as the expectation of $\mathcal{N}_B^{[k]}=\mathcal{N}_B(\mathcal{N}_B-1) \cdots (\mathcal{N}_B-k+1)$. As explained in~\cite[p.~58]{AW2009}, we have $\mathcal{N}_B^{[k]} = \prsc{\nu^{[k]}}{\boxtimes_{i=1}^k \mathbf{1}_B} = \prsc{\nu^{[k]}}{\mathbf{1}_{B^k}}$, hence $\esp{\prsc{\nu^{[k]}}{\mathbf{1}_{B^k}}}$ is the $k$-th factorial moment of $\card(Z \cap B)$. We will see below that this quantities are well-defined in $[0,+\infty]$.
\end{rem}


\section{Kac--Rice formulas and mean number of zeros}
\label{sec Kac-Rice formulas and mean number of zeros}

In this section, we state the so-called Kac--Rice formulas, that are one of the tools in the proofs of Theorem~\ref{thm moments} and Propositions~\ref{prop expectation} and~\ref{prop variance}. The Kac--Rice formulas are recalled in Section~\ref{subsec Kac-Rice formulas}. In Section~\ref{subsec Kac-Rice density and k point functions}, we related the Kac--Rice density introduced in Definition~\ref{def Kac-Rice densities} with the $k$-point function of the random set $Z=f^{-1}(0)$ defined in Section~\ref{subsec zeros of stationary Gaussian processes}. Then, in Section~\ref{subsec proof of proposition expectation}, we prove Proposition~\ref{prop expectation}.


\subsection{Kac--Rice formulas}
\label{subsec Kac-Rice formulas}

In this section, we recall the Kac--Rice formulas (see Proposition~\ref{prop Kac-Rice formula}). A standard reference for this material is \cite[Chapters~3 and~6]{AW2009}, see also~\cite[Chapter~11]{AT2007}. Note however that formulas of this kind already appear in the work of Cramér and Leadbetter~\cite{CL1965}.

First, we need to introduce the Kac--Rice densities associated with a non-degenerate Gaussian process of class $\mathcal{C}^1$. 

\begin{dfn}[Kac--Rice densities]
\label{def Kac-Rice densities}
Let $f$ be a centered Gaussian $\mathcal{C}^1$-process. Let $k \in \N^*$ and let $x=(x_i)_{1 \leq i \leq k} \in \R^k$. We denote by
\begin{equation}
\label{eq def Dkx}
D_k(x) = \det\left(\var{f(x_1),\dots,f(x_k)}\right).
\end{equation}
If $\left(f(x_1),\dots,f(x_k)\right)$ is non-degenerate, i.e.~if $D_k(x) \neq 0$, we denote by
\begin{equation}
\label{eq def Nkx}
N_k(x) = \espcond{\prod_{i=1}^k \norm{f'(x_i)}}{\forall i \in \{1,\dots,k\}, f(x_i)=0},
\end{equation}
the conditional expectation of $\norm{f'(x_1)}\cdots \norm{f'(x_k)}$ given that $f(x_1)= \dots = f(x_k)=0$, and by
\begin{equation}
\label{eq def rho kx}
\rho_k(x) = \frac{N_k(x)}{(2\pi)^\frac{k}{2}D_k(x)^\frac{1}{2}}.
\end{equation}
We refer to $\rho_k$ as the \emph{Kac--Rice density} of order $k$ associated with $f$.
\end{dfn}

\begin{rem}
\label{rem Kac-Rice densities}
By Lemma~\ref{lem non-degeneracy}, if $\kappa(x) \xrightarrow[x \to +\infty]{}0$ then, for all $k \in \N^*$ the Kac--Rice density $\rho_k$ is well-defined on $\R^k \setminus \Delta_k$. Note however that $D_k$ always vanishes along $\Delta_k$.
\end{rem}

\begin{ex}
\label{ex Kac-Rice densities}
Let $f$ be a normalized Gaussian process (see Definition~\ref{def normalized process})
\begin{itemize}
\item For all $x \in \R$, $f(x)$ and $f'(x)$ are independent $\mathcal{N}(0,1)$ variables (see Remark~\ref{rems normalized and non-degenerate}). Hence, $D_1(x) = \var{f(x)}=1$ and $N_1(x) =\esp{\norm{f'(x)}}=\sqrt{\frac{2}{\pi}}$. Thus, $\rho_1$ is constant equal to $\frac{1}{\pi}$.
\item Let $\kappa$ denote the correlation function of $f$. For all $(x,y) \in \R^2$, we have $D_2(x,y) = 1 - \kappa(y-x)^2$. Hence $\rho_2$ is well-defined on $\R^2 \setminus \Delta_2$ if and only if $\norm{\kappa(x)} < 1$ for all $x \in \R \setminus \{0\}$.
\end{itemize}
\end{ex}

\begin{ntn}[Symmetric group]
\label{ntn permutation}
Let $A$ be a non-empty finite set, we denote by $\mathfrak{S}_A$ the group of permutations of~$A$. For all $\sigma \in \mathfrak{S}_A$ and $x=(x_a)_{a \in A} \in \R^A$, we denote by $\sigma \cdot x = (x_{\sigma(a)})_{a \in A}$. If $A=\{1,\dots,k\}$, we denote $\mathfrak{S}_k = \mathfrak{S}_{A}$ for simplicity.
\end{ntn}

\begin{lem}[Symmetry]
\label{lem Kac-Rice densities symmetric}
Let $k \in \N^*$, we have $D_k(\sigma \cdot x) = D_k(x)$ for all $x \in \R^k$. Moreover, if $D_k(x) \neq 0$, then $N_k(\sigma \cdot x) = N_k(x)$ and $\rho_k(\sigma \cdot x) = \rho_k(x)$.
\end{lem}

\begin{proof}
Let $\sigma \in \mathfrak{S}_k$ and let $\Sigma$ denote the linear map $x \mapsto \sigma \cdot x$. For all $x \in \R^k$, we have:
\begin{equation*}
D_k(\sigma \cdot x) = \det \var{f(x_{\sigma(1)}),\dots,f(x_{\sigma(k)})} = \det \var{\Sigma \left(f(x_1),\dots,f(x_k)\right)} = \det(\Sigma)^2 D_k(x).
\end{equation*}
Since the matrix of $\Sigma$ in the canonical basis of $\R^k$ is a permutation matrix, $\det(\Sigma)^2=1$. This proves that $D_k$ is symmetric on $\R^k$.

If $D_k(x) \neq 0$, the first point shows that $D_k(\sigma \cdot x) \neq 0$, so that $N_k(x)$, $N_k(\sigma \cdot x)$, $\rho_k(x)$ and $\rho_k(\sigma \cdot x)$ are well-defined. To conclude the proof it is enough to check that $N_k(\sigma \cdot x) = N_k(x)$. This follows from the definition of $N_k$, see Equation~\eqref{eq def Nkx}.
\end{proof}

We can now state the Kac--Rice formula itself.

\begin{prop}[Kac--Rice formula]
\label{prop Kac-Rice formula}
Let $f$ be a centered Gaussian process of class $\mathcal{C}^1$ and let $Z$ denote its zero set. Let $k \in \N^*$ and let $\nu^{[k]}$ be the counting measure of $Z^k \setminus \Delta_k$. Let $U$ be an open subset of $\R^k$ such that, for all $x \in U \setminus \Delta_k$, $D_k(x) \neq 0$ (cf.~Definition~\ref{def Kac-Rice densities}). Let $\Phi:\R^k \to \R$ be a Borel function supported in $U$ satisfying one of the following conditions:
\begin{itemize}
\item the function $\Phi$ is non-negative;
\item the function $\Phi\rho_k$ is Lebesgue-integrable on $\R^k$;
\item the random variable $\prsc{\nu^{[k]}}{\Phi}$ is integrable.
\end{itemize}
Then we have:
\begin{equation*}
\esp{\prsc{\nu^{[k]}}{\Phi}} = \int_{x \in \R^k} \Phi(x) \rho_k(x) \dx x,
\end{equation*}
where $\dx x$ denote the Lebesgue measure of $\R^k$.
\end{prop}

\begin{proof}
We refer to~\cite[Theorem~3.2]{AW2009} for a proof of this result (see also \cite[Theorem~6.2 and 6.3]{AW2009}). Our statement of the Kac--Rice formula differs from those that can be found in~\cite{AW2009}. Let us comment upon the differences.

In~\cite{AW2009}, the authors are concerned with the so-called factorial moments of the number of zeros of~$f$ in some Borel set $B \subset \R$. As we already explained in Remark~\ref{rem factorial moments}, the $k$-th factorial moment of $\card(Z \cap B)$ is $\esp{\prsc{\nu^{[k]}}{\mathbf{1}_{B^k}}}$. Hence, Azaïs and Wschebor state and prove Proposition~\ref{prop Kac-Rice formula} in the case where $\Phi$ is the indicator function $\mathbf{1}_{B^k}$, where $B$ is an interval in~\cite[Theorem~3.2]{AW2009} and a Borel set in~\cite[Theorem~6.2 and 6.3]{AW2009}. Their proofs can be adapted to deal with a Borel function~$\Phi$. Alternatively, once the result is proved for the indicator function of a Borel set, it also holds for simple functions. Then, we conclude by approximating the positive and negative part of $\Phi$ by simple functions and applying Beppo Levi's Monotone Convergence Theorem.
\end{proof}

\begin{rem}
\label{rem U in Kac-Rice formula}
The only place where we use the Kac--Rice formula with $U \neq \R^k$ is the proof of Lemma~\ref{lem k point function}, where we prove that $\rho_k$ coincides with the $k$-point function of $Z$.
\end{rem}

\begin{rem}
\label{rem Kac-Rice formula}
We prove below that, if $f$ is of class $\mathcal{C}^k$ and its correlation function $\kappa$ is such that $\Norm{\kappa}_{k,\eta}$ tends to $0$ as $\eta \to +\infty$, then the function $\rho_k$ is bounded (see Theorem~\ref{thm clustering}). In this case, the second condition in Proposition~\ref{prop Kac-Rice formula} can be replaced by the Lebesgue-integrability of $\Phi$ on $\R^k$. In particular, this implies that for any integrable function $\Phi:\R^k \to \R$ the random variable $\prsc{\nu^{[k]}}{\Phi}$ is almost surely well-defined.
\end{rem}


\subsection{Kac--Rice density and \texorpdfstring{$k$}{}-point functions}
\label{subsec Kac-Rice density and k point functions}

In this section, we show that the Kac--Rice density $\rho_k$ introduced in Definition~\ref{def Kac-Rice densities} is in fact the $k$-point function of the point process $Z = f^{-1}(0)$ introduced in Section~\ref{subsec zeros of stationary Gaussian processes}. First, we need to prove the continuity of $\rho_k$.

\begin{lem}[Continuity]
\label{lem Dk and Nk continuous}
Let $f$ be a centered Gaussian process of class $\mathcal{C}^1$. For all $k \in \N^*$, the maps $D_k$, $N_k$ and $\rho_k$ appearing in Definition~\ref{def Kac-Rice densities} are continuous on their domains of definition.
\end{lem}

\begin{proof}
Let $k \in \N^*$, for all $x =(x_i)_{1 \leq i \leq k} \in \R^k$, let us denote by $X_k(x) =(f(x_i))_{1 \leq i \leq k}$ and $Y_k(x) =(f'(x_i))_{1 \leq i \leq k}$. Then, $(X_k(x),Y_k(x))_{x \in \R^k}$ is a continuous centered Gaussian field with values in $\R^{2k}$. We write the variance matrix of $(X_k(x),Y_k(x))$ by square blocks of size $k$ as:
\begin{equation*}
\begin{pmatrix}
\Theta_k(x) & \trans{\Xi_k(x)} \\ \Xi_k(x) & \Omega_k(x)
\end{pmatrix},
\end{equation*}
where $\Theta_k$, $\Xi_k$ and $\Omega_k$ are continuous maps on~$\R^k$. Then, $D_k = \det(\Theta_k)$ is continuous on $\R^k$.

If $x \in \R^k$ is such that $D_k(x) \neq 0$, then $Y_k(x)$ given that $X_k(x) = 0$ is a well-defined centered Gaussian vector of variance matrix $\Lambda_k(x) = \Omega_k(x) - \Xi_k(x)\Theta_k(x)^{-1} \trans{\Xi_k(x)}$ (see~\cite[Proposition~1.2]{AW2009}). Note that $\Lambda_k$ is continuous on $\{x \in \R^k \mid D_k(x) \neq 0\}$. Then,
\begin{equation*}
N_k(x) = \esp{\prod_{i=1}^k \norm{Z_i(x)}},
\end{equation*}
where $(Z_1(x),\dots,Z_k(x)) \sim \mathcal{N}(0,\Lambda_k(x))$. That is, $N_k(x) = \Pi_k(\Lambda_k(x))$, where $\Pi_k$ is the map defined in Definition~\ref{def Pi k}. Since $\Pi_k$ is continuous (see Corollary~\ref{cor Pi k}), the function $N_k$ is continuous on $\{x \in \R^k \mid D_k(x) \neq 0\}$, and so is $\rho_k$.
\end{proof}

Let us now consider a normalized centered stationary Gaussian process $f$ which is $\mathcal{C}^1$. By Lemma~\ref{lem Z as closed discrete}, its zero set $Z$ is a discrete closed subset of $\R$ almost surely. That is $Z$ is random point process in $\R$.

\begin{dfn}[$k$-point function]
\label{def k point function}
Let $x=(x_i)_{1 \leq i \leq k} \in \R^k \setminus \Delta_k$, the value at $x$ of the \emph{$k$-point function} of a random point process $Z$ is defined as:
\begin{equation*}
\lim_{\epsilon \to 0} \frac{1}{(2\epsilon)^k} \esp{\prod_{i=1}^k \card \left(Z \cap [x_i-\epsilon,x_i+\epsilon]\right)},
\end{equation*}
if this limit is well-defined.
\end{dfn}

We can now make precise our claim that $\rho_k$ is the $k$-point function of $Z$.

\begin{lem}
\label{lem k point function}
Let $k \in \N^*$ and let $f$ be a normalized centered stationary Gaussian $\mathcal{C}^1$-process. Let us denote by $Z$ the vanishing locus of $f$. Then, for all $x =(x_i)_{1 \leq i \leq k} \in \R^k$ such that $D_k(x)\neq 0$, we have:
\begin{equation*}
\frac{1}{(2\epsilon)^k} \esp{\prod_{i=1}^k \card\left(Z \cap [x_i-\epsilon,x_i+\epsilon]\right)} \xrightarrow[\epsilon \to 0]{} \rho_k(x),
\end{equation*}
where $\rho_k$ is the function appearing in Definition~\ref{def Kac-Rice densities}.
\end{lem}

\begin{proof}
Since $D_k(x)\neq 0$, by continuity of $D_k$ (see Lemma~\ref{lem Dk and Nk continuous}) there exists a neighborhood $U$ of $x$ such that $D_k$ does not vanish on $U$. Note that this implies $U \subset \R^k \setminus \Delta_k$.

Let $\epsilon \neq 0$. Without loss of generality, we can assume that $\epsilon$ is positive, and small enough that $\prod_{i=1}^k [x_i-\epsilon,x_i+\epsilon] \subset U \subset \R^k \setminus \Delta_k$. In particular, the intervals $([x_i-\epsilon,x_i+\epsilon])_{1 \leq i \leq k}$ are pairwise disjoint. Let $C$ denote the cube $[-1,1]^k$. Using the notations introduced in Definition~\ref{def ZR nuR} and Notations~\ref{ntn phiR and indicator}, we have:
\begin{align*}
\prod_{i=1}^k \card\left(Z \cap [x_i-\epsilon,x_i+\epsilon]\right) = \card\left(Z^k \cap \left(x+\epsilon C\right)\right) = \prsc{\nu^k}{\mathbf{1}_{x+\epsilon C}} = \prsc{\nu^{[k]}}{\mathbf{1}_{x+\epsilon C}},
\end{align*}
since $x+\epsilon C$ does not intersect $\Delta_k$. The function $\rho_k$ is well-defined and continuous on $U$. Then, by the Kac--Rice formula of order $k$ (see Proposition~\ref{prop Kac-Rice formula}), we have:
\begin{equation*}
\frac{1}{(2\epsilon)^k} \esp{\prod_{i=1}^k \card\left(Z \cap [x_i-\epsilon,x_i+\epsilon]\right)} = \frac{1}{(2\epsilon)^k} \int_{x+\epsilon C} \rho_k(y) \dx y \xrightarrow[\epsilon \to 0]{} \rho_k(x),
\end{equation*}
since $x+\epsilon C$ has volume $(2\epsilon)^k$ and $\rho_k$ is continuous at $x$.
\end{proof}


\subsection{Proof of Proposition~\ref{prop expectation}: expectation of the linear statistics}
\label{subsec proof of proposition expectation}

A first application of the Kac--Rice formulas (cf.~Proposition~\ref{prop Kac-Rice formula}) is the computation of the expectation of the linear statistics $\prsc{\nu_R}{\phi}$ (see Section~\ref{subsec zeros of stationary Gaussian processes}), where $R >0$ and $\phi:\R \to \R$ is either non-negative or integrable. In this section, we address this problem and prove Proposition~\ref{prop expectation}.

\begin{proof}[Proof of Proposition~\ref{prop expectation}]
Let $R>0$ and let $\phi: \R \to \R$ be non-negative or integrable. By definition of $\nu_R$ and $\phi_R$ (see Section~\ref{subsec zeros of stationary Gaussian processes}), we have $\esp{\prsc{\nu_R}{\phi}} = \esp{\prsc{\nu}{\phi_R}}$. We apply the Kac--Rice formula for $k=1$, bearing in mind that $\rho_1$ is constant equal to $\frac{1}{\pi}$ (see Example~\ref{ex Kac-Rice densities}). We obtain:
\begin{equation*}
\esp{\prsc{\nu_R}{\phi}}=\esp{\prsc{\nu}{\phi_R}} = \esp{\prsc{\nu^{[1]}}{\phi_R}} = \frac{1}{\pi} \int_{-\infty}^{+\infty} \phi_R(x)\dx x = \frac{R}{\pi} \int_{-\infty}^{+\infty} \phi(x) \dx x.
\end{equation*}
For all $\phi \in \mathcal{C}^0_c(\R)$ we have: $\prsc{\esp{\nu_R}}{\phi} = \esp{\prsc{\nu_R}{\phi}} = \prsc{\frac{R}{\pi}\dx x}{\phi}$. Hence, $\esp{\nu_R} = \frac{R}{\pi}\dx x$.
\end{proof}

As explained in Remark~\ref{rem as well defined}, applying Proposition~\ref{prop expectation} for the positive function $\norm{\phi}$ allows to prove that, if $\phi$ is integrable then, for all $R >0$, $\prsc{\nu_R}{\phi}$ is almost surely well-defined. We can do a bit better than that. For example, let $\mathcal{E}$ denote the following space of functions:
\begin{equation*}
\mathcal{E} = \left\{ \phi:\R \to \R \mvert \phi \ \text{is Lebesgue-measurable and} \ \exists C >0, \exists \alpha >1, \forall x \in \R, \norm{\phi(x)} \leq \frac{C}{1 + \norm{x}^\alpha} \right\}.
\end{equation*}
For all $C>0$ and $\alpha > 1$, we denote by $\psi_{C,\alpha}: x \mapsto \frac{C}{1+\norm{x}^\alpha}$, from $\R$ to $\R$. By Proposition~\ref{prop expectation}, almost surely, for all $C >0$ and $\alpha >1$ such that $C \in \Q$ and $\alpha \in \Q$, we have $\prsc{\nu}{\psi_{C,\alpha}} <+\infty$. Hence, almost surely, for all $\phi \in \mathcal{E}$, we have $\prsc{\nu}{\norm{\phi}} < +\infty$. A function $\phi$ belongs to $\mathcal{E}$ if and only if it is bounded and $\phi(x) = O(\norm{x}^{-\alpha})$ as $\norm{x} \to +\infty$, for some $\alpha >1$. Thus, if $\phi \in \mathcal{E}$, then $\phi_R \in \mathcal{E}$ for all $R>0$. Finally, we obtain that, almost surely, for all $\phi \in \mathcal{E}$, for all $R >0$, we have $\prsc{\nu}{\norm{\phi_R}} <+\infty$, i.e.~$\prsc{\nu_R}{\phi}=\prsc{\nu}{\phi_R}$ is well-defined. Of course, in this example, the family $\{\psi_{C,\alpha}\mid C>0,\alpha>1\}$ can be replaced by any countable family of non-negative integrable functions. The same idea shows that $\nu_R$ almost surely defines a tempered distribution.

\begin{lem}
\label{lem nu R as tempered distribution}
Using the same notations as in Proposition~\ref{prop expectation}, almost surely, for all $R >0$ we have $\nu_R \in \mathcal{S}'(\R)$.
\end{lem}

\begin{proof}
The definitions of $\mathcal{S}(\R)$ and $\mathcal{S}'(\R)$ were recalled in Definition~\ref{def SR}. For all $\phi \in \mathcal{S}(\R)$, we denote by $C(\phi) = \max_{x \in \R^n} \norm{(1+x^2)\phi(x)}$. Note that $\phi \mapsto C(\phi)$ is one of the norms defining the topology of $\mathcal{S}(\R)$. In particular, $C(\phi) \to 0$ as $\phi \to 0$ in $\mathcal{S}(\R)$.

Let $\psi :x \mapsto \frac{1}{1+x^2}$. By Proposition~\ref{prop expectation}, almost surely $\prsc{\nu}{\psi} < +\infty$. Let us consider a fixed realization of $\nu$ in the full probability event such that $\prsc{\nu}{\psi} <+\infty$. For all $R >0$ and all $\phi \in \mathcal{S}(\R)$ we have:
\begin{equation*}
\prsc{\nu_R}{\norm{\phi}} = \prsc{\nu}{\norm{\phi_R}} \leq C(\phi_R) \prsc{\nu}{\psi} \leq C(\phi) R^2 \prsc{\nu}{\psi}.
\end{equation*}
On the one hand, this shows that $\prsc{\nu_R}{\phi}$ is well-defined. On the other hand, for all $R >0$,
\begin{equation*}
\norm{\prsc{\nu_R}{\phi}} \leq C(\phi) R^2 \prsc{\nu}{\psi} \xrightarrow[\phi \to 0]{} 0.
\end{equation*}
Thus the linear form $\phi \mapsto \prsc{\nu_R}{\phi}$ is continuous on $\mathcal{S}(\R)$, i.e.~$\nu_R \in \mathcal{S}'(\R)$.
\end{proof}


\section{Proof of Proposition~\ref{prop variance}: asymptotics of the covariances}
\label{sec proof of proposition variance}

This section is concerned with the proof of Proposition~\ref{prop variance}. In all this section, we consider a Gaussian process $f$ satisfying the hypotheses of Proposition~\ref{prop variance}, that is $f$ is a normalized stationary centered Gaussian $\mathcal{C}^2$-process. Moreover, the correlation function~$\kappa$ of $f$ is such that $\kappa$ and $\kappa''$ are square-integrable functions that tend to $0$ at infinity.

First, in Section~\ref{subsec asymptotics of the covariances}, we prove that the asymptotics given in Equation~\eqref{eq asymp m2} holds. Then, we prove the positivity of the constant $\sigma$ (see~Equation~\eqref{eq def sigma}) in Section~\ref{subsec positivity of the leading constant}.


\subsection{Asymptotics of the covariances}
\label{subsec asymptotics of the covariances}

In this section, we prove that Equation~\eqref{eq asymp m2} in Proposition~\ref{prop variance} holds. The content of this section is close to what can already be found in the literature, for example in the work of Cuzick~\cite{Cuz1976}. The main difference is that we added test-functions $\phi_1$ and $\phi_2$ in Equation~\eqref{eq asymp m2}, where other authors generally consider the case $\phi_1=\phi_2=\mathbf{1}_{[0,1]}$. However, some of the notations and auxiliary results of this section will also be used in the proof of Theorem~\ref{thm moments} (see Section~\ref{subsec conclusion of the proof}). Besides, the proof of~\eqref{eq asymp m2} is a good toy-model for the proof of Theorem~\ref{thm moments}, which is another reason to write it in full here.

We first introduce a density function~$F$ (see Definition~\ref{def F}) and state some of its properties in Lemmas~\ref{lem expression F} and~\ref{lem integrability F}. The proofs of these lemmas are postponed until Appendix~\ref{sec properties of the density function F}. Then we establish Equation~\eqref{eq asymp m2}.

Since $\kappa$ tends to $0$ at infinity, by Lemma~\ref{lem non-degeneracy}, the Kac--Rice density $\rho_2$ is well-defined on $\R^2 \setminus \Delta_2$ (see Remark~\ref{rem Kac-Rice densities}). Moreover, since $f$ is stationary, we have $\rho_2(x,x+z) = \rho_2(0,z)$ for all $z \neq 0$ (see Definition~\ref{def Kac-Rice densities}).

\begin{dfn}
\label{def F}
We denote by $F:z \mapsto \rho_2(0,z)-\frac{1}{\pi^2}$ from $\R \setminus \{0\}$ to $\R$.
\end{dfn}

Note that, for all $x \neq y$, we have $\rho_2(x,y)-\rho_1(x)\rho_1(y) = F(y-x)$. It is possible to compute a somewhat more explicit expression of $F$.

\begin{lem}
\label{lem expression F}
For all $z >0$, we have:
\begin{equation*}
F(z) = F(-z) = \frac{1}{\pi^2} \left(\frac{1-\kappa(z)^2-\kappa'(z)^2}{\left(1-\kappa(z)^2\right)^\frac{3}{2}}\left(\sqrt{1-a(z)^2} + a(z) \arcsin(a(z))\right) - 1\right),
\end{equation*}
where $a(z) = \dfrac{\kappa(z)\kappa'(z)^2 - \kappa(z)^2\kappa''(z) + \kappa''(z)}{1 - \kappa(z)^2 -\kappa'(z)^2} \in [-1,1]$.
\end{lem}

\begin{proof}
See Appendix~\ref{subsec proof of Lemma expression F}.
\end{proof}

\begin{lem}
\label{lem integrability F}
Under the hypotheses of Proposition~\ref{prop variance}, we have:
\begin{align*}
F(z) & \xrightarrow[z \to 0]{} -\frac{1}{\pi^2} & &\text{and} & F(z) & \xrightarrow[\norm{z} \to +\infty]{} 0.
\end{align*}
Moreover, the function $F$ is Lebesgue-integrable on $\R$.
\end{lem}

\begin{proof}
See Appendix~\ref{subsec proof of Lemma integrability F}.
\end{proof}

Assuming that Lemmas~\ref{lem expression F} and~\ref{lem integrability F} hold, we can now prove the first part of Proposition~\ref{prop variance}. An important step is the following lemma, which will also appear in the proof of Theorem~\ref{thm moments}.

\begin{lem}
\label{lem formula m2}
Under the hypotheses of Proposition~\ref{prop variance}, for all $R>0$ we have:
\begin{equation*}
m_2(\nu_R)(\phi_1,\phi_2) = \int_{\R^2} \phi_1\!\left(\frac{x}{R}\right)\phi_2\!\left(\frac{y}{R}\right)F(y-x)\dx x \dx y + \frac{R}{\pi}\int_\R \phi_1(x)\phi_2(x) \dx x,
\end{equation*}
where $F$ is the function introduced in Definition~\ref{def F}.
\end{lem}

\begin{proof}
Let $R>0$ and let $\phi_1$ and $\phi_2$ be two Lebesgue-integrable functions such that $\phi_2$ is essentially bounded and continuous almost everywhere. Note that $\phi_1\phi_2$ is integrable. By Remark~\ref{rem as well defined}, the random variables $\prsc{\nu_R}{\phi_1}$, $\prsc{\nu_R}{\phi_2}$ and $\prsc{\nu_R}{\phi_1\phi_2}$ are almost surely well-defined and integrable. Using the Notations~\ref{ntn product indexed by A}, we have $\phi_R = (\phi_1)_R \boxtimes (\phi_2)_R$ and:
\begin{align*}
m_2(\nu_R)(\phi_1,\phi_2) &= \esp{\prsc{\nu_R}{\phi_1}\prsc{\nu_R}{\phi_2}} - \esp{\prsc{\nu_R}{\phi_1}}\esp{\prsc{\nu_R}{\phi_2}}\\
&= \esp{\prsc{\nu^2}{\phi_R}} - \esp{\prsc{\nu}{(\phi_1)_R}}\esp{\prsc{\nu}{(\phi_2)_R}}\\
&= \esp{\prsc{\nu^{[2]}}{\phi_R}} + \esp{\prsc{\nu}{(\phi_1\phi_2)_R}} - \esp{\prsc{\nu}{(\phi_1)_R}}\esp{\prsc{\nu}{(\phi_2)_R}}.
\end{align*}
By Proposition~\ref{prop expectation}, the middle term in this expression is $\esp{\prsc{\nu_R}{\phi_1\phi_2}}=\frac{R}{\pi} \int_\R \phi_1(x)\phi_2(x) \dx x$. We compute the other two terms by the Kac--Rice formulas of order $1$ and $2$. By Lemma~\ref{lem Dk and Nk continuous}, $\rho_2$ is continuous on $\R^2 \setminus \Delta_2$. By Lemma~\ref{lem integrability F} and Definition~\ref{def F}, the function $\rho_2$ is bounded on $\R^2 \setminus \Delta_2$. Thus, $\phi_R \rho_2$ is Lebesgue-integrable on $\R^2$. Then, by Lemma~\ref{lem non-degeneracy}, the hypotheses of Proposition~\ref{prop Kac-Rice formula} are satisfied. Recalling that $\rho_1$ is constant equal to $\frac{1}{\pi}$ (see Example~\ref{ex Kac-Rice densities}), we obtain:
\begin{align*}
\esp{\prsc{\nu^{[2]}}{\phi_R}} - \esp{\prsc{\nu}{(\phi_1)_R}}\esp{\prsc{\nu}{(\phi_2)_R}} &= \int_{\R^2} \phi_1\!\left(\frac{x}{R}\right)\phi_2\!\left(\frac{y}{R}\right)\left(\rho_2(x,y)-\rho_1(x)\rho_1(y)\right)\dx x \dx y\\
&= \int_{\R^2} \phi_1\!\left(\frac{x}{R}\right)\phi_2\!\left(\frac{y}{R}\right)F(y-x)\dx x \dx y.\qedhere
\end{align*}
\end{proof}

\begin{proof}[Proof of Equation~\eqref{eq asymp m2}]
Under the hypotheses of Proposition~\ref{prop variance}, we apply Lemma~\ref{lem formula m2}, which yields:
\begin{equation*}
m_2(\nu_R)(\phi_1,\phi_2) = \int_{\R^2} \phi_1\!\left(\frac{x}{R}\right)\phi_2\!\left(\frac{y}{R}\right)F(y-x)\dx x \dx y + \frac{R}{\pi}\int_\R \phi_1(x)\phi_2(x) \dx x.
\end{equation*}
By a change of variable, we obtain:
\begin{equation*}
\int_{\R^2} \phi_1\!\left(\frac{x}{R}\right)\phi_2\!\left(\frac{y}{R}\right)F(y-x)\dx x \dx y = R \int_{\R^2} \phi_1(x) \phi_2\!\left(x+\frac{z}{R}\right)F(z) \dx x \dx z.
\end{equation*}

Let us define $g:(x,z) \mapsto \phi_1(x) \phi_2(x)F(z)$ and $g_R:(x,z) \mapsto \phi_1(x) \phi_2\!\left(x+\frac{z}{R}\right)F(z)$ for all $R>0$. Since $\phi_2$ is continuous almost everywhere, $g_R$ simply converges toward $g$ almost everywhere on $\R^2$. Besides, for all $(x,z) \in \R^2$ we have:
\begin{equation*}
\norm{g_R(x,z)} \leq \Norm{\phi_2}_\infty \norm{\phi_1(x)}\norm{F(z)},
\end{equation*}
where $\Norm{\phi_2}_\infty$ stands for the essential supremum of $\phi_2$. Since $\phi_1$ and $F$ are integrable on $\R$ (see Lemma~\ref{lem integrability F}), the right-hand side is integrable on $\R^2$. By Lebesgue's Dominated Convergence Theorem, we get:
\begin{equation*}
\int_{\R^2} \phi_1(x) \phi_2\!\left(x+\frac{z}{R}\right)F(z) \dx x \dx z \xrightarrow[R \to +\infty]{} \left(\int_{x \in \R} \phi_1(x) \phi_2(x) \dx x\right) \left(\int_{z \in \R} F(z) \dx z\right).
\end{equation*}

Putting together everything we have done so far, as $R \to +\infty$, we have:
\begin{equation*}
m_2(\nu_R)(\phi_1,\phi_2) = R \left(\int_{-\infty}^{+\infty} \phi_1(x)\phi_2(x) \dx x\right) \left(\frac{1}{\pi}+\int_{-\infty}^{+\infty} F(z) \dx z\right) + o(R).
\end{equation*}
Finally, by Lemma~\ref{lem expression F} and Equation~\eqref{eq def sigma}, we have: $\frac{1}{\pi}+\int_{-\infty}^{+\infty} F(z) \dx z = \sigma^2$, hence the result.
\end{proof}


\subsection{Positivity of the leading constant}
\label{subsec positivity of the leading constant}

The goal of this section is to conclude the proof of Proposition~\ref{prop variance}, by proving that $\sigma^2 >0$, see Corollary~\ref{cor sigma positive} below. Recall that~$\sigma^2$ is given by Equation~\eqref{eq def sigma} and that $\sigma$ is its non-negative square root. It is not clear from its expression that $\sigma^2$ is positive. Indeed, Equation~\eqref{eq def sigma} can be rewritten (cf.~Section~\ref{subsec asymptotics of the covariances}) as:
\begin{equation*}
\sigma^2 = \frac{1}{\pi} + 2 \int_0^{+\infty} F(z) \dx z,
\end{equation*}
where $F$ is defined by Definition~\ref{def F}. The function $F$ is not non-negative since it tends to $-\frac{1}{\pi^2}$ as $z \to 0$ (see Lemma~\ref{lem integrability F}). In fact, on several examples $2\int_0^{+\infty} F(z)\dx z <0$ and we would need to compare this integral with $-\frac{1}{\pi}$ in order to deduce the positivity of $\sigma^2$ from the previous expression.

Our proof does not use Equation~\eqref{eq def sigma}, but relies on the Wiener--Itô expansion of $\prsc{\nu_R}{\mathbf{1}_{[0,1]}}$ derived by Kratz--Leòn in~\cite{KL1997}. It is not necessary to know about these Wiener--Itô expansions to understand what follows, and we refer the interested reader to~\cite{KL1997}.

\begin{prop}
\label{prop chaotic expansion}
Let $f$ be a normalized centered stationary Gaussian $\mathcal{C}^2$-process and let $Z$ denote its zero set. Then, for any $R>0$, there exists a square-integrable centered random variable $X_R$ such that:
\begin{equation*}
\card\left( Z \cap [0,R] \right) = \frac{R}{\pi} + \frac{1}{2\pi} \int_0^R f'(x)^2 - f(x)^2 \dx x + X_R.
\end{equation*}
Moreover, $\displaystyle\int_0^R f'(x)^2 - f(x)^2 \dx x$ is a square-integrable centered random variable and we have:
\begin{equation*}
\esp{\left(\int_0^R f'(x)^2 - f(x)^2 \dx x\right) X_R}=0.
\end{equation*}
\end{prop}

\begin{rem}
\label{rem chaotic expansion}
One can check that, since $f$ is normalized (see Definition~\ref{def normalized process}), we have:
\begin{equation*}
\esp{\int_0^R f(x)^2 \dx x} = \int_0^R \esp{f(x)^2} \dx x = \int_0^R \dx x = R
\end{equation*}
and, by Cauchy--Schwarz's Inequality:
\begin{equation*}
\esp{\left(\int_0^R f(x)^2\dx x\right)^2} = \int_0^R \int_0^R \esp{f(x)^2f(y)^2} \dx x \dx y \leq \left(\int_0^R \esp{f(x)^4}^\frac{1}{2} \dx x\right)^2 = 3R^2.
\end{equation*}
Thus $\int_0^R f(x)^2 \dx x$ is square-integrable of mean $R$. Similarly, $\int_0^R f'(x)^2 \dx x$ is square-integrable of mean $R$. Hence the difference of these terms is indeed square-integrable and centered.
\end{rem}

\begin{proof}[Proof of Proposition~\ref{prop chaotic expansion}]
This result is a simplified version of~\cite[Proposition~1]{KL1997}. Note that this result holds for a normalized process $f$ whose correlation function $\kappa$ satisfies $\kappa^{(4)}(0) <+\infty$ (see \cite[Condition~(1), p.~238]{KL1997}). Here, our process $f$ is of class $\mathcal{C}^2$, hence this condition is satisfied.

Let us denote by $\mathcal{N}_R = \card(Z\cap [0,R])$. Kratz and Leòn prove that an expansion of the form:
\begin{equation*}
\mathcal{N}_R = \esp{\mathcal{N}_R} + \sum_{q \geq 1} \mathcal{N}_R[q]
\end{equation*}
holds in the space of $L^2$-random variables, where the $(\mathcal{N}_R[q])_{q \geq 1}$ are uncorrelated centered random variables. Using Proposition~\ref{prop expectation} and setting $X_R = \sum_{q \geq 2} \mathcal{N}_R[q]$, we have:
\begin{equation*}
\mathcal{N}_R = \frac{R}{\pi} + \mathcal{N}_R[1] + X_R,
\end{equation*}
where $\mathcal{N}_R[1]$ and $X_R$ are centered $L^2$-random variables such that $\esp{\rule{0pt}{2ex}\mathcal{N}_R[1] \ X_R}=0$. Then, \cite[Proposition~1]{KL1997} gives an expression of $\mathcal{N}_R[q]$, for all $q \geq 0$. In particular, we have:
\begin{equation*}
\mathcal{N}_R[1] = a_0 b_2(0) \int_0^R H_2(f(x))H_0(f'(x))\dx x + a_2 b_0(0) \int_0^R H_0(f(x))H_2(f'(x))\dx x.
\end{equation*}
Here $H_0(X)=1$ and $H_2(X) = X^2-1$ are the Hermite polynomials of degree $0$ and $2$ respectively, $a_0 = \sqrt{\frac{2}{\pi}}$ and $a_2 = \frac{1}{\sqrt{2\pi}}$ by~\cite[Lemma~2]{KL1997}, $b_0(0) = \frac{1}{\sqrt{2\pi}}$ and $b_2(0) = \frac{1}{\sqrt{8\pi}}$ by~\cite[Proposition~1]{KL1997}. Finally, a direct computation yields:
\begin{equation*}
\mathcal{N}_R[1] = \frac{1}{2\pi} \int_0^R f'(x)^2 - f(x)^2 \dx x. \qedhere
\end{equation*}
\end{proof}

\begin{lem}
\label{lem lower bound variance NR1}
Let $f$ be a normalized centered stationary Gaussian $\mathcal{C}^2$-process and let $\kappa$ denote its correlation function. We assume that $\kappa$ and $\kappa''$ are square-integrable and that $\kappa(x)\kappa'(x) \to 0$ as $x \to +\infty$. Then,
\begin{equation*}
\frac{1}{R} \var{\int_0^R f'(x)^2 - f(x)^2 \dx x} \xrightarrow[R \to +\infty]{} 4 \int_0^{+\infty} \left(\kappa(x) +\kappa''(x)\right)^2 \dx x.
\end{equation*}
\end{lem}

\begin{proof}
As explained in Remark~\ref{rem chaotic expansion}, $\int_0^R f'(x)^2 - f(x)^2 \dx x$ is square-integrable and centered. For all $R >0$, we have:
\begin{multline*}
\var{\int_0^R f'(x)^2 - f(x)^2 \dx x} = \esp{\left(\int_0^R f'(x)^2 - f(x)^2 \dx x\right)^2}\\
= \int_0^R \int_0^R \esp{f'(x)^2f'(y)^2} - \esp{f'(x)^2f(y)^2} - \esp{f(x)^2f'(y)^2} + \esp{f(x)^2f(y)^2}\dx x \dx y.
\end{multline*}
By Wick's Formula (see~\cite[Lemma~11.6.1]{AT2007}), if $(X,Y)$ is a centered Gaussian vector in $\R^2$, then we have $\esp{X^2 Y^2} = \esp{X^2}\esp{Y^2} + 2 \esp{XY}^2$. For example, using the stationarity and normalization of $f$, we have:
\begin{equation*}
\esp{f(x)^2f(y)^2} = \esp{f(x)^2}\esp{f(y)^2} + 2\esp{f(x)f(y)}^2 = 1 +2\kappa(y-x)^2.
\end{equation*}
Applying Wick's Formula to $(f(x),f(y))$, $(f(x),f'(y))$, $(f'(x),f(y))$ and $(f'(x),f'(y))$ yields:
\begin{align*}
\var{\int_0^R f'(x)^2 - f(x)^2 \dx x} &= 2 \int_0^R \int_0^R \kappa''(y-x)^2 - 2\kappa'(y-x)^2 + \kappa(y-x)^2\dx x \dx y\\
&= 2 \int_0^R \left(\int_{-x}^{R-x} \kappa''(z)^2 - 2\kappa'(z)^2 + \kappa(z)^2\dx z\right) \dx x\\
&= 2R \int_0^1 \left(\int_{-Rx}^{R(1-x)}\kappa''(z)^2 - 2\kappa'(z)^2 + \kappa(z)^2\dx z\right) \dx x.
\end{align*}

Integrating by parts, we have:
\begin{equation*}
\int_{-Rx}^{R(1-x)}\kappa'(z)^2 \dx z = \kappa(R(1-x))\kappa'(R(1-x)) - \kappa(-Rx)\kappa'(-Rx) - \int_{-Rx}^{R(1-x)}\kappa(z) \kappa(z)''\dx x,
\end{equation*}
so that
\begin{multline*}
\int_{-Rx}^{R(1-x)}\kappa''(z)^2 - 2\kappa'(z)^2 + \kappa(z)^2\dx z = \int_{-Rx}^{R(1-x)}\left(\kappa(z)+ \kappa''(z)\right)^2\dx z\\
+\kappa(R(1-x))\kappa'(R(1-x)) - \kappa(-Rx)\kappa'(-Rx).
\end{multline*}
Recall that $\kappa\kappa'$ tends to $0$ at infinity and that $\kappa$ is even. Letting $R \to +\infty$ in the previous equation, we obtain for any $x \in (0,1)$:
\begin{equation*}
\int_{-Rx}^{R(1-x)}\kappa''(z)^2 - 2\kappa'(z)^2 + \kappa(z)^2\dx z \xrightarrow[R \to +\infty]{} \int_{-\infty}^{+\infty}\left(\kappa(z) + \kappa''(z)\right)^2\dx z,
\end{equation*}
where the right-hand side is finite since both $\kappa$ and $\kappa''$ are square-integrable. By Lebesgue's Dominated Convergence Theorem, we get:
\begin{equation*}
\frac{1}{R} \var{\int_0^R f'(x)^2 - f(x)^2 \dx x} \xrightarrow[R \to +\infty]{} 4\int_{0}^{+\infty}\left(\kappa(z) + \kappa''(z)\right)^2\dx z.
\end{equation*}
In this last step the dominating function is constant on $[0,1]$ equal to:
\begin{equation*}
2 \Norm{\kappa}_1^2 + \int_{-\infty}^{+\infty} (\kappa(z)+\kappa''(z))^2\dx z.\qedhere
\end{equation*}
\end{proof}

The following corollary proves the positivity of $\sigma^2$ and concludes the proof of Proposition~\ref{prop variance}.

\begin{cor}[Explicit lower bound on $\sigma^2$]
\label{cor sigma positive}
Let $f$ be a normalized centered stationary centered Gaussian $\mathcal{C}^2$-process and let $\kappa$ denote its correlation function. Under the hypotheses of Proposition~\ref{prop variance}, the constant $\sigma^2$ defined by Equation~\eqref{eq def sigma} satisfies:
\begin{equation*}
\sigma^2 \geq \frac{1}{\pi^2} \int_0^{+\infty} \left(\kappa(z) + \kappa''(z) \right)^2 \dx z >0.
\end{equation*}
\end{cor}

\begin{proof}
Let $Z$ denote the zero set of $f$ and let $\nu$ denote its counting measure, as in Section~\ref{subsec zeros of stationary Gaussian processes}. As we already said, we have: $\card(Z \cap [0,R]) = \prsc{\nu}{\mathbf{1}_{[0,R]}} = \prsc{\nu_R}{\mathbf{1}_{[0,1]}}$. We use the asymptotics given by Equation~\eqref{eq asymp m2} with $\phi_1=\phi_2 = \mathbf{1}_{[0,1]}$. Note that this asymptotics was already proved to hold, in Section~\ref{subsec asymptotics of the covariances}. Then, as $R \to +\infty$,
\begin{equation*}
m_2(\nu_R)(\mathbf{1}_{[0,1]},\mathbf{1}_{[0,1]}) = \var{\prsc{\nu_R}{\mathbf{1}_{[0,1]}}} = R\sigma^2 +o(R).
\end{equation*}
That is, $\frac{1}{R}\var{\card(Z \cap [0,R])} \xrightarrow[R \to +\infty]{} \sigma^2$.

By Proposition~\ref{prop chaotic expansion}, we have:
\begin{equation*}
\var{\card(Z \cap [0,R])} \geq \frac{1}{4\pi^2} \var{\int_0^R f'(x)^2 - f(x)^2 \dx x}.
\end{equation*}
We divide by $R$ and let $R \to +\infty$. By Lemma~\ref{lem lower bound variance NR1}, we have:
\begin{equation*}
\sigma^2 = \lim_{R \to +\infty} \frac{1}{R}\var{\card(Z \cap [0,R])} \geq \frac{1}{\pi^2} \int_0^{+\infty} \left(\kappa(z) + \kappa''(z) \right)^2 \dx z.
\end{equation*}

In order to conclude the proof, we need to check that the right-hand side of the previous equation is positive. It is clearly non-negative. If it were zero, then $\kappa$ would be an even function of class $\mathcal{C}^2$ such that $\kappa(0)=1$, $\kappa'(0)=0$ and $\forall z \geq 0$, $\kappa(z)+\kappa''(z) =0$. That is we would have $\kappa(z) =\cos(z)$ for all $z \in \R$. This would contradict our hypotheses on $\kappa$, for example the fact that $\kappa(z) \xrightarrow[z \to +\infty]{}0$. Thus,
\begin{equation*}
\int_0^{+\infty} \left(\kappa(z) + \kappa''(z) \right)^2 \dx z >0. \qedhere
\end{equation*}
\end{proof}


\section{Divided differences}
\label{sec divided differences}

In this section, we introduce another important tool that we will use in the proofs of Theorems~\ref{thm moments}, \ref{thm vanishing order} and~\ref{thm clustering}: the divided differences. The divided differences associated with a point $x \in \R^p$ and a function $f \in \mathcal{C}^p(\R)$ are coefficients of the Hermite interpolation polynomial of $f$ at $x$ (see Definition~\ref{def divided differences} below). As such, they are an important object in polynomial approximation and are well-studied. In Section~\ref{subsec Hermite interpolation and divided differences}, we define the divided differences and the related Hermite interpolation polynomials. In Section~\ref{subsec properties of the divided differences}, we state the properties of the divided differences that we are interested in. Most of the material of these two sections is classical and can be found in the survey~\cite{Dub2020}. Finally, in Section~\ref{subsec double divided differences and correlation function}, we study the distribution of the divided differences associated with a stationary centered Gaussian process.


\subsection{Hermite interpolation and divided differences}
\label{subsec Hermite interpolation and divided differences}

The goal of this section is to define the so-called divided differences associated with a point $x \in \R^p$ and a function $f \in \mathcal{C}^{p-1}(\R)$. First we define the evaluation at $x \in \R^p$ and introduce some useful notations. Then we define the Hermite interpolation polynomial of $f$ at $x$ and the associated divided differences in Definition~\ref{def divided differences}.

\begin{dfn}[Evaluation map]
\label{def evaluation map}
Let $p \in \N^*$ and let $x =(x_i)_{1 \leq i \leq p} \in \R^p$. For all $i \in \{1,\dots,p\}$ we denote by $c_i(x) = \card \left\{j \in \{1,\dots,i-1\} \mvert x_j = x_i\right\}$. We denote by $\ev_x:\mathcal{C}^{p-1}(\R) \to \R^p$ the \emph{evaluation map} defined by:
\begin{equation*}
\ev_x : f \longmapsto \left(\frac{f^{(c_i(x))}(x_i)}{c_i(x)!}\right)_{1 \leq i \leq p}.
\end{equation*}
\end{dfn}

\begin{ex}
\label{ex evaluation map}
If $x=(x_i)_{1 \leq i \leq p} \in \R^p \setminus \Delta_p$, then $\ev_x:f \mapsto (f(x_1),\dots,f(x_p))$ is the classical evaluation map at the points $(x_i)_{1 \leq i \leq p}$. On the diagonal, we also evaluate derivatives of $f$: if $x=(y_1,\dots,y_1,\dots,y_m,\dots,y_m)$, where the $(y_j)_{1 \leq j \leq m}$ are distinct and $y_j$ is repeated $k_j+1$ times, then $p = \sum_{j=1}^m (k_j+1)$ and
\begin{equation*}
\ev_x : f \longmapsto \left(f(y_1),f'(y_1),\dots,\frac{f^{(k_1)}(y_1)}{k_1!},\dots,f(y_m),f'(y_m),\dots,\frac{f^{(k_m)}(y_m)}{k_m!}\right).
\end{equation*}
More generally, with the notations of Section~\ref{subsec partitions, products and diagonal inclusions}, let $\I \in \pa_p$, let $y = (y_I)_{I \in \I} \in \R^\I \setminus \Delta_\I$ and let $x = \iota_\I(y) \in \Delta_{p,\I}$. Then, for any $f \in \mathcal{C}^{p-1}(\R)$, we have:
\begin{equation*}
\ev_x(f) = \left(\frac{f^{(i)}(y_I)}{i!}\right)_{I \in \I, 0 \leq i < \norm{I}}.
\end{equation*}
\end{ex}

\begin{dfns}[Newton polynomials]
\label{defs Newton polynomials}
Let $p \in \N^*$ and let $x = (x_i)_{1 \leq i \leq p} \in \R^p$.
\begin{itemize}
\item We denote by $\R_{p-1}[X]$ the space of polynomials in $X$ of degree at most $p-1$.
\item For all $j \in \{0,\dots,p-1\}$, we denote by $P_x^j = \prod_{l=1}^j (X - x_l)$ the $j$-th \emph{Newton polynomial} associated with $x$.
\item Let $M(x)$ denote the matrix of the restriction of $\ev_x$ to $\R_{p-1}[X]$, in the basis $(P_x^0,\dots,P_x^{p-1})$ of $\R_{p-1}[X]$ and the canonical basis of $\R^p$ (see Example~\ref{ex divided differences}.\ref{subex DD p=2} below).
\end{itemize}
\end{dfns}

\begin{lem}
\label{lem coeff M x}
Let $p \in \N^*$ and let $x = (x_i)_{1 \leq i \leq p} \in \R^p$. The matrix $M(x)=\begin{pmatrix}M_{ij}(x)\end{pmatrix}_{1,\leq i,j \leq p}$ is lower triangular and, for all $i \in \{1,\dots,p\}$, we have:
\begin{equation*}
M_{ii}(x) = \prod_{\{k \in \{1,\dots,i-1\} \mid x_k \neq x_i\}} (x_i-x_k).
\end{equation*}
Moreover, if $1 \leq j < i \leq p$, the coefficient $M_{ij}(x)$ vanishes when $c_i(x) \geq j$ (cf.~Definition~\ref{def evaluation map}), and is an homogeneous polynomial of degree $j-1-c_i(x)$ in $(x_i-x_l)_{1 \leq l <j}$ when $c_i(x) <j$.
\end{lem}

\begin{proof}
Let $x=(x_i)_{1 \leq i \leq p} \in \R^p$ and let $i,j \in \{1,\dots,p\}$. By definition of $M(x)$, we have
\begin{equation*}
M_{ij}(x) = \frac{(P_x^{j-1})^{(c_i(x))}(x_i)}{c_i(x)!}.
\end{equation*}
If $i < j$, then $\card\{l<j\mid x_l=x_i\} \geq c_i(x)+1$. Hence, $x_i$ is a root of $P_x^{j-1}$ of multiplicity at least $c_i(x)+1$, and $(P_x^{j-1})^{(c_i(x))}(x_i)=0$. Thus $M(x)$ is lower triangular. Then, if $i=j$, we have
\begin{equation*}
P_x^{j-1} = (X-x_i)^{c_i(x)} \prod_{k \in K}(X-x_k),
\end{equation*}
where $K = \{k \in \{1,\dots,i-1\} \mid x_k \neq x_i\}$. Hence, $M_{ii}(x) = \prod_{k \in K} (x_i-x_k)$ as claimed.

Let us now assume that $j < i$. If $c_i(x) \geq j$, since $P^{j-1}_x$ has degree $j-1$ we have $(P^{j-1}_x)^{(c_i(x))}=0$, and $M_{ij}(x)=0$. If $c_i(x) < j$, then $(P^{j-1}_x)^{(c_i(x))}$ is a sum of terms which are products of exactly $j-1-c_i(x)$ factors of the form $(X-x_l)$, where $1 \leq l <j$. Thus $M_{ij}(x)$ is some homogeneous polynomial of degree $j-1-c_i(x)$ evaluated on $(x_i-x_l)_{1 \leq l < j}$.
\end{proof}

\begin{cor}
\label{cor evx isomorphism}
For all $x \in \R^p$, the restriction of $\ev_x$ is an isomorphism from $\R_{p-1}[X]$ to $\R^p$.
\end{cor}

\begin{proof}
By Lemma~\ref{lem coeff M x}, the matrix $M(x)$ of this linear map is lower triangular and its diagonal coefficients are non-zero.
\end{proof}

We can now define the Hermite interpolation polynomial of $f$ at $x \in \R^p$ and the divided difference $[f]_p(x)$. The meaning of the name ``divided difference'' is not obvious in the following definition. The terminology will become clearer after we explained how to compute these divided differences recursively (see Lemma~\ref{lem alternative definition divided differences} below).

\begin{dfn}[Divided differences]
\label{def divided differences}
Let $p \in \N^*$ and let $x \in \R^p$. By Corollary~\ref{cor evx isomorphism}, for any $f \in \mathcal{C}^{p-1}(\R)$ there exists a unique $\pi_x^f \in \R_{p-1}[X]$ such that $\ev_x(\pi_x^f)= \ev_x(f)$. This polynomial is called the \emph{Hermite interpolation polynomial} of $f$ at $x$. The \emph{divided difference} $[f]_p(x)$ is defined as its leading coefficient.
\end{dfn}

The following lemma shows that the divided differences are the coordinates of the Hermite interpolation polynomial in the basis of the Newton polynomials defined above, see Definitions~\ref{defs Newton polynomials}.

\begin{lem}
\label{lem coefficients of Hermite interpolation polynomial}
Let $p \in \N^*$ and let $x = (x_i)_{1 \leq i \leq p} \in \R^p$. For all $f \in \mathcal{C}^{p-1}(\R)$, we have:
\begin{equation*}
\pi_x^f = \sum_{j=1}^p [f]_j(x_1,\dots,x_j) P_x^{j-1}.
\end{equation*}
\end{lem}

\begin{proof}
We prove this result by induction on $p \in \N^*$. If $p=1$, for any continuous~$f$, the polynomial $\pi_x^f$ is constant equal to $f(x_1)$. Hence, $\pi_x^f=[f]_1(x_1)P_x^0$ where $[f]_1=f$.

Let us assume that the result holds for $p \in \N^*$. Let $x =(x_i)_{1 \leq i \leq p+1} \in \R^{p+1}$, we denote by $\tilde{x} = (x_i)_{1 \leq i \leq p}$. Note that for any $f \in \mathcal{C}^p(\R)$, the components of $\ev_{\tilde{x}}(f)$ are the first~$p$ components of $\ev_x(f)$. Then, by Lemma~\ref{lem coeff M x}, we have $\ev_{\tilde{x}}(P_x^p) = 0$. Hence,
\begin{equation*}
\ev_{\tilde{x}}\left(\pi_x^f - [f]_{p+1}(x) P_x^p\right) = \ev_{\tilde{x}}(\pi_x^f) = \ev_{\tilde{x}}(f).
\end{equation*}
Moreover, by Definition~\ref{def divided differences}, the polynomial $\pi_x^f - [f]_{p+1}(x) P_x^p$ has degree at most $p-1$. Thus, $\pi_x^f - [f]_{p+1}(x) P_x^p = \pi_{\tilde{x}}^f = \sum_{j=1}^p [f]_j(x_1,\dots,x_j) P_x^{j-1}$, where the second equality is given by the induction hypothesis. This concludes the induction step and the proof.
\end{proof}

\begin{dfn}[Divided differences evaluation map]
\label{def evaluation divided differences}
Let $p \in \N^*$ and let $x =(x_i)_{1 \leq i \leq p} \in \R^p$. We denote by $[\ev]_x:\mathcal{C}^{p-1}(\R) \to \R^p$ the linear map defined by
\begin{equation*}
[\ev]_x : f \longmapsto \left([f]_j(x_1,\dots,x_j)\right)_{1 \leq j \leq p}.
\end{equation*}
\end{dfn}

\begin{lem}
\label{lem evaluation divided differences}
Let $p \in \N^*$, for all $x \in \R^p$ we have $M(x) [\ev]_x = \ev_x$, where $\ev_x$ is as in Definition~\ref{def evaluation map} and $M(x)$ is defined by Definitions~\ref{defs Newton polynomials}.
\end{lem}

\begin{proof}
Let $x \in \R^p$ and let $f \in \mathcal{C}^{p-1}(\R)$. By Lemma~\ref{lem coefficients of Hermite interpolation polynomial}, the components of $[\ev]_x(f)$ are the coordinates of the polynomial $\pi_x^f$ in the basis $(P_x^j)_{1 \leq j \leq p}$ of $\R_{p-1}[X]$. Then, by definition of $M(x)$ and $\pi_x^f$, we have:
\begin{equation*}
M(x) [\ev]_x(f) = \ev_x(\pi_x^f) = \ev_x(f).\qedhere
\end{equation*}
\end{proof}

\begin{exs}
\label{ex divided differences}
We conclude this section by giving some examples.
\begin{enumerate}
\item \label{subex DD p=2} Let $f \in \mathcal{C}^1(\R)$ and $(x_1,x_2) \in \R^2 \setminus \Delta_2$, we have $M(x) = \begin{pmatrix}
1 & 0 \\
1 & x_2-x_1
\end{pmatrix}$. Hence
\begin{equation*}
\begin{pmatrix}
[f]_1(x_1) \\ [f]_2(x_1,x_2)
\end{pmatrix} = M(x)^{-1} \begin{pmatrix}
f(x_1) \\ f(x_2)
\end{pmatrix} = \begin{pmatrix}
1 & 0 \\ \frac{-1}{x_2-x_1} & \frac{1}{x_2-x_1}
\end{pmatrix} \begin{pmatrix}
f(x_1) \\ f(x_2)
\end{pmatrix} = \begin{pmatrix}
f(x_1)\\ \frac{f(x_2)-f(x_1)}{x_2-x_1}
\end{pmatrix}.
\end{equation*}
\item \label{subex DD Taylor} Let $p \in \N^*$, let $x=(x_i)_{1 \leq i \leq p} \in \R^p$ and let $f \in \mathcal{C}^{p-1}(\R)$. If there exists $z \in \R$ such that $x_i=z$ for all $i \in \{1,\dots,p\}$, then $\ev_x(f)=\left(f(z),\dots,\frac{f^{(p-1)}(z)}{(p-1)!}\right)$ and $M(x)$ is the identity matrix of size $p$. Then,~$\pi_x^f$ is the Taylor polynomial of degree $p-1$ of $f$ at $z$ and $[f]_j(x_1,\dots,x_j) = \frac{f^{(j-1)}(z)}{(j-1)!}$, for all $j \in \{1,\dots,p\}$.
\item \label{subex DD derivative} Let $p \in \N^*$ and let $(x_i)_{1 \leq i \leq p} \in \R^p \setminus \Delta_p$. Let $f \in \mathcal{C}^p(\R)$ be such that $[f]_j(x_1,\dots,x_j) = 0$ for all $j \in \{1,\dots,p\}$. Let $i \in \{1,\dots,p\}$, we denote by $x=(x_1,\dots,x_p,x_i) \in \R^{p+1}$. By Lemma~\ref{lem coefficients of Hermite interpolation polynomial}, we have $\pi_x^f = [f]_{p+1}(x_1,\dots,x_p,x_i) \prod_{j=1}^p(X-x_j)$. Hence,
\begin{equation*}
f'(x_i) = (\pi_x^f)'(x_i) = [f]_{p+1}(x_1,\dots,x_p,x_i) \prod_{j \in \{1,\dots,p\} \setminus \{i\}}(x_i-x_j).
\end{equation*}
\end{enumerate}
\end{exs}


\subsection{Properties of the divided differences}
\label{subsec properties of the divided differences}

Let us now derive some interesting properties of the divided differences defined in Definition~\ref{def divided differences}. They will be useful in Section~\ref{sec Kac-Rice densities revisited and clustering}, to obtain new expressions of the Kac--Rice densities (cf.~Definition~\ref{def Kac-Rice densities}) and prove clustering results for these densities.

Recall that we denoted by $\sigma \cdot x$ the action of $\sigma \in \mathfrak{S}_p$ on $x \in \R^p$ by permutation of the indices (see Notation~\ref{ntn permutation}).

\begin{lem}[Symmetry]
\label{lem divided differences are symmetric}
Let $p \in \N^*$ and $f \in \mathcal{C}^{p-1}(\R)$. For all $x \in \R^p$ and all $\sigma \in \mathfrak{S}_p$, we have $\pi_{\sigma \cdot x}^f = \pi_x^f$. In particular $[f]_p(\sigma \cdot x) = [f]_p(x)$, that is the function $[f]_p:\R^p \to \R$ is symmetric.
\end{lem}

\begin{proof}
Let $x \in \R^p$. By Remark~\ref{rem diagonal inclusions}, there exists a unique $\I \in \pa_p$ such that $x \in \Delta_{p,\I}$. Moreover, there exists a unique $y = (y_I)_{I \in \I} \in \R^\I \setminus \Delta_\I$ such that $x= \iota_\I(y)$. By Definition~\ref{def divided differences}, the polynomial $\pi_x^f$ is the only element of $\R_{p-1}[X]$ such that$(\pi_x^f-f)^{(i)}(y_I)=0$ for all $I \in \I$ and all $i < \norm{I}$. The set $\{(y_I,\norm{I})\mid I \in \I\}$ is invariant under the action of $\sigma$ on $x$ by permutation of the components. Hence $\pi_{\sigma \cdot x}^f = \pi_x^f$ and, looking at the leading coefficients, we have $[f]_p(\sigma \cdot x) = [f]_p(x)$.
\end{proof}

The following result shows that the divided differences can be computed recursively, at least if the interpolation points $x_1,\dots,x_p \in \R$ are distinct. It also explains the name ``divided differences''.

\begin{lem}[Inductive definition]
\label{lem alternative definition divided differences}
Let $p \in \N^*$ and let $f \in \mathcal{C}^p(\R)$. Let $x=(x_i)_{1 \leq i \leq p+1} \in \R^{p+1}$ be such that $x_p \neq x_{p+1}$, then we have:
\begin{equation*}
[f]_{p+1}(x) = \frac{[f]_p(x_1,\dots,x_{p-1},x_{p+1})-[f]_p(x_1\dots,x_{p-1},x_p)}{x_{p+1}-x_p}.
\end{equation*}
\end{lem}

\begin{proof}
Let $\sigma \in \mathfrak{S}_{p+1}$ be defined by $\sigma(p)=p+1$, $\sigma(p+1)=p$ and $\sigma(i)=i$ for all $i \in \{1,\dots,p-1\}$. Using Notation~\ref{ntn permutation}, by Lemmas~\ref{lem coefficients of Hermite interpolation polynomial} and~\ref{lem divided differences are symmetric}, we have:
\begin{equation*}
0 = \pi_x^f - \pi_{\sigma \cdot x}^f = \sum_{j=1}^{p+1} [f]_j(x_1,\dots,x_j) P_x^{j-1} - \sum_{j=1}^{p+1} [f]_j(x_{\sigma(1)},\dots,x_{\sigma(j)}) P_{\sigma \cdot x}^{j-1}.
\end{equation*}
For all $j \in \{1,\dots,p\}$, we have $P_{\sigma\cdot x}^{j-1} = P_x^{j-1}$. Moreover, $[f]_j(x_{\sigma(1)},\dots,x_{\sigma(j)}) = [f]_j(x_1\dots,x_j)$ for all $j \in \{1,\dots,p+1\} \setminus \{p\}$. Hence, only the terms of index $p$ and $p+1$ do not cancel out in the previous sums. Dividing by $ P_x^{p-1} = \prod_{i=1}^{p-1} (X-x_i)= P_{\sigma\cdot x}^{p-1}$, we obtain:
\begin{equation*}
(x_{p+1}-x_p) [f]_{p+1}(x_1,\dots,x_p,x_{p+1}) + [f]_p(x_1,\dots,x_{p-1},x_{p+1})-[f]_p(x_1\dots,x_{p-1},x_p) = 0.\qedhere
\end{equation*}
\end{proof}

\begin{ntn}
\label{ntn x min x max}
Let $p \in \N^*$ and let $x=(x_i)_{1\leq i \leq p}$, we denote by $x_{\min} = \min \{x_i \mid 1 \leq i \leq p\}$ and by $x_{\max} = \max \{x_i \mid 1 \leq i \leq p\}$.
\end{ntn}

\begin{lem}[Rolle's Property]
\label{lem divided differences Rolle}
Let $p \in \N^*$ and let $f \in \mathcal{C}^{p-1}(\R)$. For all $x=(x_i)_{1 \leq i \leq p} \in \R^p$, there exists $\xi \in [x_{\min},x_{\max}]$ such that $[f]_p(x) = \frac{f^{(p-1)}(\xi)}{(p-1)!}$.
\end{lem}

\begin{proof}
Let $x \in \R^p$. There exist $y_1<\dots <y_m$ and $k_1,\dots, k_m \in \N$ such that, for all $j \in \{1,\dots,m\}$, exactly $k_j+1$ components of $x$ are equal to $y_j$. With these notations, $x_{\min} = y_1$ and $x_{\max} = y_m$. By Definition~\ref{def divided differences}, $(f-\pi_x^f)$ has at least $p$ zeros in $[x_{\min},x_{\max}]$, counted with multiplicity. More precisely, $\forall j \in \{1,\dots,m\}$, $\forall k \in \{0,\dots,k_j\}$, we have $(f-\pi_x^f)^{(k)}(y_j)=0$.

For all $j \in \{1,\dots,m-1\}$, there exists $z_j \in (y_j,y_{j+1})$ such that $(f-\pi_x^f)'(z_j)=0$, by Rolle's Theorem. Hence $(f-\pi_x^f)'$ has at least $p-1$ zeros in $[x_{\min},x_{\max}]$, namely $z_1,\dots,z_{m-1}$ with multiplicity~$1$, and $y_j$ with multiplicity $k_j-1$, for all $j \in \{1,\dots,m\}$. Iterating this procedure, for all $k \in \{0,\dots,p-1\}$, the function $(f-\pi_x^f)^{(k)}$ has at least $p-k$ zeros in $[x_{\min},x_{\max}]$, counted with multiplicity. In particular, there exists $\xi \in [x_{\min},x_{\max}]$ such that:
\begin{equation*}
(f-\pi_x^f)^{(p-1)}(\xi) = f^{(p-1)}(\xi) - (p-1)![f]_p(x) = 0.\qedhere
\end{equation*}
\end{proof}

\begin{lem}[Continuity]
\label{lem divided differences are continuous}
Let $f \in \mathcal{C}^{p-1}(\R)$, then the function $[f]_p:\R^p \to \R$ is continuous.
\end{lem}

\begin{proof}
We prove this result by induction on $p$. If $p=1$, then $[f]_1 =f$ is continuous on $\R$.

Let us now assume that the result holds for some $p \in \N^*$ and let $f \in \mathcal{C}^p(\R)$. Using Lemma~\ref{lem alternative definition divided differences} and the induction hypothesis, $[f]_{p+1}$ is continuous on $\{(x_1,\dots,x_{p+1}) \in \R^{p+1} \mid x_p \neq x_{p+1}\}$. By symmetry (see Lemma~\ref{lem divided differences are symmetric}), this map is in fact continuous at any point $(x_i)_{1 \leq i \leq p+1} \in \R^{p+1}$ such that $x_i \neq x_j$ for some $i, j \in \{1, \dots,p+1\}$. In order to conclude the proof, it is enough to prove that, for all $z \in \R$,
\begin{equation*}
[f]_{p+1}(x) \xrightarrow[x \to (z,z, \dots,z)]{} [f]_{p+1}(z,\dots,z).
\end{equation*}
We have seen in Example~\ref{ex divided differences}.\ref{subex DD Taylor} that $[f]_{p+1}(z,\dots,z) = \frac{f^{(p)}(z)}{p!}$. Let $x \in \R^{p+1}$, by Lemma~\ref{lem divided differences Rolle}, there exists $\xi \in [x_{\min},x_{\max}]$ such that $[f]_{p+1}(x) = \frac{f^{(p)}(\xi)}{p!}$. As $x \to (z,\dots,z)$, we have $x_{\min} \to z$ and $x_{\max} \to z$. The conclusion follows from the continuity of $f^{(p)}$.
\end{proof}

\begin{rem}
\label{rem alternative definition divided differences}
Let $p \in \N^*$ and $f \in \mathcal{C}^p(\R)$, for all $x = (x_i)_{1 \leq i \leq p+1} \in \R^p$ we have:
\begin{equation}
\label{eq rem alternative definition}
[f]_{p+1}(x) = \lim_{z \to x_{p+1}} \frac{[f]_p(x_1,\dots,x_{p-1},z) - [f]_p(x_1,\dots,x_{p-1},x_p)}{z-x_p}.
\end{equation}
If $x_p \neq x_{p+1}$, this is a consequence Lemma~\ref{lem alternative definition divided differences} and the continuity of $[f]_p$ (see Lemma~\ref{lem divided differences are continuous}). If $x_p = x_{p+1}$ this follows from the first case and the continuity of $[f]_{p+1}$. Thus, one can define the divided differences recursively as follows: if $f \in \mathcal{C}^0(\R)$ then $[f]_1 =f$, and if $f \in \mathcal{C}^p(\R)$ the map $[f]_{p+1}:\R^{p+1} \to \R$ is defined by Equation~\eqref{eq rem alternative definition}. This definition is equivalent to Definition~\ref{def divided differences}.
\end{rem}

\begin{lem}[Regularity]
\label{lem regularity of divided differences}
Let $p \in \N^*$ and let $k \in \N$, if $f \in \mathcal{C}^{p+k-1}(\R)$ then $[f]_p:\R^p \to \R$ is of class $\mathcal{C}^k$. Moreover, for all $k_1,\dots,k_p \in \N$ such that $k_1 +\dots+k_p \leq k$, for all $x=(x_i)_{1 \leq i \leq p} \in \R^p$, we have:
\begin{equation}
\label{eq partial derivatives divided differences}
\frac{1}{k_1! \dots k_p!}\frac{\partial^{k_1+\dots+k_p}[f]_p}{\partial x_1^{k_1} \dots \partial x_p^{k_p}} (x) = [f]_{p+k_1+\dots+k_p}(x_1,\dots,x_1,\dots,x_p,\dots,x_p), 
\end{equation}
where each $x_j$ is repeated $k_j+1$ times on the right-hand side.
\end{lem}

\begin{proof}
We prove this result by induction on $k$. The case $k=0$ is given by Lemma~\ref{lem divided differences are continuous}.

For $k=1$, let $p \in \N^*$ and let $f \in \mathcal{C}^p(\R)$. By Lemmas~\ref{lem alternative definition divided differences} and~\ref{lem divided differences are continuous} (see also Remark~\ref{rem alternative definition divided differences}, Equation~\eqref{eq rem alternative definition}), the map $[f]_p$ admits a continuous partial derivative with respect to the $p$-th variable, given by:
\begin{equation*}
\deron{[f]_p}{x_p}: (x_1,\dots,x_p) \mapsto [f]_{p+1}(x_1,\dots,x_p,x_p).
\end{equation*}
The symmetry of the divided differences (see Lemma~\ref{lem divided differences are symmetric}) yields that $[f]_p$ is of class $\mathcal{C}^1$, with partial derivatives given by Equation~\eqref{eq partial derivatives divided differences}.

Let $k \in \N^*$ and let us assume that the result holds for $k$ and any $p \in \N^*$. Let $p \in \N^*$ and let $f \in\mathcal{C}^{p+k}(\R)$. Using the case $k=1$, the map $[f]_p$ is $\mathcal{C}^1$ and its partial derivatives of order~$1$ are given by Equation~\eqref{eq partial derivatives divided differences}. The induction hypothesis shows that $[f]_{p+1}$ is of class $\mathcal{C}^k$, hence $[f]_p$ is of class $\mathcal{C}^{k+1}$. The induction hypothesis also shows that the partial derivatives of order at most~$k$ of $[f]_p$ are given by~\eqref{eq partial derivatives divided differences}. Let $k_1,\dots,k_p \in \N$ be such that $k_1 + \dots + k_p =k$ and let $i \in \{1,\dots,p\}$. We have:
\begin{equation*}
\frac{1}{k_i+1}\deron{}{x_i}\left(\frac{1}{k_1! \dots k_p!}\frac{\partial^{k_1+\dots+k_p}[f]_p}{\partial x_1^{k_1} \dots \partial x_p^{k_p}}\right) = \frac{1}{(k_i+1)} \deron{}{x_i}\left(x \mapsto [f]_{p+k}(x_1,\dots,x_1,\dots,x_p,\dots,x_p)\right),
\end{equation*}
where each $x_j$ is repeated $k_j+1$ times on the right-hand side. Using the case $k=1$ for $[f]_{p+k}$ proves that the partial derivatives of order $k+1$ of $[f]_p$ satisfy Equation~\eqref{eq partial derivatives divided differences}.
\end{proof}

We conclude this section by stating facts that provide some insight on divided differences. Let $p \in \N^*$ and let $x =(x_i)_{1 \leq i \leq p} \in \R^p \setminus \Delta_p$, for all $f \in \mathcal{C}^{p-1}(\R)$, we have:
\begin{equation*}
[f]_p(x) = \sum_{i=1}^p f(x_i) \prod_{l \in \{1,\dots,p\} \setminus \{i\}} \frac{1}{x_i-x_l}
\end{equation*}
This formula is proved by induction on $p \in \N^*$, using Lemma~\ref{lem alternative definition divided differences} in the induction step. Taking partial derivatives in the previous formula and using Lemma~\ref{lem regularity of divided differences} allows to derive an expression of $[f]_p(x)$ for any $p \in \N^*$, any $x=(x_i)_{1 \leq i \leq p} \in \R^p$ and any $f \in \mathcal{C}^{p-1}(\R)$. One obtains that $[f]_p(x)$ is a linear combination of the $f^{(k)}(x_i)$ with $i \in \{1,\dots,p\}$ and $k < \card\{ j \in \{1,\dots,p\} \mid x_j = x_i\}$. The coefficients of this linear combination are rational functions in $(x_i-x_j)_{1 \leq j < i \leq p}$, independent of~$f$. This can already be deduced from the fact that $[f]_p(x)$ is the last coordinate of $M(x)^{-1}\ev_x(f)$ and the expression of $M(x)$ (see Definitions~\ref{defs Newton polynomials} and Lemmas~\ref{lem coeff M x} and~\ref{lem evaluation divided differences}).


\subsection{Double divided differences and correlation function}
\label{subsec double divided differences and correlation function}

In the previous two sections, we defined and studied the divided differences of some regular enough function. The upshot is to consider the divided differences of the Gaussian process~$f$ that we are interested in. Since the evaluation $[\ev]_x$ is linear, $[\ev]_x(f)$ is a centered Gaussian vector. The goal of this section is to compute and study its variance.

Let $K:\R^2 \to \R$ denote the correlation kernel of $f$. In order to compute the coefficients of the variance matrix of $[\ev]_x(f)$, we need to take divided differences of $K$ with respect to the first variable, then take divided differences of the result with respect to the second variable. If $f$ was not stationary, this would require to develop a notion of ``partial divided differences'' and prove parametric versions of the regularity results of Section~\ref{subsec properties of the divided differences}. This can be done but is a bit cumbersome. Since we consider stationary processes in this paper, we can avoid these complications and only consider divided differences associated with the correlation function $\kappa:x \mapsto K(0,x)$. We need however to introduce some additional notations.

Let $\kappa : \R \to \R$ and let $K:\R^2 \to \R$ be defined by $K:(z,w) \mapsto \kappa(w-z)$. If $\kappa$ is $\mathcal{C}^{p-1}$ then, for all $y \in \R$, the map $K(\cdot,y):z \mapsto K(z,y)$ is of class $\mathcal{C}^{p-1}$. In particular, the divided differences $[K(\cdot,y)]_k$ are well-defined for all $k \in \{1,\dots,p\}$.

\begin{dfn}
\label{def divided differences K}
Let $p \in \N^*$, let $k \in \{1,\dots,p\}$ and let $\kappa \in \mathcal{C}^{p-1}(\R)$. For all $x \in \R^k$ and all $y \in \R$, we denote by $[\kappa]_{(k,1)}(x,y)= [K(\cdot,y)]_k(x)$, where $K:(z,w) \mapsto \kappa(w-z)$.
\end{dfn}

\begin{lem}
\label{lem divided differences kappa}
Let $p \in \N^*$ and let $\kappa \in \mathcal{C}^{p-1}(\R)$. Let $k \in \{1,\dots,p\}$, for all $x=(x_i)_{1 \leq i \leq k} \in \R^k$, for all $y \in \R$, we have:
\begin{equation*}
[\kappa]_{(k,1)}(x,y) = (-1)^{k-1}[\kappa]_k(y-x_1,\dots,y-x_k).
\end{equation*}
\end{lem}

\begin{proof}
Let $K:(z,w) \mapsto K(w-z)$. By Definition~\ref{def divided differences}, we know that $[\kappa]_{(k,1)}(x,y)=[K(\cdot,y)]_k(x)$ is the leading coefficient of $\pi_x^{K(\cdot,y)}$. Now, recalling Definition~\ref{def evaluation map}, we have:
\begin{equation*}
\ev_x(\pi_x^{K(\cdot,y)}) = \ev_x(K(\cdot,y)) = \begin{pmatrix}
\frac{1}{c_i(x)!} \deron{^{c_i(x)}K}{x^{c_i(x)}}(x_i,y)
\end{pmatrix}_{1 \leq i \leq k} = \begin{pmatrix}
(-1)^{c_i(x)}\frac{\kappa^{(c_i(x))}(y-x_i)}{c_i(x)!}
\end{pmatrix}_{1 \leq i \leq k},
\end{equation*}
and $\pi_x^{K(\cdot,y)}$ is the only polynomial in $\R_{k-1}[X]$ satisfying this condition. On the other hand, let us denote by $y-x=(y-x_1,\dots,y-x_k) \in \R^k$. For all $i \in \{1,\dots,k\}$, we have $c_i(y-x) = c_i(x)$. Then, $\pi_{y-x}^\kappa(y-X) \in \R_{k-1}[X]$ satisfies:
\begin{equation*}
\ev_x(\pi_{y-x}^\kappa(y-X)) = \begin{pmatrix}
\frac{(-1)^{c_i(x)}}{c_i(x)!} (\pi_{y-x}^\kappa)^{(c_i(x))}(y-x_i)
\end{pmatrix}_{1 \leq i \leq k} = \begin{pmatrix}
(-1)^{c_i(x)}\frac{\kappa^{(c_i(x))}(y-x_i)}{c_i(x)!}
\end{pmatrix}_{1 \leq i \leq k}.
\end{equation*}
Thus $\pi_x^{K(\cdot,y)} = \pi_{y-x}^\kappa(y-X)$, and its leading coefficient equals $(-1)^{k-1}[\kappa]_k(y-x)$.
\end{proof}

A consequence of Lemma~\ref{lem divided differences kappa} is that, if $\kappa \in \mathcal{C}^{p-1}(\R)$ and $k \in \{1,\dots,p\}$, then for all $x \in \R^k$ the function $[\kappa]_{(k,1)}(x,\cdot):w \mapsto [\kappa]_{(k,1)}(x,w)$ is of class $\mathcal{C}^{p-k}$ from $\R$ to $\R$ (see Lemma~\ref{lem regularity of divided differences}). In particular, its the divided differences of order at most $p-k+1$ are well-defined, and the following makes sense.

\begin{dfn}[Double divided differences]
\label{def double divided differences}
Let $p \in \N^*$ and let $\kappa \in \mathcal{C}^{p-1}(\R)$. Let $k$ and $l \in \N^*$ be such that $k+l \leq p+1$, we denote by $[\kappa]_{(k,l)}:\R^k \times \R^l \to \R$ the map defined by $[\kappa]_{(k,l)}(x,y) = \left[[\kappa]_{(k,1)}(x,\cdot) \right]_l(y)$ for all $x \in \R^k$ and $y \in \R^l$.
\end{dfn}

Thanks to Lemma~\ref{lem divided differences Rolle}, we can give bounds on the double divided differences $[\kappa]_{(k,l)}(x,y)$. This is the object of the following result.

\begin{lem}
\label{lem double divided differences Rolle}
Let $k$ and $l \in \N^*$ and let $\kappa \in \mathcal{C}^{k+l-2}(\R)$, for all $x=(x_i)_{1 \leq i \leq k} \in \R^k$ and all $y =(y_j)_{1 \leq j \leq l} \in \R^l$, we have:
\begin{equation*}
\norm{[\kappa]_{(k,l)}(x,y)} \leq \max \left\{\norm{\kappa^{(k+l-2)}(\xi)} \mvert y_{\min} - x_{\max} \leq \xi \leq y_{\max} - x_{\min} \right\}.
\end{equation*}
\end{lem}

\begin{proof}
Since $[\kappa]_{(k,l)}(x,y) = \left[ [\kappa]_{(k,1)}(x,\cdot) \right]_l (y)$, by Lemma~\ref{lem divided differences Rolle} there exists $w_0 \in [y_{\min},y_{\max}]$ such that:
\begin{equation*}
[\kappa]_{(k,l)}(x,y) = \frac{1}{(l-1)!} \deron{^{(l-1)}}{w^{(l-1)}}_{| w=w_0}[\kappa]_{(k,1)}(x,w).
\end{equation*}
Then, by Lemma~\ref{lem divided differences kappa} and Lemma~\ref{lem regularity of divided differences}, we have:
\begin{align*}
[\kappa]_{(k,l)}(x,y) &= \frac{(-1)^{k-1}}{(l-1)!}\deron{^{(l-1)}}{w^{(l-1)}}_{| w=w_0}[\kappa]_k(w-x_1,\dots,w-x_k)\\
&= (-1)^{k-1} \sum_{l_1 + \dots + l_k = l-1} \frac{1}{l_1! \dots l_k!} \frac{\partial^{l-1}[\kappa]_k}{\partial x_1^{l_1} \dots \partial x_k^{l_k}}(w_0-x_1,\dots,w_0-x_k)\\
&= (-1)^{k-1} \sum_{l_1 + \dots + l_k = l-1} [\kappa]_{k+l-1}(w_0-x_1,\dots,w_0-x_1,\dots,w_0-x_k,\dots,w_0-x_k),
\end{align*}
where the last two sums are indexed by $\{(l_1,\dots,l_k) \in \N^k \mid l_1 +\dots +l_k = l-1\}$, and each $w_0-x_i$ is repeated exactly $l_i +1$ times in the term indexed by $(l_1,\dots,l_k)$.

Let $(l_1,\dots,l_k) \in \N^k$ be such that $l_1 +\dots +l_k = l-1$. By Lemma~\ref{lem divided differences Rolle} there exists $\xi_{(l_1,\dots,l_k)} \in \R$ such that:
\begin{equation*}
y_{\min} - x_{\max} \leq w_0 - x_{\max} \leq \xi_{(l_1,\dots,l_k)} \leq w_0 - x_{\min} \leq y_{\max} - x_{\min},
\end{equation*}
and
\begin{equation*}
[\kappa]_{k+l-1}(w_0-x_1,\dots,w_0-x_1,\dots,w_0-x_k,\dots,w_0-x_k) = \frac{\kappa^{(k+l-2)}(\xi_{(l_1,\dots,l_k)})}{(k+l-2)!},
\end{equation*}
where each term $w_0-x_i$ is repeated $l_i +1$ times on the right-hand side. Thus,
\begin{equation*}
\norm{[\kappa]_{(k,l)}(x,y)} \leq \max \left\{\norm{\kappa^{(k+l-2)}(\xi)} \mvert y_{\min} - x_{\max} \leq \xi \leq y_{\max} - x_{\min} \right\}
\end{equation*}
provided that $\card \left\{(l_1,\dots,l_k) \in \N^k \mvert l_1 + \dots + l_k = l-1 \right\} \leq (k+l-2)!$. This cardinal is the dimension of the space of homogeneous polynomials of degree $(l-1)$ in $k$ variables. Thus, it is equal to $\binom{k+l-2}{k-1} \leq (k+l-2)!$.
\end{proof}

The \emph{double divided differences} $[\kappa]_{(k,l)}$ will appear in the coefficients of the variance matrix of the Gaussian vector $[\ev]_x(f)$. The key step in this direction is the following lemma. It also shows how to compute efficiently $[\kappa]_{(k,l)}$ from the values of $\kappa$ and its derivatives. Finally, Lemma~\ref{lem double divided differences} shows that taking divided differences in the $x$ variable then in the $y$ variable gives the same result as the converse, which is hinted by the notation but is not obvious from the definition.

\begin{lem}
\label{lem double divided differences}
Let $p \in \N^*$ and let $k,l \in \N^*$ be such that $k+l \leq p+1$. Let $\kappa \in \mathcal{C}^{p-1}(\R)$, for all $x= (x_i)_{1 \leq i \leq k} \in \R^k$ and $y=(y_j)_{1 \leq j \leq l} \in \R^l$ we have:
\begin{equation*}
\begin{pmatrix}
[\kappa]_{(i,j)}(x_1,\dots,x_i,y_1,\dots,y_j)\rule{0em}{2.5ex}
\end{pmatrix}_{\substack{1 \leq i \leq k\\ 1 \leq j \leq l}} = M(x)^{-1} \begin{pmatrix}
\dfrac{(-1)^{c_i(x)}\kappa^{(c_i(x) + c_j(y))}(y_j-x_i)}{c_i(x)! c_j(y)!}
\end{pmatrix}_{\substack{1 \leq i \leq k\\ 1 \leq j \leq l}} \trans{M(y)^{-1}},
\end{equation*}
where $c_i(\cdot)$ is as in Definition~\ref{def evaluation map} and $M(\cdot)$ is as in Definitions~\ref{defs Newton polynomials}. In particular,
\begin{equation*}
\left[[\kappa]_{(k,1)}(x,\cdot)\right]_l(y) = [\kappa]_{(k,l)}(x,y) = \left[[\kappa]_{(1,l)}(\cdot,y)\right]_k(x).
\end{equation*}
\end{lem}

\begin{proof}
Let $K:(z,w) \mapsto \kappa(w-z)$. We denote by $C(x,y)$ the matrix
\begin{equation*}
C(x,y) = \begin{pmatrix}
\dfrac{(-1)^{c_i(x)}\kappa^{(c_i(x) + c_j(y))}(y_j-x_i)}{c_i(x)! c_j(y)!}
\end{pmatrix}_{\substack{1 \leq i \leq k\\ 1 \leq j \leq l}} = \begin{pmatrix}
\dfrac{1}{c_i(x)! c_j(y)!} \dfrac{\partial^{c_i(x)+c_j(y)}K}{\partial z^{c_i(x)} \partial w^{c_j(y)}}(x_i,y_j)
\end{pmatrix}_{\substack{1 \leq i \leq k\\ 1 \leq j \leq l}}.
\end{equation*}
The $j$-th column of $C(x,y)$ equals:
\begin{equation*}
\frac{1}{c_j(y)!} \deron{^{c_j(y)}}{w^{c_j(y)}}_{| w=y_j}\left( \frac{1}{c_i(x)!} \deron{^{c_i(x)}K}{z^{c_i(x)}}(x_i,w)\right)_{1 \leq i \leq k} = \frac{1}{c_j(y)!} \deron{^{c_j(y)}}{w^{c_j(y)}}_{| w=y_j} \ev_x\left(K(\cdot,w)\right).
\end{equation*}
Then, by Lemma~\ref{lem evaluation divided differences}, the $j$-th column of $M(x)^{-1}C(x,y)$ is:
\begin{align*}
\frac{1}{c_j(y)!} \deron{^{c_j(y)}}{w^{c_j(y)}}_{| w=y_j} [\ev]_x(K(\cdot,w)) &= \left(\frac{1}{c_j(y)!} \deron{^{c_j(y)}}{w^{c_j(y)}}_{| w=y_j} [K(\cdot,w)]_i(x_1,\dots,x_i)\right)_{1 \leq i \leq k}\\
&= \left(\frac{1}{c_j(y)!} \deron{^{c_j(y)}[\kappa]_{(i,1)}}{w^{c_j(y)}}(x_1,\dots,x_i,y_j)\right)_{1 \leq i \leq k}.
\end{align*}
This shows that the $i$-th row of $M(x)^{-1}C(x,y)$ equals $\trans{\ev_y\left([\kappa]_{(i,1)}(x_1,\dots,x_i,\cdot)\right)}$. Then, the $i$-th row of $M(x)^{-1}C(x,y)\trans{M(y)^{-1}}$ equals:
\begin{align*}
\trans{\left(M(y)^{-1}\ev_y\left([\kappa]_{(i,1)}(x_1,\dots,x_i,\cdot)\right)\right)} &= \trans{[\ev]_y\left([\kappa]_{(i,1)}(x_1,\dots,x_i,\cdot)\right)}\\
&= \left(\left[[\kappa]_{(i,1)}(x_1,\dots,x_i,\cdot)\right]_j(y_1,\dots,y_j)\right)_{1 \leq j \leq l}\\
&= \left([\kappa]_{(i,j)}(x_1,\dots,x_i,y_1,\dots,y_j)\right)_{1 \leq j \leq l}.
\end{align*}
This proves that the coefficients of $M(x)^{-1}C(x,y)\trans{M(y)^{-1}}$ are as claimed.

The previous computation proved that the bottom-right coefficient of $M(x)^{-1} C(x,y) \trans{M(y)^{-1}}$ equals $\left[ [\kappa]_{(k,1)}(x,\cdot) \right]_l(y) = [\kappa]_{(k,l)}(x,y)$. This reflects the fact that we first multiplied $C(x,y)$ by $M(x)^{-1}$ on the left, thus acting on each column of $C(x,y)$ and taking divided differences in the $x$ variables, then we multiplied the result by $\trans{M(y)^{-1}}$ on the right, thus acting on the rows and taking divided differences in the $y$ variables. If we first multiply $C(x,y)$ by $\trans{M(y)^{-1}}$ on the right then multiply the result by $M(x)^{-1}$ on the left, we first act on the rows of $C(x,y)$ then on the columns of $C(x,y)\trans{M(y)^{-1}}$. In this case, we start by computing divided differences in the $y$ variables, then we take divided differences in the $x$ variables. The same kind of computation as above shows that the bottom-right coefficient of $M(x)^{-1} C(x,y) \trans{M(y)^{-1}}$ equals $\left[[\kappa]_{(1,l)}(\cdot,y)\right]_k(x)$. Thus, the desired relation is just a consequence of the associativity of the matrix product.
\end{proof}

We conclude this section by studying the distribution of the divided differences associated with a regular enough Gaussian process. Note that the following result shows that, if $\kappa$ is the correlation function of a $\mathcal{C}^{p-1}$ Gaussian process, then $[\kappa]_{(k,l)}$ is continuous on $\R^k \times \R^l$, for all $k,l \in \{1,\dots,p\}$.

\begin{lem}[Distribution of divided differences]
\label{lem variance divided differences}
Let $p \in \N^*$, let $f$ be a stationary centered Gaussian process of class $\mathcal{C}^{p-1}$ and let~$\kappa$ denote the correlation function of $f$. The map $x \mapsto [\ev]_x(f)$ from $\R^p$ to itself defines a continuous centered Gaussian field. Its distribution is characterized by the fact that for all $k,l \in \{1,\dots,p\}$, for all $x \in \R^k$ and $y \in \R^l$, $\esp{\rule{0em}{2ex}[f]_k(x) [f]_l(y)} = [\kappa]_{(k,l)}(x,y)$. Moreover, the distribution of $\left([\ev]_x(f)\right)_{x \in \R^p}$ is invariant under the diagonal action of $\R$ on $\R^p$ by translation. That is, for all $t \in \R$, we have $\left([\ev]_{x+(t,\dots,t)}(f)\right)_{x \in \R^p} = \left([\ev]_x(f)\right)_{x \in \R^p}$ in distribution.
\end{lem}

\begin{proof}
Since $[\ev]_x$ is linear for all $x \in \R^p$, the finite-dimensional marginal distributions of the field $\left([\ev]_x(f)\right)_{x \in \R^p}$ are centered and Gaussian. By Lemma~\ref{lem regularity of divided differences}, since $f$ is almost surely $\mathcal{C}^{p-1}$, then $x \mapsto [\ev]_x(f)$ is almost surely continuous. Thus, $\left([\ev]_x(f)\right)_{x \in \R^p}$ is a continuous centered Gaussian field, and characterizing its distribution amounts to computing the variance matrix of $[\ev]_x(f)$ and $[\ev]_y(f)$ for any $x,y \in \R^p$.

Recall that, since $f$ is of class $\mathcal{C}^{p-1}$, its correlation function $\kappa$ is at least $\mathcal{C}^{2p-2}$. Let $x =(x_i)_{1 \leq i \leq p}$ and $y=(y_j)_{1 \leq j \leq p} \in \R^p$, by Lemmas~\ref{lem evaluation divided differences} and~\ref{lem double divided differences}, the variance matrix of $[\ev]_x(f)$ and $[\ev]_y(f)$ equals:
\begin{align*}
\esp{[\ev]_x(f) \trans{[\ev]_y(f)}} &= M(x)^{-1} \esp{ \ev_x(f) \trans{ \ev_y(f)}} \trans{M(y)^{-1}}\\
&= M(x)^{-1} \begin{pmatrix}
\dfrac{(-1)^{c_i(x)}\kappa^{(c_i(x) + c_j(y))}(y_j-x_i)}{c_i(x)! c_j(y)!}
\end{pmatrix}_{\substack{1 \leq i ,j \leq p}} \trans{M(y)^{-1}}\\
&= \begin{pmatrix}
[\kappa]_{(i,j)}(x_1,\dots,x_i,y_1,\dots,y_j)\rule{0em}{2.5ex}
\end{pmatrix}_{\substack{1 \leq i,j \leq p}}
\end{align*}
where $M(x)$ (resp.~$M(y)$) is defined in Definitions~\ref{defs Newton polynomials}, and is invertible by Lemma~\ref{lem coeff M x}. Equivalently, for any $x \in \R^k$ and $y \in \R^l$ with $1 \leq k,l \leq p$, we have $\esp{\rule{0em}{2ex}[f]_k(x) [f]_l(y)} = [\kappa]_{(k,l)}(x,y)$.

By Lemma~\ref{lem coeff M x}, for all $x \in \R^p$ and all $t \in \R$, we have $M(x+(t,\dots,t)) = M(x)$. Hence, using Lemma~\ref{lem evaluation divided differences},
\begin{equation*}
[\ev]_{x+(t,\dots,t)}(f) = M(x)^{-1} \ev_{x+(t,\dots,t)}(f).
\end{equation*}
The stationarity of $f$ implies that for any $t \in \R$, $(\ev_{x+(t,\dots,t)}(f))_{x \in \R^p} = (\ev_x(f))_{x \in \R^p}$ in distribution. Thus, $([\ev]_{x+(t,\dots,t)}(f))_{x \in \R^p}$ is distributed as $\left(M(x)^{-1} \ev_x(f)\right)_{x \in \R^k} = \left([\ev]_x(f)\right)_{x \in \R^k}$. One can also check this distributional invariance directly on the expression of the variance matrix $\esp{[\ev]_x(f) \trans{[\ev]_y(f)}}$ above.
\end{proof}


\section{Kac--Rice densities revisited and clustering}
\label{sec Kac-Rice densities revisited and clustering}

The purpose of this section is to derive alternative expressions for the Kac--Rice density $\rho_k$ defined by Equation~\eqref{eq def rho kx}. The upshot is to be able to choose the nicest of these expressions depending on the point $x \in \R^k$ we are considering. These new expressions use the divided differences introduced in Section \ref{sec divided differences}. In particular, the divided differences allow us to replace $(f(x_1),\dots,f(x_k))$ in Equation~\eqref{eq def rho kx} by another Gaussian vector which is never degenerate, even on the diagonal. We also study the properties of $\rho_k$ using these new expressions. This allows us to prove Theorems~\ref{thm vanishing order} and~\ref{thm clustering}.

In Section~\ref{subsec graph partition}, we define a nice partition of $\R^k$ that we use in the following. In Section~\ref{subsec Kac--Rice densities revisited}, we derive the alternative expressions of the Kac--Rice densities that we are interested in, using the formalism of divided differences introduced in Section~\ref{sec divided differences}. The main result of Section~\ref{subsec Kac--Rice densities revisited} is Proposition~\ref{prop Kac-Rice revisited}. In Section~\ref{subsec proof of theorem vanishing order}, we deduce Theorem~\ref{thm vanishing order} from Proposition~\ref{prop Kac-Rice revisited}. In Section~\ref{subsec variance and covariance matrices}, we introduce notations allowing to study the distribution of the random vectors appearing in the definition of $\rho_k$. Then, we study the clustering properties of the Kac--Rice densities in Sections~\ref{subsec denominator clustering},~\ref{subsec numerator clustering} and~\ref{subsec proof of theorem clustering}. We prove Theorem~\ref{thm clustering} in Section~\ref{subsec proof of theorem clustering}. Several results of this section will also be useful in the proof of Theorem~\ref{thm moments} in Section~\ref{sec proof theorem moments}.


\subsection{Graph partitions}
\label{subsec graph partition}

In this section, given a finite set $A \neq \emptyset$ and a scale parameter $\eta \geq 0$, we define a partition of the Cartesian product $\R^A$ into disjoint pieces. These pieces are indexed by the set $\pa_A$ of partitions of~$A$. In order to do this, we first need to define the graph and the partition associated with a point $x \in \R^A$ and the scale parameter $\eta$. Along the way, we also endow $\pa_A$ with a partially ordered set structure.

\begin{dfn}
\label{def G eta}
Let $A$ be a non-empty finite set and let $\eta \geq 0$. For any $x =(x_a)_{a \in A} \in \R^A$, we define a graph $G_\eta(x)$ as follows:
\begin{itemize}
\item the vertices of $G_\eta(x)$ are the elements of $A$;
\item two vertices $a$ and $b \in A$ are joined by an edge of $G_\eta(x)$ if and only if $a \neq b$ and $\norm{x_a-x_b} \leq \eta$.
\end{itemize}
\end{dfn}

We are not interested in the graph $G_\eta(x)$ itself, but rather in the partition of $A$ defined by its connected components. This partition encodes how the components $(x_a)_{a \in A}$ of $x$ are clustered in~$\R$, at scale $\eta$. An example of this construction is given on Figure~\ref{fig clusters} below.

\begin{dfn}
\label{def I eta}
Let $A$ be a non-empty finite set and let $\eta \geq 0$. We define a map $\I_\eta : \R^A \to \pa_A$ as follows: for all $x=(x_a)_{a \in A} \in \R^A$, $\I_\eta(x)$ is the partition of $A$ given by the connected components of $G_\eta(x)$. That is $a$ and $b \in A$ belong to the same element of $\I_\eta(x)$ if and only if they are in the same connected component of $G_\eta(x)$. An element $I \in \I_\eta(x)$, or equivalently the set $\{x_i \mid i \in I\}$, is called a \emph{cluster} of components of $x$ at scale $\eta$.
\end{dfn}

\begin{figure}[ht]
\hspace{0.03\textwidth}
\subfloat[A configuration $x = (x_1,\dots,x_6) \in \R^6$. The circle shows the points at distance at most $\eta$ from $x_4$. \label{fig config}]{\begin{tikzpicture}[x=2.7cm,y=2.7cm]
\draw[-,color=black] (-1.3,0) -- (1.3,0);
\clip(-1.3,-0.58) rectangle (1.3,0.5);
\draw [dash pattern=on 3pt off 3pt] (-0.6,0) circle (1.34cm);
\draw [line width=0.5pt] (-0.6,0)-- (-0.17,0.26);
\fill [color=black] (-1,0) circle (1.5pt);
\draw[color=black] (-0.98,-0.12) node {$x_1$};
\fill [color=black] (-0.8,0) circle (1.5pt);
\draw[color=black] (-0.75,-0.12) node {$x_2$};
\fill [color=black] (-0.6,0) circle (1.5pt);
\draw[color=black] (-0.5,-0.12) node {$x_4$};
\fill [color=black] (0,0) circle (1.5pt);
\draw[color=black] (0.1,-0.12) node {$x_5$};
\fill [color=black] (1,0) circle (1.5pt);
\draw[color=black] (1.1,-0.12) node {$x_3$};
\fill [color=black] (0.8,0) circle (1.5pt);
\draw[color=black] (0.75,-0.12) node {$x_6$};
\draw[color=black] (-0.4,0.21) node {$\eta$};
\end{tikzpicture}}
\hspace{0.04\textwidth}
\subfloat[The graph $G_\eta(x)$ associated with the configuration $x$ and the distance $\eta$ on the left. \label{fig graph}]{\begin{tikzpicture}
\tikz{
	\node (a) [circle,draw] at (0.3,1) {1};
	\node (b) [circle,draw] at (1.4,1) {2};
	\node (c) [circle,draw] at (2.5,1) {3};
	\node (d) [circle,draw] at (3.6,1) {4};
	\node (e) [circle,draw] at (4.7,1) {5};
	\node (f) [circle,draw] at (5.8,1) {6};
	
	\graph{
		(a) -- (b);
		(a) --[bend left=60] (d);
		(b) --[bend left=45] (d);
		(c) --[bend right=60] (f)
	};
}
\end{tikzpicture}}
\caption{Example of a configuration of six points in $\R$. The partition $\I_\eta(x)$ defined by $\eta$ and $x = (x_1,\dots,x_6)$ is $\I_\eta(x) = \left\{\{1,2,4\},\{3,6\},\{5\}\right\} \in \pa_6$. \label{fig clusters}}
\end{figure}

Let us now define a partial order $\leq$ on the set $\pa_A$ of partitions of $A$.

\begin{dfn}[Partial order on partitions]
\label{def partial order pa p}
Let $A$ be a non-empty finite set and let $\I, \J \in \pa_A$. We denote $\J \leq \I$ if $\J$ is \emph{finer} than~$\I$, that is for all $J \in \J$, there exists $I \in \I$ such that $J \subset I$. We denote by $\J < \I$ the fact that $\J \leq \I$ and $\I \neq \J$.
\end{dfn}

One can check that $\leq$ is a partial order on $\pa_A$ such that $\I \mapsto \norm{\I}$ is decreasing. It admits a minimum equal to $\{\{a\} \mid a \in A\}$, and a maximum equal to $\{A\}$.

\begin{ntn}
\label{ntn Imin Imax}
Let $A\neq \emptyset$ be a finite set, we denote the minimum (resp.~maximum) of $(\pa_A,\leq)$ by $\I_{\min}(A) = \{\{a\} \mid a \in A\}$ (resp.~$\I_{\max}(A) =\{A\}$). If $A$ is of the form $\{1,\dots,k\}$, we use the simpler notation $\I_{\min}(k)$ for $\I_{\min}(A)$ (resp.~$\I_{\max}(k)$ for $\I_{\max}(A)$).
\end{ntn}

Let $A$ be a non-empty finite set and let $\I,\J \in \pa_A$. We have $\J \leq \I$ if and only if every $I \in \I$ is obtained as the disjoint union of elements of $\J$. Equivalently, for all $I \in \I$, the set $\{J \in \J \mid J \subset I\}$ is a partition of $I$. This justifies the introduction of the following notation, that will be used in Sections~\ref{subsec variance and covariance matrices} and~\ref{subsec denominator clustering}.

\begin{ntn}[Induced partition]
\label{ntn J I}
Let $A \neq \emptyset$ be a finite set and let $\I,\J \in \pa_A$ be such that $\J \leq \I$. For all $I \in \I$, we denote by $\J_I = \{J \in \J \mid J \subset I\} \in \pa_I$. Note that if $\J \leq \I$ in $\pa_A$, we have $\J = \bigsqcup_{I \in \I} \J_I$.
\end{ntn}

\begin{lem}
\label{lem partial order}
Let $A$ be a non-empty finite set, for all $x \in \R^A$ the map $\eta \mapsto \I_\eta(x)$ is non-decreasing from $[0,+\infty)$ to $\pa_A$.
\end{lem}

\begin{proof}
Let $x\in \R^A$ and let $0 \leq \eta \leq \eta'$. Let $a,b \in A$, if $a$ and $b$ belong to the same cluster of $\I_\eta(x)$, then they are in the same connected component of $G_\eta(x)$ by definition. Every edge of $G_\eta(x)$ is also an edge of $G_{\eta'}(x)$, by Definition~\ref{def G eta}. Hence $a$ and $b$ belong to the same cluster of $\I_{\eta'}(x)$. Thus $\I_\eta(x) \leq \I_{\eta'}(x)$.
\end{proof}

\begin{dfn}
\label{def R A I eta}
Let $A$ be a non-empty finite set, let $\eta \geq 0$ and let $\I \in \pa_A$. We define:
\begin{equation*}
\R^{A}_{\I,\eta} = \left\{ x \in \R^A \mvert \I_\eta(x) = \I \right\}.
\end{equation*}
\end{dfn}

\begin{rems}
\label{rem R A I eta}
The sets $\R^A_{\I,\eta}$ we just defined satisfy the following properties.
\begin{itemize}
\item For any finite set $A\neq \emptyset$ and any $\eta \geq 0$, we can partition $\R^A$ as follows: $\R^A = \bigsqcup_{\I \in \pa_A} \R^{A}_{\I,\eta}$.
\item Let $\eta \geq 0$, let $\I \in \pa_A$ and let $(x_a)_{a \in A} \in \R^{A}_{\I,\eta}$. For any cluster $I \in \I$ and any $i,j \in I$, we have: $\norm{x_i-x_j} \leq (\norm{I}-1)\eta \leq \norm{A}\eta$.
\end{itemize}
\end{rems}

\begin{ex}
\label{ex I eta}
Let $A \neq \emptyset$ be a finite set and let $x=(x_a)_{a \in A} \in \R^A$. If $\eta = 0$, then $i$ and $j$ are in the same cluster of $\I_0(x)$ if and only if $x_i=x_j$. That is, for all $\I \in \pa_A$, we have $\R^A_{\I,0} = \Delta_{A,\I}$ (see Definition~\ref{def diagonals}). In particular, the partition $\I$ appearing in Theorem~\ref{thm vanishing order} is $\I_0(y)$.
\end{ex}

\begin{lem}
\label{lem clusters are not interlaced}
Let $A$ be a non-empty finite set, let $\I \in \pa_A$ and let $I,J \in \I$ be such that $I \neq J$. Let $\eta \geq 0$, for any $x=(x_a)_{a \in A} \in \R^A_{\I,\eta}$, we have either $(\underline{x}_J)_{\min} > (\underline{x}_I)_{\max} + \eta$ or $(\underline{x}_J)_{\max} < (\underline{x}_I)_{\min} - \eta$ (see Notations~\ref{ntn product indexed by A} and~\ref{ntn x min x max}).
\end{lem}

\begin{proof}
We can write $I = \{i_1,\dots,i_{\norm{I}}\}$ in such a way that:
\begin{equation*}
(\underline{x}_I)_{\min}-\eta < x_{i_1} \leq x_{i_2} \leq \dots \leq x_{i_{\norm{I}}} < (\underline{x}_I)_{\max}-\eta.
\end{equation*}
Moreover, since $\I_\eta(x) = \I$, we have $x_{i_{k+1}} - x_{i_k} \leq \eta$ for all $k \in \{1,\dots,\norm{I}-1\}$. Let $j \in J$, if we had $x_j \in [(\underline{x}_I)_{\min}-\eta,(\underline{x}_I)_{\max}+\eta]$, there would exists $i \in I$ such that $\norm{x_i - x_j} \leq \eta$. This would contradict the fact that $i$ and $j$ are in different clusters of components of $x$ at scale $\eta$. Hence, for all $j \in J$, we have either $x_j > (\underline{x}_I)_{\max} + \eta$ or $x_j < (\underline{x}_I)_{\min} - \eta$. Symmetrically, for all $i \in I$, we have either $x_i > (\underline{x}_J)_{\max} + \eta$ or $x_i < (\underline{x}_J)_{\min} - \eta$.

If we had both $(\underline{x}_J)_{\min} < (\underline{x}_I)_{\min} - \eta$ and $(\underline{x}_J)_{\max} > (\underline{x}_I)_{\max} + \eta$, we would have, for all $i \in I$, $x_i \in [(\underline{x}_J)_{\min},(\underline{x}_J)_{\max}]$, which is absurd. Hence, we have either $(\underline{x}_J)_{\min} > (\underline{x}_I)_{\max} + \eta$ or $(\underline{x}_J)_{\max} < (\underline{x}_I)_{\min} - \eta$.
\end{proof}


\subsection{Kac--Rice densities revisited}
\label{subsec Kac--Rice densities revisited}

The goal of this section is to derive new expressions of the Kac--Rice density $\rho_k$ (cf.~Definition~\ref{def Kac-Rice densities}), in terms of divided differences studied in Section~\ref{sec divided differences}. This is done in Proposition~\ref{prop Kac-Rice revisited} below.

In all this section, we denote by $A$ a non-empty finite set and by $f$ a normalized stationary centered Gaussian process which is at least of class $\mathcal{C}^1$.

\begin{dfn}[Evaluation map associated with a partition]
\label{def evaluation map partition}
Let $x =(x_a)_{a \in A} \in \R^A$. Following the notations of Section~\ref{subsec partitions, products and diagonal inclusions}, for any $\I \in \pa_A$ we denote by $\ev_x^\I: \mathcal{C}^{\norm{A}-1}(\R) \to \R^{\norm{A}}$ the linear map defined by $\ev_x^\I: f \mapsto ([\ev]_{\underline{x}_I}(f))_{I \in \I}$, where $[\ev]_x$ is defined in Definition~\ref{def evaluation divided differences}.
\end{dfn}

\begin{rem}
\label{rem ordering of partitions}
Given $\I \in \pa_A$, we need to fix an ordering of $\I$ and an ordering of each $I \in \I$ for $\ev_x^\I$ to be well-defined. Here and throughout the paper, we implicitly assume that such orderings are fixed whenever necessary. The precise choice of these orderings will be of no consequence.
\end{rem}

\begin{exs}
\label{exs evaluation map partition}
Let us give some examples.
\begin{enumerate}
\item \label{subex evaluation map partition 1} If $\I = \{\{1,3\},\{2,4,5\}\} \in \pa_5$ and $x=(x_i)_{1 \leq i\leq 5}$, for all $f \in \mathcal{C}^4(\R)$ we have:
\begin{equation*}
\ev_x^\I(f) = \left([f]_1(x_1),[f]_2(x_1,x_3),[f]_1(x_2),[f]_2(x_2,x_4),[f]_3(x_2,x_4,x_5)\right).
\end{equation*}
\item \label{subex evaluation map partition 2} For all $x=(x_a)_{a \in A} \in \R^A$, we have $\ev_x^{\I_{\min}(A)}: f \mapsto (f(x_a))_{a \in A}$, where $\I_{\min}(A)$ is as in Notation~\ref{ntn Imin Imax}.
\item \label{subex evaluation map partition 3} If $\I = \I_{\max}(A)=\{A\}$, then for all $x\in \R^A$, we have $\ev_x^{\{A\}} = [\ev]_x$ after choosing some ordering of $A$. In particular, we have $M(x) \ev_x^{\{A\}} = \ev_x$ by Lemma~\ref{lem coefficients of Hermite interpolation polynomial}.
\item \label{subex evaluation map partition 4} More generally, let $\I \in \pa_A$. Let us choose an ordering of each $I \in \I$, and let us also choose an ordering of $\I$, say $\I = \{I_i \mid 1 \leq i \leq \norm{I}\}$. This yields an ordering of $A$ such that if $a \in I_i$ and $b \in I_j$ with $i <j$ then $a<b$, and whose restriction to any $I \in \I$ coincides with the ordering of $I$ we fixed. Using this ordering, we can identify $A$ with $\{1,\dots,\norm{A}\}$. Then, for all $x \in \R^A\setminus \Delta_A$, we have:
\begin{equation}
\label{eq ex evaluation map partition}
\begin{pmatrix}
M(\underline{x}_{I_1}) & & \\ & \ddots & \\ & & M(\underline{x}_{I_{\norm{\I}}})
\end{pmatrix} \ev_x^\I = \ev_x.
\end{equation}
Note that Equation~\eqref{eq ex evaluation map partition} holds independently of our choices or orderings of $\I$ and each $I \in \I$, as long as they are consistent from one term to the other.
\end{enumerate}
\end{exs}

Recall that the Kac--Rice densities were defined by Definition~\ref{def Kac-Rice densities}. One of the key ideas in this paper is that we can find alternative expressions of $\rho_k$ (see Equation~\eqref{eq def rho kx}). These alternative expressions are indexed by $\I \in \pa_k$ (see Definition~\ref{def Kac-Rice densities partition} below). Then, we are to choose the right $\I$, that is the right expression for $\rho_k$, depending on the point $x \in \R^k$ at which we want to evaluate~$\rho_k$. 

\begin{dfn}[Kac--Rice densities associated with a partition]
\label{def Kac-Rice densities partition}
Let $A$ be a non-empty finite set and let $f$ be a centered Gaussian process of class $\mathcal{C}^{\norm{A}}$. Let $\I \in \pa_A$, for all $x = (x_a)_{a \in A} \in \R^A$, we denote by
\begin{equation}
\label{eq def DIx}
D_\I(x) = \det\left(\var{\ev_x^\I(f)}\right).
\end{equation}
Moreover, if $\ev_x^\I(f)$ is non-degenerate, i.e.~if $D_\I(x) \neq 0$, we denote by:
\begin{equation}
\label{eq def NIx}
N_\I(x) = \espcond{\prod_{I \in \I} \prod_{i \in I} \norm{[f]_{\norm{I}+1}(\underline{x}_I,x_i)}}{\ev_x^\I(f)=0},
\end{equation}
the conditional expectation of $\prod_{I \in \I} \prod_{i \in I} \norm{[f]_{\norm{I}+1}(\underline{x}_I,x_i)}$ given that $\ev_x^\I(f)=0$. Finally we denote by:
\begin{equation}
\label{eq def rho Ix}
\rho_\I(x) = \left(\prod_{I \in \I} \prod_{\{(i,j) \in I^2 \mid i \neq j\}} \norm{x_i-x_j}^\frac{1}{2}\right) \frac{N_\I(x)}{(2\pi)^\frac{\norm{A}}{2}D_\I(x)^\frac{1}{2}}.
\end{equation}
\end{dfn}

\begin{rems}
\label{rem def Kac-Rice densities partition}
This definition requires some comments.
\begin{itemize}
\item Both $D_\I(x)$ and $N_\I(x)$ only depend on the joint distribution of the divided differences
\begin{equation}
\label{eq divided differences in DI and NI}
\left\{[f]_j(x_{a_1},\dots,x_{a_j}) \mvert 1 \leq j \leq \norm{A}+1 \ \text{and} \ a_1,\dots,a_j \in A\right\}.
\end{equation}
By Lemma~\ref{lem variance divided differences}, this distribution is a centered Gaussian. This remains true after conditioning on $\ev_x^\I(f)=0$, as soon as $\ev_x^\I(f)$ is non-degenerate. In particular, $D_\I(x)$, $N_\I(x)$ and $\rho_\I$ are well-defined.
\item By Lemma~\ref{lem variance divided differences}, the joint distribution of the divided differences~\eqref{eq divided differences in DI and NI} is invariant by diagonal translation. In particular, for any $x \in \R^p$ and any $t \in \R$ we have $D_\I(x+(t,\dots,t)) = D_\I(x)$, $N_\I(x+(t,\dots,t)) = N_\I(x)$ and $\rho_\I(x+(t,\dots,t)) = \rho_\I(x)$.
\item The definitions of $D_\I(x)$, $N_\I(x)$ and $\rho_\I(x)$ do not depend on the orderings we choose on $\I$ and on each $I \in \I$. This is not obvious, and will be proved in Lemmas~\ref{lem expression DIx} and~\ref{lem expression NIx} below.
\end{itemize}
\end{rems}

\begin{lem}[Continuity]
\label{lem DI and NI continuous}
If $f$ is of class $\mathcal{C}^{\norm{A}}$, then for all $\I \in \pa_A$, the maps $D_\I$, $N_\I$ and $\rho_\I$ are continuous on their domains of definition.
\end{lem}

\begin{proof}
The proof is similar to that of Lemma~\ref{lem Dk and Nk continuous}. Let $\I \in \pa_A$, for all $x =(x_a)_{a \in A} \in \R^A$, we denote by $X_\I(x) = \ev_x^\I(x)$ and by $Y_\I(x) = \left([f]_{\norm{I}+1}(\underline{x}_I,x_i)\right)_{I \in \I, i \in I}$. By Lemma~\ref{lem variance divided differences}, $(X_\I(x),Y_\I(x))_{x \in \R^A}$ is a continuous centered Gaussian field with values in $\R^{2\norm{A}}$. We write the variance matrix of $(X_\I(x),Y_\I(x))$ by square blocks of size $\norm{A}$ as:
\begin{equation*}
\begin{pmatrix}
\Theta_\I(x) & \trans{\Xi_\I(x)} \\ \Xi_\I(x) & \Omega_\I(x)
\end{pmatrix},
\end{equation*}
where $\Theta_\I$, $\Xi_\I$ and $\Omega_\I$ are continuous maps on~$\R^A$. Then, $D_\I = \det(\Theta_\I)$ is continuous on $\R^A$.

If $x \in \R^A$ is such that $D_\I(x) \neq 0$, then $Y_\I(x)$ given that $X_\I(x) = 0$ is a well-defined centered Gaussian variable, whose variance matrix $\Lambda_\I(x) = \Omega_\I(x) - \Xi_\I(x)\Theta_\I(x)^{-1} \trans{\Xi_\I(x)}$ depends continuously on $x$. Then, $N_\I(x) = \Pi_{\norm{A}}(\Lambda_\I(x))$, where $\Pi_{\norm{A}}$ is defined by Definition~\ref{def Pi k}. By Corollary~\ref{cor Pi k}, $\Pi_{\norm{A}}$ is continuous. Hence $N_\I$ is continuous on $\{x \in \R^A \mid D_\I(x) \neq 0\}$, and so is $\rho_\I$.
\end{proof}

\begin{lem}
\label{lem expression DIx}
Let us assume that $f$ is $\mathcal{C}^{\norm{A}}$. Let $\I \in \pa_A$, for all $x =(x_a)_{a \in A} \in \R^A$ we have:
\begin{equation}
\label{eq lem expression DIx}
D_{\norm{A}}(x) = \left(\prod_{I \in \I} \ \prod_{\{(i,j) \in I^2 \mid i \neq j\}} \norm{x_i-x_j}\right) D_\I(x),
\end{equation}
where $D_{\norm{A}}$ is the symmetric function defined by Equation~\eqref{eq def Dkx}. In particular, $D_\I(x)$ is independent of the ordering on~$\I$ and of the ordering on each $I \in \I$ used to define $\ev_x^\I$.
\end{lem}

\begin{proof}
Let $\I \in \pa_A$. As in Example~\ref{exs evaluation map partition}.\ref{subex evaluation map partition 4}, we choose an ordering of each $I \in \I$ and an ordering of~$\I$, say $\I = \{I_i \mid 1 \leq i \leq \norm{I}\}$. This defines an ordering of $A$, that coincides with the one on each $I \in \I$ and such that if $i < j$ then the elements of $I_i$ are smaller than those of $I_j$. Let us identify $A$ with $\{1,\dots,\norm{A}\}$ using this ordering. For all $x =(x_a) \in \R^A \setminus \Delta_A$, we have:
\begin{equation}
\label{eq relation evaluations}
\begin{pmatrix} M(\underline{x}_{I_1}) & & \\ & \ddots & \\ & & M(\underline{x}_{I_{\norm{\I}}})
\end{pmatrix} \ev_x^\I(f) = \ev_x(f) = (f(x_1),\dots,f(x_{\norm{A}})).
\end{equation}
Taking the determinant of the variance of these random vectors, we obtain:
\begin{equation*}
D_{\norm{A}}(x) = \left(\prod_{j=1}^{\norm{\I}} \det \begin{pmatrix}
M(\underline{x}_{I_j})
\end{pmatrix}^2\right) D_\I(x) = \left(\prod_{I \in \I} \det \begin{pmatrix}
M(\underline{x}_I)
\end{pmatrix}^2\right) D_\I(x).
\end{equation*}
Let $I \in \I$ and let us assume that $I = \{i_1,\dots,i_{\norm{I}}\}$. Since $x \notin \Delta_A$, by Lemma~\ref{lem coeff M x} we have
\begin{equation*}
\det(M(\underline{x}_I))^2 = \prod_{j=1}^{\norm{I}} \prod_{l=1}^{j-1} (x_{i_j}-x_{i_l})^2 = \prod_{1 \leq l < j \leq \norm{I}} (x_{i_j}-x_{i_l})^2 = \prod_{\{(i,j) \in I^2 \mid i \neq j\}} \norm{x_i - x_j}.
\end{equation*}
This proves Equation~\eqref{eq lem expression DIx} for any $x \in \R^A \setminus \Delta_A$. Since $\R^A \setminus \Delta_A$ is dense in $\R^A$ and both sides of Equation~\eqref{eq lem expression DIx} are continuous functions (see Lemmas~\ref{lem Dk and Nk continuous} and~\ref{lem DI and NI continuous}), this relation holds for all~$x \in \R^A$.

By Lemma~\ref{lem Kac-Rice densities symmetric}, the function $D_{\norm{A}}$ is symmetric on $\R^{\norm{A}}$. Hence the left-hand side of Equation~\eqref{eq lem expression DIx} does not depend on the ordering of $A$ we used to identify $A$ and $\{1,\dots,\norm{A}\}$. Let $D_\I(x)$ be defined for all $x \in \R^A$ by Equation~\eqref{eq def DIx}, and let $\tilde{D}_\I(x)$ be defined similarly but with other choices of ordering on $\I$ or on some $I \in \I$. Using Equation~\eqref{eq lem expression DIx}, for all $x \in \R^A \setminus \Delta_A$, we have:
\begin{equation*}
D_\I(x) = \left(\prod_{I \in \I} \ \prod_{\{(i,j) \in I^2 \mid i \neq j\}} \norm{x_i-x_j}\right)^{-1}D_{\norm{A}}(x) = \tilde{D}_\I(x).
\end{equation*}
Indeed the middle term does not depend on the ordering of $\I$, nor on the orderings of each $I \in \I$. Then $D_\I = \tilde{D}_\I$, since these are continuous functions on $\R^A$ that coincide on a dense subset.
\end{proof}

\begin{cor}
\label{cor expression DIx}
Let us assume that $f \in \mathcal{C}^{\norm{A}}$. Let $\I,\J \in \pa_A$ be such that $\J \leq \I$, then for all $x= (x_a)_{a \in A}$ we have:
\begin{equation*}
D_\J(x) = \left(\prod_{I \in \I} \prod_{\{(J,J') \in \J_I^2 \mid J \neq J'\}} \prod_{(i,j) \in J \times J'} \norm{x_i - x_j}\right) D_\I(x),
\end{equation*}
where for all $I \in \I$, we denoted by $\J_I =\{J \in \J \mid J \subset I\} \in \pa_\I$, as in Notation~\ref{ntn J I}.
\end{cor}

\begin{proof}
By Lemma~\ref{lem expression DIx}, for all $x=(x_a)_{a \in A} \in \R^A$ we have:
\begin{align*}
\left(\prod_{I \in \I} \prod_{\{(i,j) \in I^2 \mid i \neq j\}} \norm{x_i-x_j}\right) D_\I(x) &=\left(\prod_{J \in \J} \prod_{\{(i,j) \in J^2 \mid i \neq j\}} \norm{x_i-x_j}\right) D_\J(x) \\
&= \left(\prod_{I \in \I} \prod_{J \in \J_I} \prod_{\{(i,j) \in J^2 \mid i \neq j\}} \norm{x_i-x_j}\right) D_\J(x).
\end{align*}
Thus, the result is true for all $x \in \R^A \setminus \Delta_A$. Since this set is dense in $\R^A$ and since $D_\I$ and $D_\J$ are continuous (see Lemma~\ref{lem DI and NI continuous}), this concludes the proof.
\end{proof}

\begin{lem}
\label{lem expression NIx}
Let us assume that $f$ is $\mathcal{C}^{\norm{A}}$. Let $\I \in \pa_A$, for all $x =(x_a)_{a \in A} \in \R^A$ such that $D_{\norm{A}}(x) \neq 0$, we have:
\begin{equation}
\label{eq lem expression NIx}
N_{\norm{A}}(x) = \left(\prod_{I \in \I} \ \prod_{\{(i,j) \in I^2 \mid i \neq j\}} \norm{x_i-x_j}\right) N_\I(x),
\end{equation}
where $N_{\norm{A}}$ is the symmetric function defined by Equation~\eqref{eq def Nkx}. Moreover, for all $x \in \R^A$ such that $D_\I(x) \neq 0$, $N_\I(x)$ is independent of the ordering on $\I$ and of the ordering on each $I \in \I$.
\end{lem}

\begin{proof}
Let $\I \in \pa_A$, as in the proof of Lemma~\ref{lem expression DIx} above, we assume that $\I = \{I_i \mid 1 \leq i \leq \norm{\I}\}$ is ordered, as well as each $I \in \I$. This defines an ordering of $A$ that we use to identify $A$ and $\{1,\dots,\norm{A}\}$. By Lemma~\ref{lem Kac-Rice densities symmetric}, neither the condition $D_{\norm{A}}(x) \neq 0$ nor the left-hand side of Equation~\eqref{eq lem expression NIx} depend on a choice of ordering of $A$.

Let $x \in \R^A$, if $D_\I(x) = 0$, then $D_{\norm{A}}(x) = 0$, by Lemma~\ref{lem expression DIx}. Conversely, if $x$ is such that $D_\I(x) \neq 0$, by continuity of $D_\I$ (see Lemma~\ref{lem expression DIx}) there exists a neighborhood~$U$ of~$x$ on which $D_\I$ does not vanish. By Lemma~\ref{lem expression DIx}, for all $y \in U \setminus \Delta_A$, we have $D_{\norm{A}}(y) \neq 0$. Thus $\{x \in \R^A \mid D_{\norm{A}}(x) \neq 0\}$ is a dense subset of $\{x \in \R^A \mid D_\I(x) \neq 0\}$. In particular, both $N_{\norm{A}}$ and $N_\I$ are well-defined on $\{x \in \R^A \mid D_{\norm{A}}(x) \neq 0\} \subset \R^A \setminus \Delta_A$.

Let $x = (x_a)_{a \in A} \in \R^A$ be such that $D_{\norm{A}}(x) \neq 0$. Since $D_{\norm{A}}(x) \neq 0$ we have $x \notin \Delta_A$. Then, as in the proof of Lemma~\ref{lem expression DIx}, Equation~\eqref{eq relation evaluations} holds under the identification of $A$ and $\{1,\dots,\norm{A}\}$. By Lemma~\ref{lem coeff M x}, for all $j \in \{1, \dots,\norm{\I}\}$ the matrix $M(\underline{x}_{I_j})$ is invertible. Hence, it is equivalent to condition on $f(x_1) = \dots = f(x_{\norm{A}}) = 0$ and on $\ev_x^\I(f) = 0$. Thus,
\begin{equation*}
N_{\norm{A}}(x) = \espcond{\prod_{I \in \I} \prod_{i \in I} \norm{f'(x_i)}}{\ev_x^\I(f)=0}.
\end{equation*}
Let $I = \{i_1,\dots,i_{\norm{I}}\} \in \I$. The condition $\ev_x^\I(f)=0$ implies that $[\ev]_{\underline{x}_I}(f) =0$, that is $[f]_j(x_{i_1},\dots,x_{i_j}) = 0$ for all $j \in \{1,\dots, \norm{I}\}$. Under this condition, by Example~\ref{ex divided differences}.\ref{subex DD derivative}, we have:
\begin{equation*}
f'(x_i) = [f]_{\norm{I}+1}(x_{i_1},\dots,x_{i_{\norm{I}}},x_i) \prod_{\substack{1 \leq j \leq \norm{I}\\ i_j \neq i}} (x_i-x_{i_j}) = [f]_{\norm{I}+1}(\underline{x}_I,x_i) \prod_{j \in I \setminus \{i\}} (x_i-x_j),
\end{equation*}
for all $i \in I$. Hence,
\begin{equation*}
\prod_{i \in I} \norm{f'(x_i)} = \left(\prod_{i \in I} \norm{[f]_{\norm{I}+1}(\underline{x}_I,x_i)}\right) \left(\prod_{\{(i,j) \in I^2 \mid i \neq j\}}\norm{x_i-x_j}\right).
\end{equation*}
This proves that Equation~\eqref{eq lem expression NIx} holds for all $x \in \R^A$ such that $D_{\norm{A}}(x) \neq 0$.

Let $N_\I(x)$ be defined by Equation~\eqref{eq def NIx} and let $\tilde{N}_\I$ be defined similarly but with other choices of ordering on $\I$ or on some $I \in \I$. Recall that the function $N_{\norm{A}}$ is symmetric by Lemma~\ref{lem Kac-Rice densities symmetric}. Using Equation~\eqref{eq lem expression NIx}, for all $x \in \R^A$ such that $D_{\norm{A}}(x) \neq 0$ we have:
\begin{equation*}
N_\I(x) = \left(\prod_{I \in \I} \ \prod_{\{(i,j) \in I^2 \mid i \neq j\}} \norm{x_i-x_j}^{-1}\right) N_{\norm{A}}(x) = \tilde{N}_\I(x),
\end{equation*}
since the middle term does not depend on our choices of orderings on $\I$ and on each $I \in \I$. By Lemmas~\ref{lem DI and NI continuous} and~\ref{lem expression DIx}, both $N_\I$ and $\tilde{N}_\I$ are defined and continuous on $\{x \in \R^k \mid D_\I(x) \neq 0\}$. They coincide on the dense subset $\{x \in \R^k \mid D_{\norm{A}}(x) \neq 0\}$, hence everywhere.
\end{proof}

\begin{cor}
\label{cor expression NIx}
Let us assume that $f \in \mathcal{C}^{\norm{A}}$. Let $\I,\J \in \pa_A$ be such that $\J \leq \I$, then for all $x= (x_a)_{a \in A}$ such that $D_\J(x) \neq 0$, we have:
\begin{equation*}
N_\J(x) = \left(\prod_{I \in \I} \prod_{\{(J,J') \in \J_I^2 \mid J \neq J'\}} \prod_{(i,j) \in J \times J'} \norm{x_i - x_j}\right) N_\I(x),
\end{equation*}
where for all $I \in \I$, $\J_I =\{J \in \J \mid J \subset I\} \in \pa_I$, as in Notation~\ref{ntn J I}.
\end{cor}

\begin{proof}
Let $x=(x_a)_{a \in A} \in \R^A$ be such that $D_\J(x) \neq 0$. By Corollary~\ref{cor expression DIx}, we have $D_\I(x) \neq 0$ and both $N_\J(x)$ and $N_\I(x)$ are well-defined. Then, by Lemma~\ref{lem expression NIx},
\begin{align*}
\left(\prod_{I \in \I} \prod_{\{(i,j) \in I^2 \mid i \neq j\}} \norm{x_i-x_j}\right) N_\I(x) &=\left(\prod_{J \in \J} \prod_{\{(i,j) \in J^2 \mid i \neq j\}} \norm{x_i-x_j}\right) N_\J(x) \\
&= \left(\prod_{I \in \I} \prod_{J \in \J_I} \prod_{\{(i,j) \in J^2 \mid i \neq j\}} \norm{x_i-x_j}\right) N_\J(x).
\end{align*}
Thus, the result is true for all $x \in \R^A \setminus \Delta_A$ such that $D_\J(x) \neq 0$. Since $\{ x \in \R^A \setminus \Delta_A \mid D_\J(x) \neq 0\}$ is dense in $\{x \in \R^A \mid D_\J(x) \neq 0\}$ and since $N_\I$ and $N_\J$ are continuous (see Lemma~\ref{lem DI and NI continuous}), this concludes the proof.
\end{proof}

Recall that, in Theorems~\ref{thm moments} and~\ref{thm clustering}, we consider a normalized stationary centered Gaussian process $f$ whose correlation function $\kappa$ tends to $0$ at infinity. In particular, it satisfies the hypotheses of Lemma~\ref{lem non-degeneracy}.

\begin{lem}[Positivity]
\label{lem partition such that DIx positive}
Let us assume that $f$ is of class $\mathcal{C}^{\norm{A}}$ and that its correlation function $\kappa$ is such that $\kappa(x) \xrightarrow[x \to +\infty]{}0$. Let $x \in \R^A$ and let us denote by $\I = \I_0(x)$. Then $D_\I(x) >0$ and $N_\I(x) >0$.
\end{lem}

\begin{proof}
Let $x\in \R^A$ and let $\I = \I_0(x) \in \pa_A$ be the partition defined by Definition~\ref{def I eta}. Since $x \in \R^A_{\I,0} = \Delta_{A,\I}$, there exists $y =(y_I)_{I \in \I} \in \R^\I \setminus \Delta_\I$ such that $x = \iota_\I(y)$ (see Definition~\ref{def diagonal inclusions} and Remark~\ref{rem diagonal inclusions}). Let $I \in \I$, we have $\underline{x}_I = (y_I,\dots,y_I) \in \R^I$. Then, by Definition~\ref{def evaluation map partition} and Example~\ref{ex divided differences}.\ref{subex DD Taylor}, we have $[\ev]_{\underline{x}_I}(f) = \left(\frac{f^{(i)}(y_I)}{i!}\right)_{0 \leq i < \norm{I}}$ and $\ev_x^\I(f) = \left(\frac{f^{(i)}(y_I)}{i!}\right)_{I \in \I, 0 \leq i < \norm{I}}$. Since $\kappa$ tends to $0$ at infinity, this Gaussian vector is non-degenerate by Lemma~\ref{lem non-degeneracy}, that is~$D_\I(x) >0$.

The previous expression of $\ev_x^\I(f)$ shows that the condition $\ev_x^\I(f)=0$ is equivalent to: $\forall I \in \I$, $\forall i \in \{0,\dots,\norm{I}-1\}$, $f^{(i)}(y_I)=0$. For all $i \in I$, by Example~\ref{ex divided differences}.\ref{subex DD Taylor} we have:
\begin{equation*}
[f]_{\norm{I}+1}(\underline{x}_I,x_i) = [f]_{\norm{I}+1}(y_I,\dots,y_I) = \frac{f^{(\norm{I})}(y_I)}{\norm{I}!}.
\end{equation*}
Hence,
\begin{equation}
\label{eq expression NIx}
N_\I(x) = \left(\prod_{I \in \I} \norm{I}!^{-\norm{I}}\right)\espcond{\prod_{I \in \I} \norm{f^{(\norm{I)}}(y_I)}^{\norm{I}}}{\forall I \in \I, \forall i \in \{0,\dots,\norm{I}-1\}, f^{(i)}(y_I)=0}.
\end{equation}
Let us denote by $U = (f^{(i)}(y_I))_{I \in \I, 0 \leq i < \norm{I}}$ and by $V = (f^{(\norm{I})}(y_I))_{I \in \I}$. Since $\kappa(x) \xrightarrow[x \to +\infty]{}0$, by Lemma~\ref{lem non-degeneracy} the centered Gaussian vector $(U,V)$ is non-degenerate. Let us denote by $\begin{pmatrix} A & \trans{B} \\ B & C \end{pmatrix}$ its block variance matrix. The variance matrix of $V$ given that $U=0$ is $C - B A^{-1} \trans{B}$, see \cite[Proposition~1.2]{AW2009}. Note that $C - B A^{-1} \trans{B}$ is the Schur complement of $A$. In particular,
\begin{equation*}
\det \begin{pmatrix} A & \trans{B} \\ B & C \end{pmatrix} = \det(A) \det(C - B A^{-1} \trans{B}) >0.
\end{equation*}
Thus, the distribution of $V$ given that $U=0$ is a non-degenerate centered Gaussian. Hence it admits a positive density with respect to the Lebesgue measure of $\R^\I$. Finally, by Equation~\eqref{eq expression NIx} we have $N_\I(x) >0$, which concludes the proof.
\end{proof}

\begin{cor}[Vanishing locus]
\label{cor vanishing locus DI}
Let us assume that $f$ is of class $\mathcal{C}^{\norm{A}}$ and $\kappa(x) \xrightarrow[x \to +\infty]{}0$. Let $\I \in \pa_A$, for all $x \in \R^A$, $D_\I(x) = 0$ if and only if there exist $i \in I \in \I$ and $j \in J \in \I$ such that $I \neq J$ and $x_i=x_j$. Moreover, if $D_\I(x) \neq 0$ we have $N_\I(x) >0$. In particular, we have $D_{\{A\}}(x) >0$ and $N_{\{A\}}(x) >0$ for all $x \in \R^A$.
\end{cor}

\begin{proof}
Let $\I \in \pa_A$, by Corollary~\ref{cor expression DIx}, for all $x \in \R^A$ we have:
\begin{equation}
\label{eq cor vanishing locus DI}
D_\I(x) = D_{\{A\}}(x) \prod_{\{(I,J) \in \I^2 \mid I \neq J\}} \prod_{(i,j) \in I \times J} \norm{x_i - x_j}.
\end{equation}
Let $x \in \R^A$, we already know that $D_{\{A\}}(x) \geq 0$. If $D_{\{A\}}(x)=0$, then by Equation~\eqref{eq cor vanishing locus DI} we would have $D_\I(x)=0$ for all $\I \in \pa_A$. Since we assumed that $\kappa(x) \xrightarrow[x \to +\infty]{}0$, this would contradict the result of Lemma~\ref{lem partition such that DIx positive}. Hence, for all $x \in \R^A$, we have $D_{\{A\}}(x) >0$. Then, Equation~\eqref{eq cor vanishing locus DI} shows that $D_\I(x) = 0$ if and only if there exist $i \in I \in \I$ and $j \in J \in \I$ such that $I \neq J$ and $x_i=x_j$.

Similarly, let $x \in \R^A$. Since $D_{\{A\}}(x) >0$, we know that $N_{\{A\}}(x)$ is well-defined and non-negative. For all $\I \in \pa_A$ such that $D_\I(x) \neq 0$, we have:
\begin{equation}
\label{eq cor vanishing locus NI}
N_\I(x) = N_{\{A\}}(x) \prod_{\{(I,J) \in \I^2 \mid I \neq J\}} \prod_{(i,j) \in I \times J} \norm{x_i - x_j},
\end{equation}
by Corollary~\ref{cor expression NIx}. If we had $N_{\{A\}}(x)=0$, we would have $N_\I(x)=0$ for all $\I \in \pa_A$ such that $D_\I(x) >0$, which would contradict Lemma~\ref{lem partition such that DIx positive}. Thus, $N_{\{A\}}(x) >0$ for all $x \in \R^A$. Finally, by Equations~\eqref{eq cor vanishing locus DI} and~\eqref{eq cor vanishing locus NI}, if $\I \in \pa_A$ is such that $D_\I(x) \neq 0$, then $N_\I(x) >0$.
\end{proof}

We can now state the main result of this section.

\begin{prop}[Equality of Kac--Rice densities]
\label{prop Kac-Rice revisited}
Let $f$ be a normalized stationary centered Gaussian process whose correlation function tends to $0$ at infinity. Let $A$ be a non-empty finite set and us assume that $f$ is of class~$\mathcal{C}^{\norm{A}}$. Then, $\rho_{\{A\}}$ is a well-defined continuous map from $\R^A$ to $\R$ such that for all $x \in \R^A$, for all $\I \in \pa_A$, if $D_\I(x) \neq 0$ then $\rho_{\{A\}}(x) = \rho_\I(x)$. Moreover, $\rho_{\{A\}}(x) = 0$ if and only if $x \in \Delta_A$.

In particular, if $f$ is of class $\mathcal{C}^k$ with $k \in \N^*$, then the Kac--Rice density $\rho_k$ (see Definition~\ref{def Kac-Rice densities}) can be uniquely extended into a continuous map from $\R^k$ to $\R$ whose vanishing locus is~$\Delta_k$. Moreover, for all $x \in \R^k$, for all $\I \in \pa_k$ such that $D_\I(x) \neq 0$, we have $\rho_k(x) = \rho_\I(x)$.
\end{prop}

\begin{proof}
Let $A \neq \emptyset$ be a finite set and let us assume that $f$ is $\mathcal{C}^{\norm{A}}$. By Corollary~\ref{cor vanishing locus DI}, the maps $N_{\{A\}}$ and $\rho_{\{A\}}$ are well-defined from $\R^A$ to~$\R$. By Lemma~\ref{lem DI and NI continuous} and Equation~\eqref{eq def rho Ix}, these maps are continuous on~$\R^A$ and $\rho_{\{A\}}$ vanishes along $\Delta_A$. In fact, by Corollary~\ref{cor vanishing locus DI}, we have $\rho_{\{A\}}(x)=0$ if and only if $x \in \Delta_A$.

Let $\I \in \pa_A$, by Corollary~\ref{cor vanishing locus DI}, for all $x \in \R^A \setminus \Delta_A$ we have $D_\I(x) \neq 0$. Then, by Corollaries~\ref{cor expression DIx} and~\ref{cor expression NIx}, for all $x \in \R^A \setminus \Delta_A$, $\rho_\I(x) = \rho_{\{A\}}(x)$. Since $\rho_\I$ and $\rho_{\{A\}}$ are continuous and $\R^A \setminus \Delta_A$ is dense in $\R^A$, we have $\rho_\I(x) = \rho_{\{A\}}(x)$ for any $x \in \R^A$ such that $\rho_\I(x)$ is well-defined.

Let $k \in \N^*$ and let us assume that $f$ is of class $\mathcal{C}^k$. The map $\rho_k$ defined by Equation~\eqref{eq def rho kx} is equal to the map $\rho_{\I_{\min}(k)}$ of Definition~\ref{def Kac-Rice densities partition}, where $\I_{\min}(k) = \{\{i\} \mid 1 \leq i \leq k\}$. Applying the first point of Proposition~\ref{prop Kac-Rice revisited} with $A = \{1,\dots,k\}$ and denoting by $\I_{\max}(k) = \{\{1,\dots,k\}\}$, we obtain:
\begin{equation*}
\rho_k(x) = \rho_{\I_{\min}(k)}(x) = \rho_{\I_{\max}(k)}(x),
\end{equation*}
for all $x \in \R^k \setminus \Delta_k$. Then $\rho_{\I_{\max}(k)}$ is the desired continuous extension of $\rho_k$ to $\R^k$. It is necessarily unique by density of $\R^k \setminus \Delta_k$ in $\R^k$.
\end{proof}

\begin{rem}
\label{rem Kac Rice density partitions}
Given any ordering of $A$, we can identify $A$ and $\{1,\dots,\norm{A}\}$. Then, assuming that~$\kappa$ tends to $0$ at infinity, for any $x \in \R^A \setminus \Delta_A$, we have $\rho_{\{A\}}(x) = \rho_{\I_{\min}(A)}(x) = \rho_{\norm{A}}(x)$. Then, by Lemma~\ref{lem Kac-Rice densities symmetric}, the map $\rho_{\{A\}}$ is symmetric on $\R^A \setminus \Delta_A$, hence on $\R^A$ by continuity. Moreover, for all $\I \in \pa_A$, the function $\rho_\I$ does not depend on the choices of ordering of $\I$ and of each $I \in \I$, by Lemmas~\ref{lem expression DIx} and~\ref{lem expression NIx}.
\end{rem}


\subsection{Proof of Theorem~\ref{thm vanishing order}: vanishing order of the \texorpdfstring{$k$}{}-point function}
\label{subsec proof of theorem vanishing order}

In this section we prove Theorem~\ref{thm vanishing order}. This result is a consequence of Lemmas~\ref{lem DI and NI continuous}, \ref{lem expression DIx} and~\ref{lem expression NIx}, that were proved in the previous section.

\begin{proof}[Proof of Theorem~\ref{thm vanishing order}]
Let $k \in \N^*$ and let $f$ be a normalized stationary centered Gaussian $\mathcal{C}^k$-process. Let $Z = f^{-1}(0)$ be the point process of the zeros of $f$. Let $y =(y_i)_{1 \leq i \leq k} \in \R^k$ and let $\I = \I_0(y) \in \pa_k$. Recalling Definition~\ref{def I eta}, this partition is the only one such that, for any $i,j \in \{1,\dots,k\}$, we have $(y_i=y_j) \iff(\exists I \in \I, \{i,j\} \subset I)$. As in Theorem~\ref{thm vanishing order} and Remark~\ref{rem diagonal inclusions and vanishing order}, we have $y \in \Delta_{\I,k}$ and there exists a unique $(y_I)_{I \in \I} \in \R^\I \setminus \Delta_\I$ such that $y = \iota_\I\left((y_I)_{I \in \I}\right)$. It is characterized by the fact that $y_I$ is the common value of the $(y_i)_{i \in \I}$, for all $I \in \I$.

Let us assume that $(f^{(i)}(y_I))_{I \in \I, 0 \leq i < \norm{I}}$ is non-degenerate. As already discussed in the proof of Lemma~\ref{lem partition such that DIx positive}, this is equivalent to the non-degeneracy of $\ev^\I_y(f)$, see Definition~\ref{def evaluation map} and Example~\ref{ex evaluation map}.\ref{subex DD Taylor}, hence $D_\I(y) >0$. By Lemma~\ref{lem DI and NI continuous}, there exists a neighborhood $U$ of $y$ in $\R^k$ such that $D_\I(x) >0$ for all $x \in U$. By Lemma~\ref{lem expression DIx}, for all $x \in U \setminus \Delta_k$ we have $D_k(x) >0$, so that the Kac--Rice density $\rho_k$ is well-defined on $U \setminus \Delta_k$ (cf.~Definition~\ref{def Kac-Rice densities}). Then, by Lemma~\ref{lem k point function}, the $k$-point function of the point process $Z$ is well-defined and equal to $\rho_k$ on $U \setminus \Delta_k$.

By Lemmas~\ref{lem expression DIx} and~\ref{lem expression NIx}, for any $x \in U \setminus \Delta_k$, we have:
\begin{equation*}
\rho_k(x) = \frac{N_k(x)}{(2\pi)^\frac{k}{2}D_k(x)^\frac{1}{2}} = \left(\prod_{I \in \I} \prod_{\{(i,j) \in I^2 \mid i \neq j\}} \norm{x_i-x_j}^\frac{1}{2}\right) \frac{N_\I(x)}{(2\pi)^\frac{p}{2}D_\I(x)^\frac{1}{2}}
\end{equation*}
By continuity of $D_\I$ and $N_\I$ on $U$ (cf.~Lemma~\ref{lem DI and NI continuous}), we have
\begin{equation*}
\left(\prod_{I \in \I} \prod_{\{(i,j) \in I^2 \mid i < j\}} \frac{1}{\norm{x_i-x_j}}\right) \rho_k(x) \xrightarrow[x \to y]{} \frac{N_\I(y)}{(2\pi)^\frac{p}{2}D_\I(y)^\frac{1}{2}},
\end{equation*}
and it is enough to check that the right-hand side of the previous equation equals $\ell(y)$ (cf.~Equation~\eqref{eq def ly}) to prove the first part of Theorem~\ref{thm vanishing order}.

Let $I \in \I$, for all $i \in I$, we have $y_i = y_I$. Hence, $\ev_y^\I(f) =\left(\frac{f^{(i)}(y_I)}{i!}\right)_{I \in \I, 0 \leq i < \norm{I}}$ and
\begin{equation}
\label{eq proof DIy}
D_\I(y) = \det \var{\ev_y^\I(f)} = \left(\prod_{I \in \I} \prod_{i=0}^{\norm{I}-1} \frac{1}{i!}\right)^2 \det \var{\left(f^{(i)}(y_I)\right)_{I \in \I, 0 \leq i < \norm{I}}}.
\end{equation}
Combining Equations~\eqref{eq proof DIy} and~\eqref{eq expression NIx}, we have $(2\pi)^{-\frac{p}{2}}N_\I(y) D_\I(y)^{-\frac{1}{2}} = \ell(y)$, as expected.

Let us now assume that $(f^{(i)}(y_I))_{I \in \I, 0 \leq i \leq \norm{I}}$ is non-degenerate. Proceeding as in the proof of Lemma~\ref{lem partition such that DIx positive}, this condition ensures that $N_\I(y) >0$. Hence, $\ell(y) >0$, which concludes the proof.
\end{proof}


\subsection{Variance and covariance matrices}
\label{subsec variance and covariance matrices}

In this section, we study the distribution of the random vectors appearing in the definitions of the functions $D_\I$, $N_\I$ and $\rho_I$ (see Definition~\ref{def Kac-Rice densities partition}). We introduce notations for their variance and covariance matrices that will be used in the proof of Theorem~\ref{thm clustering}. Then, we derive some estimates for the coefficients of these matrices. In all this section, we denote by $A$ a non-empty finite set and by $f$ a $\mathcal{C}^{\norm{A}}$ Gaussian process which is assumed to be stationary centered and normalized. Moreover, we denote by $\kappa$ the correlation function of $f$.

Let $\I \in \pa_A$ and $x \in \R^A$, using the same notations as in the proof of Lemma~\ref{lem DI and NI continuous} above, we denote by $X_\I(x)$ and $Y_\I(x) \in \R^{\norm{A}}$ the following centered Gaussian vectors:
\begin{align}
\label{eq def XI}
X_\I(x) &= \ev_x^\I(f),\\
\label{eq def YI}
Y_\I(x) &= \left([f]_{\norm{I}+1}(\underline{x}_I,x_i)\right)_{I \in \I, i \in I},
\end{align}
We also denote by:
\begin{equation}
\label{eq def variance I}
\begin{pmatrix}
\Theta_\I(x) & \trans{\Xi_\I(x)} \\ \Xi_\I(x) & \Omega_\I(x),
\end{pmatrix}
\end{equation}
the variance matrix of $(X_\I(x),Y_\I(x))$, by blocks of size $\norm{A}$. Finally, if $D_\I(x) \neq 0$, i.e.~if $\Theta_\I(x)$ is invertible, we denote by $\Lambda_\I(x)$ the variance of $Y_\I(x)$ given that $X_\I(x) =0$. By~\cite[Proposition~1.2]{AW2009}, we have:
\begin{equation}
\label{eq def Lambda I}
\Lambda_\I(x) = \Omega_\I(x) - \Xi_\I(x) \Theta_\I(x)^{-1} \trans{\Xi_\I(x)}.
\end{equation}
Note that $X_\I$, $Y_\I$, $\Theta_\I$, $\Xi_I$, $\Omega_\I$ and $\Lambda_\I$ depend on how we order $\I$ and each $I \in \I$, but recall that $D_\I$, $N_\I$ and $\rho_\I$ do not (cf.~Lemmas~\ref{lem expression DIx} and~\ref{lem expression NIx}).

We have $\Theta_\I(x) = \var{X_\I(x)} = \var{\ev_x^\I(f)}$. By Definition~\ref{def evaluation map partition}, we can write $\Theta_\I(x)$ as a block matrix indexed by the elements of $\I$: $\Theta_\I(x) = \left(\Theta_{IJ}(\underline{x}_I,\underline{x}_J)\right)_{I,J \in \I}$, where for any $I,J \in \I$,
\begin{equation*}
\Theta_{IJ}(\underline{x}_I,\underline{x}_J) = \esp{[\ev]_{\underline{x}_I}(f) \trans{[\ev]_{\underline{x}_J}(f)}} = \begin{pmatrix}
\esp{\rule{0em}{2ex}[f]_k(x_{i_1},\dots,x_{i_k}) [f]_l(x_{j_1},\dots,x_{j_l})}
\end{pmatrix}_{\substack{1 \leq k \leq \norm{I}\\ 1 \leq l \leq \norm{J}}}.
\end{equation*}
In this last equation, we denoted by $i_1,\dots,i_{\norm{I}}$ (resp.~$j_1,\dots,j_{\norm{J}}$) the elements of $I$ (resp.~$J$). By Lemma~\ref{lem variance divided differences}, we finally obtain that:
\begin{equation}
\label{eq def Theta IJ}
\Theta_{IJ}(\underline{x}_I,\underline{x}_J) = \begin{pmatrix}
[\kappa]_{(k,l)}(x_{i_1},\dots,x_{i_k},x_{j_1},\dots,x_{j_l})
\end{pmatrix}_{\substack{1 \leq k \leq \norm{I}\\ 1 \leq l \leq \norm{J}}},
\end{equation}
where $[\kappa]_{(k,l)}$ is the double divided differences introduced in Definition~\ref{def double divided differences}. Similarly, we can decompose $\Xi_\I(x) = \left(\Xi_{IJ}(\underline{x}_I,\underline{x}_J)\right)_{I,J \in \I}$, $\Omega_\I(x) = \left(\Omega_{IJ}(\underline{x}_I,\underline{x}_J)\right)_{I,J \in \I}$ and $\Lambda_\I(x) = \left(\Lambda_{IJ}(x)\right)_{I,J \in \I}$ by blocks, where for all $I$ and $J \in \I$:
\begin{align}
\label{eq def XI IJ}
\Xi_{IJ}(\underline{x}_I,\underline{x}_J) &= \begin{pmatrix}
[\kappa]_{(\norm{I}+1,l)}(\underline{x}_I,x_{i_k},x_{j_1},\dots,x_{j_l})
\end{pmatrix}_{\substack{1 \leq k \leq \norm{I}\\ 1 \leq l \leq \norm{J}}},\\
\label{eq def Omega IJ}
\Omega_{IJ}(\underline{x}_I,\underline{x}_J) &= \begin{pmatrix}
[\kappa]_{(\norm{I}+1,\norm{J}+1)}(\underline{x}_I,x_{i_k},\underline{x}_J,x_{j_l})
\end{pmatrix}_{\substack{1 \leq k \leq \norm{I}\\ 1 \leq l \leq \norm{J}}},\\
\intertext{and, if $D_\I(x) \neq 0$,}
\label{eq def Lambda IJ}
\Lambda_{IJ}(x) &= \Omega_{IJ}(x) - \left(\Xi_{IK}(\underline{x}_I,\underline{x}_K)\right)_{K \in \I}\Theta_\I(x)^{-1} \left(\trans{\Xi_{JL}(\underline{x}_J,\underline{x}_L)}\right)_{L \in \I}.
\end{align}
Note that, unlike $\Theta_{IJ}$, $\Xi_{IJ}$ and $\Omega_{IJ}$, the function $\Lambda_{IJ}$ depends a priori on all the components of $x$. Note also that the diagonal blocks are such that, for any $I \in \I$, $\Theta_{II}(\underline{x}_I,\underline{x}_I) = \Theta_{\{I\}}(\underline{x}_I)$, $\Xi_{II}(\underline{x}_I,\underline{x}_I) = \Xi_{\{I\}}(\underline{x}_I)$ and $\Omega_{II}(\underline{x}_I,\underline{x}_I) = \Omega_{\{I\}}(\underline{x}_I)$.

\begin{ntn}[Sup-norm]
\label{ntn norm sup matrix}
Let $U=(U_{ij})$ be a matrix, we denote by $\Norm{U}_\infty = \max_{i,j} \norm{U_{ij}}$ its sup-norm.
\end{ntn}

Recall that we defined the norms $\Norm{\cdot}_{k,\eta}$ and $\Norm{\cdot}_k$, for $k \in \N$ and $\eta \geq 0$, in Notation~\ref{ntn norm kappa}.

\begin{lem}
\label{lem boundedness Theta Xi Omega I}
For all $\I \in \pa_A$, for all $x \in \R^A$ we have $\Norm{\Theta_\I(x)}_\infty \leq \Norm{\kappa}_{2\norm{A}}$, $\Norm{\Xi_\I(x)}_\infty \leq \Norm{\kappa}_{2\norm{A}}$ and $\Norm{\Omega_\I(x)}_\infty \leq \Norm{\kappa}_{2\norm{A}}$.
\end{lem}

\begin{proof}
Let us prove this result for $\Theta_\I(x)$. It is enough to prove that for all $I,J \in \I$, we have $\Norm{\Theta_{IJ}(\underline{x}_I,\underline{x}_J)}_\infty \leq \Norm{\kappa}_{2\norm{A}}$. Let $I,J \in \I$, let $k \in \{1,\dots,\norm{I}\}$ and let $l \in \{1,\dots,\norm{J}\}$. By Lemma~\ref{lem double divided differences Rolle},
\begin{equation*}
\norm{[\kappa]_{(k,l)}(x_{i_1},\dots,x_{i_k},x_{j_1},\dots,x_{j_l})} \leq \Norm{\kappa}_{k+l-2} \leq \Norm{\kappa}_{2\norm{A}}.
\end{equation*}
Thus, by Equation~\eqref{eq def Theta IJ}, we have $\Norm{\Theta_{IJ}(\underline{x}_I,\underline{x}_J)}_\infty \leq \Norm{\kappa}_{2\norm{A}}$. The proof is similar for $\Xi_\I(x)$ and $\Omega_\I(x)$, using Equations~\eqref{eq def XI IJ} and~\eqref{eq def Omega IJ}.
\end{proof}

\begin{lem}
\label{lem boundedness Lambda I}
For all $\I \in \pa_A$, for all $x \in \R^A$ such that $D_\I(x) \neq 0$, we have $\Norm{\Lambda_\I(x)}_\infty \leq \Norm{\kappa}_{2\norm{A}}$.
\end{lem}

\begin{proof}
Since $\Lambda_\I(x)$ is a variance matrix, by the Cauchy--Schwarz Inequality it is enough to prove that its diagonal coefficients are bounded by $\Norm{\kappa}_{2\norm{A}}$.

Since $\Theta_\I(x)$ is a variance matrix of determinant $D_\I(x)$, if $D_\I(x) \neq 0$ then $\Theta_\I(x)$ is symmetric and positive definite. Then $\Theta_\I(x)^{-1}$ is also symmetric and positive definite, so that the diagonal coefficients of $\Xi_\I(x) \Theta_\I(x)^{-1} \trans{\Xi_\I(x)}$ are non-negative. Indeed these diagonal coefficients are of the form $Z \Theta_\I(x)^{-1} \trans{Z}$, where $Z$ is one of the rows of $\Xi_\I(x)$.

Thus, by Equation~\eqref{eq def Lambda I}, the diagonal coefficients of $\Lambda_\I(x)$ are non-negative and bounded from above by the corresponding diagonal coefficients of $\Omega_\I(x)$. Finally, $\Norm{\Lambda_\I(x)}_\infty \leq \Norm{\Omega_\I(x)}_\infty$ and the conclusion follows from Lemma~\ref{lem boundedness Theta Xi Omega I}.
\end{proof}

\begin{cor}
\label{cor boundedness NI}
There exists $C>0$ such that, for all $\I \in \pa_A$, for all $x \in \R^A$ such that $D_\I(x) \neq 0$, we have $\norm{N_\I(x)} \leq C$.
\end{cor}

\begin{proof}
We use the notations of Appendix~\ref{sec a Gaussian lemma}, in particular Definition~\ref{def Pi k}. Let $\I \in \pa_A$, for all $x \in \R^A$ such that $D_\I(x) \neq 0$, we have $N_\I(x) = \Pi_{\norm{A}}(\Lambda_\I(x))$. By Lemma~\ref{lem boundedness Lambda I}, the map $\Lambda_\I$ takes values in the compact ball of center $0$ and radius $\Norm{\kappa}_{2\norm{A}}$, in the space of symmetric matrices endowed with~$\Norm{\cdot}_\infty$. By Corollary~\ref{cor Pi k}, the map $\Pi_{\norm{A}}$ is continuous on this ball, hence bounded. Thus $N_\I$ is bounded on $\{x \in \R^A \mid D_\I(x) \neq 0\}$. The conclusion follows from the finiteness of $\pa_A$.
\end{proof}

\begin{lem}
\label{lem variance estimates large eta}
Let $\I,\J \in \pa_A$ be such that $\J \leq \I$ and let us denote $\I=\{I_1,\dots,I_{\norm{\I}}\}$. For all $i \in \{1,\dots,\norm{\I}\}$, we denote by $\J_i = \J_{I_i}$ (see Notation~\ref{ntn J I}) for simplicity. Then, for all $\eta \geq 0$, for all $x \in \R^A_{\I,\eta}$ we have:
\begin{align*}
&\Norm{\Theta_\J(x) - \begin{pmatrix}
\Theta_{\J_1}(\underline{x}_{I_1}) & 0 & \cdots & 0 \\
0 & \Theta_{\J_2}(\underline{x}_{I_2}) & \ddots & \vdots\\
\vdots & \ddots & \ddots & 0\\
0 & \cdots & 0 & \Theta_{\J_{\norm{\I}}}(\underline{x}_{I_{\norm{\I}}})
\end{pmatrix} }_\infty \leq \Norm{\kappa}_{\norm{A},\eta},\\
&\Norm{\Xi_\J(x) - \begin{pmatrix}
\Xi_{\J_1}(\underline{x}_{I_1}) & 0 & \cdots & 0 \\
0 & \Xi_{\J_2}(\underline{x}_{I_2}) & \ddots & \vdots\\
\vdots & \ddots & \ddots & 0\\
0 & \cdots & 0 & \Xi_{\J_{\norm{\I}}}(\underline{x}_{I_{\norm{\I}}})
\end{pmatrix} }_\infty \leq \Norm{\kappa}_{\norm{A},\eta},\\
\intertext{and}
&\Norm{\Omega_\J(x) - \begin{pmatrix}
\Omega_{\J_1}(\underline{x}_{I_1}) & 0 & \cdots & 0 \\
0 & \Omega_{\J_2}(\underline{x}_{I_2}) & \ddots & \vdots\\
\vdots & \ddots & \ddots & 0\\
0 & \cdots & 0 & \Omega_{\J_{\norm{\I}}}(\underline{x}_{I_{\norm{\I}}})
\end{pmatrix} }_\infty \leq \Norm{\kappa}_{\norm{A},\eta}.
\end{align*}
\end{lem}

The statement of Lemma~\ref{lem variance estimates large eta} may seem a little obscure, so let us start by commenting upon it. In the following, we only consider $\Theta_\J(x)$, but the cases of $\Xi_\J(x)$ and $\Omega_\J(x)$ are similar.

If $\J = \I$, then $\J_i = \{I\in \I \mid I \subset I_i\} =\{I_i\}$ for all $i \in \{1,\dots,\norm{\I}\}$. In this case, using the block decomposition of $\Theta_\I(x)$, we have $\Theta_\I(x) = \left(\Theta_{I_i I_j}(\underline{x}_{I_i},\underline{x}_{I_j})\right)_{1 \leq i,j \leq \norm{\I}}$ and, by definition, $\Theta_{\J_i}(\underline{x}_{I_i}) = \Theta_{\{I_i\}}(\underline{x}_{I_i}) = \Theta_{I_i I_i}(\underline{x}_{I_i},\underline{x}_{I_i})$. Then, Lemma~\ref{lem variance estimates large eta} simply states that the sup-norms of the off-diagonal blocks $\Theta_{I_i I_j}(\underline{x}_{I_i},\underline{x}_{I_j})$ with $i \neq j$ are bounded by $\Norm{\kappa}_{\norm{A},\eta}$ on $\R^A_{\I,\eta}$. This can be deduced from Lemmas~\ref{lem double divided differences Rolle} and \ref{lem clusters are not interlaced}.

It turns out that this result remains true if we refine $\I$ by considering some $\J \leq \I$. In this case, for each $i \in \{1,\dots,\norm{\I}\}$ we need to replace $\{I_i\}$ by another partition of $I_i$, namely $\J_i$. Then, $\Theta_\J(x)$ can be written as a block matrix whose diagonal blocks are the $\Theta_{\J_i}(\underline{x}_{I_i})$ with $i \in \{1,\dots,\norm{\I}\}$. The proof of Lemma~\ref{lem variance estimates large eta} in this general case is again a matter of bounding the sup-norm of the off-diagonal blocks in this decomposition of $\Theta_\J(x)$.

\begin{proof}[Proof of Lemma~\ref{lem variance estimates large eta}]
We only give a formal proof of the first inequality. The proofs of the other two inequalities are similar.

Let $\J \leq \I$, using the notations of Equation~\eqref{eq def Theta IJ}, we have $\Theta_\J(x) = \left(\Theta_{IJ}(\underline{x}_I,\underline{x}_J)\right)_{I,J \in \J}$ for all $x \in \R^A$, where we fixed some ordering of $\J$. Note that, in the statement of Lemma~\ref{lem variance estimates large eta}, we implicitly assume that $\J$ is ordered in such a way that if $I \in \J_i$ and $J \in \J_j$ with $i <j$, then $I$ comes before $J$.

Let us regroup the blocks of $\Theta_\J(x)$ according to $\I$. For any $k$ and $l \in \{1,\dots,\norm{\I}\}$ we denote by $\Theta_{kl}(x) = \left(\Theta_{IJ}(\underline{x}_I,\underline{x}_J)\right)_{I \in \J_k, J \in \J_l}$. Then $\Theta_\J(x) = \left(\Theta_{kl}(x)\right)_{1 \leq k,l \leq \norm{\I}}$. Furthermore, for any $k \in \{1,\dots,\norm{\I}\}$, we have:
\begin{equation*}
\Theta_{kk}(x) = \left(\Theta_{IJ}(\underline{x}_I,\underline{x}_J)\right)_{I,J \in \J_k} = \Theta_{\J_k}(\underline{x}_{I_k}).
\end{equation*}
Hence, it is enough to prove that $\Norm{\Theta_{kl}(x)}_\infty \leq \Norm{\kappa}_{\norm{A},\eta}$, for any $k$ and $l \in \{1,\dots,\norm{\I}\}$ such that $k \neq l$ and any $x \in \R^A_{\I,\eta}$.

Let $k, l \in \{1,\dots,\norm{\I}\}$ be distinct, let $\eta \geq 0$ and let $x \in \R^A_{\I,\eta}$. Let $I = \{i_1,\dots,i_{\norm{I}}\} \in \J_k$ and $J = \{j_1,\dots,j_{\norm{J}}\} \in \J_l$, since $I$ and $J$ are disjoint subsets of $A$ we have $\norm{I}+\norm{J} \leq \norm{A}$. Then, by Equation~\eqref{eq def Theta IJ}, we have:
\begin{equation*}
\Theta_{IJ}(\underline{x}_I,\underline{x}_J) = \begin{pmatrix}
[\kappa]_{(p,q)}(x_{i_1},\dots,x_{i_p},x_{j_1},\dots,x_{j_q})
\end{pmatrix}_{\substack{1 \leq p \leq \norm{I}\\ 1 \leq q \leq \norm{J}}}.
\end{equation*}
By Lemma~\ref{lem double divided differences Rolle}, for all $p \in \{1,\dots,\norm{I}\}$ and all $q \in \{1,\dots,\norm{J}\}$, we have:
\begin{equation*}
\norm{[\kappa]_{(p,q)}(x_{i_1},\dots,x_{i_p},x_{j_1},\dots,x_{j_q})} \leq \max_{0 \leq m \leq \norm{A}} \ \sup_{(\underline{x}_J)_{\min} - (\underline{x}_I)_{\max}\leq \xi \leq (\underline{x}_J)_{\max} - (\underline{x}_I)_{\min}} \norm{\kappa^{(m)}(\xi)},
\end{equation*}
where we used the facts that $p+q-2 \leq \norm{I} + \norm{J} \leq \norm{A}$ and:
\begin{align*}
(\underline{x}_I)_{\min} &\leq \min \{x_{i_1},\dots,x_{i_p}\}, & (\underline{x}_I)_{\max} &\geq \max \{x_{i_1},\dots,x_{i_p}\},\\
(\underline{x}_J)_{\min} &\leq \min \{x_{j_1},\dots,x_{j_q}\}, & (\underline{x}_J)_{\max} &\geq \max \{x_{j_1},\dots,x_{j_q}\}.
\end{align*}
Since $I \in \J_k$, we have $I \subset I_k$ so that $(\underline{x}_I)_{\min} \geq (\underline{x}_{I_k})_{\min}$ and $(\underline{x}_I)_{\max} \leq (\underline{x}_{I_k})_{\max}$. Similarly, we have $(\underline{x}_J)_{\min} \geq (\underline{x}_{I_l})_{\min}$ and $(\underline{x}_J)_{\max} \leq (\underline{x}_{I_l})_{\max}$. Hence,
\begin{equation*}
(\underline{x}_{I_l})_{\min} - (\underline{x}_{I_k})_{\max} \leq (\underline{x}_J)_{\min} - (\underline{x}_I)_{\max} \leq (\underline{x}_J)_{\max} - (\underline{x}_I)_{\min} \leq (\underline{x}_{I_l})_{\max} - (\underline{x}_{I_k})_{\min}.
\end{equation*}
Since $x \in \R^A_{\I,\eta}$, by Lemma~\ref{lem clusters are not interlaced}, either $(\underline{x}_{I_l})_{\min} - (\underline{x}_{I_k})_{\max} \geq \eta$ or $(\underline{x}_{I_l})_{\max} - (\underline{x}_{I_k})_{\min} \leq -\eta$. In the first case, we have $[(\underline{x}_J)_{\min} - (\underline{x}_I)_{\max},(\underline{x}_J)_{\max} - (\underline{x}_I)_{\min}] \subset [\eta,+\infty)$, while in the second we have $[(\underline{x}_J)_{\min} - (\underline{x}_I)_{\max},(\underline{x}_J)_{\max} - (\underline{x}_I)_{\min}] \subset (-\infty,-\eta]$. In both cases, using the parity of $\kappa$ and its derivatives, we get:
\begin{equation*}
\norm{[\kappa]_{(p,q)}(x_{i_1},\dots,x_{i_p},x_{j_1},\dots,x_{j_q})} \leq \Norm{\kappa}_{\norm{A},\eta},
\end{equation*}
for all $p \in \{1,\dots,\norm{I}\}$ and all $q \in \{1,\dots,\norm{J}\}$. Thus, $\Norm{\Theta_{IJ}(\underline{x}_I,\underline{x}_J)}_\infty \leq \Norm{\kappa}_{\norm{A},\eta}$ for all $I \in \J_k$ and $J \in \J_l$. Finally, $\Norm{\Theta_{kj}(x)}_\infty \leq \Norm{\kappa}_{\norm{A},\eta}$, which concludes the proof.
\end{proof}

\begin{cor}
\label{cor splitting DI}
There exists $C >0$ such that, for all $\I,\J \in \pa_A$ such that $\J \leq \I$, for all $\eta \geq 0$, for all $x \in \R^A_{\I,\eta}$ we have:
\begin{equation*}
\norm{D_\J(x) - \prod_{I \in \I} D_{\J_I}(\underline{x}_I)} \leq C \Norm{\kappa}^2_{\norm{A},\eta}.
\end{equation*}
\end{cor}

\begin{proof}
Let $\I, \J \in \pa_A$ be such that $\J \leq \I$, and let us denote by $I_1,\dots,I_{\norm{\I}}$ the elements of $\I$. As in Lemma~\ref{lem variance estimates large eta}, let us denote $\J_i = \J_{I_i}$ for all $i \in \{1,\dots,\norm{\I}\}$.

Let $\eta \geq 0$. By Lemma~\ref{lem variance estimates large eta}, for all $x \in \R^A_{\I,\eta}$ we have:
\begin{equation*}
D_\J(x) = \det \Theta_\J(x) = \prod_{i=1}^{\norm{\I}} \det \Theta_{\J_i}(\underline{x}_{I_i}) + O\!\left(\Norm{\kappa}^2_{\norm{A},\eta}\right) = \prod_{I \in \I} D_{\J_I}(\underline{x}_I) + O\!\left(\Norm{\kappa}^2_{\norm{A},\eta}\right).
\end{equation*}
Moreover, by Lemma~\ref{lem boundedness Theta Xi Omega I}, the coefficients of $\Theta_\J(x)$ are bounded by $\Norm{\kappa}_{2\norm{A}}$. Hence, the constant implied in the error term $O\!\left(\Norm{\kappa}^2_{\norm{A},\eta}\right)$ depends only on $\kappa$ and $\norm{A}$.
\end{proof}


\subsection{Denominator clustering}
\label{subsec denominator clustering}

The purpose of this section is to study the clustering properties of the denominators $D_\I$ of the Kac--Rice densities $\rho_\I$ (recall Definition~\ref{def Kac-Rice densities partition}), where $\I$ is a partition of some finite set $A$. In all this section, we consider a finite set $A \neq \emptyset$ and some $\mathcal{C}^{\norm{A}}$ Gaussian process $f$ whose correlation function is denoted by $\kappa$. Recall that $\kappa$ is of class $\mathcal{C}^{2\norm{A}}$ with bounded derivatives of any order up to $2\norm{A}$, and that its norm $\Norm{\kappa}_{\norm{A},\eta}$ is defined by Notation~\ref{ntn norm kappa} for any $\eta \geq 0$. We assume that $f$ is stationary centered and normalized, and that $\kappa$ tends to $0$ at infinity.

\begin{lem}
\label{lem lower bound DI base case}
Let $\eta \geq 0$, there exists $\epsilon_{\{A\},\eta} >0$ such that $\forall x \in \R^A_{\{A\},\eta}$, $D_{\{A\}}(x) \geq \epsilon_{\{A\},\eta}$.
\end{lem}

\begin{proof}
By Lemma~\ref{lem expression DIx}, the definition of $D_{\{A\}}$ does not depend on how we ordered the elements of $A$. That is $D_{\{A\}}$ is a symmetric function on $\R^A$. Without loss of generality, let us order the elements of $A$, say $A = \{a_i \mid 1 \leq i \leq \norm{A}\}$.

Let $\eta \geq 0$ and let $x=(x_a)_{a \in A}$. As explained in Remarks~\ref{rem def Kac-Rice densities partition}, we have:
\begin{equation*}
D_{\{A\}}(x) = D_{\{A\}}(x_{a_1}, \dots, x_{a_{\norm{A}}}) = D_{\{A\}}(0,x_{a_2}-x_{a_1}, \dots,x_{a_{\norm{A}}} -x_{a_1}).
\end{equation*}
Moreover, by Definition~\ref{def R A I eta}, for all $i \in \{2,\dots,\norm{A}\}$ we have $\norm{x_{a_i} - x_{a_1}} \leq \norm{A} \eta$ (see also Remarks~\ref{rem R A I eta}). By Lemma~\ref{lem DI and NI continuous} and Corollary~\ref{cor vanishing locus DI}, the function $D_{\{A\}}$ is continuous and positive on the compact set $\{0\} \times [-\norm{A}\eta,\norm{A}\eta]^{\norm{A}-1}$. Hence, there exists $\epsilon_{\{A\},\eta} >0$ such that $D_{\{A\}}(x) \geq \epsilon_{\{A\},\eta}$ for all $x \in \R^A_{\{A\},\eta}$.
\end{proof}

\begin{lem}[Uniform lower bound]
\label{lem lower bound DI induction step}
Let us assume that $\Norm{\kappa}_{\norm{A},\eta} \xrightarrow[\eta \to +\infty]{} 0$. Then, for all $\eta > 0$, there exists $\epsilon_\eta>0$ such that $\forall \I \in \pa_A$, $\forall x \in \R^A_{\I,\eta}$, $D_\I(x) \geq \epsilon_\eta$.
\end{lem}

\begin{proof}
Let us prove that for all $\I \in \pa_A$ the following statement is true:
\begin{equation}
\label{eq statement *}
\text{for all} \ \eta >0, \ \text{there exists} \ \epsilon_{\I,\eta} >0 \ \text{such that} \ \forall x \in \R^A_{\I,\eta},\, D_\I(x) \geq \epsilon_{\I,\eta}.
\end{equation}
Since $\pa_A$ is finite, the conclusion then follows by setting $\epsilon_\eta = \min \{\epsilon_{\I,\eta} \mid \I \in \pa_A\}$ for all $\eta >0$.

Recall that we defined a partial order $\leq$ on $\pa_A$ in Definition~\ref{def partial order pa p}, and that $\I_{\max}(A) = \{A\}$ is the maximum of $(\pa_A,\leq)$. We will prove that~\eqref{eq statement *} holds for any $\I \in \pa_A$ by a backward induction on $\I \in \pa_A$. If $\I = \I_{\max}(A) = \{A\}$, then~\eqref{eq statement *} holds by Lemma~\ref{lem lower bound DI base case}, so we already took care of the base case.

Let $\J \in \pa_A \setminus \{\I_{\max}(A)\}$ and let us assume that~\eqref{eq statement *} holds for any $\I \in \pa_A$ such that $\J < \I$. Let $\eta >0$ and let $\tau \geq \eta$. If $x \in \R^A_{\J,\eta}$, then by Lemma~\ref{lem partial order} we have $\I_\tau(x) \geq \I_\eta(x) = \J$. Hence,
\begin{equation}
\label{eq splitting RAI eta}
\R^A_{\J,\eta} = \R^A_{\J,\eta} \cap \left(\bigsqcup_{\I \in \pa_A} \R^A_{\I,\tau}\right) = \bigsqcup_{\I \in \pa_A} \left(\R^A_{\J,\eta} \cap \R^A_{\I,\tau}\right) = \bigsqcup_{\I \geq \J} \left(\R^A_{\J,\eta} \cap \R^A_{\I,\tau}\right).
\end{equation}

Let $x \in \R^A_{\J,\eta} \cap \R^A_{\J,\tau}$, that is the components of $x$ form clusters that are encoded by $\J$, the points of a given cluster are at distance of order $\eta$ from one another, and two distinct clusters are further apart than $\tau$. Since $x \in \R^A_{\J,\tau}$, applying Corollary~\ref{cor splitting DI} with $\I = \J$, we get:
\begin{equation}
\label{eq splitting DI}
D_\J(x) = \prod_{J \in \J} D_{\{J\}}(\underline{x}_{J}) + O\!\left(\Norm{\kappa}^2_{\norm{A},\tau}\right).
\end{equation}
Let $J \in \J$, applying Lemma~\ref{lem lower bound DI base case} with $A= J$, there exists $\epsilon_{\{J\},\eta} >0$ such that $D_{\{J\}}$ is bounded from below by $\epsilon_{\{J\},\eta}$ on $\R^J_{\{J\},\eta}$. Since $x \in \R^A_{\J,\eta}$, we have $\underline{x}_J \in \R^J_{\{J\},\eta}$ and $D_{\{J\}}(\underline{x}_J) \geq \epsilon_{\{J\},\eta}$. Now, let us assume $\tau \geq \eta$ to be large enough for the error term in Equation~\eqref{eq splitting DI} to be bounded by $\frac{1}{2}\prod_{J\in \J} \epsilon_{\{J\},\eta}$. This is possible because $\Norm{\kappa}_{\norm{A},\tau}$ tends to $0$ as $\tau \to +\infty$. Then, for all $x \in \R^A_{\J,\eta} \cap \R^A_{\J,\tau}$, we have $D_\J(x) \geq \frac{1}{2}\prod_{J\in \J} \epsilon_{\{J\},\eta} >0$.

Let $\I \in \pa_A$ be such that $\J < \I$. By Corollary~\ref{cor expression DIx}, for all $x \in \R^A$ we have:
\begin{equation*}
D_\J(x) = \left(\prod_{I \in \I} \prod_{\{(J,J') \in \J_I^2 \mid J \neq J'\}} \prod_{(i,j) \in J \times J'} \norm{x_i-x_j} \right) D_\I(x).
\end{equation*}
Let $x \in \R^A_{\J,\eta} \cap \R^A_{\I,\tau}$, where $\tau \geq \eta$ is the one we chose previously. Using the induction hypothesis~\eqref{eq statement *} for $\I$, there exits $\epsilon_{\I,\tau}>0$ such that $D_\I$ is bounded from below by $\epsilon_{\I,\tau}$ on $\R^A_{\I,\tau}$. In particular, $D_\I(x) \geq \epsilon_{\I,\tau}$. Moreover, let $J$ and $J' \in \J$ be such that $J \neq J'$, then for all $i \in J$ and $j \in J'$ we have $\norm{x_i-x_j} > \eta$. Hence, $D_\J(x) \geq \eta^{\alpha(\I,\J)} \epsilon_{\I,\tau} >0$, where
\begin{equation*}
\alpha(\I,\J) = \sum_{I \in \I} \sum_{\{(J,J') \in \J_I^2 \mid J \neq J'\}} \norm{J} \norm{J'} = \sum_{I \in \I} \left(\left(\sum_{J \in \J_I} \norm{J}\right)^2 - \sum_{J \in \J_I} \norm{J}^2\right) = \sum_{I \in \I} \norm{I}^2 - \sum_{J \in \J} \norm{J}^2 \geq 0.
\end{equation*}

We set
\begin{equation*}
\epsilon_{\J,\eta} = \min \left\{\frac{1}{2} \prod_{J \in \J}\epsilon_{\{J\},\eta}\right\} \cup \left\{\eta^{\alpha(\I,\J)} \epsilon_{\I,\tau} \mid \I > \J\right\} >0.
\end{equation*}
Then, by Equation~\eqref{eq splitting RAI eta}, for all $x \in \R^A_{\J,\eta}$ we have $D_\J(x) \geq \epsilon_{\J,\eta}$. Thus~\eqref{eq statement *} holds for $\J$, which concludes the induction step and the proof.
\end{proof}

\begin{lem}[Denominator clustering]
\label{lem clustering DI}
Let us assume that $\Norm{\kappa}_{\norm{A},\eta} \xrightarrow[\eta \to +\infty]{} 0$. Let $\eta \geq 1$, let $\I, \J \in \pa_A$ be such that $\J \leq \I$ and let $x \in \R^A_{\J,1} \cap \R^A_{\I,\eta}$, we have:
\begin{equation*}
\prod_{I \in \I} D_{\J_I}(\underline{x}_I) = D_\J(x) \left(1 + O\!\left(\Norm{\kappa}^2_{\norm{A},\eta}\right)\right),
\end{equation*}
where $\J_I$ is defined as in Notation~\ref{ntn J I} for all $I \in \I$. Moreover, the constant implied in the error term does not depend on $\eta$, $\I,\J$ nor $x \in \R^A_{\J,1} \cap \R^A_{\I,\eta}$.
\end{lem}

\begin{proof}
Let $\eta \geq 1$, let $\I,\J \in \pa_A$ be such that $\J \leq \I$ and let $x \in \R^A_{\J,1} \cap \R^A_{\I,\eta}$. Since $x \in \R^A_{\J,1}$ we have $D_\J(x) \geq \epsilon_1$, where $\epsilon_1 >0$ is given by Lemma~\ref{lem lower bound DI induction step}. By Corollary~\ref{cor splitting DI}, since $x \in \R^A_{\I,\eta}$, we have:
\begin{equation*}
\norm{\frac{\prod_{I \in \I} D_{\J_I}(\underline{x}_I)}{D_\J(x)}-1} \leq \frac{C \Norm{\kappa}^2_{\norm{A},\eta}}{D_\J(x)} \leq \frac{C \Norm{\kappa}^2_{\norm{A},\eta}}{\epsilon_1},
\end{equation*}
where $C >0$ is independent of $\I$, $\J$, $\eta$ and $x$. This yields the result.
\end{proof}

\begin{rem}
\label{rem clustering DI}
In Lemma~\ref{lem clustering DI}, we deal with $x \in \R^A_{\J,1} \cap \R^A_{\I,\eta}$ where $\eta \geq 1$ and $\J \leq \I$. It means that components of $x$ whose indices lie in the same cluster of $\J$ are at distance less than $1$, while components whose indices lie in different clusters of $\I$ are further away than $\eta$.

One could be under the impression that the hypothesis that $\J \leq \I$ is restrictive. In fact it is not since, if $x \in \R^A_{\I,\eta}$ with $\eta \geq 1$, then by Lemma~\ref{lem partial order} we have $\I_1(x) \leq \I_\eta(x) =\I$. In particular, for any $\eta \geq 1$ and any $\I \in \pa_A$ we have:
\begin{equation*}
\R^A_{\I,\eta} = \R^A_{\I,\eta} \cap \bigsqcup_{\J \in \pa_A} \R^A_{\J,1} = \bigsqcup_{\J \in \pa_A} \left(\R^A_{\J,1} \cap \R^A_{\I,\eta}\right) = \bigsqcup_{\J \leq \I} \left(\R^A_{\J,1} \cap \R^A_{\I,\eta}\right).
\end{equation*}
\end{rem}


\subsection{Numerator clustering}
\label{subsec numerator clustering}

Let us now consider the clustering properties of the numerators $N_\I$ of the Kac--Rice densities introduced in Definition~\ref{def Kac-Rice densities partition}. In this section, we aim to prove a result similar to Lemma~\ref{lem clustering DI} for these functions. This is achieved in Lemma~\ref{lem clustering NI mult} below. Once again, in all this section $A$ is a non-empty finite set, and $f$ is a normalized stationary centered $\mathcal{C}^{\norm{A}}$ Gaussian process, whose correlation function $\kappa$ tends to $0$ at infinity.

First, we study the variance matrix of the Gaussian vector appearing in the definition of $N_\I$. Given $\I \in \pa_A$ and $x \in \R^A$ such that $D_\I(x) \neq 0$, recall that $Y_\I(x)$ given that $X_\I(x) = 0$ is a well-defined centered Gaussian vector in $\R^{\norm{A}}$ of variance matrix $\Lambda_\I(x)$ (see Equations~\eqref{eq def XI},~\eqref{eq def YI} and~\eqref{eq def Lambda I}).

\begin{lem}
\label{lem estimate Lambda I}
Let us assume that $\Norm{\kappa}_{\norm{A},\eta} \xrightarrow[\eta \to +\infty]{} 0$. Let $\I =\{I_1,\dots,I_{\norm{\I}}\} \in \pa_A$ and let $\J \in \pa_A$ be such that $\J \leq \I$. For all $i \in \{1,\dots,\norm{\I}\}$, we denote by $\J_i = \J_{I_i}$ (see Notation~\ref{ntn J I}) for simplicity. Then, for all $\eta \geq 1$, for all $x \in \R^A_{\J,1} \cap \R^A_{\I,\eta}$ we have:
\begin{equation*}
\Lambda_\J(x) = \begin{pmatrix}
\Lambda_{\J_1}(\underline{x}_{I_1}) & 0 & \cdots & 0 \\
0 & \Lambda_{\J_2}(\underline{x}_{I_2}) & \ddots & \vdots\\
\vdots & \ddots & \ddots & 0\\
0 & \cdots & 0 & \Lambda_{\J_{\norm{\I}}}(\underline{x}_{I_{\norm{\I}}})
\end{pmatrix} + O\!\left(\Norm{\kappa}_{\norm{A},\eta}\right),
\end{equation*}
where the error term does not depend on $\eta$, $\I$, $\J$ nor $x \in \R^A_{\J,1} \cap \R^A_{\I,\eta}$.
\end{lem}

\begin{proof}
First, let us consider $\Theta_\J$ (see Equation~\eqref{eq def variance I}). By Lemma~\ref{lem boundedness Theta Xi Omega I}, for any $x \in \R^A$, the symmetric matrix $\Theta_\J(x)$ belongs to the compact ball $\mathcal{B}$ of center $0$ and radius $\Norm{\kappa}_{2\norm{A}}$, for the sup-norm. For all $x \in \R^A_{\J,1}$, we have $\det \left(\Theta_\J(x)\right) = D_\J(x) \geq \epsilon_1$, where $\epsilon_1 >0$ is given by Lemma~\ref{lem lower bound DI induction step}. Hence, $\Theta_\J(x)$ belongs to $\mathcal{B} \cap \det^{-1}([\epsilon_1,+\infty))$, which is a compact set of invertible matrices. By continuity of the inverse on this compact set, there exists $C >0$, depending only on $\Norm{\kappa}_{2\norm{A}}$ and~$\epsilon_1$, such that for all $x \in \R^A_{\J,1}$, $\Norm{\Theta_\J(x)^{-1}}_\infty \leq C$.

Let $\eta \geq 1$ and let $x \in \R^A_{\J,1} \cap \R^A_{\I,\eta}$. Since $\J \leq \I$, by Lemma~\ref{lem variance estimates large eta}, we have:
\begin{align*}
\begin{pmatrix}
\Theta_{\J_1}(\underline{x}_{I_1}) & 0 & \cdots & 0 \\
0 & \Theta_{\J_2}(\underline{x}_{I_2}) & \ddots & \vdots\\
\vdots & \ddots & \ddots & 0\\
0 & \cdots & 0 & \Theta_{\J_{\norm{\I}}}(\underline{x}_{I_{\norm{\I}}})
\end{pmatrix} &= \Theta_\J(x) + O\!\left(\Norm{\kappa}_{\norm{A},\eta}\right)\\
&= \Theta_\J(x) \left(\Id_{\norm{A}} + O\!\left(\Norm{\kappa}_{\norm{A},\eta}\right)\right),
\end{align*}
where $\Id_{\norm{A}}$ stands for the identity matrix of size $\norm{A}$. Note that we used the fact that $\Theta_\J(x)^{-1}$ is bounded to get the second equality, and that the error terms are independent of $\I$, $\J$, $\eta$ and $x$. Using once again the boundedness of $\Theta_\J(x)^{-1}$, we obtain after taking the inverse:
\begin{align*}
\begin{pmatrix}
\Theta_{\J_1}(\underline{x}_{I_1})^{-1} & 0 & \cdots & 0 \\
0 & \Theta_{\J_2}(\underline{x}_{I_2})^{-1} & \ddots & \vdots\\
\vdots & \ddots & \ddots & 0\\
0 & \cdots & 0 & \Theta_{\J_{\norm{\I}}}(\underline{x}_{I_{\norm{\I}}})^{-1}
\end{pmatrix} &= \Theta_\J(x)^{-1} \left(\Id_{\norm{A}} + O\!\left(\Norm{\kappa}_{\norm{A},\eta}\right)\right)\\
&= \Theta_\J(x)^{-1} + O\!\left(\Norm{\kappa}_{\norm{A},\eta}\right),
\end{align*}
where the error terms are uniform in $\eta$, $\I$, $\J$ and $x$.

In order to conclude the proof, we start from the definition of $\Lambda_\J(x)$ (see Equation~\eqref{eq def Lambda I}) and use the previous estimates for $\Theta_\J(x)^{-1}$. We also use the estimates of Lemma~\ref{lem variance estimates large eta} for $\Xi_\J(x)$ and $\Omega_\J(x)$, as well as the uniform boundedness of $\Xi_\J$ (see Lemma~\ref{lem boundedness Theta Xi Omega I}) and $\Theta_\J(x)^{-1}$. We obtain:
\begin{equation*}
\Lambda_\J(x) = \begin{pmatrix}
\Lambda_{\J_1}(\underline{x}_{I_1}) & 0 & \cdots & 0 \\
0 & \Lambda_{\J_2}(\underline{x}_{I_2}) & \ddots & \vdots\\
\vdots & \ddots & \ddots & 0\\
0 & \cdots & 0 & \Lambda_{\J_{\norm{\I}}}(\underline{x}_{I_{\norm{\I}}})
\end{pmatrix} + O\!\left(\Norm{\kappa}_{\norm{A},\eta}\right),
\end{equation*}
where the error term does not depend on $\eta \geq 1$, nor on $\I$ and $\J \in \pa_A$ such that $\J \leq \I$, nor on $x \in \R^A_{\J,1} \cap \R^A_{\I,\eta}$.
\end{proof}

Recall that $N_\I(x) = \Pi_{\norm{A}}(\Lambda_\I(x))$, where $\Pi_{\norm{A}}$ is the function defined by Definition~\ref{def Pi k}. The estimate of Lemma~\ref{lem estimate Lambda I} allows to derive the following additive estimate.

\begin{lem}
\label{lem clustering NI add}
Let us assume that $\Norm{\kappa}_{\norm{A},\eta} \xrightarrow[\eta \to +\infty]{} 0$. Let $\eta \geq 1$, let $\I, \J \in \pa_A$ be such that $\J \leq \I$ and let $x \in \R^A_{\J,1} \cap \R^A_{\I,\eta}$, we have:
\begin{equation*}
N_\J(x) = \prod_{I \in \I} N_{\J_I}(\underline{x}_I) + O\!\left(\left(\Norm{\kappa}_{\norm{A},\eta}\right)^\frac{1}{2}\right),
\end{equation*}
where the constant implied in the error term does not depend on $\eta$, $\I$, $\J$ nor $x \in \R^A_{\J,1} \cap \R^A_{\I,\eta}$.
\end{lem}

\begin{proof}
Let $\eta \geq 1$, let $\I,\J \in \pa_A$ be such that $\J \leq \I$ and let $x \in \R^A_{\J,1} \cap \R^A_{\I,\eta}$. Since $x \in \R^A_{\J,1}$, we have $D_{\J}(x) \geq \epsilon_1 >0$ (see Lemma~\ref{lem lower bound DI induction step}), so that $N_\J(x)$ is well-defined. Similarly, for any $I \in \I$, we have $\underline{x}_I \in \R^I_{\J_I,1}$, so that $N_{\J_I}(\underline{x}_I)$ is also well-defined.

Let us denote by $I_1,\dots,I_{\norm{\I}}$ the elements of $\I$. For all $i \in \{1,\dots,\norm{\I}\}$, we set $\J_i = \J_{I_i}$, and we denote by:
\begin{equation*}
\tilde{\Lambda}_\J(x) = \begin{pmatrix}
\Lambda_{\J_1}(\underline{x}_{I_1}) & 0 & \cdots & 0 \\
0 & \Lambda_{\J_2}(\underline{x}_{I_2}) & \ddots & \vdots\\
\vdots & \ddots & \ddots & 0\\
0 & \cdots & 0 & \Lambda_{\J_{\norm{\I}}}(\underline{x}_{I_{\norm{\I}}})
\end{pmatrix}.
\end{equation*}
By Definitions~\ref{def Kac-Rice densities partition} and~\ref{def Pi k} and the definition of $\Lambda_\J(x)$ (see the beginning of Section~\ref{subsec variance and covariance matrices}), we have $N_\J(x) = \Pi_{\norm{A}}\left(\Lambda_\J(x)\right)$. Let $(Z_i)_{1 \leq i \leq \norm{A}} \sim \mathcal{N}(0,\tilde{\Lambda}_\J(x))$ in $\R^{\norm{A}}$, we have in the same way:
\begin{equation*}
\Pi_{\norm{A}}(\tilde{\Lambda}_\J(x)) = \prod_{i=1}^{\norm{\I}} \esp{\prod_{j=1}^{\norm{I_i}}\norm{Z_{\norm{I_1}+\norm{I_2}+ \dots + \norm{I_{i-1}}+j}}} = \prod_{i=1}^{\norm{\I}} \Pi_{\norm{I_i}}(\Lambda_{\J_i}(\underline{x}_{I_i})) = \prod_{I \in \I} N_{\J_I}(\underline{x}_I).
\end{equation*}

By Lemma~\ref{lem estimate Lambda I}, we have $\Norm{\Lambda_\J(x) - \tilde{\Lambda}_\J(x)}_\infty = O\!\left(\Norm{\kappa}_{\norm{A},\eta}\right)$. Moreover, by Lemma~\ref{lem boundedness Lambda I}, we have $\Norm{\Lambda_\J(x)}_\infty = O\!\left(\Norm{\kappa}_{2\norm{A}}\right)$, so that both $\Lambda_\J(x)$ and $\tilde{\Lambda}_\J(x)$ lie in a ball of center $0$ and radius $O\!\left(\Norm{\kappa}_{2\norm{A}}\right)$ in the space of symmetric matrices. Note that the constant implied in these estimates is independent of $\eta$, $\I$, $\J$ and $x$. By Corollary~\ref{cor Pi k}, the map $\Pi_{\norm{A}}$ is $\frac{1}{2}$-Hölder on compact sets. Hence, there exists $C>0$, depending only on $\norm{A}$ and $\kappa$, such that:
\begin{align*}
\norm{N_\J(x) - \prod_{I \in \I} N_{\J_I}(\underline{x}_I)} &= \norm{\Pi_{\norm{A}}(\Lambda_\J(x) - \Pi_{\norm{A}}(\tilde{\Lambda}_\J(x))}\\
&\leq C \Norm{\Lambda_\J(x) - \tilde{\Lambda}_\J(x)}^\frac{1}{2}_\infty\\
&= O\!\left(\left(\Norm{\kappa}_{\norm{A},\eta}\right)^\frac{1}{2}\right),
\end{align*}
where the constant implied in the error term does not depend on $\eta$, $\I$, $\J$ nor $x$.
\end{proof}

\begin{cor}
\label{cor clustering NI add}
Let us assume that $\Norm{\kappa}_{\norm{A},\eta} \xrightarrow[\eta \to +\infty]{} 0$. Let $\eta \geq 1$ and let $\I \in \pa_A$, for all $x \in \R^A_{\I,\eta}$, we have:
\begin{equation*}
N_\I(x) = \prod_{I \in \I} N_{\{I\}}(\underline{x}_I) + O\!\left(\left(\Norm{\kappa}_{\norm{A},\eta}\right)^\frac{1}{2}\right),
\end{equation*}
where the constant implied in the error term does not depend on $\eta$, $\I$ nor $x \in \R^A_{\I,\eta}$.
\end{cor}

\begin{proof}
Let $\eta \geq 1$ and let $\I \in \pa_A$. Let $x \in \R^A_{\I,\eta}$ and let us denote by $\J = \I_1(x)$. By Lemma~\ref{lem partial order}, we have $\J \leq \I$. Applying Lemma~\ref{lem clustering NI add} and Corollary~\ref{cor expression NIx}, we get:
\begin{equation*}
\left(\prod_{I \in \I}\prod_{\{(J,J') \in \J_I^2 \mid J \neq J'\}} \prod_{(i,j) \in J \times J'} \norm{x_i - x_j}\right) \norm{N_\I(x) - \prod_{I\in \I} N_{\{I\}}(\underline{x}_I)} = O\!\left(\left(\Norm{\kappa}_{\norm{A},\eta}\right)^\frac{1}{2}\right),
\end{equation*}
where the constant implied in the error term does not depend on $\eta$, $\I$, $\J$ nor $x$. Since $x \in \R^A_{\J,1}$, if $i \in J \in \J$ and $j \in J' \in \J$ with $J \neq J'$ then $\norm{x_i-x_j} \geq 1$. Hence,
\begin{equation*}
\left(\prod_{I \in \I}\prod_{\{(J,J') \in \J_I^2 \mid J \neq J'\}} \prod_{(i,j) \in J \times J'} \norm{x_i - x_j}\right) \geq 1.
\end{equation*}
This yields the result.
\end{proof}

In the remainder of this section, we show that the additive estimate of Corollary~\ref{cor clustering NI add} yields a multiplicative estimate similar to the one derived in Lemma~\ref{lem clustering DI}. The key step is to prove that $N_\I(x)$ is bounded from below by a positive constant of the relevant domain.

\begin{lem}
\label{lem lower bound NI base case}
Let $\eta\geq 0$, there exists $\epsilon'_{\{A\},\eta} >0$ such that $\forall x \in \R^A_{\{A\},\eta}$, $N_{\{A\}}(x) \geq \epsilon'_{\{A\},\eta}$.
\end{lem}

\begin{proof}
The proof is similar to that of Lemma~\ref{lem lower bound DI base case}. Note that $N_{\{A\}}$ is a well-defined continuous positive function on $\R^A$, by Lemma~\ref{lem DI and NI continuous} and Corollary~\ref{cor vanishing locus DI}.

Let $\eta \geq 0$. Using the stationarity of $f$ (see Remarks~\ref{rem def Kac-Rice densities partition}), it is enough to prove that there exists $\epsilon'_{\{A\},\eta}>0$ such that $N_{\{A\}}(x) \geq \epsilon'_{\{A\},\eta}$ for all $x \in \{0\} \times [-\norm{A}\eta,\norm{A}\eta]^{\norm{A}-1}$. This is true, by compactness of this set.
\end{proof}

\begin{lem}[Uniform lower bound]
\label{lem lower bound NI induction step}
Let us assume that $\Norm{\kappa}_{\norm{A},\eta} \xrightarrow[\eta \to +\infty]{} 0$. Then, for all $\eta \geq 1$, there exists $\epsilon'_\eta>0$ such that $\forall \I \in \pa_A$, $\forall x \in \R^A_{\I,\eta}$, $N_\I(x) \geq \epsilon'_\eta$.
\end{lem}

\begin{proof}
The proof is similar to that of Lemma~\ref{lem lower bound DI induction step}. We prove by a backward induction on $(\pa_A,\leq )$ that for any $\I \in \pa_A$ the following statement is true:
\begin{equation}
\label{eq statement **}
\text{for all} \ \eta \geq 1, \ \text{there exists} \ \epsilon'_{\I,\eta} >0 \ \text{such that} \ \forall x \in \R^A_{\I,\eta},\, N_\I(x) \geq \epsilon'_{\I,\eta}.
\end{equation}
Then we set $\epsilon'_\eta = \min \left\{ \epsilon'_{\I,\eta} \mvert \I \in \pa_A \right\}$, which is the positive lower bound we are looking for.

The base case of the induction is for $\I = \I_{\max}(A) = \{A\}$. It is given by Lemma~\ref{lem lower bound NI base case}.

The induction step is similar to the induction step in the proof of Lemma~\ref{lem lower bound DI induction step}. In Lemma~\ref{lem lower bound DI induction step}, the two key elements are the additive estimate of Equation~\eqref{eq splitting DI} and the relation given by Corollary~\ref{cor expression DIx}. Here, the analogous results are the additive estimate of Corollary~\ref{cor clustering NI add} and the relation given by Corollary~\ref{cor expression NIx}.
\end{proof}

\begin{lem}[Numerator clustering]
\label{lem clustering NI mult}
Let us assume that $\Norm{\kappa}_{\norm{A},\eta} \xrightarrow[\eta \to +\infty]{} 0$. Let $\eta \geq 1$, let $\I, \J \in \pa_A$ be such that $\J \leq \I$ and let $x \in \R^A_{\J,1} \cap \R^A_{\I,\eta}$, we have:
\begin{equation*}
\prod_{I \in \I} N_{\J_I}(\underline{x}_I) = N_\J(x) \left(1 + O\!\left(\left(\Norm{\kappa}_{\norm{A},\eta}\right)^\frac{1}{2}\right)\right),
\end{equation*}
where the constant implied in the error term does not depend on $\eta$, $\I,\J$ nor $x \in \R^A_{\J,1} \cap \R^A_{\I,\eta}$.
\end{lem}

\begin{proof}
Since, $x \in \R^A_{\J,1} \cap \R^A_{\I,\eta}$, we have:
\begin{equation*}
\prod_{I \in \I} N_{\J_I}(\underline{x}_I) = N_\J(x) + O\!\left(\left(\Norm{\kappa}_{\norm{A},\eta}\right)^\frac{1}{2}\right)
\end{equation*}
by Lemma~\ref{lem clustering NI add}. The result follows from the fact that $N_\J(x) \geq \epsilon'_1$, where $\epsilon'_1>0$ is given by Lemma~\ref{lem lower bound NI induction step}.
\end{proof}


\subsection{Proof of Theorem~\ref{thm clustering}: clustering for \texorpdfstring{$k$}{}-point functions}
\label{subsec proof of theorem clustering}

In this section we prove Theorem~\ref{thm clustering}. This result will be deduced from Lemma~\ref{lem bound rho I} and Proposition~\ref{prop clustering rho I}, which will also be useful in the proof of Theorem~\ref{thm moments} in Section~\ref{sec proof theorem moments}.

As usual, in all this section, $A$ denotes a non-empty finite set and $f$ is a normalized centered stationary Gaussian process of class $\mathcal{C}^{\norm{A}}$ whose correlation function is denoted by $\kappa$. We assume in the following that $\Norm{\kappa}_{\norm{A},\eta} \xrightarrow[\eta \to +\infty]{} 0$. In particular, $\kappa$ tends to $0$ at infinity, so that the conclusion of Lemma~\ref{lem non-degeneracy} holds true. Under these assumptions, the Kac--Rice density $\rho_{\{A\}}$ is well-defined on~$\R^A$. Moreover, it coincides with $\rho_{\norm{A}}$ on $\R^A \setminus \Delta_A$ (see Proposition~\ref{prop Kac-Rice revisited}), which is also the $\norm{A}$-point function of the point process $Z= f^{-1}(0)$ (see Lemma~\ref{lem k point function}).

\begin{lem}[Boundedness]
\label{lem bound rho I}
If $\Norm{\kappa}_{\norm{A},\eta} \xrightarrow[\eta \to +\infty]{} 0$ then there exists $C>0$ such that, for all $x =(x_a)_{a \in A} \in \R^A$,
\begin{equation*}
\rho_{\{A\}}(x) \leq C \left(\prod_{a \neq b} \min(\norm{x_a-x_b},1)\right)^\frac{1}{2}.
\end{equation*}
In particular, $\rho_{\{A\}}$ is bounded on $\R^A$.
\end{lem}

\begin{proof}
Let $x =(x_a)_{a \in A} \in \R^A$ and let us denote by $\I = \I_1(x)$ for simplicity. We have $x \in \R^A_{\I,1}$, hence $D_\I(x) \geq \epsilon_1$, where $\epsilon_1 >0$ is given by Lemma~\ref{lem lower bound DI induction step}. In particular, by Proposition~\ref{prop Kac-Rice revisited}, we have $\rho_{\{A\}}(x) = \rho_\I(x)$. Then, by Equation~\eqref{eq def rho Ix},
\begin{equation*}
\rho_{\{A\}}(x) = \left(\prod_{I \in \I} \prod_{\{(i,j) \in I^2 \mid i \neq j\}} \norm{x_i-x_j}^\frac{1}{2}\right) \frac{N_\I(x)}{(2\pi)^\frac{\norm{A}}{2} D_\I(x)^\frac{1}{2}}.
\end{equation*}
By Corollary~\ref{cor boundedness NI}, there exists $C'>0$ independent of $\I$ and $x$ such that $N_\I(x) \leq C'$. Hence,
\begin{equation*}
\rho_{\{A\}}(x) \leq \frac{C'}{(2\pi)^\frac{\norm{A}}{2} (\epsilon_1)^\frac{1}{2}} \left(\prod_{I \in \I} \prod_{\{(i,j) \in I^2 \mid i \neq j\}} \norm{x_i-x_j}^\frac{1}{2}\right).
\end{equation*}

Let $a$ and $b \in A$. If $a$ and $b$ belong to the same cluster of $\I$, then $\norm{x_a-x_b} \leq \norm{A}$ (see Remark~\ref{rem R A I eta}) and $\norm{x_a-x_b} \leq \norm{A} \min(\norm{x_a-x_b},1)$. If $a$ and $b$ belong to different cluster of $\I$, then $\norm{x_a-x_b} \geq 1$ by definition, so that $\min(\norm{x_a-x_b},1) = 1$. Thus,
\begin{align*}
\prod_{I \in \I} \prod_{\{(i,j) \in I^2 \mid i \neq j\}} \norm{x_i-x_j}^\frac{1}{2} &\leq \norm{A}^{\norm{A}^2} \left(\prod_{I \in \I} \prod_{\{(i,j) \in I^2 \mid i \neq j\}} \min(\norm{x_i-x_j},1)\right)^\frac{1}{2}\\
&\leq \norm{A}^{\norm{A}^2} \left(\prod_{\{(i,j) \in A^2 \mid i \neq j\}} \min(\norm{x_i-x_j},1)\right)^\frac{1}{2}.
\end{align*}
This proves the result with $C = \norm{A}^{\norm{A}^2}C'(2\pi)^{-\frac{\norm{A}}{2}}(\epsilon_1)^{-\frac{1}{2}}$.
\end{proof}

\begin{rem}
\label{rem bound rho I}
Note that this bound is the best we can hope for near the diagonal, because of Theorem~\ref{thm vanishing order}.
\end{rem}

\begin{prop}[Clustering]
\label{prop clustering rho I}
Let $A$ be a non-empty finite set. Let $f$ be a normalized centered stationary Gaussian process of class $\mathcal{C}^{\norm{A}}$, whose correlation function $\kappa$ satisfies $\Norm{\kappa}_{\norm{A},\eta} \xrightarrow[\eta \to +\infty]{} 0$. Then, for all $\eta \geq 1$, for all $\I \in \pa_A$, for all $x \in \R^A_{\I, \eta}$, we have:
\begin{equation*}
\prod_{I \in \I} \rho_{\{I\}}(\underline{x}_I) = \rho_{\{A\}}(x) \left(1+O\!\left(\Norm{\kappa}_{\norm{A},\eta}\right)^\frac{1}{2}\right),
\end{equation*}
where the constant involved in the error term does not depend on $\eta$, $\I$ nor $x$.
\end{prop}

\begin{proof}
First, note that $\rho_{\{A\}}$ (resp.~$\rho_{\{I\}}$) is well-defined on $\R^A$ (resp.~$\R^I$), see Proposition~\ref{prop Kac-Rice revisited}. Note also that, for any $\J \in \pa_A$, we have $\rho_{\{A\}} = \rho_\J$ whenever $\rho_\J$ is well-defined. In this proof, we use the previous fact, and we choose $\J \in \pa_A$ depending on the point $x \in \R^A_{\I,\eta}$ in order to use the nicest expression of $\rho_{\{A\}}$ we can find.

Let $\eta \geq 1$, let $\I \in \pa_A$ and let $x =(x_a)_{a \in A} \in \R^A_{\I,\eta}$. Let us denote by $\J = \I_1(x)$, so that $\J \leq \I$ and $x \in \R^A_{\J,1} \cap \R^A_{\I,\eta}$. By Corollary~\ref{cor vanishing locus DI}, since $x \in \R^A_{\J,1}$, we have $D_\J(x) >0$ so that $\rho_\J(x)$ is well-defined. Similarly, for any $I \in \I$, we have $\underline{x}_\I \in \R^I_{\J_I,1}$ so that $\rho_{\J_I}(\underline{x}_I)$ is well-defined.

By Proposition~\ref{prop Kac-Rice revisited} and Equation~\eqref{eq def rho Ix}, we have:
\begin{align*}
\prod_{I \in \I} \rho_{\{I\}}(\underline{x}_I) &= \prod_{I \in \I} \rho_{\J_I}(\underline{x}_I)\\
&= \prod_{I \in \I} \left(\left(\prod_{J \in \J_I} \prod_{\{(i,j) \in J^2 \mid i \neq j\}} \norm{x_i-x_j}^\frac{1}{2}\right)\frac{N_{\J_I}(\underline{x}_I)}{(2\pi)^\frac{\norm{I}}{2} D_{\J_I}(\underline{x}_I)^\frac{1}{2}}\right)\\
&= \left(\prod_{J \in \J} \prod_{\{(i,j) \in J^2 \mid i\neq j\}} \norm{x_i - x_j}^\frac{1}{2}\right) \frac{\prod_{I \in \I} N_{\J_I}(\underline{x}_I)}{(2\pi)^\frac{\norm{A}}{2} \left(\prod_{I \in \I} D_{\J_I}(\underline{x}_I)\right)^\frac{1}{2}}.
\end{align*}
Since $x \in \R^A_{\J,1}\cap \R^A_{\I,\eta}$, by Lemmas~\ref{lem clustering DI} and~\ref{lem clustering NI mult}, we obtain:
\begin{equation*}
\prod_{I \in \I} \rho_{\{I\}}(\underline{x}_I) = \left(\prod_{J \in \J} \prod_{\{(i,j) \in J^2 \mid i\neq j\}} \norm{x_i - x_j}^\frac{1}{2}\right) \frac{N_\J(x)}{(2\pi)^\frac{\norm{A}}{2} D_\J(x)^\frac{1}{2}} \left(1 +O\!\left(\left(\Norm{\kappa}_{\norm{A},\eta}\right)^\frac{1}{2}\right)\right).
\end{equation*}
The conclusion follows from $\rho_{\{A\}}(x) = \rho_{\J}(x)$ and the definition of $\rho_{\J}$, see Equation~\eqref{eq def rho Ix} and Proposition~\ref{prop Kac-Rice revisited}.
\end{proof}

In Proposition~\ref{prop clustering rho I}, we only consider points $x \in \R^A_{\I,\eta}$. In fact, an estimate of the same kind remains valid if we replace $\R^A_{\I,\eta}$ with $\bigsqcup_{\J \leq \I} \R^A_{\J,\eta}$. Equivalently, we only need $\I_\eta(x) \leq \I$ instead of $\I_\eta(x)=\I$. That is, we need the components of $x$ whose indices lie in different clusters of $\I$ to be far from one another, but we do not ask anything regarding components whose indices lie in the same cluster of $\I$. The precise statement is the following.

\begin{cor}
\label{cor clustering rho I}
In the setting of Proposition~\ref{prop clustering rho I}, for all $\eta \geq 1$, for all $\I \in \pa_A$, for all $x \in \R^A$, such that $\I_\eta(x) \leq \I$, we have:
\begin{equation*}
\prod_{I \in \I} \rho_{\{I\}}(\underline{x}_I) = \rho_{\{A\}}(x) \left(1+O\!\left(\Norm{\kappa}_{\norm{A},\eta}\right)^\frac{1}{2}\right),
\end{equation*}
where the constant involved in the error term does not depend on $\eta$, $\I$ nor $x$.
\end{cor}

\begin{proof}
Let $\eta \geq 1$, let $\I \in \pa_A$ and let $x \in \R^A$ such that $\I_\eta(x) \leq \I$. Let us denote by $\J = \I_\eta(x)$ for simplicity. Since $x \in \R^A_{\J,\eta}$, by Proposition~\ref{prop clustering rho I} we have:
\begin{align*}
\rho_{\{A\}}(x) &= \left(\prod_{J \in \J} \rho_{\{J\}}(\underline{x}_J)\right) \left(1 + O\!\left(\left(\Norm{\kappa}_{\norm{A},\eta}\right)^\frac{1}{2}\right)\right)\\
&= \left(\prod_{I \in \I} \prod_{J \in \J_I} \rho_{\{J\}}(\underline{x}_J)\right) \left(1 + O\!\left(\left(\Norm{\kappa}_{\norm{A},\eta}\right)^\frac{1}{2}\right)\right).
\end{align*}
Let $I \in \I$, we have $\underline{x}_I \in \R^I_{\J_I,\eta}$. Using Proposition~\ref{prop clustering rho I} once again, we have:
\begin{equation*}
\prod_{J \in \J_I} \rho_{\{J\}}(\underline{x}_J) = \rho_{\{I\}}(\underline{x}_I) \left(1 + O\!\left(\left(\Norm{\kappa}_{\norm{A},\eta}\right)^\frac{1}{2}\right)\right).
\end{equation*}
This yields the result. The uniformity of the error term follows from the finiteness of $A$.
\end{proof}

We can now prove Theorem~\ref{thm clustering}, which is just a special case of Lemma~\ref{lem bound rho I} and Corollary~\ref{cor clustering rho I}.

\begin{proof}[Proof of Theorem~\ref{thm clustering}]
Let $k \in \N^*$ and let $f$ be a $\mathcal{C}^k$ Gaussian process which is normalized centered and stationary. We assume that its correlation function $\kappa$ satisfies $\Norm{\kappa}_{k,\eta} \xrightarrow[\eta \to +\infty]{}0$. By Lemmas~\ref{lem non-degeneracy} and~\ref{lem k point function}, for all $l \in \{1,\dots,k\}$, the $l$-point function of the point process $Z=f^{-1}(0)$ is the Kac--Rice density $\rho_l$ defined in Definition~\ref{def Kac-Rice densities}. This function is well-defined on $\R^l \setminus \Delta_l$ and admits a unique continuous extension to $\R^l$, which vanishes on $\Delta_l$, by Proposition~\ref{prop Kac-Rice revisited}.

By Proposition~\ref{prop Kac-Rice revisited}, the continuous extension of $\rho_k$ to $\R^k$ is the function $\rho_{\I_{\max}(k)}$. Then, by Lemma~\ref{lem bound rho I}, for all $x=(x_i)_{1 \leq i \leq k} \in \R^k \setminus \Delta_k$ we have:
\begin{equation*}
\rho_k(x) = \rho_{\I_{\max}(k)}(x) \leq C \prod_{1 \leq i < j \leq k} \max\left(\norm{x_i-x_j},1\right)
\end{equation*}
for some positive constant $C$.

Let $\eta \geq 1$, let $\I \in \pa_k$ and let $x=(x_i)_{1 \leq i \leq k} \in \R^k \setminus \Delta_k$. The condition:
\begin{equation*}
\forall I, J \in \I \ \text{such that} \ I \neq J, \ \forall i \in I, \ \forall j \in J, \ \norm{x_i -x _j} > \eta,
\end{equation*}
appearing in Theorem~\ref{thm clustering} is equivalent to $\I_\eta(x) \leq \I$. Let us assume that $x$ satisfies this condition. Then, by Corollary~\ref{cor clustering rho I} applied with $A = \{1,\dots,k\}$, we have:
\begin{equation*}
\prod_{I \in \I} \rho_{\{I\}}(\underline{x}_I) = \rho_{\I_{\max}(k)}(x) \left(1+O\!\left(\Norm{\kappa}_{k,\eta}\right)^\frac{1}{2}\right).
\end{equation*}
Finally, we have the equality $\rho_k(x) = \rho_{\I_{\max}(k)}(x)$ and similarly, $\rho_{\{I\}}(\underline{x}_I) = \rho_{\norm{I}}(\underline{x}_I)$ for all $I \in \I$ since $\underline{x}_I \in \R^I \setminus \Delta_I$. Hence the result.
\end{proof}


\section{Proof of Theorem~\ref{thm moments}: central moments asymptotics}
\label{sec proof theorem moments}

This section deals with the proof of Theorem~\ref{thm moments}. The proof follows the lines of that of~\cite[Theorem~1.12]{AL2019}. We still give the proof in full, since we believe that the setting of the present article makes it accessible to a wider audience than~\cite{AL2019}. In all this section, we consider a Gaussian process $f$ which is at least of class $\mathcal{C}^1$, stationary, centered and normalized. We denote by $\kappa$ is correlation function, which is assumed to tend to $0$ at infinity. In particular, the process $f$ satisfies the conclusion of Lemma~\ref{lem non-degeneracy}.

In Section~\ref{subsec an integral expression of the central moments}, we derive an integral expression of the central moments we are interested in. In order to understand this integral, we split $\R^p$ as $\bigsqcup_{\I \in \pa_p} \R^p_{\I,\eta}$, for some $\eta >0$, and study the contribution of each~$\R^p_{\I,\eta}$ to the integral. In Section~\ref{subsec an upper bound on the contribution of each piece}, we give an upper bound for the contribution of each of these sets. Then, in Sections~\ref{subsec contribution of the partitions with an isolated point} and~\ref{subsec contribution of the partitions into pairs}, we study the contributions of the pieces of the form $\R^p_{\I,\eta}$, where respectively $\I$ contains a singleton and $\I$ is a partition into pairs. We conclude the proof of Theorem~\ref{thm moments} in Section~\ref{subsec conclusion of the proof}.


\subsection{An integral expression of the central moments}
\label{subsec an integral expression of the central moments}

Let $R >0$, recall that $\nu_R$ denotes the counting measure of $Z_R =\{x \in \R \mid f(Rx)=0\}$. Let $p \geq 2$ be an integer and let $\phi_1,\dots,\phi_p$ be test-functions in the sense of Definition~\ref{def test-function}. In this section, we derive an integral expression of the quantity $m_p(\nu_R)(\phi_1,\dots,\phi_p)$ defined by Definition~\ref{def mp nu R}. This requires to introduce the following definition.

\begin{dfn}[Subsets adapted to a partition]
\label{def subsets adapted to I}
Let $A$ be a finite set and let $\I \in \pa_A$, we denote by $S_A(\I)$ the set of subsets of $A$ \emph{adapted} to $\I$, that is:
\begin{equation*}
\mathcal{S}_A(\I) = \left\{B \subset A \mvert \forall I \in \I, \text{ if } \card(I) \geq 2, \text{ then } I \subset B \right\}.
\end{equation*}
Equivalently, $B \in \mathcal{S}_A(\I)$ if and only if $\I \leq \{B\} \sqcup \I_{\min}(A \setminus B)$ where $\I_{\min}(A \setminus B) = \{\{b\} \mid b \notin B\}$ and $\leq$ is as in Definition~\ref{def partial order pa p}. If $A$ is of the form $\{1,\dots,p\}$, we simply denote by $S_p(\I)=S_A(\I)$.
\end{dfn}

Let $A$ be a finite set and let $\I \in \pa_A$. Let $B \in \mathcal{S}_A(\I)$, we have $\I \leq \{B\} \sqcup \I_{\min}(A \setminus B)$, so that $\I_B = \{ I \in \I \mid I \subset B \}$ is a well-defined element of $\pa_B$, as in Notation~\ref{ntn J I}. In fact, we have:
\begin{equation}
\label{eq subsets adapted to I}
\I = \I_B \sqcup \I_{\min}(A \setminus B).
\end{equation}

\begin{lem}
\label{lem subsets adapted to I}
Let $A$ be any finite set, then the map $(B,\I) \mapsto (B,\I_B)$ defines a bijection from $\{(B,\I) \mid \I\in \pa_A, B \in \mathcal{S}_A(\I) \}$ to $\{ (B,\J) \mid B \subset A, \J \in \pa_B \}$.
\end{lem}

\begin{proof}
This map is well-defined. By Equation~\eqref{eq subsets adapted to I}, the map $(B,\J) \mapsto (B, \J \sqcup \I_{\min}(A \setminus B))$ from $\{ (B,\J) \mid B \subset A, \J \in \pa_B \}$ to $\{(B,\I) \mid \I\in \pa_A, B \in \mathcal{S}_A(\I) \}$ is the inverse of $(B,\I) \mapsto (B,\I_B)$.
\end{proof}

For any non-empty finite set $A$, let us denote by $\dx \underline{x}_A$ the Lebesgue measure on $\R^A$. The following lemma gives the integral expression of $m_p(\nu_R)(\phi_1,\dots,\phi_p)$ we are looking for. This integral expression is a sum of integrals over $\R^\I$, indexed by $\mathcal{I}\in \pa_p$. For each of these terms, the integrand function is itself a sum of functions that are indexed by partitions of some partitions induced by $\I$. Hence we need to consider partitions of partitions, which is a bit cumbersome.

\begin{lem}[Integral expression of the central moments]
\label{lem integral expression}
Let $p \geq 2$ and let us assume that $f$ is of class $\mathcal{C}^p$. Let $\phi_1,\dots,\phi_p$ be test-functions in the sense of Definition~\ref{def test-function}. For all $R>0$, we have:
\begin{equation*}
m_p(\nu_R)(\phi_1,\dots,\phi_p)= \sum_{\I \in \pa_p} \int_{\R^\I} \iota_\I^*\phi_R(\underline{x}_\I) F_\I(\underline{x}_\I) \dx\underline{x}_\I,
\end{equation*}
where, for any finite set $A \neq \emptyset$, any $\I \in \pa_A$ and any $\underline{x}_\I = (x_I)_{I \in \I} \in \R^\I$,
\begin{equation}
\label{eq def F I}
F_\I(\underline{x}_\I) = \sum_{B \in \mathcal{S}_A(\I)} \left(\frac{-1}{\pi}\right)^{\norm{A} - \norm{B}} \rho_{\{\I_B\}}(\underline{x}_{\I_B}).
\end{equation}
Here, we use the convention that $\rho_{\{\emptyset\}}$ is constant equal to $1$. Note that $\I_B = \emptyset$ if and only if $B =\emptyset$.
\end{lem}

\begin{proof}
The proof follows the lines of~\cite[Lemma~3.1]{AL2019}. Recall that $m_p(\nu_R)$ is defined by Definition~\ref{def mp nu R}. We develop this product using Notations~\ref{ntn product indexed by A}, we get:
\begin{align*}
m_p(\nu_R)(\phi_1,\dots,\phi_p) &= \sum_{B \subset \{1,\dots,p\}} (-1)^{p-\norm{B}} \esp{\prod_{i \in B} \prsc{\nu_R}{\phi_i}} \prod_{i \notin B}\esp{\prsc{\nu_R}{\phi_i}}\\
&= \sum_{B \subset \{1,\dots,p\}} (-1)^{p-\norm{B}} \esp{\prsc{\nu^B}{(\phi_B)_R}} \prod_{i \notin B}\esp{\prsc{\nu_R}{\phi_i}}\\
&= \sum_{B \subset \{1,\dots,p\}} \sum_{\I \in \pa_B} (-1)^{p-\norm{B}} \esp{\prsc{\nu^{[\I]}}{\iota_\I^*((\phi_B)_R)}} \prod_{i \notin B}\esp{\prsc{\nu_R}{\phi_i}},
\end{align*}
where the last equality comes from Lemma~\ref{lem decomposition nu ks}.

Let $B \subset \{1,\dots,p\}$ and $\I \in \pa_B$. Note first that $\iota_\I^*((\phi_B)_R) = \phi_B\left(\frac{\iota_\I(\cdot)}{R}\right) = (\iota_\I^* \phi_B)_R$. Then, we can identify $\R^\I$ with $\R^{\norm{\I}}$ by ordering $\I$. Since the $(\phi_i)_{1 \leq i\leq p}$ are integrable on $\R$ and essentially bounded, $\iota_\I^* \phi_B$ is integrable on $\R^\I$. By Propositions~\ref{prop Kac-Rice formula} and~\ref{prop Kac-Rice revisited}, we obtain:
\begin{align*}
\esp{\prsc{\nu^{[\I]}}{\iota_\I^*((\phi_B)_R)}} &= \esp{\prsc{\nu^{[\norm{\I}]}}{(\iota_\I^*\phi_B)_R)}}\\
&= \int_{\R^{\norm{\I}}} \iota_\I^*\phi_B\!\left(\frac{x_1}{R},\dots,\frac{x_{\norm{\I}}}{R}\right) \rho_{\norm{\I}}(x_1,\dots,x_{\norm{\I}}) \dx x_1\dots \dx x_{\norm{\I}}\\
&= \int_{\R^\I} \iota_\I^*\phi_B\!\left(\frac{\underline{x}_\I}{R}\right) \rho_{\{\I\}}(\underline{x}_\I) \dx \underline{x}_\I,
\end{align*}
where $\rho_{\{\I\}}:\R^\I \to \R$ is defined by Equation~\eqref{eq def rho Ix}. As in Section~\ref{subsec proof of proposition expectation}, for any $i \notin B$, we have $\esp{\prsc{\nu_R}{\phi_i}} =\displaystyle \frac{1}{\pi} \int_\R \phi_i\!\left(\frac{x}{R}\right)\dx x$. Hence,
\begin{equation*}
m_p(\nu_R)(\phi_1,\dots,\phi_p) = \sum_{B \subset \{1,\dots,p\}} \sum_{\I \in \pa_B} \left(\frac{-1}{\pi}\right)^{p-\norm{B}} \int_{\R^\I} \iota_\I^*\phi_B\!\left(\frac{\underline{x}_\I}{R}\right) \rho_{\{\I\}}(\underline{x}_\I) \dx \underline{x}_\I \prod_{i \notin B} \int_{\R} \phi_i\!\left(\frac{x}{R}\right) \dx x.
\end{equation*}
By Lemma~\ref{lem subsets adapted to I}, we can exchange the two sums and obtain the following:
\begin{equation*}
\sum_{\I \in \pa_p} \sum_{B \in \mathcal{S}_p(\I)} \left(\frac{-1}{\pi}\right)^{p-\norm{B}} \left(\int_{\R^{\I_B}} \iota_{\I_B}^*\phi_B\!\left(\frac{\underline{x}_{\I_B}}{R}\!\right) \rho_{\{\I_B\}}(\underline{x}_{\I_B}) \dx \underline{x}_{\I_B}\right)\prod_{i \notin B} \int_\R \phi_i\!\left(\frac{x}{R}\right)\dx x.
\end{equation*}
We conclude the proof by applying Fubini's Theorem, which yields:
\begin{equation*}
m_p(\nu_R)(\phi_1,\dots,\phi_p) =\sum_{\I \in \pa_p} \int_{\R^\I} \iota_\I^*\phi_R(\underline{x}_\I) \left(\sum_{B \in \mathcal{S}_p(\I)}\left(\frac{-1}{\pi}\right)^{p-\norm{B}} \rho_{\{\I_B\}}(\underline{x}_{\I_B})\right) \dx \underline{x}_\I. \qedhere
\end{equation*}
\end{proof}

\begin{ex}
\label{ex covariances}
Let $A =\{a,b\}$, then $\pa_A$ contains only two elements: $\I_{\min}(A) = \{\{a\},\{b\}\}$ and $\I_{\max}(A)=\{A\}$. In the first case, $\mathcal{S}_A(\I_{\min}(A)) = \{\emptyset, \{a\},\{b\}, A\}$, and one can check that, $F_{\{\{a\},\{b\}\}} : (x,y) \mapsto \rho_{\{\I_{\min}(A)\}}(x,y) - \frac{1}{\pi^2}$. In fact, for all $(x,y) \in \R^2 \setminus \Delta_2$, we have:
\begin{equation*}
F_{\{\{a\},\{b\}\}}(x,y) =\rho_2(x,y)-\frac{1}{\pi^2} = F(y-x),
\end{equation*}
where $F$ is as in Definition~\ref{def F}. In the second case, $\mathcal{S}_A(\{A\})=\{A\}$ and $F_{\{A\}} = \frac{1}{\pi}$. Then, if $\phi_a$ and $\phi_b$ are integrable and $\phi_b$ is bounded and continuous almost everywhere, we have:
\begin{equation*}
\int_{\R^{\{A\}}}(\phi_a)_R(x)(\phi_b)_R(x)F_{\{A\}}(x)\dx x = \frac{R}{\pi} \int_\R \phi_a(x)\phi_b(x) \dx x.
\end{equation*}
Moreover, the computations of Section~\ref{subsec asymptotics of the covariances} show that:
\begin{align*}
\int_{\R^{\{\{a\},\{b\}\}}} (\phi_a)_R(x)(\phi_b)_R(y) F_{\{\{a\},\{b\}\}}(x,y)\dx x \dx y &= \int_{\R^2} (\phi_a)_R(x)(\phi_b)_R(y) F(y-x)\dx x \dx y\\
&= R\left(\sigma^2-\frac{1}{\pi}\right)\int_{\R^2} \phi_a(x)\phi_b(x) \dx x +o(R),
\end{align*}
where $\sigma^2$ is defined by Equation~\eqref{eq def sigma}.
\end{ex}

\begin{lem}[Boundedness]
\label{lem FI bounded}
Let $A$ be a finite set. Let us assume that $f$ is of class $\mathcal{C}^{\norm{A}}$ and that $\Norm{\kappa}_{\norm{A},\eta} \xrightarrow[\eta \to +\infty]{}0$. Then, for all $\I \in \pa_A$, the function $F_\I$ is bounded on $\R^\I$.
\end{lem}

\begin{proof}
Let $B \in \mathcal{S}_A(\I)$. Since $\I_B \in \pa_B$, we have $\norm{\I_B} \leq \norm{B} \leq \norm{A}$. Hence $\Norm{\kappa}_{\norm{\I_B},\eta} \xrightarrow[\eta \to +\infty]{}0$, and $\rho_{\{\I_B\}}$ is bounded by Lemma~\ref{lem bound rho I}. The conclusion follows from the expression of $F_\I$ given by Equation~\eqref{eq def F I}.
\end{proof}

In order to compute the central moments of $\nu_R$, we have to estimate the integral of $F_\I$ over $\R^\I$ for any partition $\I \in \pa_p$. This is done by writing $\R^\I$ as $\bigsqcup_{\J \in \pa_\I} \R^\I_{\J,\eta}$ for some well-chosen $\eta >0$ and proving estimates for the contribution of each $\R^\I_{\J,\eta}$. In the following sections, we will derive estimates for $F_\I$ on $\R^\I_{\J,\eta}$, depending on the combinatorial properties of the partitions $\I$ and $\J$.

Let us conclude this section by choosing the scale parameter $\eta$. In the following, we will work, not with a fixed $\eta >0$, but with a scale parameter depending on $R$, given by the following lemma.

\begin{lem}[Scale parameter]
\label{lem scale parameter}
Let $p \geq 2$ be an integer. Let us assume that $f$ is of class $\mathcal{C}^p$ and that $\Norm{\kappa}_{p,\eta} = o(\eta^{-4p})$ as $\eta \to +\infty$. Then, there exists a function $\eta:(0,+\infty) \to (0,+\infty)$ such that, as $R \to +\infty$ we have: $\eta(R) \to +\infty$, $\eta(R) = o(R^\frac{1}{4})$ and $\Norm{\kappa}_{p,\eta(R)} = o(R^{-p})$.
\end{lem}

\begin{proof}
Let $\epsilon:\tau \mapsto \tau^{4p}\Norm{\kappa}_{p,\tau}$ from $[0,+\infty)$ to itself. We have $\epsilon(\tau) \to 0$ as $\tau \to +\infty$ since $\Norm{\kappa}_{p,\tau} = o(\tau^{-4p})$. Note that this implies that $\epsilon$ is bounded on $[0,+\infty)$. For all $R >0$ we denote by $\alpha(R) = \max\left(R^{-\frac{1}{8}},\left(\sup \{\epsilon(\tau)\mid \tau \geq R^\frac{1}{8}\}\right)^\frac{1}{8p}\right)$. Since $\epsilon(\tau) \xrightarrow[\tau \to +\infty]{}0$ at infinity, we have $\alpha(R) \xrightarrow[R \to +\infty]{}0$. Let us set $\eta(R) = R^\frac{1}{4}\alpha(R)$ for all $R>0$, and let us check that $\eta$ satisfies the desired conditions.

Since $\alpha$ goes to $0$ at infinity, we have $\eta(R) = o(R^\frac{1}{4})$ as $R \to +\infty$. Besides, for all $R>0$ we have $\alpha(R) \geq R^{-\frac{1}{8}}$, hence $\eta(R) \geq R^\frac{1}{8}$ and $\eta(R) \xrightarrow[R \to +\infty]{}+\infty$. Then, let $R >0$, we have:
\begin{equation*}
\Norm{\kappa}_{p,\eta(R)} = \frac{\epsilon(\eta(R))}{\eta(R)^{4p}} = \frac{1}{R^p} \frac{\epsilon(\eta(R))}{\alpha(R)^{4p}}.
\end{equation*}
Since $\eta(R) \geq R^\frac{1}{8}$, we have $\epsilon(\eta(R)) \leq \sup \{\epsilon(\tau) \mid \tau \geq R^\frac{1}{8}\} \leq \alpha(R)^{8p}$. Finally,
\begin{equation*}
\Norm{\kappa}_{p,\eta(R)} \leq \frac{\alpha(R)^{4p}}{R^p} = o(R^{-p}).\qedhere
\end{equation*}
\end{proof}


\subsection{An upper bound on the contribution of each piece}
\label{subsec an upper bound on the contribution of each piece}

In this section, we give upper bounds for the contribution of each addend in the expression of $m_p(\nu_R)(\phi_1,\dots,\phi_p)$ derived in Lemma~\ref{lem integral expression}. We bound the contribution of the integral of the term indexed by $\I \in \pa_p$ over $\R^\I_{\J,\eta(R)}$ in terms of the combinatorial properties of $\I$ and $\J$. See Lemma~\ref{lem upper bound} below for a precise statement.

\begin{rem}
\label{rem take it easy}
Note that $\J$ is partition of $\I$, which is itself a partition of $\{1,\dots,p\}$. This is the main reason why our formalism is so heavy. A good starting point is to understand what happens for the term indexed by $\I=\I_{\min}(p) = \{\{i\} \mid 1 \leq i \leq p\}$ in Lemma~\ref{lem integral expression}. In this case, all the important ideas of the proof appear, but we can simplify the formalism a bit since $\I \simeq \{1,\dots,p\}$ canonically.
\end{rem}

\begin{lem}
\label{lem partitions into pairs}
Let $A$ be a non-empty finite set and let us assume that $f$ is a $\mathcal{C}^{\norm{A}}$-process such that $\Norm{\kappa}_{\norm{A},\eta} = o(\eta^{-4\norm{A}})$ at infinity. Let $\eta:[0,+\infty) \to [0,+\infty)$ be a function satisfying the conditions of Lemma~\ref{lem scale parameter} with $p=\norm{A}$.

Let $\I \in \pa_A$ and let $S \subset A$ be of even cardinality and such that $\I_S = \{\{s\} \mid s \in S\} \subset \I$. Let $\J' \in \pp_{\I_S}$ and $\J'' \in \pa_{\I \setminus \I_S}$, we denote by $\J = \J' \sqcup \J'' \in \pa_\I$. Then, the following holds uniformly for all $\underline{x}_\I \in \R^\I_{\J,\eta(R)}$:
\begin{equation}
\label{eq partitions into pairs mult}
F_\I(\underline{x}_\I) = F_{\I \setminus \I_S}(\underline{x}_{\I \setminus \I_S}) \prod_{J \in \J'} F_J(\underline{x}_J) +o\!\left(R^{-\frac{\norm{A}}{2}}\right).
\end{equation}
\end{lem}

\begin{proof}
If $S$ is empty, then $\I_S = \emptyset$ and $\J' = \emptyset$ by convention. Hence, the result holds in this case. Let us now assume that $S$ is not empty. Then $\I_S$ is non-empty and contains an even number of elements, so that $\pp_{\I_S}$ is non-empty.

Let $\J' \in \pp_{\I_S}$, let $\J'' \in \pa_{\I \setminus \I_S}$ and let $\J = \J' \sqcup \J''$. Let $J \in \J'$, then there exists $s$ and $t \in S$ such that $s \neq t$ and $J=\{\{s\},\{t\}\}$. Let us denote by $A'= A \setminus \{s,t\}$, we have $\I \leq \{A',\{s\},\{t\}\}$ so that $\I = \I_{A'} \sqcup\{\{s\},\{t\}\}$. Recall that $F_\I$ is defined as a sum indexed by $B \in \mathcal{S}_A(\I)$, see Equation~\eqref{eq def F I}. Since $\{s\}$ and $\{t\} \in \I$, we have:
\begin{equation*}
\mathcal{S}_A(\I) = \bigsqcup_{B \in \mathcal{S}_{A'}(\I_{A'})} \{B, B\sqcup \{s\}, B\sqcup \{t\}, B\sqcup \{s,t\}\}.
\end{equation*}
Let $B \in S_{A'}(\I_{A'})$, we regroup the four terms corresponding to $B$, $B \sqcup \{s\}$, $B \sqcup \{t\}$ and $B \sqcup \{s,t\}$ in the sum defining $F_\I$ (see Lemma~\ref{lem integral expression}). For all $\underline{x}_\I \in \R^\I$, we obtain:
\begin{equation}
\label{eq difference four terms}
\left(\!\frac{-1}{\pi}\!\right)^{\!\norm{A}-\norm{B}}\!\left(\pi^2\rho_{\{\I_{B \sqcup \{s,t\}}\}}(\underline{x}_{\I_{B \sqcup \{s,t\}}}) - \pi\rho_{\{\I_{B \sqcup \{s\}}\}}(\underline{x}_{\I_{B \sqcup \{s\}}})- \pi \rho_{\{\I_{B \sqcup \{t\}}\}}(\underline{x}_{\I_{B \sqcup \{t\}}}) + \rho_{\{\I_B\}}(\underline{x}_{\I_B})\!\right).
\end{equation}
Note that $\I_{B \sqcup\{s,t\}} = \I_B \sqcup J$. Since $J = \{\{s\},\{t\}\} \in \J$, if $\underline{x}_\I \in \R^\I_{\J,\eta(R)}$, then for any $I \in \I \setminus J$ we have $\norm{x_I-x_{\{s\}}} \geq \eta(R)$ and $\norm{x_I-x_{\{t\}}} \geq \eta(R)$. In particular, the following holds:
\begin{equation*}
\I_{\eta(R)}\left(\underline{x}_{\I_B \sqcup J}\right) = \J_{\I_B \sqcup J} \leq \{\I_B,J\}.
\end{equation*}
Since we chose $\eta$ so that $\Norm{\kappa}_{\norm{A},\eta(R)} = o(R^{-\norm{A}})$, applying Corollary~\ref{cor clustering rho I} we get:
\begin{equation*}
\rho_{\{\I_{B \sqcup \{s,t\}}\}}(\underline{x}_{\I_{B \sqcup \{s,t\}}}) = \rho_{\{\I_B \sqcup J\}}(\underline{x}_{\I_B \sqcup J}) = \rho_{\{\I_B\}}(\underline{x}_{\I_B}) \rho_{\{J\}}(\underline{x}_J)\left(1 +o\!\left(R^{-\frac{\norm{A}}{2}}\right)\right).
\end{equation*}
We proceed similarly with the three other terms in Equation~\eqref{eq difference four terms}. Bearing in mind that $\rho_{\{J\}}$ and $\rho_{\{\I_B\}}$ are bounded (see Lemma~\ref{lem bound rho I}), we obtain that~\eqref{eq difference four terms} equals:
\begin{equation*}
\left(\frac{-1}{\pi}\right)^{\norm{A}-\norm{B}-2} \rho_{\{\I_B\}}(\underline{x}_{\I_B}) F_J(\underline{x}_J)+o\!\left(R^{-\frac{\norm{A}}{2}}\right).
\end{equation*}
Summing these terms over $B \subset A'$, we get that, for all $\underline{x}_\I \in \R^\I_{\J,\eta(R)}$:
\begin{equation}
\label{eq inductive step}
F_\I(\underline{x}_\I) = F_{\I \setminus J}(\underline{x}_{\I \setminus J})F_J(\underline{x}_J) +o\!\left(R^{-\frac{\norm{A}}{2}}\right).
\end{equation}
We can repeat the argument for $F_{\I \setminus J}(\underline{x}_{\I \setminus J})$, with $\J'$ replaced by $\J'\setminus \{J\}$. More formally, we prove by induction on the cardinality of $\J'$ that Equation~\eqref{eq partitions into pairs mult} holds uniformly on $\R^\I_{\J,\eta(R)}$. The result is true if $\J'=\emptyset$, and the inductive step is given by Equation~\eqref{eq inductive step}.
\end{proof}

\begin{lem}
\label{lem upper bound}
In the same setting as Lemma~\ref{lem partitions into pairs}, let $(\phi_a)_{a \in A}$ be Lebesgue-integrable and essentially bounded functions. Then, as $R \to +\infty$, we have:
\begin{equation*}
\int_{\R^\I_{\J,\eta(R)}} \iota_\I^*\phi_R(\underline{x}_\I) F_\I(\underline{x}_\I) \dx \underline{x}_\I = O\left(R^{\norm{\J}} \eta(R)^{\norm{\I}-2\norm{\J'}-\norm{\J''}}\right),
\end{equation*}
where $\eta:[0,+\infty) \to [0,+\infty)$ is a function satisfying the conditions of Lemma~\ref{lem scale parameter}.
\end{lem}

\begin{proof}
Let $\underline{x}_I = (x_I)_{I \in \I} \in \R^\I_{\J,\eta(R)}$, using the estimate of Lemma~\ref{lem partitions into pairs}, we have:
\begin{equation*}
\iota_\I^*\phi_R(\underline{x}_\I) F_\I(\underline{x}_\I) = \iota_\I^*\phi_R(\underline{x}_\I)F_{\I \setminus \I_S}(\underline{x}_{\I \setminus \I_S}) \prod_{J \in \J'}F_J(\underline{x}_J) +o\!\left(R^{-\frac{\norm{A}}{2}}\right)\iota_\I^*\phi_R(\underline{x}_\I).
\end{equation*}
Then, since $\I = \left(\I \setminus \I_S\right) \sqcup \bigsqcup_{J \in \J'} J$, by Fubini's Theorem,
\begin{multline}
\label{eq upper bound first ineq}
\norm{\int_{\R^\I_{\J,\eta(R)}} \iota_\I^*\phi_R(\underline{x}_\I)F_{\I \setminus \I_S}(\underline{x}_{\I \setminus \I_S}) \prod_{J \in \J'}F_J(\underline{x}_J) \dx \underline{x}_\I}\\
\leq \int_{\R^\I_{\J,\eta(R)}} \norm{\iota_\I^*\phi_R(\underline{x}_\I)F_{\I \setminus \I_S}(\underline{x}_{\I \setminus \I_S}) \prod_{J \in \J'}F_J(\underline{x}_J)} \dx \underline{x}_\I \\
\leq \int_{\R^{\I \setminus \I_S}_{\J'',\eta(R)}} \norm{F_{\I \setminus \I_S}(\underline{x}_{\I \setminus \I_S})}\prod_{I \notin \I_S}\prod_{i \in I}\norm{\phi_i\left(\frac{x_I}{R}\right)} \dx \underline{x}_{\I \setminus \I_S} \\
\times \prod_{J \in \J'} \int_{\R^J}\norm{F_J(\underline{x}_J)}\prod_{I \in J} \prod_{i \in I}\norm{\phi_i\left(\frac{x_I}{R}\right)}\dx \underline{x}_J,
\end{multline}
where we used the fact that $\R^\I_{\J,\eta(R)} \subset \R^{\I \setminus \I_S}_{\J'',\eta(R)} \times \R^{\I_S} = \R^{\I \setminus \I_S}_{\J'',\eta(R)} \times \prod_{J \in \J'} \R^J$. The same kind of computation shows that:
\begin{equation}
\label{eq bonus equation}
\norm{\int_{\R^\I_{\J,\eta(R)}} \iota_\I^*\phi_R(\underline{x}_\I)\dx \underline{x}_\I} \leq \int_{\R^{\I \setminus \I_S}_{\J'',\eta(R)}} \prod_{I \notin \I_S}\prod_{i \in I}\norm{\phi_i\left(\frac{x_I}{R}\right)} \dx \underline{x}_{\I \setminus \I_S}
\times \prod_{J \in \J'} \int_{\R^J}\prod_{I \in J} \prod_{i \in I}\norm{\phi_i\left(\frac{x_I}{R}\right)}\dx \underline{x}_J.
\end{equation}

Let $J=\{\{a\},\{b\}\} \in \J'$. Since $\phi_a$ and $\phi_b$ are integrable on $\R$, then $\phi_a \boxtimes \phi_b$ is integrable on~$\R^2$. Then, as explained in Example~\ref{ex covariances}, we have:
\begin{equation}
\label{eq upper bound J part}
\int_{\R^J}\norm{F_J(\underline{x}_J)}\prod_{I \in J} \prod_{i \in I}\norm{\phi_i\left(\frac{x_I}{R}\right)}\dx \underline{x}_J =\int_{\R^2} \norm{\phi_a\left(\frac{x}{R}\right)\phi_b\left(\frac{y}{R}\right)}\norm{F(y-x)} \dx x \dx y,
\end{equation}
where $F$ is defined by Definition~\ref{def F}. Then, recall that $\phi_a$ and $\phi_b$ are essentially bounded. Denoting by $\Norm{\phi_b}_\infty$ the essential supremum of $\phi_b$, we have:
\begin{align*}
\int_{\R^2} \norm{\phi_a\left(\frac{x}{R}\right)\phi_b\left(\frac{y}{R}\right)}\norm{F(y-x)} \dx x \dx y &\leq \Norm{\phi_b}_\infty \int_{\R^2} \norm{\phi_a\left(\frac{x}{R}\right)}\norm{F(z)} \dx x \dx z\\
&\leq R \Norm{\phi_b}_\infty \left(\int_\R \norm{\phi_a(x)} \dx x\right) \left(\int_\R \norm{F(z)}\dx z\right).
\end{align*}
Since $\phi_a$ and $F$ are integrable (see Lemma~\ref{lem integrability F}), the previous term is $O(R)$. Hence, by Equations~\eqref{eq upper bound first ineq} and~\eqref{eq upper bound J part}, we have:
\begin{equation*}
\norm{\int_{\R^\I_{\J,\eta(R)}} \iota_\I^*\phi_R(\underline{x}_\I)F_{\I \setminus \I_S}(\underline{x}_{\I \setminus \I_S}) \prod_{J \in \J'}F_J(\underline{x}_J) \dx \underline{x}_\I} \leq O(R^{\norm{\J'}})\int_{\R^{\I \setminus \I_S}_{\J'',\eta(R)}} \prod_{I \notin \I_S}\prod_{i \in I}\norm{\phi_i\left(\frac{x_I}{R}\right)} \dx \underline{x}_{\I \setminus \I_S},
\end{equation*}
where we also used the boundedness of $F_{\I \setminus \I_S}$ on $\R^{\I \setminus \I_S}$ (see Lemma~\ref{lem FI bounded}). On the other hand, we have:
\begin{equation*}
\int_{\R^J}\prod_{i \in I}\norm{\phi_i\left(\frac{x_I}{R}\right)}\dx \underline{x}_J =R^2\int_{\R^2} \norm{\phi_a(x)\phi_b(y)} \dx x \dx y = O(R^2),
\end{equation*}
so that, by Equation~\eqref{eq bonus equation},
\begin{equation*}
o\!\left(R^{-\frac{\norm{A}}{2}}\right)\norm{\int_{\R^\I_{\J,\eta(R)}} \iota_\I^*\phi_R(\underline{x}_\I)\dx \underline{x}_\I} = o\left(R^{2\norm{\J'}-\frac{\norm{A}}{2}}\right)\int_{\R^{\I \setminus \I_S}_{\J'',\eta(R)}} \prod_{I \notin \I_S}\prod_{i \in I}\norm{\phi_i\left(\frac{x_I}{R}\right)} \dx \underline{x}_{\I \setminus \I_S}.
\end{equation*}
Since $\J' \in \pp_{\I_S}$ and $\I \in \pa_A$, we have $\norm{\J'} = \frac{\norm{\I_S}}{2} \leq \frac{\norm{\I}}{2} \leq \frac{\norm{A}}{2}$. Hence, in order to conclude the proof, it is enough to prove that:
\begin{equation}
\label{eq upper bound goal I'}
\int_{\R^{\I \setminus \I_S}_{\J'',\eta(R)}} \prod_{I \notin \I_S}\prod_{i \in I}\norm{\phi_i\left(\frac{x_I}{R}\right)} \dx \underline{x}_{\I \setminus \I_S} = O\left(R^{\norm{\J''}}\eta(R)^{\norm{\I}-2\norm{\J'}-\norm{\J''}}\right).
\end{equation}

Note that $\R^{\I \setminus \I_S}_{\J'',\eta(R)} \subset \prod_{J \in \J''} \R^J_{\{J\},\eta(R)}$. Hence, we have:
\begin{equation}
\label{eq1 in upper bound}
\int_{\R^{\I \setminus \I_S}_{\J'',\eta(R)}} \prod_{I \notin \I_S}\prod_{i \in I}\norm{\phi_i\left(\frac{x_I}{R}\right)} \dx \underline{x}_{\I \setminus \I_S} \leq \prod_{J \in \J''} \int_{\R^J_{\{J\},\eta(R)}} \prod_{I \in J}\prod_{i \in I}\norm{\phi_i\left(\frac{x_I}{R}\right)} \dx \underline{x}_J
\end{equation}
Let $J \in \J''$ and let $B = \bigsqcup_{I \in J} I$. Recall that, for all $a \in B$, the function $\phi_a$ is essentially bounded and let us denote by $\Norm{\phi_a}_\infty$ its essential supremum. Let us choose a preferred element $b \in B$ and let $I_0 \in J$ be defined by $b \in I_0$. Then, we have:
\begin{equation}
\label{eq2 in upper bound}
\int_{\R^J_{\{J\},\eta(R)}} \prod_{I \in J}\prod_{i \in I}\norm{\phi_i\left(\frac{x_I}{R}\right)} \dx \underline{x}_J \leq \left(\prod_{a \in B \setminus \{b\}} \Norm{\phi_a}_\infty\right) \int_{\R^J_{\{J\},\eta(R)}} \norm{\phi_b\left(\frac{x_{I_0}}{R}\right)} \dx \underline{x}_J.
\end{equation}
Let $\underline{x}_J \in \R^J_{\{J\},\eta(R)}$, for all $I \in J \setminus \{I_0\}$, we have $\norm{x_I-x_{I_0}} \leq \norm{J} \eta(R)$ (cf.~Remark~\ref{rem R A I eta}). Thus,
\begin{equation}
\label{eq3 in upper bound}
\int_{\R^J_{\{J\},\eta(R)}} \norm{\phi_b\left(\frac{x_{I_0}}{R}\right)} \dx \underline{x}_J \leq \left(\norm{J}\eta(R)\right)^{\norm{J}-1} \int_\R \norm{\phi_b\left(\frac{x}{R}\right)} \dx x = R\left(\norm{J}\eta(R)\right)^{\norm{J}-1}\int_\R \norm{\phi_b(x)} \dx x,
\end{equation}
and this term is $O\!\left(R\eta(R)^{\norm{J}-1}\right)$ since $\phi_b$ is integrable. Finally, by Equations~\eqref{eq1 in upper bound}, \eqref{eq2 in upper bound} and~\eqref{eq3 in upper bound}, we obtain:
\begin{equation*}
\int_{\R^{\I \setminus \I_S}_{\J'',\eta(R)}} \prod_{I \notin \I_S}\prod_{i \in I}\norm{\phi_i\left(\frac{x_I}{R}\right)} \dx \underline{x}_{\I \setminus \I_S} = \prod_{J \in \J''}O\!\left(R\eta(R)^{\norm{J}-1}\right) = O\!\left(R^{\norm{\J''}} \eta(R)^{\sum_{J \in \J''} \norm{J}-\norm{\J''}}\right).
\end{equation*}
Since $\J'' \in \pa_{\I \setminus \I_S}$, we have $\sum_{J \in \J''} \norm{J} = \norm{\I \setminus \I_S} = \norm{\I} - \norm{\I_S}$. Finally, since $\J' \in \pp_{\I_S}$, we have $\norm{\I_S} = 2\norm{\J'}$. Thus, $\sum_{J \in \J''} \norm{J}-\norm{\J''} = \norm{\I}-2\norm{\J'}-\norm{\J''}$, and we just proved Equation~\eqref{eq upper bound goal I'}, which yields the result.
\end{proof}


\subsection{Contribution of the partitions with an isolated point}
\label{subsec contribution of the partitions with an isolated point}

The result of Lemma~\ref{lem upper bound} may lead to think that the main contribution in the integral over $\R^\I$ appearing in Lemma~\ref{lem integral expression} comes from the subsets of the form $\R^\I_{\J,\eta(R)}$, where $\norm{\J}$ is large. In this section, we prove that the function $F_\I$ is uniformly small over $\R^\I_{\J,\eta(R)}$ if $\J$ contains a singleton. In particular, the contribution of $\R^\I_{\J,\eta(R)}$ to the integral over $\R^\I$ appearing in Lemma~\ref{lem integral expression} is small if $\norm{\J} > \frac{\norm{\I}}{2}$.

\begin{lem}
\label{lem partitions with a singleton}
Let $A$ be a finite set of cardinality at least $2$, and let us assume that $f$ is a $\mathcal{C}^{\norm{A}}$-process such that $\Norm{\kappa}_{\norm{A},\eta} = o(\eta^{-4\norm{A}})$ at infinity. Let $\eta:[0,+\infty) \to [0,+\infty)$ be a function satisfying the conditions of Lemma~\ref{lem scale parameter} with $p=\norm{A}$.

Let $\I \in \pa_A$, we assume that there exists $i \in A$ such that $\{i\} \in \I$. Then, for all $\J \in \pa_\I$ such that $\left\{\{i\}\right\} \in \J$, we have $F_\I(\underline{x}_\I) = o\!\left(R^{-\frac{\norm{A}}{2}}\right)$ as $R \to +\infty$, uniformly in $\underline{x}_\I \in \R^\I_{\J,\eta(R)}$.
\end{lem}

\begin{proof}
Recall that we defined $\mathcal{S}_A(\I)$ in Definition~\ref{def subsets adapted to I}. Let $A' = A \setminus \{i\}$. Since $\{i\} \in \I$, we have $\I = \I_{A'} \sqcup \{\{i\}\}$, and:
\begin{equation*}
\mathcal{S}_A(\I) = \bigsqcup_{B \in \mathcal{S}_{A'}(\I_{A'})} \{B, B\sqcup \{i\}\}.
\end{equation*}
In other terms, we can split $\mathcal{S}_A(\I)$ into the subsets that contain $i$ and those that do not, and $B \mapsto B \sqcup \{i\}$ is a bijection from $\mathcal{S}_{A'}(\I_{A'}) = \{B \in \mathcal{S}_A(\I) \mid i \notin B\}$ to $\{B \in \mathcal{S}_A(\I) \mid i \in B\}$.

Let $\underline{x}_\I=(x_I)_{I \in \I} \in \R^\I$ and let $B \in \mathcal{S}_A(\I)$ be such that $i \notin B$. Regrouping the terms corresponding to $B$ and $B \sqcup\{i\}$ in the sum defining $F_\I(\underline{x}_\I)$ (see Equation~\eqref{eq def F I}), we obtain:
\begin{equation}
\label{eq difference two terms}
\left(\frac{-1}{\pi}\right)^{\norm{A} - \norm{B}-1}\left(\rho_{\{\I_{B\sqcup\{i\}}\}}(\underline{x}_{\I_{B\sqcup\{i\}}}) - \frac{1}{\pi}\rho_{\{\I_B\}}(\underline{x}_{\I_B})\right).
\end{equation}
Note that we have $\I = \I_B \sqcup \{\{a\} \mid a \in A \setminus B\}$ and $\I_{B\sqcup \{i\}} = \I_B \sqcup \{\{i\}\}$. Let us now assume that $\underline{x}_\I \in \R^\I_{\J,\eta(R)}$. Since $\{\{i\}\} \in \J$, for all $I \in \I \setminus \{\{i\}\}$ we have $\norm{x_I-x_{\{i\}}} \geq \eta(R)$, and $\I_{\eta(R)}\left(\underline{x}_{\I_{B \sqcup \{i\}}}\right) = \J_{\I_{B \sqcup \{i\}}} \leq \{\I_B,\{i\}\}$. Then, by Lemma~\ref{lem bound rho I} and Corollary~\ref{cor clustering rho I}, the right-hand side of Equation~\eqref{eq difference two terms} is $o\!\left(R^{-\frac{\norm{A}}{2}}\right)$, uniformly over $\R^\I_{\J,\eta(R)}$. Here, we also used the fact that $\rho_1$ is constant equal to $\frac{1}{\pi}$ (see Example~\ref{ex Kac-Rice densities}) and that $\Norm{\kappa}_{\norm{A},\eta(R)} = o(R^{-\norm{A}})$, since $\eta$ satisfies the conclusion of Lemma~\ref{lem scale parameter}.
\end{proof}

\begin{cor}
\label{cor partitions with a singleton}
In the same setting as Lemma~\ref{lem partitions with a singleton}, let $(\phi_a)_{a \in A}$ be Lebesgue-integrable and essentially bounded functions. Then, as $R \to +\infty$, we have:
\begin{equation*}
\int_{\R^\I_{\J,\eta(R)}} \iota_\I^*\phi_R(\underline{x}_\I) F_\I(\underline{x}_\I) \dx \underline{x}_\I = o\!\left(R^\frac{\norm{A}}{2}\right),
\end{equation*}
where $\eta:[0,+\infty) \to [0,+\infty)$ is a function satisfying the conditions of Lemma~\ref{lem scale parameter}.
\end{cor}

\begin{proof}
Our hypotheses ensure that $\iota_\I^*\phi$ is integrable on $\R^\I$. Then, by Lemma~\ref{lem partitions with a singleton}, we have:
\begin{equation*}
\int_{\R^\I_{\J,\eta(R)}} \iota_\I^*\phi_R(\underline{x}_\I) F_\I(\underline{x}_\I) \dx \underline{x}_\I = o\!\left(R^{-\frac{\norm{A}}{2}}\right) \int_{\R^\I} \norm{\iota_\I^*\phi\left(\frac{\underline{x}_\I}{R}\right)} \dx \underline{x}_\I = o\!\left(R^{\norm{\I}-\frac{\norm{A}}{2}}\right).
\end{equation*}
Since $\I \in \pa_A$, we have $\norm{\I} \leq \norm{A}$, which concludes the proof.
\end{proof}

\begin{ex}
\label{ex covariances 2}
Let us consider the simple case where $A = \{a,b\}$ and $\I = \I_{\min}(A) = \{\{a\},\{b\}\}$. Applying Corollary~\ref{cor partitions with a singleton} with $\J =\I_{\min}(\I) = \left\{\rule{0em}{2ex}\{\{a\}\},\{\{b\}\}\right\}$, we have:
\begin{equation*}
\int_{\R^\I_{\J,\eta(R)}} \iota_\I^*\phi_R(\underline{x}_\I) F_\I(\underline{x}_\I) \dx \underline{x}_\I = \int_{\{(x,y) \in \R^2 \mid \norm{x-y} > \eta(R)\}} \phi_a\left(\frac{x}{R}\right)\phi_b\left(\frac{y}{R}\right) F_\I(x,y) \dx x \dx y = o(R).
\end{equation*}
Besides, by Example~\ref{ex covariances}, we know that:
\begin{equation*}
\int_{\R^2} \phi_a\left(\frac{x}{R}\right)\phi_b\left(\frac{y}{R}\right) F_\I(x,y) \dx x \dx y = R\left(\sigma^2-\frac{1}{\pi}\right)\int_{\R^2} \phi_a(x)\phi_b(x) \dx x +o(R),
\end{equation*}
where $\sigma^2$ is the constant appearing in Proposition~\ref{prop variance}. Since $\pa_\I = \{\J,\{\I\}\}$, this shows that:
\begin{align*}
\int_{\R^\I_{\{\I\},\eta(R)}} \iota_\I^*\phi_R(\underline{x}_\I) F_\I(\underline{x}_\I) \dx \underline{x}_\I &= \int_{\{(x,y) \in \R^2 \mid \norm{x-y} \leq \eta(R)\}} \phi_a\left(\frac{x}{R}\right)\phi_b\left(\frac{y}{R}\right) F_\I(x,y) \dx x \dx y\\
&= R\left(\sigma^2-\frac{1}{\pi}\right)\int_{\R^2} \phi_a(x)\phi_b(x) \dx x +o(R).
\end{align*}
\end{ex}


\subsection{Contribution of the partitions into pairs}
\label{subsec contribution of the partitions into pairs}

The goal of this section is to study the behavior of the function $F_\I$, introduced in Lemma~\ref{lem integral expression}, on some particular pieces of the decomposition $\R^\I = \bigsqcup_{\J \in \pa_\I} \R^\I_{\J,\eta(R)}$. More precisely, we will consider pieces indexed by partitions into pairs.

\begin{dfn}
\label{def Cp}
Let $A$ be a non-empty finite set. Let $\I \in \pa_A$, we denote by $\I'=\{ I \in \I \mid \norm{I} \geq 2\}$ and by $\I'' = \{I \in \I \mid \norm{I}=1\}$. We denote by $\mathfrak{C}_A$ the set of couples $(\I,\J)$ such that $\I \in \pa_A$, $\J \in \pa_\I$ and the following two conditions are satisfied:
\begin{enumerate}
\item \label{cond 1} for all $I \in \I'$, we have $\norm{I} = 2$;
\item \label{cond 2} there exists $\J' \in \pp_{\I''}$ such that $\J = \J' \sqcup \J''$, where $\J'' = \I_{\min}(\I')= \{\{I\} \mid I \in \I'\}$.
\end{enumerate}
As usual, if $A = \{1,\dots,p\}$, we denote by $\mathfrak{C}_p=\mathfrak{C}_A$.
\end{dfn}

\begin{rem}
\label{rem conditions 1 and 2}
If $(\I,\J) \in \mathfrak{C}_A$, then $\I''$ admits a partition into pairs, hence $\norm{\I''}$ is even. Moreover, $\norm{A} = 2\norm{\I'}+\norm{\I''}$ is also even.
\end{rem}

\begin{lem}
\label{lem partitions into pairs 2}
Let $A$ be a non-empty finite set and let us assume that $f$ is a $\mathcal{C}^{\norm{A}}$-process such that $\Norm{\kappa}_{\norm{A},\eta} = o(\eta^{-4\norm{A}})$ at infinity. Let $\eta:[0,+\infty) \to [0,+\infty)$ be a function satisfying the conditions of Lemma~\ref{lem scale parameter} with $p=\norm{A}$.

Let $(\I,\J) \in \mathfrak{C}_A$ and let $\I'$, $\I''$, $\J'$ and $\J''$ be as in Definition~\ref{def Cp}. Then the following holds uniformly for all $\underline{x}_\I \in \R^\I_{\J,\eta(R)}$:
\begin{equation*}
F_\I(\underline{x}_\I) = \left(\frac{1}{\pi}\right)^{\norm{\I'}} \prod_{J \in \J'}F_J(\underline{x}_J) +o\!\left(R^{-\frac{\norm{A}}{2}}\right).
\end{equation*}
\end{lem}

\begin{proof}
Let us denote by $S = \bigsqcup_{I \in \I''} I$. We have $\norm{S} = \norm{\I''}$, and by Remark~\ref{rem conditions 1 and 2} this cardinality is even. By Lemma~\ref{lem partitions into pairs}, we have the following uniform estimate: for all $\underline{x}_\I \in \R^\I_{\J,\eta(R)}$,
\begin{equation*}
F_\I(\underline{x}_\I) = F_{\I'}(\underline{x}_{\I'}) \prod_{J \in \J'}F_J(\underline{x}_J) +o\!\left(R^{-\frac{\norm{A}}{2}}\right).
\end{equation*}

Note that we have $\I' \in \pa_{A \setminus S}$. Recalling Definition~\ref{def subsets adapted to I}, we have $\mathcal{S}_{A \setminus S}(\I') = \{A \setminus S\}$ because of Condition~\ref{cond 1} in Definition~\ref{def Cp}. Hence we obtain $F_{\I'} = \rho_{\{\I'\}}$ by Equation~\eqref{eq def F I}. By Condition~\ref{cond 2} in Definition~\ref{def Cp} we have $\J_{\I'} = \J'' = \I_{\min}(\I')$. Then, for any $\underline{x}_\I \in \R^\I_{\J,\eta(R)}$ we have $\underline{x}_{\I'} \in \R^{\I'}_{\J'',\eta(R)}$, and by Proposition~\ref{prop clustering rho I}, we obtain:
\begin{equation*}
F_{\I'}(\underline{x}_{\I'}) = \rho_{\{\I'\}}(\underline{x}_{\I'}) = \left(\prod_{J \in \J''} \rho_{\{J\}}(\underline{x}_J)\right)\left(1+O\!\left(\Norm{\kappa}_{\norm{A},\eta(R)}\right)^\frac{1}{2}\right),
\end{equation*}
uniformly for all $\underline{x}_\I \in \R^\I_{\J,\eta(R)}$. Our hypotheses on $\kappa$ and $\eta$ ensure that $\Norm{\kappa}_{\norm{A},\eta(R)} = o\!\left(R^{-\frac{\norm{A}}{2}}\right)$. Moreover, for all $J \in \J''$ we have $\norm{J}=1$, so that $\rho_{\{J\}} = \rho_1$ is constant equal to $\frac{1}{\pi}$, see Example~\ref{ex Kac-Rice densities}. The conclusion follows from the boundedness of the functions $F_J$ with $J \in \J'$ (see Lemma~\ref{lem FI bounded}) and from $\norm{\J''} = \norm{\I'}$.
\end{proof}


\subsection{Conclusion of the proof}
\label{subsec conclusion of the proof}

In this section, we finally conclude the proof of Theorem~\ref{thm moments}. In all this section, we fix an integer $p \geq 2$. We consider a normalized centered stationary Gaussian process $f$ of class $\mathcal{C}^p$. The correlation function $\kappa$ of $f$ is such that $\Norm{\kappa}_{p,\eta} = o(\eta^{-4p})$ as $\eta \to +\infty$. We fix a function $\eta:[0,+\infty) \to [0,+\infty)$ such that as $R \to +\infty$ we have: $\eta(R) \to +\infty$, $\eta(R) = o(R^\frac{1}{4})$, and $\Norm{\kappa}_{p,\eta(R)} = o(R^{-p})$. The existence of such a function was proved in Lemma~\ref{lem scale parameter}. Finally, we consider test-functions $\phi_1,\dots,\phi_p$ is the sense of Definition~\ref{def test-function}.

\begin{lem}[Error terms]
\label{lem order of the error terms}
Let $\I \in \pa_p$ and $\J \in \pa_\I$ be such that $(\I,\J) \notin \mathfrak{C}_p$, where $\mathfrak{C}_p$ is defined by Definition~\ref{def Cp}. Then, as $R \to +\infty$, we have:
\begin{equation}
\label{eq order of the error terms}
\int_{\R^\I_{\J,\eta(R)}} \iota_\I^*\phi_R(\underline{x}_\I) F_\I(\underline{x}_\I)\dx \underline{x}_\I = o(R^\frac{p}{2}).
\end{equation}
\end{lem}

\begin{proof}
As in Definition~\ref{def Cp}, let us denote by $\I' = \{I \in \I \mid \norm{I} \geq 2\}$ and by $\I'' = \{I \in \I \mid \norm{I} =1\}$. Since $\I = \I' \sqcup \I'' \in \pa_p$, we have:
\begin{equation}
\label{eq order error term 1}
2\norm{\I'}+\norm{\I''} \leq p
\end{equation}
and equality holds if and only if $\norm{I}=2$ for all $I \in \I'$, that is if and only if Condition~\ref{cond 1} of Definition~\ref{def Cp} is satisfied.

Let $J \in \J$ be such that $\norm{J}=1$. There exists $I \in \I$ such that $J = \{I\}$. If $I \in \I''$, that is if $I$ is a singleton, then we are in the situation studied in Section~\ref{subsec contribution of the partitions with an isolated point}. In this case, Equation~\eqref{eq order of the error terms} holds by Corollary~\ref{cor partitions with a singleton}.

In the following, we assume that we are not in the previous situation. That is, for all $J \in \J$ such that $\norm{J}=1$, we have $J \subset \I'$. Let us prove that Equation~\eqref{eq order of the error terms} holds in this case. We denote by $\J' = \{J \in \J \mid J \subset \I'', \norm{J}=2\}$ and by $\J'' = \J \setminus \J'$. Finally, let us denote by $S= \bigsqcup_{J \in \J'} \bigsqcup_{I \in J} I \subset \{1,\dots,p\}$. By definition of $S$ we have $\I_S = \{\{s\} \mid s \in S\} \subset \I'' \subset \I$ and $\J' \in \pp_{\I_S}$. We also have $\J'' \in \pa_{\I \setminus \I_S}$ and $\J = \J' \sqcup \J''$. By Lemma~\ref{lem upper bound}, the left-hand side of Equation~\eqref{eq order of the error terms} equals $O\!\left(R^{\norm{\J}}\eta(R)^{\norm{\I}-2\norm{\J'}-\norm{\J''}}\right)$. Since $\eta(R) =o(R^\frac{1}{4})$, there exists a function $\alpha$ such that $\alpha(R) \xrightarrow[R \to +\infty]{} 0$ and $\eta(R)= \alpha(R)R^\frac{1}{4}$. Then, left-hand side of~\eqref{eq order of the error terms} is:
\begin{equation}
\label{eq big O}
O\!\left(R^{\norm{\J}+\frac{\norm{\I}-2\norm{\J'}-\norm{\J''}}{4}}\alpha(R)^{\norm{\I}-2\norm{\J'}-\norm{\J''}}\right).
\end{equation}

Let $J \in \J''$. Since $\J''$ is a partition of $\I \setminus \I_S = \I' \sqcup (\I'' \setminus \I_S)$, if $J \cap \I' = \emptyset$ then $J \subset \I'' \setminus \I_S$. In this case, we assumed that $\norm{J} \neq 1$, and moreover $\norm{J} \neq 2$ by definition of $\J'$ and $\J''$. Thus, either there exists $I \in \I' \cap J$, or $J \subset \I'' \setminus \I_S$ and $\norm{J} \geq 3$. This proves that:
\begin{equation*}
\norm{\J''} \leq \norm{\I'}+\frac{1}{3}(\norm{\I''}-\norm{\I_S}).
\end{equation*}
Since $\norm{\I_S}=2\norm{\J'}$, we have $2\norm{\J'}+3\norm{\J''} \leq 3\norm{\I'} + \norm{\I''}$. Hence,
\begin{equation*}
\norm{\J}+\frac{\norm{\I}-2\norm{\J'}-\norm{\J''}}{4} = \frac{1}{4}\left(2\norm{\J'}+3\norm{\J''}+\norm{\I'}+\norm{\I''}\right) \leq \frac{1}{2}(2\norm{\I'}+\norm{\I''}).
\end{equation*}
Then, by Equation~\eqref{eq order error term 1}, we have:
\begin{equation}
\label{eq order error term 2}
\norm{\J}+\frac{\norm{\I}-2\norm{\J'}-\norm{\J''}}{4} \leq \frac{p}{2}.
\end{equation}
If this inequality is strict, then \eqref{eq big O} equals a $o(R^\frac{p}{2})$, which proves that Equation~\eqref{eq order of the error terms} holds.

If equality holds in Equation~\eqref{eq order error term 2}, then it must hold in Equation~\eqref{eq order error term 1}, which implies that $(\I,\J)$ satisfies Condition~\ref{cond 1} in Definition~\ref{def Cp}. In this case, \eqref{eq big O} is a $O\!\left(R^\frac{p}{2}\alpha(R)^{\norm{\I}-2\norm{\J'}-\norm{\J''}}\right)$ and, since $\alpha(R) = o(1)$, it is enough to check that $2\norm{\J'}+\norm{\J''} < \norm{\I}$. Since $\J = \J' \sqcup \J'' \in \pa_\I$ and $\norm{J}=2$ for all $J \in \J'$, we have:
\begin{equation}
\label{eq order error term 3}
2\norm{\J'} + \norm{\J''} \leq \norm{\I},
\end{equation}
and equality holds if and only if $\norm{J}=1$ for all $J \in \J''$. If equality held in Equation~\eqref{eq order error term 3} then, under our assumptions, we would have $\norm{J}=1$ and $J \subset \I'$ for all $J \in \J''$, hence $\J'' \subset \I_{\min}(\I')$. Since $\J' \in \pa_{\I_S}$ and $\I_S \subset \I'' = \I \setminus \I'$, the only possibility for this to happen is that $\J'' = \I_{\min}(\I')$ and $\I_S = \I''$. Thus, if equality held in Equation~\eqref{eq order error term 3} then $(\I,\J)$ would satisfy Condition~\ref{cond 2} of Definition~\ref{def Cp}. Since we already assumed that $(\I,\J)$ satisfies Condition~\ref{cond 1}, this would imply $(\I,\J) \in \mathfrak{C}_p$, which is a contradiction. Finally, the inequality is strict in Equation~\eqref{eq order error term 3}. Hence, \eqref{eq big O} is a $o(R^\frac{p}{2})$, which concludes the proof.
\end{proof}

\begin{lem}[Leading terms]
\label{lem leading term}
Let $\I \in \pa_p$ and $\J \in \pa_\I$ be such that $(\I,\J) \in \mathfrak{C}_p$. Let $\I',\I'',\J'$ and $\J''$ be as in Definition~\ref{def Cp}. Then, as $R \to +\infty$, we have:
\begin{multline*}
\int_{\underline{x}_\I \in \R^\I_{\J,\eta(R)}} \iota_\I^*\phi_R(\underline{x}_\I) F_\I(\underline{x}_\I) \dx \underline{x}_\I= \\ \left(\prod_{\{i,j\} \in \I'} \frac{R}{\pi} \int_\R \phi_i(x)\phi_j(x) \dx x\right)\left(\prod_{J \in \J'} \int_{\R^J} \phi_J\left(\frac{\underline{x}_J}{R}\right) F_J(\underline{x}_J) \dx \underline{x}_J \right)+ o\!\left(R^\frac{p}{2}\right).
\end{multline*}
\end{lem}

\begin{proof}
For any point $\underline{x}_\I \in \R^\I$, we have:
\begin{equation*}
\iota_\I^*\phi_R(\underline{x}_\I) = \prod_{I \in \I} \prod_{i \in I} \phi_i\left(\frac{x_I}{R}\right) = \left(\prod_{\{i,j\}=I \in \I'} \phi_i\left(\frac{x_I}{R}\right)\phi_j\left(\frac{x_I}{R}\right)\right) \left(\prod_{\{i\} \in \I''} \phi_i\left(\frac{x_{\{i\}}}{R}\right)\right).
\end{equation*}
Note that, for all $J = \{\{i\},\{j\}\} \in \J'$ we have $\phi_J = \phi_{\{i\}} \boxtimes \phi_{\{j\}} = \phi_i \boxtimes \phi_j$ (see Notation~\ref{ntn product indexed by A}). Hence, by Lemma~\ref{lem partitions into pairs 2}, for all $\underline{x}_I \in \R^\I_{\J,\eta(R)}$, we have that $\iota_\I^*\phi_R(\underline{x}_\I) F_\I(\underline{x}_\I)$ equals:
\begin{equation}
\label{eq leading term}
\left(\prod_{\{i,j\}=I \in \I'}\frac{1}{\pi} \phi_i\left(\frac{x_I}{R}\right)\phi_j\left(\frac{x_I}{R}\right)\right) \left(\prod_{J \in \J'} \phi_J\left(\frac{\underline{x}_J}{R}\right)F_J(\underline{x}_J)\right)+\iota_\I^*\phi_R(\underline{x}_\I) o\!\left(R^{-\frac{p}{2}}\right),
\end{equation}
and we want to compute the integral of~\eqref{eq leading term} over $\R^\I_{\J,\eta(R)}$. Our assumptions on the the test-functions $\phi_1,\dots,\phi_p$ ensure that $\iota_\I^*\phi$ is integrable on $\R^\I$. Hence, since $\norm{\I} \leq p$, the contribution to the integral of the error term in Equation~\eqref{eq leading term} is:
\begin{equation*}
\int_{\R^\I_{\J,\eta(R)}}\iota_\I^*\phi_R(\underline{x}_\I) o\!\left(R^{-\frac{p}{2}}\right) \dx \underline{x}_\I = o\!\left(R^{-\frac{p}{2}}\right) \int_{\R^\I_{\J,\eta(R)}}\norm{\iota_\I^*\phi_R(\underline{x}_\I)}\dx \underline{x}_\I = o\!\left(R^{\norm{\I}-\frac{p}{2}}\right) = o\!\left(R^\frac{p}{2}\right).
\end{equation*}

Let us check that:
\begin{equation}
\label{eq subsets leading term}
\R^\I_{\J,\eta(R)} \subset \left(\prod_{I \in \I'} \R\right)\times \left(\prod_{J \in \J'} \R^J_{\{J\},\eta(R)}\right) \subset \bigsqcup_{\K \geq \J} \R^\I_{\K,\eta(R)}.
\end{equation}
In order to prove the first inclusion, let $\underline{x}_\I \in \R^\I_{\J,\eta(R)}$. Since $\J$ satisfies Condition~\ref{cond 2} of Definition~\ref{def Cp}, the associated graph $G_{\eta(R)}(\underline{x}_\I)$ (see Definition~\ref{def G eta}) is formed of $\norm{\I'}$ isolated vertices $\{I \mid I \in \I'\}$ and $\norm{\J'}$ pairs of vertices of the form $\{I,I'\} \in \J'$ with an edge between $I$ and $I'$. Hence, for all $J = \{I,I'\} \in \J'$, we have $\norm{x_I-x_{I'}}\leq \eta(R)$, and $\underline{x}_J \in \R^{J}_{\{J\},\eta(R)}$. Let us now prove the second inclusion in Equation~\eqref{eq subsets leading term}. Let $\underline{x}_\I=(x_I)_{I \in \I} \in \R^\I$ and let us assume that $\underline{x}_\I$ belongs to the middle set in Equation~\eqref{eq subsets leading term}. By Definition~\ref{def G eta}, for all $J=\{I,I'\} \in \J'$, the graph $G_{\eta(R)}(\underline{x}_\I)$ associated with $\underline{x}_\I$ has an edge between $I$ and $I'$. Hence the associated partition $\K = \I_{\eta(R)}(\underline{x}_\I)$ (see Definition~\ref{def I eta}) is such that, for all $J \in \J'$ there exists $K \in \K$ such that $J \subset K$. Since $\J$ satisfies Condition~\ref{cond 2} in Definition~\ref{def Cp}, this is enough to ensure that $\K \geq \J$. Note that, in general, both inclusions in Equation~\eqref{eq subsets leading term} are strict.

Let $\K \in \pa_\I$ be such that $\K > \J$. This ensures that $(\I,\K)\notin \mathfrak{C}_p$.Then, by Lemma~\ref{lem order of the error terms}, the integral of the leading term in Equation~\eqref{eq leading term} over $\left(\left(\prod_{I \in \I'} \R\right)\times \left(\prod_{J \in \J''} \R^J_{\{J\},\eta(R)}\right)\right) \cap \R^\I_{\K,\eta(R)}$ equals $o(R^\frac{p}{2})$. This proves that:
\begin{multline}
\label{eq integral over J}
\int_{\underline{x}_\I \in \R^\I_{\J,\eta(R)}} \iota_\I^*\phi_R(\underline{x}_\I) F_\I(\underline{x}_\I)\dx \underline{x}_\I=\\
\left(\prod_{\{i,j\} \in \I'} \frac{R}{\pi} \int_\R \phi_i(x)\phi_j(x) \dx x\right) \left(\prod_{J \in \J'} \int_{\R^J_{\{J\},\eta(R)}} \phi_J\left(\frac{\underline{x}_J}{R}\right)F_J(\underline{x}_J) \dx \underline{x}_J \right)+ o\!\left(R^\frac{p}{2}\right).
\end{multline}

In order to conclude the proof, we need to replace the integral over $\R^{J}_{\{J\},\eta(R)}$ by an integral over $\R^J$ in Equation~\eqref{eq integral over J}, for all $J \in \J'$. Let $J = \{I,I'\} \in \J'$, we have:
\begin{align*}
\int_{\R^J_{\{J\},\eta(R)}} \phi_J\left(\frac{\underline{x}_J}{R}\right)F_J(\underline{x}_J) \dx \underline{x}_J &= \int_{\R^J} \phi_J\left(\frac{\underline{x}_J}{R}\right)F_J(\underline{x}_J) \dx \underline{x}_J - \int_{\R^J_{\{\{I\},\{I'\}\},\eta(R)}} \phi_J\left(\frac{\underline{x}_J}{R}\right)F_J(\underline{x}_J) \dx \underline{x}_J\\
&= \int_{\R^J} \phi_J\left(\frac{\underline{x}_J}{R}\right)F_J(\underline{x}_J) \dx \underline{x}_J +o(R),
\end{align*}
by Corollary~\ref{cor partitions with a singleton}. Moreover, as in Example~\ref{ex covariances 2}, we have:
\begin{equation*}
\int_{\R^J} \phi_J\left(\frac{\underline{x}_J}{R}\right)F_J(\underline{x}_J) \dx \underline{x}_J = O(R).
\end{equation*}
Since $\norm{\I'} + \norm{\J'} = \norm{\I'}+\frac{1}{2}\norm{\I''} = \frac{p}{2}$, this yields the claimed estimate.
\end{proof}

\begin{proof}[Proof of Theorem~\ref{thm moments}]
By Lemma~\ref{lem integral expression} and the fact that $\R^\I = \bigsqcup_{\J \in \pa_\I} \R^\I_{\J,\eta(R)}$ for all $\I \in \pa_p$, we have:
\begin{equation*}
m_p(\nu_R)(\phi_1,\dots,\phi_p) = \sum_{\I \in \pa_p} \sum_{\J \in \pa_\I} \int_{\underline{x}_\I\in \R^\I_{\J,\eta(R)}} \left(\iota_\I^*\phi_R\right)(\underline{x}_\I) F_\I(\underline{x}_\I) \dx \underline{x}_\I.
\end{equation*}
By Lemma~\ref{lem order of the error terms}, up to an error term of the form $o(R^\frac{p}{2})$ we need only consider the terms in this double sum indexed by $(\I,\J) \in \mathfrak{C}_p$, where $\mathfrak{C}_p$ is defined by Definition~\ref{def Cp}. The expression of these terms is given by Lemma~\ref{lem leading term}. Thus, we have:
\begin{multline*}
m_p(\nu_R)(\phi_1,\dots,\phi_p) = \\
\sum_{(\I,\J) \in \mathfrak{C}_p} \left(\prod_{\{i,j\} \in \I'} \frac{R}{\pi} \int_\R \phi_i(x)\phi_j(x) \dx x\right)\left(\prod_{J \in \J'} \int_{\R^J} \phi_J\left(\frac{\underline{x}_J}{R}\right) F_J(\underline{x}_J) \dx \underline{x}_J \right) + o(R^\frac{p}{2}),
\end{multline*}
where we used the same notations as in Definition~\ref{def Cp}. Recall that $\I' = \{ I \in \I \mid \norm{I}=2\}$ and $\I'' = \{ I \in \I \mid \norm{I}=1\}$, so that $\I = \I' \sqcup \I''$ thanks to Condition~\ref{cond 1} in Definition~\ref{def Cp}. Recall also that $\J''=\{\{I\} \mid I \in \I'\}$ and $\J = \J' \sqcup \J''$ for some $\J' \in \pp_{\I''}$, thanks to Condition~\ref{cond 2} in Definition~\ref{def Cp}.

We just wrote $m_p(\nu_R)(\phi_1,\dots,\phi_p)$ as a sum of terms indexed by $\mathfrak{C}_p$.
In the following, we define a bijection $\Phi: \mathfrak{C}_p \to \{(\Pi,\mathcal{S}) \mid \Pi \in \pp_p, \mathcal{S} \subset \Pi\}$, which will allow us to rewrite $m_p(\nu_R)(\phi_1,\dots,\phi_p)$ as sum over $\{(\Pi,\mathcal{S}) \mid \Pi \in \pp_p, \mathcal{S} \subset \Pi\}$, by a change of variable. 
\begin{itemize}
\item Let $(\I,\J) \in \mathfrak{C}_p$ and let us denote by $S = \{i \mid \{i\} \in \I''\}$ and by $\mathcal{S} = \{\{i,j\} \mid \{\{i\},\{j\}\} \in \J'\}$. Since $\J' \in \pp_{\I''}$, we have $\mathcal{S} \in \pp_S$. Since $\I'$ is a partition of $\bigsqcup_{I \in \I'} I = \{1,\dots,p\} \setminus S$ into pairs, we define a partition into pairs $\Pi \in \pp_p$ by $\Pi = \I' \sqcup \mathcal{S}$. We obtain a couple $(\Pi,\mathcal{S})$ where $\Pi \in \pp_p$ and $\mathcal{S} \subset \Pi$. Let us denote this couple by $\Phi(\I,\J)$.
\item Conversely, let $\Pi \in \pp_p$ and let $\mathcal{S} \subset \Pi$. We set $S = \bigsqcup_{I \in \mathcal{S}} I$ and $\I' = \Pi \setminus \mathcal{S}$, so that $\I'$ is a partition into pairs of $\{1,\dots,p\} \setminus S$. Let us denote by $\I'' = \{\{i\} \mid i \in S\}$ and by $\J' = \{\{\{i\},\{j\}\} \mid \{i,j\} \in \mathcal{S} \}$, so that $\J' \in \pp_{\I''}$. Finally, let $\I = \I' \sqcup \I'' \in \pa_p$, and let $\J'' = \{ \{I\} \mid I \in \I'\}$ so that $\J = \J'\sqcup \J'' \in \pa_\I$. We just defined a couple $(\I,\J) \in \mathfrak{C}_p$ that we denote by $\Psi(\Pi,\mathcal{S})$.
\end{itemize}

By construction, $\Phi$ is a bijection from $\mathfrak{C}_p$ to $\{(\Pi,\mathcal{S}) \mid \Pi \in \pp_p, \mathcal{S} \subset \Pi\}$ such that $\Psi = \Phi^{-1}$. Moreover, for all $(\I,\J) \in \mathfrak{C}_p$, denoting by $(\Pi,\mathcal{S}) = \Phi(\I,\J)$, we have:
\begin{multline*}
\left(\prod_{\{i,j\} \in \I'} \frac{R}{\pi} \int_\R \phi_i(x)\phi_j(x) \dx x\right)\left(\prod_{J \in \J'} \int_{\R^J} \phi_J\left(\frac{\underline{x}_J}{R}\right) F_J(\underline{x}_J) \dx \underline{x}_J \right) =\\
\left(\prod_{\{i,j\} \in \Pi \setminus \mathcal{S}} \frac{R}{\pi} \int_\R \phi_i(x)\phi_j(x) \dx x\right)\left(\prod_{\{i,j\} \in \mathcal{S}} \int_{\R^2} \phi_i\left(\frac{x}{R}\right)\phi_j\left(\frac{y}{R}\right) F(y-x) \dx x \dx y \right),
\end{multline*}
where we used Example~\ref{ex covariances}. Hence,
\begin{multline}
\label{eq final mp}
m_p(\nu_R)(\phi_1,\dots,\phi_p) =\\
\sum_{\Pi \in \pp_p} \sum_{\mathcal{S} \subset \Pi} \left(\prod_{\{i,j\} \in \Pi \setminus \mathcal{S}} \frac{R}{\pi} \int_\R \phi_i(x)\phi_j(x) \dx x\right)\left(\prod_{\{i,j\} \in \mathcal{S}} \int_{\R^2} \phi_i\left(\frac{x}{R}\right)\phi_j\left(\frac{y}{R}\right) F(y-x) \dx x \dx y \right)\\+ o(R^\frac{p}{2}).
\end{multline}
The leading term on the right-hand side of Equation~\eqref{eq final mp} can be rewritten as:
\begin{equation*}
\sum_{\Pi \in \pp_p} \prod_{\{i,j\} \in \Pi} \left(\int_{\R^2} \phi_i\left(\frac{x}{R}\right)\phi_j\left(\frac{y}{R}\right) F(y-x) \dx x \dx y + \frac{R}{\pi} \int_\R \phi_i(x)\phi_j(x) \dx x\right).
\end{equation*}
Using the expression of $m_2(\nu_R)(\phi_i,\phi_j)$ derived in Lemma~\ref{lem formula m2}, this last term equals
\begin{equation*}
\sum_{\Pi \in \pp_p} \prod_{\{i,j\} \in \Pi} m_2(\nu_R)(\phi_i,\phi_j) = \sum_{\left\{ \{a_i,b_i\} \mvert 1 \leq i \leq \frac{p}{2}\right\} \in \pp_p} \prod_{i=1}^\frac{p}{2} m_2(\nu_R)(\phi_{a_i},\phi_{b_i}),
\end{equation*}
which proves the first part of Theorem~\ref{thm moments}.

Let $\phi:\R \to \R$ be a test-function in the sense of Definition~\ref{def test-function}. Using what we just proved with $\phi_i=\phi$ for all $i \in \{1,\dots,p\}$, we have:
\begin{equation*}
m_p(\prsc{\nu_R}{\phi}) = m_p(\nu_R)(\phi,\dots,\phi) = \card(\pp_p) \var{\prsc{\nu_R}{\phi}}^\frac{p}{2} + o(R^\frac{p}{2}).
\end{equation*}
Since $\card(\pp_p)=\mu_p$ (see Notation~\ref{ntn mu p}), this proves the expression of $m_p(\prsc{\nu_R}{\phi})$ stated in Theorem~\ref{thm moments}.
\end{proof}


\section{Limit theorems}
\label{sec limit theorems}

The purpose of this section is to deduce the functional limit Theorems~\ref{thm LLN} and~\ref{thm CLT} from Theorem~\ref{thm moments}. We prove the Law of Large Numbers (Theorem~\ref{thm LLN}) in Section~\ref{subsec proof thm LLN} and the Central Limit Theorem~\ref{thm CLT} in Section~\ref{subsec proof thm CLT}.


\subsection{Proof of Theorem~\ref{thm LLN}: Law of Large Numbers}
\label{subsec proof thm LLN}

This section is concerned with the proof of Theorem~\ref{thm LLN}. In the first part of the proof, we derive a Law of Large Numbers for the linear statistics $\prsc{\nu_R}{\phi}$, for a fixed test-function $\phi$. We use the moments estimates of Theorem~\ref{thm moments} and Markov's Inequality to obtain a quantitative convergence in probability. Then, we deduce the almost sure convergence from the Borel--Cantelli Lemma. In the second part of the proof, we prove a functional Law of Large Number for the random measures $\nu_R$. This uses the first part of Theorem~\ref{thm LLN}, together with the separability of the space $\mathcal{C}^0_c(\R)$ of continuous functions with compact support equipped with the sup-norm $\Norm{\cdot}_\infty$.

\begin{proof}[Proof of Theorem~\ref{thm LLN}]
Let $p \in \N^*$ and let $f$ be a normalized stationary centered Gaussian process of class~$\mathcal{C}^{2p}$. Let $\kappa$ denote the correlation function of $f$. We assumed that $\Norm{\kappa}_{2p,\eta} = o(\eta^{-8p})$ as $\eta \to +\infty$, where $\Norm{\cdot}_{2p,\eta}$ is defined as in Notation~\ref{ntn norm kappa}. Let $(R_n)_{n \in \N}$ be a sequence of positive numbers such that $\sum_{n \in \N} R_n^{-p}<+\infty$.

Let $\phi:\R \to \R$ be a test-function as in Definition~\ref{def test-function},we have:
\begin{equation*}
\esp{\sum_{n \in \N} \left(\frac{1}{R_n}\prsc{\nu_{R_n}}{\phi}-\frac{1}{\pi}\int_\R \phi(x) \dx x\right)^{2p}} = \sum_{n \in \N} \frac{1}{R_n^{2p}}m_{2p}(\prsc{\nu_{R_n}}{\phi}).
\end{equation*}
Using Proposition~\ref{prop variance} if $p=1$ and Theorem~\ref{thm moments} otherwise, we have $R_n^{-2p}m_{2p}(\prsc{\nu_{R_n}}{\phi}) = O(R_n^{-p})$. Thus, the sum on the right-hand side of the previous equation is finite. This proves that, almost surely, we have:
\begin{equation*}
\sum_{n \in \N}\left(\frac{1}{R_n}\prsc{\nu_{R_n}}{\phi}-\frac{1}{\pi}\int_\R \phi(x) \dx x\right)^{2p} < +\infty,
\end{equation*}
hence $\displaystyle\frac{1}{R_n}\prsc{\nu_{R_n}}{\phi} \xrightarrow[n \to +\infty]{}\frac{1}{\pi}\int_\R \phi(x) \dx x$. This proves our first claim. The almost sure convergence of $\frac{1}{R_n}\card(Z \cap [0,R_n])$ is obtained by applying this result with $\phi=\mathbf{1}_{[0,1]}$.

Recall that the space $\mathcal{C}^0_c(\R)$ of continuous functions with compact support is separable for the topology induced by $\Norm{\cdot}_\infty$. Let $(\phi_k)_{k \in \N^*}$ denote a dense sequence in $\left(\mathcal{C}^0_c(\R),\Norm{\cdot}_\infty\right)$. We also denote by $\phi_0 = \mathbf{1}_\R$. For any $N \in \N$, let $\chi_N:\R \to \R$ denote the even continuous function defined by:
\begin{equation*}
\chi_N: x \longmapsto \left\{ \begin{aligned}
& 1 & &\text{if} \ \norm{x} \leq N,\\
& 1-(\norm{x}-N) & &\text{if} \ N \leq \norm{x} \leq N+1,\\
& 0 & &\text{if} \ \norm{x} \geq N+1.
\end{aligned}\right.
\end{equation*}
Note that $\phi_0 \notin \mathcal{C}^0_c(\R)$ but that $\chi_N\phi_0 = \chi_N \in \mathcal{C}^0_c(\R)$ for all $N \in \N$.

Using the first part of Theorem~\ref{thm LLN} proved above and the countability of $\N^2$, the following happens almost surely:
\begin{equation}
\label{eq as CV countable}
\forall (k,N) \in \N^2, \qquad \frac{1}{R_n}\prsc{\nu_{R_n}}{\chi_N\phi_k} \xrightarrow[n \to +\infty]{} \frac{1}{\pi} \int_\R \chi_N(x)\phi_k(x) \dx x. 
\end{equation}
In the following, we consider a realization of the random process $f$ such that~\eqref{eq as CV countable} holds. For this realization, we will prove by an approximation argument that:
\begin{equation}
\label{eq goal LLN}
\forall \phi \in \mathcal{C}^0_c(\R), \qquad \frac{1}{R_n}\prsc{\nu_{R_n}}{\phi} \xrightarrow[n \to +\infty]{}\frac{1}{\pi}\int_\R \phi(x) \dx x,
\end{equation}
i.e.~that $\displaystyle \frac{1}{R_n}\nu_{R_n} \xrightarrow[n \to +\infty]{} \frac{1}{\pi}\dx x$ in the weak-$*$ sense in the topological dual of $\left(\mathcal{C}^0_c(\R),\Norm{\cdot}_\infty\right)$. This yields the result.

Let us consider a realization of $f$ such that~\eqref{eq as CV countable} holds. Let $\phi \in \mathcal{C}^0_c(\R)$ and let $N \in \N$ be large enough that the support of $\phi$ is included in $[-N,N]$. By Equation~\eqref{eq as CV countable} with $k=0$, the non-negative sequence $\left(R_n^{-1}\prsc{\nu_{R_n}}{\chi_N}\right)_{n \in \N}$ converges towards $\frac{2N+1}{\pi}$. In particular, this sequence is bounded by some constant $C_N > 0$.

Let $\epsilon >0$ and let $k \in \N^*$ be such that $\Norm{\phi - \phi_k}_\infty \leq \epsilon$. For all $x \in \R$, we have:
\begin{equation}
\label{eq test function LLN}
\norm{\phi(x)-\chi_N(x)\phi_k(x)} = \norm{\chi_N(x)\left(\phi(x)-\phi_k(x)\right)} = \chi_N(x)\norm{\phi(x)-\phi_k(x)} \leq \epsilon \chi_N(x).
\end{equation}
Then, for all $n \in \N$, we have:
\begin{multline}
\label{eq 3 terms LLN}
\norm{\frac{1}{R_n}\prsc{\nu_{R_n}}{\phi}-\frac{1}{\pi} \int_\R \phi(x) \dx x} \leq \frac{1}{R_n}\norm{\prsc{\nu_{R_n}}{\phi-\chi_N \phi_k}} + \norm{\frac{1}{R_n}\prsc{\nu_{R_n}}{\chi_N\phi_k}-\frac{1}{\pi} \int_\R \chi_N\phi_k(x) \dx x}\\
+ \frac{1}{\pi} \norm{\int_\R \phi(x) - \chi_N(x) \phi_k(x) \dx x}.
\end{multline}
Using Equation~\eqref{eq test function LLN}, the first term on the right-hand side of Equation~\eqref{eq 3 terms LLN} satisfies:
\begin{equation*}
\frac{1}{R_n}\norm{\prsc{\nu_{R_n}}{\phi-\chi_N \phi_k}} \leq \frac{1}{R_n}\prsc{\nu_{R_n}}{\norm{\phi - \chi_N\phi_k}} \leq \epsilon \frac{1}{R_n}\prsc{\nu_{R_n}}{\chi_N} \leq \epsilon C_N.
\end{equation*}
Similarly, the third term on the right-hand side of Equation~\eqref{eq 3 terms LLN} satisfies:
\begin{equation*}
\frac{1}{\pi}\norm{\int_\R \phi(x)-\chi_N(x)\phi_k(x) \dx x} \leq \frac{1}{\pi} \int_\R \norm{\phi(x)-\chi_N(x)\phi_k(x)} \dx x \leq \frac{\epsilon}{\pi} \int_\R \chi_N(x) \dx x \leq \epsilon\frac{2N+1}{\pi}.
\end{equation*}
Using our hypothesis that~\eqref{eq as CV countable} holds for $(k,N)$ the middle term on the right-hand side of Equation~\eqref{eq 3 terms LLN} goes to $0$ as $n \to +\infty$. Finally, for all $n$ large enough we have:
\begin{equation*}
\norm{\frac{1}{R_n}\prsc{\nu_{R_n}}{\phi}-\frac{1}{\pi} \int_\R \phi(x) \dx x} \leq \epsilon \left(C_N +1 + \frac{2N+1}{\pi}\right).
\end{equation*}
This proves that Equation~\eqref{eq goal LLN} holds for $\phi$, hence for all $\phi \in \mathcal{C}^0_c(\R)$, as claimed. Thus~\eqref{eq goal LLN} holds almost surely, which concludes the proof.
\end{proof}


\subsection{Proof of Theorem~\ref{thm CLT}: Central Limit Theorem}
\label{subsec proof thm CLT}

In this section we deduce Theorem~\ref{thm CLT} from the moments estimates of Theorem~\ref{thm moments}. The Central Limit Theorem for a fixed test-function follows from Theorem~\ref{thm moments} by the method of moments, see~\cite[Chapter~30]{Bil1995}. Then, we obtain the functional Central Limit Theorem by the Lévy--Fernique Theorem, which is a generalization of Lévy's Continuity Theorem for random generalized functions. The result of Fernique~\cite[Theorem~III.6.5]{Fer1967} is not exactly what we need, and we refer to~\cite{BDW2018} instead, for a version of this result that suits us better.

\begin{proof}[Proof of Theorem~\ref{thm CLT}]
Let $f$ be a $\mathcal{C}^\infty$ centered Gaussian process which is stationary and normalized. Let $\kappa$ denote its correlation function, we assume that $\kappa \in \mathcal{S}(\R)$, see Definition~\ref{def SR}. Let $\sigma$ be the positive constant appearing in the variance asymptotics of Proposition~\ref{prop variance}, which is defined by Equation~\eqref{eq def sigma}.

Let $R>0$, we denote by $\nu_R$ the counting measure of $\{x \in \R \mid f(Rx)=0\}$. Let $T_R$ be the random measure on $\R$ defined by:
\begin{equation*}
T_R = \frac{1}{R^\frac{1}{2}\sigma}\left(\nu_R - \frac{R}{\pi} \dx x\right),
\end{equation*}
where $\dx x$ stands for the Lebesgue measure of $\R$. Let us consider a test-function $\phi:\R \to \R$, see Definition~\ref{def test-function}. By Proposition~\ref{prop expectation} and Remark~\ref{rem as well defined},
\begin{equation*}
\prsc{T_R}{\phi} = \frac{1}{R^\frac{1}{2}\sigma}\left(\prsc{\nu_R}{\phi} - \frac{R}{\pi} \int_\R \phi(x) \dx x\right)
\end{equation*}
is an almost-surely well-defined centered random variable. Then, by Theorem~\ref{thm moments} and Proposition~\ref{prop variance}, for any integer $p \geq 2$, we have:
\begin{equation*}
\esp{\prsc{T_R}{\phi}^p} \xrightarrow[R \to +\infty]{} \mu_p \Norm{\phi}^p_{L^2},
\end{equation*}
where $\mu_p$ is the $p$-th moment of an $\mathcal{N}(0,1)$ random variable, as in Notation~\ref{ntn mu p}. Hence, we have:
\begin{equation}
\label{eq CLT}
\prsc{T_R}{\phi} \xrightarrow[R \to +\infty]{} \mathcal{N}\left(0,\Norm{\phi}_{L^2}^2\right)
\end{equation}
in distribution, by~\cite[Theorem~30.2]{Bil1995}. Note that if $\Norm{\phi}_{L^2}=0$ then $\phi$ vanishes almost everywhere. In this case $\prsc{T_R}{\phi}=0$ almost surely, by Proposition~\ref{prop expectation} and Remark~\ref{rem as well defined}. This proves the first claim in Theorem~\ref{thm CLT}. The Central Limit Theorem follows since $\card(Z \cap [0,R]) = \prsc{\nu_R}{\mathbf{1}_{[0,1]}}$.

Let $R>0$, Lemma~\ref{lem nu R as tempered distribution} shows that $T_R$ is an almost surely well-defined random element of $\mathcal{S}'(\R)$. For all $\phi \in \mathcal{S}(\R)$, the convergence in distribution of Equation~\eqref{eq CLT} and the definition of the standard Gaussian White Noise (cf.~Definition~\ref{def white noise}) show that:
\begin{equation*}
\prsc{T_R}{\phi} \xrightarrow[R \to +\infty]{} \prsc{W}{\phi}.
\end{equation*}
By~\cite[Corollary~2.4]{BDW2018}, we have $T_R \xrightarrow[R \to +\infty]{}W$ in distribution in $\mathcal{S}'(\R)$. This concludes the proof.
\end{proof}


\appendix

\section{Examples of smooth non-degenerate processes}
\label{sec examples of smooth non-degenerate processes}

In this section, we build examples of Gaussian processes satisfying the hypotheses of Theorems~\ref{thm moments}, \ref{thm vanishing order} and~\ref{thm clustering}. First, we need to recall the definition of the spectral measure of a stationary Gaussian process and some of its properties. This is done in Section~\ref{subsec spectral measure}. In Section~\ref{subsec non-degeneracy spectral measure and ergodicity}, we give a non-degeneracy criterion on the spectral measure for a $\mathcal{C}^p$ Gaussian process. Finally, in Section~\ref{subsec smooth non-degenerate processes with fast-decreasing correlations}, we give examples of Gaussian processes whose correlation functions lie in the Schwartz space $\mathcal{S}(\R)$ of smooth fast-decreasing functions.


\subsection{Spectral measure}
\label{subsec spectral measure}

This section is concerned with the definition and the properties of the spectral measure of a stationary Gaussian process. Let $f:\R \to \R$ be a non-zero stationary centered Gaussian process of class $\mathcal{C}^0$ and let $\kappa$ denote its correlation function. We assume that $f$ is normalized so that $\var{f(0)}=1$. Then, $\kappa:\R \to \R$ is a continuous function such that $\kappa(0)=1$. Moreover, $\kappa$ is \emph{positive semi-definite}, in the sense that, for all $m \in \N^*$, for all $x_1,\dots,x_m \in \R$, for all $a_1,\dots,a_m \in \C$, we have:
\begin{equation}
\label{eq positive semi-definite}
\sum_{1 \leq j,k \leq m} a_j \bar{a_k} \kappa(x_k-x_j) = \esp{\norm{\sum_{j=1}^m a_j f(x_j)}^2} \geq 0.
\end{equation}
By Bochner's Theorem, there exists a unique Borel probability measure $\lambda$ on $\R$ such that $\kappa$ is the characteristic function of $\lambda$, i.e.:
\begin{equation}
\label{eq spectral measure}
\forall x \in \R, \quad \kappa(x) = \int_{-\infty}^{+\infty} e^{ix\xi}\dx \lambda(\xi).
\end{equation}
Since $\kappa$ is real-valued, $\lambda$ is \emph{symmetric}. That is $(-\Id)_*\lambda = \lambda$, where $\Id$ is the identity of~$\R$.

\begin{dfn}[Spectral measure]
\label{def spectral measure}
Let $f$ be a stationary centered Gaussian process of class $\mathcal{C}^0$, normalized so that its correlation function $\kappa$ satisfies $\kappa(0)=1$. The unique symmetric Borel probability measure $\lambda$ such that~\eqref{eq spectral measure} holds is called the \emph{spectral measure} of $f$.
\end{dfn}

Conversely, let $\lambda$ be a symmetric Borel probability measure on $\R$ and let $\kappa$ denote its characteristic function. Then, $\kappa$ is a continuous real-valued function such that $\kappa(0)=1$, and $\kappa$ is positive semi-definite (cf.~Equation~\eqref{eq positive semi-definite}). By a theorem of Kolmogorov (see~\cite[p.~19]{AW2009}, for example), there exists a stationary centered Gaussian process $f$ whose correlation function equals $\kappa$.

Let us now relate the properties of the process $f$, its correlation function $\kappa$ and its spectral measure $\lambda$. As explained in Section~\ref{subsec stationary Gaussian processes and correlation functions}, if $f$ is of class $\mathcal{C}^p$ for some $p \in \N^*$, then $\kappa$ is of class $\mathcal{C}^{2p}$. In this case, $\lambda$ admits a finite moment of order $2p$, that is $\int_{\R} \xi^{2p} \dx \lambda(\xi) < +\infty$. Conversely, if $\lambda$ admits a finite moment of order $2p$ then $\kappa$ is $\mathcal{C}^{2p}$. If these conditions are satisfied then, for all $j \in \{0,\dots,2p\}$, for all $x \in \R$, we have:
\begin{equation}
\label{eq derivatives kappa}
\kappa^{(j)}(x) = \int_{-\infty}^{+\infty} (i\xi)^j e^{ix\xi}\dx \lambda(\xi).
\end{equation}
The fact that $\kappa$ is of class $\mathcal{C}^{2p}$ is not enough to ensure that $f$ is a $\mathcal{C}^p$-process. However, by Kolmogorov's Theorem~\cite[Appendix~A.9]{NS2016}, the process $f$ is of class~$\mathcal{C}^{p-1}$, in the sense that there exists a version of $f$ which is of class $\mathcal{C}^{p-1}$. We can say a bit more about the regularity of $f$:
\begin{itemize}
\item for all $\alpha \in (0,1)$, the process $f^{(p-1)}$ is almost surely $\alpha$-Hölder (see~\cite[Appendix~A.11.2]{NS2016});
\item for all $x \in \R$, the variable $f^{(2p)}(x)$ is well-defined and Gaussian (cf.~\cite[Proposition~1.13]{AW2009});
\item if there exists $\alpha >0$ such that $\kappa^{(2p)}(0)-\kappa^{(2p)}(x) = O(\norm{x}^\alpha)$ as $x \to 0$, then there exists a version of $f$ of class~$\mathcal{C}^p$ (cf.~\cite[Corollary~1.7.b]{AW2009}).
\end{itemize}


\subsection{Non-degeneracy, spectral measure and ergodicity}
\label{subsec non-degeneracy spectral measure and ergodicity}

In this section, we give a condition on the spectral measure of a stationary Gaussian process implying that the finite-dimensional marginal distributions of this process are non-degenerate. A similar criterion already appeared in~\cite[p.~64]{AW2009}. Then, we use this result to prove Lemma~\ref{lem non-degeneracy}.

\begin{lem}[Non-degeneracy]
\label{lem spectral measure and non-degeneracy}
Let $f:\R \to \R$ be a stationary Gaussian process of class $\mathcal{C}^p$. Let $\lambda$ denote its spectral measure. If the support of $\lambda$ has an accumulation point in $\R$ then: for all $m \in \N^*$, for any $k_1,\dots,k_m \in \{0,\dots,p\}$ and any $x_1,\dots,x_m \in \R$ such that the couples $((k_j,x_j))_{1 \leq j \leq m}$ are pairwise distinct, the centered Gaussian vector $\left(f^{(k_j)}(x_j)\right)_{1 \leq j \leq m}$ is non-degenerate.
\end{lem}

\begin{proof}
Let us assume that the support of $\lambda$ has an accumulation point in $\R$. Let $m \in \N^*$, let $k_1,\dots,k_m \in \{0,\dots,p\}$ and let $x_1,\dots,x_m \in \R$ be such that the couples $((k_j,x_j))_{1 \leq j \leq m}$ are pairwise distinct. The Gaussian vector $\left(f^{(k_j)}(x_j)\right)_{1 \leq j \leq m}$ is non-degenerate if and only if its variance matrix $\Lambda = \left((-1)^{k_j}\kappa^{(k_j+k_l)}(x_l-x_j)\right)_{1 \leq i,j \leq m}$ is non-singular, that is, if and only if $\ker(\Lambda)=\{0\}$.

Let $a = \trans{(a_1,\dots,a_m)} \in \ker(\Lambda)$, by Equations~\eqref{eq covariance} and~\eqref{eq derivatives kappa}, we have:
\begin{align*}
0 = \trans{a}\Lambda a &= \sum_{1 \leq j,l \leq m} a_ja_l(-1)^{k_j}\kappa^{(k_j+k_l)}(x_l-x_j)\\
&= \sum_{1 \leq j,l \leq m} a_j a_l \int_{-\infty}^{+\infty} (i\xi)^{k_l}(-i\xi)^{k_j} e^{i(x_l-x_j)\xi} \dx\lambda(\xi)\\
&= \int_{-\infty}^{+\infty} \norm{\sum_{j=1}^m a_j(-i\xi)^{k_j} e^{-ix_j\xi}}^2 \dx \lambda(\xi).
\end{align*}
Hence $g_a:\xi \mapsto \sum_{j=1}^m a_j (-i\xi)^{k_j} e^{-ix_j\xi}$ vanishes on the support of $\lambda$. Since $g_a$ is analytic and the support of $\lambda$ has an accumulation point, we have $g_a=0$. Besides, $g_a$ is the Fourier transform, in the sense of tempered generalized functions, of $\sum_{j=1}^m a_j \left(\delta_{x_j}\right)^{(k_j)}$, where $\delta_x$ stands for the unit Dirac mass at $x \in \R$. Since the Fourier transform is an isomorphism from $\mathcal{S}'(\R)$ to itself, we have $\sum_{j=1}^m a_j \left(\delta_{x_j}\right)^{(k_j)}=0$. The couples $((k_j,x_j))_{1 \leq j \leq m}$ being pairwise distinct, this implies that $a=0$. Thus $\ker(\Lambda) = \{0\}$ and $\left(f^{(k_j)}(x_j)\right)_{1 \leq j \leq m}$ is non-degenerate.
\end{proof}

We can now prove Lemma~\ref{lem non-degeneracy} as a corollary of Lemma~\ref{lem spectral measure and non-degeneracy}.

\begin{proof}[Proof of Lemma~\ref{lem non-degeneracy}]
Since $\kappa(x) \xrightarrow[x \to +\infty]{}0$, the random process $f$ is ergodic, cf.~\cite[Theorem~6.5.4]{Adl2010}. By the Fomin--Grenander--Maruyama Theorem (see~\cite[Section~6]{NS2016}), this is equivalent to the fact that the spectral measure $\lambda$ of $f$ has no atom. Note that, in~\cite{NS2016}, the authors only state one implication in the body of the text, which is not the one we are interested in. The fact that equivalence holds appears as a footnote. Then, any point in the support of $\lambda$ must be an accumulation point. The conclusion follows by Lemma~\ref{lem spectral measure and non-degeneracy}.
\end{proof}


\subsection{Smooth non-degenerate processes with fast-decreasing correlations}
\label{subsec smooth non-degenerate processes with fast-decreasing correlations}

In this section, we build examples of normalized stationary centered Gaussian $\mathcal{C}^p$-processes whose correlation functions, as well as their derivatives, decay as $O(x^{-k})$ at infinity.

\begin{lem}
\label{lem spectral density and decay}
Let $\lambda$ be a probability measure on $\R$ admitting a density $g$ with respect to the Lebesgue measure. Let $k \in \N$ and $p \in \N^*$, we assume that $g$ is even, of class $\mathcal{C}^k$, and satisfies the following conditions:
\begin{enumerate}
\item $\int_{\R} \xi^2 g(\xi) \dx \xi =1$;
\label{cond 1 lem spectral density and decay}
\item $\int_{\R} \xi^{2p+2} g(\xi) \dx \xi <+\infty$;
\label{cond 2 lem spectral density and decay}
\item for all $j \in \{0,\dots,k\}$, $\int_{\R} \norm{\xi}^p \norm{g^{(j)}(\xi)} \dx \xi <+\infty$;
\label{cond 3 lem spectral density and decay}
\item for all $j \in \{0,\dots,k-1\}$, $\norm{g^{(j)}(\xi)} = o(\norm{\xi}^{-p})$ as $\norm{\xi} \to +\infty$.
\label{cond 4 lem spectral density and decay}
\end{enumerate}
Then $\lambda$ is the spectral measure of normalized stationary centered Gaussian process $f$ of class $\mathcal{C}^p$. Moreover, denoting by~$\kappa$ the correlation function of $f$, we have $\kappa^{(j)}(x) = o(\norm{x}^{-k})$ as $\norm{x} \to +\infty$, for all $j \in \{0,\dots,p\}$.
\end{lem}

\begin{proof}
Since $g$ is even and continuous, $\lambda$ is a symmetric Borel probability measure. Let $\kappa$ denote the characteristic function of $\lambda$. By Condition~\ref{cond 2 lem spectral density and decay}, the function $\kappa$ is of class $\mathcal{C}^{2p+2}$. As discussed in Section~\ref{subsec spectral measure}, this means that $\kappa$ is the correlation function of a stationary centered Gaussian process~$f$, at least of class $\mathcal{C}^p$. Since $\lambda$ is a probability measure, we have $\kappa(0)=1$. Moreover, by Condition~\ref{cond 1 lem spectral density and decay}, we have $\kappa''(0)=-1$. Thus $f$ is normalized (see Definition~\ref{def normalized process}).

For any $j \in \{0,\dots,p\}$, an expression of $\kappa^{(j)}$ is given by Equation~\eqref{eq derivatives kappa}. Using Conditions~\ref{cond 3 lem spectral density and decay} and~\ref{cond 4 lem spectral density and decay}, we integrate by parts $k$ times, and obtain for any $x \neq 0$:
\begin{equation*}
\kappa^{(j)}(x) = \frac{i^{j+k}}{x^k} \sum_{m=0}^{\min(j,k)}\binom{k}{m}\frac{j!}{(j-m)!} \int_{-\infty}^{+\infty}e^{ix\xi}\xi^{j-m}g^{(k-m)}(\xi) \dx \xi.
\end{equation*}
By Condition~\ref{cond 3 lem spectral density and decay}, the function $\xi \mapsto \xi^{j-m}g^{(k-m)}(\xi)$ is integrable for all $m \leq \min(j,k)$. Hence, by the Riemann--Lebesgue Lemma, the integrals in the previous expression tend to~$0$ as $\norm{x} \to +\infty$, which shows that $\kappa^{(j)}(x) = o(\norm{x}^{-k})$.
\end{proof}

Lemmas~\ref{lem spectral measure and non-degeneracy} and~\ref{lem spectral density and decay} allow to build examples of Gaussian processes satisfying the hypotheses of our main results, see Theorems~\ref{thm moments}, \ref{thm vanishing order} and~\ref{thm clustering}. Let us conclude this section by giving a few of them. Recall that the Schwartz space $\mathcal{S}(\R)$ is defined by Definition~\ref{def SR}.

\begin{exs}
\label{ex non-degenerate and fast-decreasing}
Let $g$ be a non-negative continuous even function such that $\int_{\R} \xi^4 g(\xi) \dx \xi < +\infty$ and $\int_{\R} g(\xi) \dx \xi = 1=\int_{\R} \xi^2 g(\xi) \dx \xi$. Let $\lambda$ denote the measure having the density $g$ with respect to the Lebesgue measure. As explained above, $\lambda$ is the spectral measure of a normalized stationary centered Gaussian $\mathcal{C}^1$-process $f$, whose correlation function is denoted by $\kappa$.
\begin{itemize}
\item If $g$ is $\mathcal{C}^1$ and we have $g(\xi) = O(\xi^{-8})$ and $g'(\xi)=O(\xi^{-4})$ as $\xi \to +\infty$, then $f$ is $\mathcal{C}^2$ and $\kappa(x)$, $\kappa'(x)$ and $\kappa''(x)$ are $o(\norm{x}^{-1})$ as $x \to +\infty$. In particular, $f$ satisfies the hypotheses of Proposition~\ref{prop variance}.
\item Let $p \in \N^*$, if $g(\xi) = O(\xi^{-2p-4})$ as $\xi \to +\infty$ then $f$ is of class $\mathcal{C}^p$ and $\kappa^{(j)}(x) \xrightarrow[x \to +\infty]{}0$ for all $j \in \{0,\dots,p\}$. In particular, $f$ satisfies the hypotheses of Theorems~\ref{thm vanishing order} and~\ref{thm clustering}.
\item If $g \in \mathcal{S}(\R)$, then $f$ is $\mathcal{C}^\infty$ and $\kappa \in \mathcal{S}(\R)$. Hence $f$ satisfies the hypotheses of Proposition~\ref{prop variance} and Theorems~\ref{thm moments}, \ref{thm vanishing order} and~\ref{thm clustering}.
\item If $g : \xi \mapsto \frac{1}{\sqrt{2\pi}}\exp\left(-\frac{1}{2}\xi^2\right)$ is the standard Gaussian density, then $\kappa:x \mapsto \exp\left(-\frac{1}{2}x^2\right)$ and $f$ is the Bargmann--Fock process discussed in Section~\ref{subsec moments asymptotics}.
\item If $g=\frac{1}{2\sqrt{3}}\mathbf{1}_{[-\sqrt{3},\sqrt{3}]}$ then $\kappa : x \mapsto \sinc(\sqrt{3}x)$, where $\sinc:x \mapsto \sum_{k \geq 0} \frac{(-1)^k x^{2k}}{(2k+1)!}$ is the smooth extension of $x \mapsto \frac{1}{x}\sin(x)$ to $\R$. The density $g$ is not regular enough to apply Lemma~\ref{lem spectral density and decay}, but $\kappa$ is still the correlation function of a normalized stationary centered Gaussian process $f$ of class $\mathcal{C}^\infty$ (cf.~Section~\ref{subsec spectral measure}). By Lemma~\ref{lem spectral measure and non-degeneracy}, $f$ satisfies the hypotheses of Theorem~\ref{thm vanishing order}. Moreover, one can check that $\kappa$ and all its derivatives are square-integrable on $\R$ and tend to~$0$ at infinity, hence $f$ satisfies the hypotheses of Proposition~\ref{prop variance} and Theorem~\ref{thm clustering} for any $p \in \N^*$.
\item If $g: \xi \mapsto \frac{1}{\sqrt{2}}\exp\left(-\sqrt{2}\norm{\xi}\right)$, then $\kappa:x \mapsto \left(1 + \frac{x^2}{2}\right)^{-1}$ (one can check this using his favorite computational software). As above $g$ is not regular enough to apply Lemma~\ref{lem spectral density and decay} directly. Yet, $\kappa$ is the correlation function of a normalized stationary centered Gaussian $\mathcal{C}^\infty$-process satisfying the hypotheses of Proposition~\ref{prop variance} and Theorems~\ref{thm vanishing order} and~\ref{thm clustering} for any $p \in \N^*$.
\end{itemize}
\end{exs}

\begin{rem}
\label{rem normalization g}
Let $h$ be a non-negative even function. If we have $\int_\R h(\xi) \dx \xi = A >0$ and $\int_\R \xi^2 h(\xi) \dx \xi =B$, then $g: \xi \mapsto \sqrt{\frac{B}{A^3}} h\left(\sqrt{\frac{B}{A}}\xi\right)$ is such that $\int_R g(\xi) \dx \xi = 1=\int_\R \xi^2 g(\xi) \dx \xi $. Moreover, $g$ is non-negative, even and it has the same regularity and asymptotic behavior as~$h$.
\end{rem}


\section{Properties of the density function \texorpdfstring{$F$}{}}
\label{sec properties of the density function F}


\subsection{Proof of Lemma~\ref{lem expression F}}
\label{subsec proof of Lemma expression F}

The goal of this section is to prove Lemma~\ref{lem expression F}, which gives an expression of the function $F$ defined by Definition~\ref{def F}. Recall that in Lemma~\ref{lem expression F} we work under the hypotheses of Proposition~\ref{prop variance}. In particular, $f$ is a $\mathcal{C}^2$ Gaussian process whose correlation function $\kappa$ tends to $0$ at infinity.

\begin{lem}
\label{lem expression N20z}
For all $z >0$ we have $N_2(0,z) = \dfrac{2}{\pi}b(z) \left(\sqrt{1-a(z)^2} + a(z) \arcsin(a(z))\right)$, where
\begin{align*}
a(z) &= \frac{\kappa(z)\kappa'(z)^2 - \kappa(z)^2\kappa''(z) + \kappa''(z)}{1 - \kappa(z)^2 -\kappa'(z)^2} & &\text{and} & b(z) &= 1 - \frac{\kappa'(z)^2}{1-\kappa(z)^2}.
\end{align*}
Moreover, $\norm{a(z)} \leq 1$.
\end{lem}

\begin{proof}
Let $z >0$. The random vector $(f(0),f(z),f'(0),f'(z))$ is a centered Gaussian in $\R^4$ whose variance matrix equals:
\begin{equation*}
\begin{pmatrix}
1 & \kappa(z) & 0 & \kappa'(z) \\
\kappa(z) & 1 & -\kappa'(z) & 0 \\
0 & -\kappa'(z) & 1 & -\kappa''(z)\\
\kappa'(z) & 0 & -\kappa''(z) & 1
\end{pmatrix}.
\end{equation*}
Since $\kappa$ tends to $0$ at infinity, by Lemma~\ref{lem non-degeneracy} the Gaussian vector $(f(0),f(z))$ is non-degenerate, i.e.~$\kappa(z)^2 < 1$. Then, by~\cite[Proposition~1.2]{AW2009}, $(f'(0),f'(z))$ given that $f(0)=0=f(z)$ is a centered Gaussian vector in~$\R^2$. Moreover, its variance matrix is:
\begin{align*}
\Lambda(z) &= \begin{pmatrix}
1 & -\kappa''(z) \\ -\kappa''(z) & 1
\end{pmatrix} - \begin{pmatrix}
0 & -\kappa'(z) \\ \kappa(z) & 0
\end{pmatrix} \begin{pmatrix}
1 & \kappa(z) \\ \kappa(z) & 1
\end{pmatrix}^{-1} \begin{pmatrix}
0 & \kappa'(z) \\ -\kappa'(z) & 0
\end{pmatrix}\\
&= \begin{pmatrix}
b(z) & c(z) \\ c(z) & b(z)
\end{pmatrix},
\end{align*}
where
\begin{align*}
b(z) &= 1 - \frac{\kappa'(z)^2}{1-\kappa(z)^2} & &\text{and} & c(z) &= -\kappa''(z) - \frac{\kappa(z)\kappa'(z)^2}{1-\kappa(z)^2}.
\end{align*}
Note that, since $\Lambda(z)$ is a variance matrix, we have $b(z) \geq \norm{c(z)} \geq 0$.

Let $(X,Y) \sim \mathcal{N}(0,\Lambda(z))$ in $\R^2$, we have $N_2(0,z) = \esp{\norm{X}\norm{Y}}$. If $b(z) > \norm{c(z)}$, then $\Lambda(z)$ is invertible and $\Lambda^{-1}(z) = \left(\begin{smallmatrix} A(z) & B(z) \\ B(z) & A(z) \end{smallmatrix} \right)$,
where
\begin{align*}
A(z) &= \frac{b(z)}{b(z)^2 - c(z)^2} & &\text{and} & B(z) &= -\frac{c(z)}{b(z)^2 - c(z)^2}.
\end{align*}
Then, using~\cite[Equation~(A.1)]{BD1997}, we have:
\begin{align*}
\esp{\norm{X}\norm{Y}} &= \frac{1}{2\pi\sqrt{b(z)^2 - c(z)^2}} \int_{\R^2} \norm{x}\norm{y} \exp\left(-A(z)\frac{x^2 + y^2}{2}-B(z)xy\right)\dx x \dx y \\
&= \frac{2}{\pi}b(z) \left(\sqrt{1 - \frac{c(z)^2}{b(z)^2}} + \frac{c(z)}{b(z)} \arcsin\left(\frac{c(z)}{b(z)}\right)\right)
\end{align*}
On the other hand, if $b(z) = \norm{c(z)}$ then $\norm{X} = \norm{Y}$ almost surely. Hence,
\begin{equation*}
\esp{\norm{X}\norm{Y}} = \esp{X^2} = b(z) = \frac{2}{\pi}b(z) \left(\sqrt{1 - \frac{c(z)^2}{b(z)^2}} + \frac{c(z)}{b(z)} \arcsin\left(\frac{c(z)}{b(z)}\right)\right).
\end{equation*}

To conclude, note that $a(z) = -\frac{c(z)}{b(z)}$, so that $\norm{a(z)} \leq 1$ and:
\begin{equation*}
N_2(0,z) = \esp{\norm{X}\norm{Y}} = \frac{2}{\pi}b(z) \left(\sqrt{1-a(z)^2} + a(z) \arcsin(a(z))\right).\qedhere
\end{equation*}
\end{proof}

\begin{proof}[Proof of Lemma~\ref{lem expression F}]
By definition of $F$ and $\rho_2$ (see Definitions~\ref{def F} and~\ref{def Kac-Rice densities}), for all $z \neq 0$, we have:
\begin{equation*}
F(z) = \frac{1}{2\pi} \frac{N_2(0,z)}{D_2(0,z)^\frac{1}{2}} - \frac{1}{\pi^2}.
\end{equation*}
Note that, $N_2$ and $D_2$ are symmetric functions on $\R^2 \setminus \Delta_2$. Then,
using the stationarity of $f$, we have $N_2(0,z) = N_2(z,0) = N_2(0,-z)$, and similarly $D_2(0,z) = D_2(0,-z)$. Thus $F(z) = F(-z)$ for all $z \neq 0$.

Let $z >0$, we have $D_2(0,z) =1 - \kappa(z)^2$ (see Example~\ref{ex Kac-Rice densities}) and the expression of $N_2(0,z)$ is given by Lemma~\ref{lem expression N20z}. Then, a direct computation yields:
\begin{equation*}
F(z) = \frac{1}{\pi^2} \left(\frac{1-\kappa(z)^2-\kappa'(z)^2}{\left(1-\kappa(z)^2\right)^\frac{3}{2}}\left(\sqrt{1-a(z)^2} + a(z) \arcsin(a(z))\right) - 1\right),
\end{equation*}
where $a(z)$ is defined as in Lemmas~\ref{lem expression F} and~\ref{lem expression N20z}. In particular, $\norm{a(z)} \leq 1$.
\end{proof}


\subsection{Proof of Lemma~\ref{lem integrability F}}
\label{subsec proof of Lemma integrability F}

In this section, we prove the integrability of the function $F$ defined by Definition~\ref{def F}, under the hypotheses of Proposition~\ref{prop variance}.

\begin{lem}
\label{lem integrability F 0}
Under the hypotheses of Proposition~\ref{prop variance}, we have $F(z) \xrightarrow[z \to 0]{} -\frac{1}{\pi^2}$.
\end{lem}

\begin{proof}
Let us consider the expression of $F$ derived in Lemma~\ref{lem expression F}. Note that, for all $z >0$, we have:
\begin{equation*}
0 \leq \sqrt{1-a(z)^2} + a(z) \arcsin(a(z)) \leq 1 + \frac{\pi}{2}.
\end{equation*}
Hence, it is enough to prove that:
\begin{equation*}
\frac{1-\kappa(z)^2 - \kappa'(z)^2}{\left(1-\kappa(z)^2\right)^\frac{3}{2}} \xrightarrow[z \to 0]{} 0.
\end{equation*}
We know that $\kappa$ is $\mathcal{C}^4$. Moreover, $\kappa(0)=1=-\kappa''(0)$ and $\kappa'(0) = 0 = \kappa^{(3)}(0)$. Thus, as $z \to 0$, we have:
\begin{align*}
\kappa(z) &= 1 - \frac{z^2}{2} + O(z^4) & &\text{and} & \kappa'(z) &= -z +O(z^3).
\end{align*}
These estimates yield that $\dfrac{1-\kappa(z)^2 - \kappa'(z)^2}{\left(1-\kappa(z)^2\right)^\frac{3}{2}} = O(z)$ as $z \to 0$, which concludes the proof.
\end{proof}

\begin{lem}
\label{lem integrability F infty}
Under the hypotheses of Proposition~\ref{prop variance}, as $z \to +\infty$, we have:
\begin{equation*}
F(z) = O(\kappa(z)^2+\kappa'(z)^2+\kappa''(z)^2).
\end{equation*}
\end{lem}

\begin{proof}
Once again, we use the expression of $F$ derived in Lemma~\ref{lem expression F}. First, note that since $\kappa(z)$, $\kappa'(z)$ and $\kappa''(z)$ tend to $0$ as $z\to +\infty$, we have: $a(z) = O(\kappa(z)) + O(\kappa''(z))$ as $z \to +\infty$. Then,
\begin{equation*}
\sqrt{1-a(z)^2} + a(z) \arcsin(a(z)) = 1 + O(a(z)^2) = 1 + O(\kappa(z)^2) + O(\kappa''(z)^2),
\end{equation*}
as $z \to +\infty$. On the other hand, as $z \to +\infty$, we have:
\begin{equation*}
\frac{1-\kappa(z)^2 - \kappa'(z)^2}{\left(1-\kappa(z)^2\right)^\frac{3}{2}} = 1 + O(\kappa(z)^2) + O(\kappa'(z)^2).
\end{equation*}
These two estimates yield that $F(z) = O(\kappa(z)^2) + O(\kappa'(z)^2) + O(\kappa''(z)^2)$ as $z \to +\infty$.
\end{proof}

\begin{proof}[Proof of Lemma~\ref{lem integrability F}]
First, note that $F$ is well-defined and continuous on $\R \setminus \{0\}$. Indeed, $\rho_2$ is continuous on $\R^2 \setminus \Delta_2$ by Lemma~\ref{lem Dk and Nk continuous}. By Lemma~\ref{lem integrability F 0}, $F(z) \xrightarrow[z \to 0]{}-\frac{1}{\pi^2}$. In particular, $F$ is integrable near $0$. Since $F$ is even and $\kappa(z)$, $\kappa'(z)$ and $\kappa''(z)$ tend to $0$ as $z \to +\infty$, we have $F(z) \xrightarrow[\norm{z} \to +\infty]{}0$, by Lemma~\ref{lem integrability F infty}. 

Under the hypotheses of Proposition~\ref{prop variance}, both $\kappa$ and $\kappa''$ are square-integrable. By an integration by parts and the Cauchy-Schwarz Inequality, this implies that $\kappa'$ is also square-integrable. Finally, $F$ is integrable at infinity by Lemma~\ref{lem integrability F infty}.
\end{proof}


\section{A Gaussian lemma}
\label{sec a Gaussian lemma}

In this appendix, we prove an estimate that we used in our study of the Kac--Rice numerators $N_\I$ (see Definition~\ref{def Kac-Rice densities partition}). More precisely, Corollary~\ref{cor Pi k} below is used in the proofs of Lemmas~\ref{lem DI and NI continuous} and~\ref{lem clustering NI add}.

Let $k \in \N^*$, we denote by $\sym_k(\R)$ the space of symmetric matrices of size $k$ with real coefficients and by $\sym_k^+(\R) \subset \sym_k(\R)$ the subset of positive semi-definite matrices. We equip $\sym_k(\R)$ with the sup-norm $\Norm{\cdot}_\infty$, see Notation~\ref{ntn norm sup matrix}.

\begin{dfn}
\label{def Pi k}
Let $U \in \sym_k^+(\R)$ and let $(X_1,\dots,X_k) \sim \mathcal{N}(0,U)$, we denote by
\begin{equation*}
\Pi_k(U) = \esp{\prod_{i=1}^k \norm{X_i}}.
\end{equation*}
\end{dfn}

\begin{lem}
\label{lem Pi k}
Let $k \in \N^*$, there exists $C_k >0$ such that, for all $U$ and $V \in \sym_k^+(\R)$:
\begin{equation*}
\norm{\Pi_k(V) - \Pi_k(U)} \leq C_k \Norm{V-U}_\infty^\frac{1}{2} \left(\max(\Norm{U}_\infty,\Norm{V}_\infty)\right)^\frac{k-1}{2}.
\end{equation*}
\end{lem}

\begin{proof}
Let $(X_1,\dots,X_k) \sim \mathcal{N}(0,U)$ and $(Y_1,\dots,Y_k) \sim \mathcal{N}(0,V)$ be centered Gaussian vectors in~$\R^k$ of variance matrices $U$ and $V$, respectively.

Let us first assume that $V-U \in \sym_k^+(\R)$ and let $T=(T_1,\dots,T_k) \sim \mathcal{N}(0,V-U)$ be independent of $X$. In this case $X+T \sim \mathcal{N}(0,V)$, and we can assume that $Y=X+T$ without loss of generality. Then,
\begin{align*}
\norm{\Pi_k(V)-\Pi_k(U)} &= \norm{\esp{\prod_{i=1}^k \norm{Y_i}}-\esp{\prod_{i=1}^k \norm{X_i}}}\\
&\leq \esp{\norm{\prod_{i=1}^k
Y_i - \prod_{i=1}^k X_i}}\\
&\leq \sum_{j=1}^k \esp{\left(\prod_{i=1}^{j-1}\norm{Y_i}\right)\norm{Y_j-X_j}\left(\prod_{i=j+1}^k \norm{X_i}\right)}\\
&\leq \sum_{j=1}^k \prod_{i=1}^{j-1} \left(\esp{\norm{Y_i}^k}^\frac{1}{k}\right) \esp{\norm{T_j}^k}^\frac{1}{k} \left(\prod_{i=j+1}^k \esp{\norm{X_i}^k}^\frac{1}{k}\right),
\end{align*}
where we obtained the last line by applying Hölder's Inequality. Now, for all $i \in \{1,\dots,k\}$, we have:
\begin{equation*}
\esp{\norm{X_i}^k} \leq \esp{(X_i)^{2k}}^\frac{1}{2} = (\mu_{2k})^\frac{1}{2} \var{X_i}^\frac{k}{2} \leq (\mu_{2k})^\frac{1}{2} \Norm{U}_\infty^\frac{k}{2},
\end{equation*}
where $\mu_{2k}$ stands for the $2k$-th moment of an $\mathcal{N}(0,1)$ real variable, as in Notation~\ref{ntn mu p}. Similarly, for all $i \in \{1,\dots,k\}$, we have $\esp{\norm{Y_i}^k}\leq (\mu_{2k})^\frac{1}{2} \Norm{V}_\infty^\frac{k}{2}$ and $\esp{\norm{T_i}^k} \leq (\mu_{2k})^\frac{1}{2} \Norm{V-U}_\infty^\frac{k}{2}$. Hence,
\begin{equation}
\label{eq Pi k}
\norm{\Pi_k(V)-\Pi_k(U)} \leq k (\mu_{2k})^\frac{1}{2} \Norm{V-U}_\infty^\frac{1}{2} \left(\max(\Norm{U}_\infty,\Norm{V}_\infty)\right)^\frac{k-1}{2}.
\end{equation}
This concludes the proof in the special case where $V-U$ is positive semi-definite.

Let us now consider the general case and let us denote by $\epsilon = \Norm{V-U}_\infty$. Let $\Id_k$ denote the identity matrix of size $k$ and let $W = U + k\epsilon \Id_k$. Then, $W-U = k \epsilon \Id_k \in \sym_k^+(\R)$. Moreover, we have $W-V = k\epsilon \Id_k + U-V$. Since for all $x= (x_i)_{1 \leq i \leq k} \in \R^k$, we have
\begin{equation*}
\norm{\trans{x} (U-V) x} \leq \epsilon \sum_{1 \leq i,j \leq k} \norm{x_i}\norm{x_j} = \epsilon \left(\sum_{i=1}^k \norm{x_i}\right)^2 \leq k \epsilon \sum_{i=1}^k x_i^2,
\end{equation*}
the matrix $W-V$ is also positive semi-definite. Let $Z =(Z_1,\dots,Z_k) \sim \mathcal{N}(0,W)$, using Equation~\eqref{eq Pi k}, we obtain:
\begin{multline*}
\norm{\Pi_k(V) - \Pi_k(U)} \leq \norm{\Pi_k(W) - \Pi_k(V)} + \norm{\Pi_k(W) - \Pi_k(U)}\\
\leq k (\mu_{2k})^\frac{1}{2} \left(\Norm{W-U}_\infty^\frac{1}{2} +\Norm{W-V}_\infty^\frac{1}{2}\right) \left(\max(\Norm{U}_\infty,\Norm{V}_\infty,\Norm{W}_\infty)\right)^\frac{k-1}{2}.
\end{multline*}
We know that $\Norm{W-V}_\infty \leq \Norm{W-U}_\infty + \Norm{V-U}_\infty$. Hence, by definition of $W$ and $\epsilon$, we get $\Norm{W-U}_\infty = k\epsilon = k \Norm{V-U}_\infty$ and $\Norm{W-V}_\infty \leq (k+1)\Norm{V-U}_\infty$. Moreover,
\begin{equation*}
\Norm{W}_\infty \leq \Norm{U}_\infty + \Norm{W-U}_\infty = \Norm{U}_\infty + k \Norm{V-U}_\infty \leq (2k+1) \max(\Norm{U}_\infty,\Norm{V}_\infty). 
\end{equation*}
Finally, setting $C_k = k(2k+1)^\frac{k+1}{2}(\mu_{2k})^\frac{1}{2}$, we have:
\begin{equation*}
\norm{\Pi_k(V) - \Pi_k(U)} \leq C_k \Norm{V-U}_\infty^\frac{1}{2} \left(\max(\Norm{U}_\infty,\Norm{V}_\infty)\right)^\frac{k-1}{2}. \qedhere
\end{equation*}
\end{proof}

\begin{cor}[Regularity]
\label{cor Pi k}
Let $k \in \N^*$, the map $\Pi_k:\sym_k^+(\R) \to \R$ defined by Definition~\ref{def Pi k} is $\frac{1}{2}$-Hölder on compact subsets of $\sym_k^+(\R)$, for the sup-norm $\Norm{\cdot}_\infty$. In particular, the map $\Pi_k$ is continuous.
\end{cor}

\begin{proof}
Let $K \subset \sym_k^+(\R)$ be a compact subset and let $M = \max_{U \in K} \Norm{U}_\infty$. By Lemma~\ref{lem Pi k}, for all $U,V \in K$, we have $\norm{\Pi_k(V) -\Pi_k(U)} \leq C_k M^\frac{k-1}{2} \Norm{V-U}_\infty^\frac{1}{2}$. Thus, $\Pi_k$ is $\frac{1}{2}$-Hölder on compact subsets of $\sym_k^+(\R)$, hence continuous.
\end{proof}


\paragraph{Acknowledgments.} Thomas Letendre thanks Julien Fageot for useful discussions about Fernique's Theorem, Benoit Laslier for his help in the proof of Lemma~\ref{lem Pi k} and Hugo Vanneuville for pointing out the relation between ergodicity and decay of correlations. The authors are grateful to Damien Gayet for suggesting they write this paper in the first place, and to Misha Sodin for bringing to their attention the intrinsic interest of clustering properties for $k$-point functions. They also thank Jean-Yves Welschinger for his support and Louis Gass for spotting an error in an earlier version of Lemma~\ref{lem partitions into pairs}. Finally, the authors thank the anonymous referee for their careful reading of the paper, and their comments that helped to improve the exposition of the main results.

\bibliographystyle{amsplain}
\bibliography{ZerosGaussianProcesses}

\providecommand{\bysame}{\leavevmode\hbox to3em{\hrulefill}\thinspace}
\providecommand{\MR}{\relax\ifhmode\unskip\space\fi MR }
\providecommand{\MRhref}[2]{%
  \href{http://www.ams.org/mathscinet-getitem?mr=#1}{#2}
}
\providecommand{\href}[2]{#2}
\begin{thebibliography}{10}

\bibitem{Adl2010}
R.~J. Adler, \emph{The geometry of random fields}, Classics in Applied
  Mathematics, vol.~62, Society for Industrial and Applied Mathematics (SIAM),
  Philadelphia, PA, 2010, Reprint of the 1981 original.

\bibitem{AT2007}
R.~J. Adler and J.~E. Taylor, \emph{Random fields and geometry}, 1st ed.,
  Monographs in Mathematics, Springer, New York, 2007.

\bibitem{Anc2019}
M.~Ancona, \emph{Random sections of line bundles over real {R}iemann surfaces},
  Int. Math. Res. Not. IMRN (2021), no.~9, 7004–7059.

\bibitem{AL2019}
M.~Ancona and T.~Letendre, \emph{Roots of {K}ostlan polynomials: moments,
  strong {L}aw of {L}arge {N}umbers and {C}entral {L}imit {T}heorem}, Ann. H.
  Lebesgue (2021), to appear, arXiv: 1911.12182.

\bibitem{AAD+2019}
D.~Armentano, J.-M. Azaïs, F.~Dalmao, J.~R. Leòn, and E.~Mordecki, \emph{On
  the finiteness of the moments of the measure of level sets of random fields},
  arXiv: 1909.10243 (2019).

\bibitem{AW2009}
J.-M. Azaïs and M.~Wschebor, \emph{Level sets and extrema of random processes
  and fields}, John Wiley \& Sons, Hoboken, NJ, 2009.

\bibitem{BDFZ2020}
R.~Basu, A.~Dembo, N.~Feldheim, and O.~Zeitouni, \emph{Exponential
  concentration for zeroes of stationary {G}aussian processes}, Int. Math. Res.
  Not. IMRN (2020), no.~23, 9769--9796.

\bibitem{BDW2018}
H.~Biermé, O.~Durieu, and Y.~Wang, \emph{Generalized random fields and
  {L}\'{e}vy's continuity theorem on the space of tempered distributions},
  Commun. Stoch. Anal. \textbf{12} (2018), no.~4, 427--445.

\bibitem{Bil1995}
P.~Billingsley, \emph{Probability and measure}, 3rd ed., Wiley Series in
  Probability and Mathematical Statistics, John Wiley \& Sons, New York, 1995.

\bibitem{BYY2019}
B.~B{\l}aszczyszyn, D.~Yogeshwaran, and J.~E. Yukich, \emph{Limit theory for
  geometric statistics of point processes having fast decay of correlations},
  Ann. Probab. \textbf{47} (2019), no.~2, 835--895.

\bibitem{BD1997}
P.~Bleher and X.~Di, \emph{Correlations between zeros of a random polynomial},
  J. Statist. Phys. \textbf{88} (1997), no.~1--2, 269--305.

\bibitem{BSZ2000}
P.~Bleher, B.~Shiffman, and S.~Zelditch, \emph{Universality and scaling of
  correlations between zeros on complex manifolds}, Invent. Math. \textbf{142}
  (2000), no.~2, 351--395.

\bibitem{CL1965}
H.~Cramér and M.~R. Leadbetter, \emph{The moments of the number of crossings
  of a level by a stationary normal process}, Ann. Math. Statist. \textbf{36}
  (1965), 1656--1663.

\bibitem{Cuz1975}
J.~Cuzick, \emph{Conditions for finite moments of the number of zero crossings
  for {G}aussian processes}, Ann. Probability \textbf{3} (1975), no.~5,
  849--858.

\bibitem{Cuz1976}
\bysame, \emph{A central limit theorem for the number of zeros of a stationary
  {G}aussian process}, Ann. Probability \textbf{4} (1976), no.~4, 547--556.

\bibitem{Cuz1978}
\bysame, \emph{Local nondeterminism and the zeros of {G}aussian processes},
  Ann. Probability \textbf{6} (1978), no.~1, 72--84.

\bibitem{Dal2015}
F.~Dalmao, \emph{Asymptotic variance and {CLT} for the number of zeros of
  {K}ostlan {S}hub {S}male random polynomials}, C. R. Math. Acad. Sci. Paris
  \textbf{353} (2015), no.~12, 1141--1145.

\bibitem{DV2020}
Y.~Do and V.~Vu, \emph{Central limit theorems for the real zeros of {W}eyl
  polynomials}, Amer. J. Math. \textbf{142} (2020), no.~5, 1327--1369.

\bibitem{Dub2020}
F.~Dubeau, \emph{On {H}ermite interpolation and divided differences}, Surv.
  Math. Appl. \textbf{15} (2020), 257--279.

\bibitem{Fer1967}
X.~Fernique, \emph{Processus linéaires, processus généralisés}, Ann. Inst.
  Fourier \textbf{17} (1967), no.~1, 1--92.

\bibitem{GW2011}
D.~Gayet and J.-Y. Welschinger, \emph{Exponential rarefaction of real curves
  with many components}, Publ. Math. Inst. Hautes Études Sci. (2011), no.~113,
  69--96.

\bibitem{Gem1972}
D.~Geman, \emph{On the variance of the number of zeros of a stationary
  {G}aussian process}, Ann. Math. Statist. \textbf{43} (1972), 977--982.

\bibitem{Kac1943}
M.~Kac, \emph{On the average number of real roots of a random algebraic
  equation}, Bull. Amer. Math. Soc. \textbf{49} (1943), 314--320.

\bibitem{Kos1993}
E.~Kostlan, \emph{On the distribution of roots of random polynomials}, From
  {T}opology to {C}omputation: {P}roceedings of the {S}malefest ({B}erkeley,
  {CA}, 1990), Springer, New York, 1993, pp.~419--431.

\bibitem{Kra2006}
M.~F. Kratz, \emph{Level crossings and other level functionals of stationary
  {G}aussian processes}, Probab. Surv. \textbf{3} (2006), 230--288.

\bibitem{KL1997}
M.~F. Kratz and J.~R. Leòn, \emph{Hermite polynomial expansion for non-smooth
  functionals of stationary {G}aussian processes: crossings and extremes},
  Stochastic Process. Appl. \textbf{66} (1997), no.~2, 237--252.

\bibitem{KL2001}
\bysame, \emph{Central limit theorems for level functionals of stationary
  {G}aussian processes and fields}, J. Theoret. Probab. \textbf{14} (2001),
  no.~3, 639--672.

\bibitem{Lac2020}
R.~Lachièze-Rey, \emph{Variance linearity for real {G}aussian zeros}, arXiv:
  2006.10341 (2020).

\bibitem{LP2019}
T.~Letendre and M.~Puchol, \emph{Variance of the volume of random real
  algebraic submanifolds {II}}, Indiana Univ. Math. J. \textbf{68} (2019),
  no.~6, 1649--1720.

\bibitem{NS2012}
F.~Nazarov and M.~Sodin, \emph{Correlation functions for random complex zeroes:
  strong clustering and local universality}, Comm. Math. Phys. \textbf{310}
  (2012), no.~1, 75--98.

\bibitem{NS2016}
\bysame, \emph{Asymptotic laws for the spatial distribution and the number of
  connected components of zero sets of {G}aussian random functions}, Zh. Mat.
  Fiz. Anal. Geom. \textbf{12} (2016), no.~3, 205--278.

\bibitem{Pit1996}
V.~I. Piterbarg, \emph{Asymptotic methods in the theory of {G}aussian processes
  and fields}, Translations of Mathematical Monographs, vol. 148, American
  Mathematical Society,, Providence, RI, 1996, Translated from the Russian by
  V. V. Piterbarg.

\bibitem{Ric1944}
S.~O. Rice, \emph{Mathematical analysis of random noise}, Bell System Tech. J.
  \textbf{23} (1944), 282--332.

\bibitem{SZ1999}
B.~Shiffman and S.~Zelditch, \emph{Distribution of zeros of random and quantum
  chaotic sections of positive line bundles}, Comm. Math. Phys. \textbf{200}
  (1999), no.~3, 661--683.

\bibitem{Ylv1965}
D.~N. Ylvisaker, \emph{The expected number of zeros of a stationary {G}aussian
  process}, Ann. Math. Statist. \textbf{36} (1965), 1043--1046.

\end{thebibliography}

\end{document}